%% file: rumin.tex
\documentclass{amsart}
\input{header}

\begin{document}

\title{The bigraded Rumin complex via differential forms}
\author{Jeffrey S. Case}
\address{Department of Mathematics \\ Penn State University \\ University Park, PA 16802 \\ USA}
\email{jscase@psu.edu}
\begin{abstract}
 We give a new CR invariant treatment of the bigraded Rumin complex and related cohomology groups via differential forms.
 A key benefit is the identification of balanced $A_\infty$-structures on the Rumin and bigraded Rumin complexes.
 We also prove related Hodge decomposition theorems.
 Among many applications, we give a sharp upper bound on the dimension of the Kohn--Rossi groups $H^{0,q}(M^{2n+1})$, $1\leq q\leq n-1$, of a closed strictly pseudoconvex manifold with a contact form of nonnegative pseudohermitian Ricci curvature;
 we prove a sharp CR analogue of the Fr\"olicher inequalities in terms of the second page of a natural spectral sequence;
 we give new proofs of selected topological properties of closed Sasakian manifolds;
 and we generalize the Lee class $\mathcal{L}\in H^1(M;\mathscr{P})$ --- whose vanishing is necessary and sufficient for the existence of a pseudo-Einstein contact form --- to all nondegenerate orientable CR manifolds.
\end{abstract}
\maketitle

\tableofcontents

\input{part-background}
\input{part-intrinsic}
\input{part-hodge}
\input{part-applications}

\subsection*{Acknowledgements}
I would like to thank Yuya Takeuchi for helpful discussions, including the argument in the second paragraph of \cref{sasakian-three-manifold};
Peter Garfield for providing me a copy of his Ph.D.\ thesis~\cite{Garfield2001};
Jack Lee for encouragement and discussions about his work~\cite{GarfieldLee1998} and Garfield's thesis;
the anonymous referees for pointing out the likelihood that the Rumin and bigraded Rumin complexes carry $A_\infty$-structures and for pointing out a number of ways to improve the exposition;
and Alex Tao for pointing out many typos in an early version of this paper.
I also would like to thank the University of Washington for providing a productive research environment while I was completing this work.

This work was partially supported by the Simons Foundation (Grant \#524601).

\bibliography{bib}
\end{document}

%% file: header.tex
\usepackage{amsmath}
\usepackage{amssymb}
\usepackage{amsthm}
\usepackage[msc-links,abbrev]{amsrefs}
\usepackage{hyperref}
\usepackage[noabbrev,capitalize]{cleveref}
\usepackage{mathrsfs}
\usepackage{mathtools}
\usepackage{slashed}
\usepackage{tikz-cd}
\usepackage{enumitem}

\setenumerate{label=(\roman*)}

\DeclareMathOperator{\im}{im}

\DeclareMathOperator{\Ric}{Ric}

\DeclareMathOperator{\End}{End}

\DeclareMathOperator{\hash}{\sharp}
\DeclareMathOperator{\ohash}{\overline{\sharp}}
\DeclareMathOperator{\ddbarhash}{\sharp\overline{\sharp}}

\DeclareMathOperator{\sgn}{sgn}

\DeclareMathOperator{\Imaginary}{Im}
\DeclareMathOperator{\Real}{Re}

\DeclareMathOperator{\contr}{\lrcorner}
\DeclareMathOperator{\lcpcontr}{lcpcontr}

\DeclareMathOperator{\rwedge}{\diamond}
\DeclareMathOperator{\krwedge}{\bullet}
\DeclareMathOperator{\hodge}{\star}
\DeclareMathOperator{\ohodge}{\overline{\star}}
\DeclareMathOperator{\cBoxb}{\widetilde{\Box}_b}

\DeclareMathOperator{\cuplength}{cup}
\DeclareMathOperator{\Lef}{Lef}
\DeclareMathOperator{\chern}{ch}

\newcommand{\defn}[1]{{\boldmath\bfseries#1}}

\newcommand{\Alpha}{A}
\newcommand{\Beta}{B}

\newcommand{\DiffOperator}[1]{\mathrm{DO}(#1)}
\newcommand{\PseudoDiffOperator}[1]{\mathrm{\Psi DO}(#1)}

\newcommand{\trace}{\Lambda}

\newcommand{\CH}{H_{\mathbb{C}}}

\newcommand{\CT}{T_{\mathbb{C}}}
\newcommand{\CTM}{T_{\mathbb{C}}M}

\newcommand{\COmega}{\Omega_{\mathbb{C}}}
\newcommand{\CmR}{\mathcal{R}_{\mathbb{C}}}
\newcommand{\CsA}{\mathscr{A}_{\mathbb{C}}}

\newcommand{\CsR}{\mathscr{R}_{\mathbb{C}}}

\newcommand{\ow}{\overline{w}}
\newcommand{\oz}{\overline{z}}

\newcommand{\otau}{\overline{\tau}}
\newcommand{\oomega}{\overline{\omega}}

\newcommand{\oBox}{\overline{\Box}}

\newcommand{\cu}{\widetilde{u}}

\newcommand{\cH}{\widetilde{H}}

\newcommand{\cL}{\widetilde{L}}

\newcommand{\cN}{\widetilde{N}}

\newcommand{\cQ}{\widetilde{Q}}

\newcommand{\cS}{\widetilde{S}}

\newcommand{\ctau}{\widetilde{\tau}}
\newcommand{\comega}{\widetilde{\omega}}

\newcommand{\cmH}{\widetilde{\mathcal{H}}}

\newcommand{\hA}{\widehat{A}}
\newcommand{\hE}{\widehat{E}}

\newcommand{\hP}{\widehat{P}}

\newcommand{\hR}{\widehat{R}}
\newcommand{\hS}{\widehat{S}}
\newcommand{\hT}{\widehat{T}}
\newcommand{\hW}{\widehat{W}}

\newcommand{\hOmega}{\widehat{\Omega}}
\newcommand{\hnabla}{\widehat{\nabla}}

\newcommand{\htheta}{\widehat{\theta}}
\newcommand{\hiota}{\widehat{\iota}}

\newcommand{\htau}{\widehat{\tau}}

\newcommand{\homega}{\widehat{\omega}}

\newcommand{\hmH}{\widehat{\mathcal{H}}}

\newcommand{\lp}{\langle}
\newcommand{\rp}{\rangle}
\newcommand{\lv}{\lvert}
\newcommand{\rv}{\rvert}
\newcommand{\lV}{\lVert}
\newcommand{\rV}{\rVert}
\newcommand{\llp}{\lp\!\lp}
\newcommand{\rrp}{\rp\!\rp}
\newcommand{\vertiii}[1]{{\left\vert\kern-0.25ex\left\vert\kern-0.25ex\left\vert #1 
    \right\vert\kern-0.25ex\right\vert\kern-0.25ex\right\vert}}

\newcommand{\db}{\partial_b}
\newcommand{\dbbar}{\overline{\partial}_b}
\newcommand{\dhor}{\partial_0}
\newcommand{\nablab}{\nabla_b}
\newcommand{\nablabbar}{\overline{\nabla}_b}

\newcommand{\mB}{\mathcal{B}}
\newcommand{\mC}{\mathcal{C}}

\newcommand{\mH}{\mathcal{H}}

\newcommand{\mL}{\mathcal{L}}

\newcommand{\mP}{\mathcal{P}}

\newcommand{\mR}{\mathcal{R}}
\newcommand{\mS}{\mathcal{S}}

\newcommand{\bC}{\mathbb{C}}

\newcommand{\bH}{\mathbb{H}}

\newcommand{\bN}{\mathbb{N}}

\newcommand{\bR}{\mathbb{R}}

\newcommand{\bZ}{\mathbb{Z}}

\newcommand{\sA}{\mathscr{A}}

\newcommand{\sI}{\mathscr{I}}
\newcommand{\sJ}{\mathscr{J}}

\newcommand{\sO}{\mathscr{O}}
\newcommand{\sP}{\mathscr{P}}

\newcommand{\sR}{\mathscr{R}}
\newcommand{\sS}{\mathscr{S}}

\newcommand{\nablas}{\slashed{\nabla}}
\newcommand{\onablas}{\overline{\nablas}}

\DeclareMathOperator{\Boxs}{\slashed{\Box}}

\def\sideremark#1{\ifvmode\leavevmode\fi\vadjust{\vbox to0pt{\vss
 \hbox to 0pt{\hskip\hsize\hskip1em
 \vbox{\hsize3cm\tiny\raggedright\pretolerance10000
 \noindent #1\hfill}\hss}\vbox to8pt{\vfil}\vss}}}

\newcommand{\suchthatcolon}{\mathrel{}:\mathrel{}}

\newtheorem{theorem}{Theorem}[section]
\newtheorem{proposition}[theorem]{Proposition}
\newtheorem{lemma}[theorem]{Lemma}
\newtheorem{corollary}[theorem]{Corollary}

\theoremstyle{definition}
\newtheorem{definition}[theorem]{Definition}

\newtheorem{conjecture}[theorem]{Conjecture}
\newtheorem{example}[theorem]{Example}

\theoremstyle{remark}
\newtheorem{remark}[theorem]{Remark}

\numberwithin{equation}{section}

%% file: part-background.tex
\part{Introduction and background}
\label{part:background}

The de Rham and Dolbeault cohomology algebras are two important invariants associated to a complex manifold.
Many of the deep and important links between these two algebras begin with the Dolbeault double complex and pass through the Hodge decomposition theorem.
It is natural to ask whether CR manifolds --- abstract analogues of boundaries of complex manifolds --- admit analogues of the de Rham and Dolbeault cohomology algebras.
Addressing this fundamental question has led to many important developments:

CR manifolds are special cases of contact manifolds.
The \defn{Rumin complex}~\cite{Rumin1990}
\[ 0 \longrightarrow \mR^0 \overset{d}{\longrightarrow} \mR^1 \overset{d}{\longrightarrow} \dotsm \overset{d}{\longrightarrow} \mR^{2n+1} \overset{d}{\longrightarrow} 0 \]
is a contact invariant refinement of the de Rham complex on a $(2n+1)$-dimensional contact manifold.
In Rumin's definition, if $k\leq n$, then $\mR^k$ is a space of equivalence classes of differential $k$-forms;
if $k\geq n+1$, then $\mR^k$ is a subspace of the space of differential $k$-forms.
In both cases, the operator $d$ is induced by the exterior derivative.
A key property of the Rumin complex is that the cohomology groups
\begin{equation*}
 H_R^k(M;\bR) := \frac{\ker \left( d \colon \mR^k \to \mR^{k+1} \right)}{\im \left( d \colon \mR^{k-1} \to \mR^k \right)}
\end{equation*}
are isomorphic to the de Rham cohomology groups~\cite{Rumin1990}.
However, Rumin's construction does not recover the de Rham cohomology \emph{algebra}.
Note also that, in contrast with the de Rham complex, $d\colon\mR^n\to\mR^{n+1}$ is a second-order operator, though $d\colon\mR^k\to\mR^{k+1}$, $k\not=n$, is a first-order operator.
Nevertheless, Rumin proved~\cite{Rumin1994} a Hodge decomposition theorem using these operators and used it to derive some cohomological vanishing theorems.

The \defn{Kohn--Rossi complex}~\cite{KohnRossi1965}
\[ 0 \longrightarrow \mR^{p,0} \overset{\dbbar}{\longrightarrow} \mR^{p,1} \overset{\dbbar}{\longrightarrow} \dotsm \overset{\dbbar}{\longrightarrow} \mR^{p,n} \overset{\dbbar}{\longrightarrow} 0 \]
is a CR invariant complex on a $(2n+1)$-dimensional CR manifold.
Kohn and Rossi~\cite{KohnRossi1965} introduced this complex to study Hodge theory for the Dolbeault complex on complex manifolds with boundary.
Kohn and his collaborators~\cites{KohnRossi1965,FollandKohn1972,Kohn1965} also proved a Hodge decomposition theorem on a closed, strictly pseudoconvex pseudohermitian $(2n+1)$-manifold which implies that if $q\not\in\{0,n\}$, then the \defn{Kohn--Rossi cohomology group}
\[ H^{p,q}(M) := \frac{\ker \left( \dbbar \colon \mR^{p,q} \to \mR^{p,q+1} \right)}{\im \left( \dbbar \colon \mR^{p,q-1} \to \mR^{p,q} \right)} \]
is finite-dimensional.
Tanaka~\cite{Tanaka1975} and Lee~\cite{Lee1988} used this Hodge theorem to derive some vanishing results for the groups $H^{0,q}(M)$.
The assumption $q\not\in\{0,n\}$ is essential:
If $(M^{2n+1},T^{1,0})$ is embeddable in $\bC^N$ for some $N$, then the space $H^{p,0}(M)$ of CR holomorphic $p$-forms is infinite-dimensional.
Notably, the extrinsic description~\cites{KohnRossi1965,FollandKohn1972,ChenShaw2001} of $\mR^{p,q}$ as the space of tangential $(p,q)$-forms requires either imposing an hermitian metric on the complex manifold or working with quotient spaces.
The intrinsic description~\cites{KohnRossi1965,Tanaka1975} requires either a pseudohermitian structure or working with sections of a bundle distinct from the exterior bundle $\Lambda^\bullet T^\ast M$.
In particular, early definitions of the space $\mR^{p,q}$ are generally \emph{not} obtained from the complexification of the space $\mR^{p+q}$ in the Rumin complex.
To the best of our knowledge, the Kohn--Rossi cohomology groups have not been given the structure of an algebra.

The \defn{bigraded Rumin complex}~\cites{Garfield2001,GarfieldLee1998}
\begin{equation*}
 \begin{tikzpicture}[baseline=(current bounding box.center),yscale=0.8,xscale=1.0]
  \node (0_0) at (0,0) {$\mR^{0,0}$};
  \node (1_0) at (1,1) {$\mR^{1,0}$};
  \node (1_1) at (2,0) {$\mR^{1,1}$};
  \node (0_1) at (1,-1) {$\mR^{0,1}$};
  \node (n-1_0) at (3,3) {$\mR^{n-1,0}$};
  \node (n_0) at (4,4) {$\mR^{n,0}$};
  \node (n-1_1) at (4,2) {$\mR^{n-1,1}$};
  \node (n+1_1) at (7,3) {$\mR^{n+1,1}$};
  \node (n+1_0) at (6,4) {$\mR^{n+1,0}$};
  \node (n_1) at (6,2) {$\mR^{n,1}$};
  \node (0_n-1) at (3,-3) {$\mR^{0,n-1}$};
  \node (2_n) at (7,-3) {$\mR^{2,n}$};
  \node (2_n-1) at (6,-2) {$\mR^{2,n-1}$};
  \node (1_n) at (6,-4) {$\mR^{1,n}$};
  \node (0_n) at (4,-4) {$\mR^{0,n}$};
  \node (1_n-1) at (4,-2) {$\mR^{1,n-1}$};
  \node (n_n-1) at (8,0) {$\mR^{n,n-1}$};
  \node (n+1_n) at (10,0) {$\mR^{n+1,n}$};
  \node (n+1_n-1) at (9,1) {$\mR^{n+1,n-1}$};
  \node (n_n) at (9,-1) {$\mR^{n,n}$};
  \draw [->] (0_0) to node[pos=0.8, left] {\tiny $\db$} (1_0);
  \draw [->] (1_0) to node[pos=0.8, left] {\tiny $\db$} (1.7,1.7) ;
  \draw [loosely dotted, line width=0.3mm] (1.9,1.9) -- (2.2,2.2);
  \draw [->] (2.3,2.3) to node[pos=0.8, left] {\tiny $\db$} (n-1_0);
  \draw [->] (n-1_0) to node[pos=0.8, left] {\tiny $\db$} (n_0);
  \draw [->] (0_1) to node[pos=0.8, left] {\tiny $\db$} (1_1);
  \draw [->] (1_1) to node[pos=0.8, left] {\tiny $\db$} (2.7, 0.7) ;
  \draw [loosely dotted, line width=0.3mm] (2.9,0.9) -- (3.2, 1.2);
  \draw [->] (3.3, 1.3) to node[pos=0.8, left] {\tiny $\db$} (n-1_1);
  \draw [->] (n-1_1) to node[pos=0.2, left] {\tiny $\db$} (n+1_0);
  \draw [loosely dotted, line width=0.3mm] (1.9,-1.9) -- (2.2,-2.2);
  \draw [->] (5.5,1.5) -- (n_1);
  \draw [->] (n_1) to node[pos=0.8, left] {\tiny $\db$} (n+1_1);
  \draw [->] (0_n-1) to node[pos=0.8, left] {\tiny $\db$} (1_n-1);
  \draw [->] (1_n-1) -- (4.5,-1.5);
  \draw [loosely dotted, line width=0.3mm] (4.6,-1.4) -- (4.9,-1.1);
  \draw [->] (0_n) to node[pos=0.2, left] {\tiny $\db$} (2_n-1);
  \draw [->] (2_n-1) to node[pos=0.8, left] {\tiny $\db$} (6.7,-1.3);
  \draw [loosely dotted, line width=0.3mm] (6.9,-1.1) -- (7.2,-0.8);
  \draw [->] (7.3,-0.7) to node[pos=0.8, left] {\tiny $\db$} (n_n-1);
  \draw [->] (n_n-1) to node[pos=0.8, left] {\tiny $\db$} (n+1_n-1);
  \draw [->] (1_n) to node[pos=0.8, left] {\tiny $\db$} (2_n);
  \draw [->] (2_n) to node[pos=0.8, left] {\tiny $\db$} (7.7,-2.3);
  \draw [loosely dotted, line width=0.3mm] (7.9,-2.1) -- (8.2,-1.8);
  \draw [->] (8.3,-1.7) to node[pos=0.8, left] {\tiny $\db$} (n_n);
  \draw [->] (n_n) to node[pos=0.8, left] {\tiny $\db$} (n+1_n);
  \draw [->] (0_0) to node[left] {\tiny $\dbbar$} (0_1);
  \draw [->] (0_1) to node[left] {\tiny $\dbbar$} (1.7,-1.7);
  \draw [loosely dotted, line width=0.3mm] (5.15,1.15) -- (5.5,1.5);
  \draw [->] (2.3,-2.3) to node[left] {\tiny $\dbbar$} (0_n-1);
  \draw [->] (0_n-1) to node[left] {\tiny $\dbbar$} (0_n);
  \draw [->] (1_0) to node[left] {\tiny $\dbbar$} (1_1);
  \draw [->] (1_1) to node[left] {\tiny $\dbbar$} (2.7,-0.7);
  \draw [loosely dotted, line width=0.3mm] (2.9,-0.9) -- (3.2,-1.2);
  \draw [->] (3.3,-1.3) to node[left] {\tiny $\dbbar$} (1_n-1);
  \draw [->] (1_n-1) to node[pos=0.2, left] {\tiny $\dbbar$} (1_n);
  \draw [->] (5.5,-1.5) -- (2_n-1);
  \draw [loosely dotted, line width=0.3mm] (5.15,-1.15) -- (5.5,-1.5);
  \draw [->] (2_n-1) to node[left] {\tiny $\dbbar$} (2_n);
  \draw [->] (n-1_0) to node[pos=0.2, left] {\tiny $\dbbar$} (n-1_1);
  \draw [->] (n-1_1) -- (4.5,1.5);
  \draw [loosely dotted, line width=0.3mm] (4.6,1.4) -- (4.9,1.1) ;
  \draw [->] (n_0) to node[pos=0.2, left] {\tiny $\dbbar$} (n_1);
  \draw [->] (n_1) to node[left] {\tiny $\dbbar$} (6.7,1.3);
  \draw [loosely dotted, line width=0.3mm] (6.9,1.1) -- (7.2,0.8);
  \draw [->] (7.3,0.7) to node[left] {\tiny $\dbbar$} (n_n-1);
  \draw [->] (n_n-1) to node[left] {\tiny $\dbbar$} (n_n);
  \draw [->] (n+1_0) to node[left] {\tiny $\dbbar$} (n+1_1);
  \draw [->] (n+1_1) to node[left] {\tiny $\dbbar$} (7.7,2.3);
  \draw [loosely dotted, line width=0.3mm] (7.9,2.1) -- (8.2,1.8);
  \draw [->] (8.3,1.7) to node[left] {\tiny $\dbbar$} (n+1_n-1);
  \draw [->] (n+1_n-1) to node[left] {\tiny $\dbbar$} (n+1_n);
  \draw [->] (n_0) to node[above] {\tiny $\dhor$} (n+1_0);
  \draw [->] (n-1_1) to node[above] {\tiny $\dhor$} (n_1);
  \draw [->] (1_n-1) to node[above] {\tiny $\dhor$} (2_n-1);
  \draw [->] (0_n) to node[above] {\tiny $\dhor$} (1_n);
 \end{tikzpicture}
\end{equation*}
is a CR invariant bigraded complex on a $(2n+1)$-dimensional CR manifold.
This complex, constructed by Garfield and Lee~\cites{Garfield2001,GarfieldLee1998}, is obtained by using the CR structure to induce a bigrading on the Rumin complex.
However, it is not a double complex:
The second-order operator $d\colon\mR^n\to\mR^{n+1}$ splits into a sum of \emph{three} operators.
One feature of the bigraded Rumin complex is that its downward diagonals are equivalent to the Kohn--Rossi complexes.
In particular, the first and limiting pages of the \defn{Garfield spectral sequence}~\cite{Garfield2001} --- the spectral sequence determined by the bigraded Rumin complex --- recover the Kohn--Rossi and de Rham cohomology groups, respectively.
Garfield and Lee also proved a Hodge decomposition theorem~\cites{Garfield2001,GarfieldLee1998} for the spaces $\mR^{p,q}$ on a closed, strictly pseudoconvex pseudohermitian manifold.

This article has two main purposes.
First, we give a new construction of the Rumin and bigraded Rumin complexes as balanced $A_\infty$-algebras~\cites{Keller2001,Markl1992} whose elements are (complex-valued) differential forms.
In particular, our construction of the Rumin complex recovers the de Rham cohomology algebra;
and our construction of the bigraded Rumin complex gives Kohn--Rossi cohomology the structure of an algebra.
Second, we provide tools --- including a useful long exact sequence, the Garfield spectral sequence, and some Hodge decomposition theorems --- which lead to effective applications of our complexes.
One of our Hodge theorems completes and corrects a sketched proof announced by Garfield and Lee~\cite{GarfieldLee1998} (cf.\ \cite{Garfield2001}).
We also present some applications of these tools to cohomological vanishing results, to topological properties of Sasakian manifolds~\cite{BoyerGalicki2008}, and to pseudo-Einstein manifolds~\cite{Lee1988}.

Our construction of the bigraded Rumin complex of a $(2n+1)$-dimensional CR manifold differs from that of Garfield and Lee only in the definition of the spaces $\mR^{p,q}$, $p+q\leq n$.
Our initial insight is that if $\Omega$ is an equivalence class in the space $\mR_{GL}^{p,q}$ defined by Garfield and Lee, then there is a unique complex-valued differential $(p+q)$-form $\omega\in\Omega$ such that $\theta\wedge\omega\wedge d\theta^{n+1-p-q}=0$ and $\theta \wedge d\omega \wedge d\theta^{n-p-q}=0$ for any local contact form $\theta$.
This choice
\begin{enumerate}
 \item defines an isomorphism $\mR^{p,q}\cong\mR_{GL}^{p,q}$;
 \item is CR invariant;
 \item is compatible with the exterior derivative on differential forms;
 \item leads to balanced $A_\infty$-structures on the Rumin and bigraded Rumin complexes.
\end{enumerate} 
Note that our choice of $\omega$ involves taking a derivative;
for this reason, it is most natural to regard the spaces in the bigraded Rumin complex as subsheaves of the sheaf of complex-valued differential forms.
When a contact form is fixed, the sheaf $\mR^{p,q}$ can be identified with the sheaf of local sections of an appropriate subbundle of $\bC\otimes\Lambda^{p+q}T^\ast M$.
This is the same bundle which appears in Garfield's work~\cite{Garfield2001}, and when $p=0$, it is also the same bundle which appears in the Kohn--Rossi complex~\cite{KohnRossi1965}.
Our construction of the Rumin complex on a contact manifold is along similar lines.

Our constructions are local and meaningful on all nondegenerate CR manifolds.
In particular, we do not need to impose assumptions on orientability, on (local) embeddability, or on the signature of the Levi form.
On orientable CR manifolds, we use this level generality to define a cohomology class $\mL \in H^1(M;\sP)$, where $\sP$ is the sheaf of CR pluriharmonic functions, whose vanishing is necessary and sufficient for the existence of a pseudo-Einstein contact form.
This generalizes a construction of Lee~\cite{Lee1988} to CR manifolds which are not locally embeddable.

Our construction of the bigraded Rumin complex is inspired by Takeuchi's observation~\cite{Takeuchi2019} that there is a CR invariant formulation of the operator $d^c := i(\dbbar-\db)$ on smooth functions which takes values in the space of differential one-forms.
Case and Yang~\cite{CaseYang2020} adapted Takeuchi's definition to give a CR invariant formulation of the $\dbbar$-operator on smooth functions which takes values in the space of complex-valued differential one-forms.
The benefits of Takeuchi's formulation are that
\begin{enumerate}
 \item the CR invariant, differential form-valued operator $dd^c$ characterizes CR pluriharmonic functions~\cites{Lee1988,Takeuchi2019} in all dimensions; and
 \item for CR manifolds realized as the boundary of a complex manifold, $d^cu$ recovers the restriction of $d^c\cu$ to the boundary, where $\cu$ is an extension of $u$ to the inside which is formally harmonic to low order~\cites{Hirachi2013,Takeuchi2019}.
\end{enumerate}
We expect that this perspective can be used to recover the relationship between the Kohn--Rossi and Dolbeault complexes on a complex manifold with boundary.

Given the length of this article, we have split it into four parts.
We have also opted to start each part with a more detailed description of its contents.
The rest of this introduction gives a broad overview of the contents of each part.

The rest of \cref{part:background} contains background material on CR and pseudohermitian manifolds, as well as a collection of facts about symplectic vector spaces needed to make the canonical choice of element $\omega\in\Omega$ described above.

In \cref{part:intrinsic} we collect most of our CR invariant, differential form-based constructions.
We begin this part by constructing the Rumin complex of a contact manifold via differential forms and by defining the cohomology groups $H_R^k(M;\bR)$.
We then construct the bigraded Rumin complex and define cohomology groups $H_R^{p,q}(M)$ and $H_R^k(M;\sP)$.
Next, we explicitly describe the long exact sequence
\begin{equation*}
 \dotsm \longrightarrow H_R^k(M;\bR) \longrightarrow H_R^{0,k}(M) \longrightarrow H_R^k(M;\sP) \longrightarrow H_R^{k+1}(M;\bR) \longrightarrow \dotsm
\end{equation*}
relating these groups.
We then introduce the balanced $A_\infty$-structures on the Rumin and bigraded Rumin complexes.
We conclude by showing that
\begin{enumerate}
 \item $H_R^k(M;\bR)$ and $H_R^{p,q}(M)$ are isomorphic to the corresponding de Rham and Kohn--Rossi cohomology groups, respectively;
 \item our $A_\infty$-structure induces the usual algebra structure on de Rham cohomology and a similar algebra structure on Kohn--Rossi cohomology; and
 \item for small $k$, the groups $H_R^k(M;\sP)$ are isomorphic to the corresponding sheaf cohomology group with coefficients in $\sP$.
\end{enumerate}

In \cref{part:hodge} we prove Hodge decomposition theorems for the Rumin and bigraded Rumin complex.
We also prove a Hodge decomposition theorem that is relevant for the second page of the Garfield spectral sequence.
The proofs of the first two Hodge decomposition theorems closely follow the original proofs of Rumin~\cite{Rumin1994} and of Garfield and Lee~\cites{Garfield2001,GarfieldLee1998}, respectively, with some streamlining to highlight the similarities of the two proofs.
The proof of the last Hodge decomposition theorem is modeled on Popovici's proof~\cite{Popovici2016} of a Hodge decomposition theorem for the second page of the Fr\"olicher spectral sequence~\cite{Frolicher1955}.
However, the fact that the Kohn Laplacian is not hypoelliptic on $\mR^{p,0}$ introduces new difficulties relative to Popovici's proof.
These are overcome using properties~\cite{BealsGreiner1988} of the Szeg\H o projection.

In \cref{part:applications} we discuss four classes of applications of the bigraded Rumin complex and our Hodge theorems.
First, we give estimates for the dimensions of the Kohn--Rossi cohomology groups $H^{p,q}(M)$ for strictly pseudoconvex manifolds with nonnegative curvature.
When $p=0$, this refines vanishing theorems of Tanaka~\cite{Tanaka1975} and Lee~\cite{Lee1988}.
Second, we present some results which indicate that the spaces $E_2^{p,q}$ on the second page of the Garfield spectral sequence should be regarded as the CR analogues of the Dolbeault cohomology groups.
For example, we show that on closed, embeddable, strictly pseudoconvex CR manifolds, the space $E_2^{p,q}$ is always finite-dimensional, and that
\begin{equation*}
 \dim H^k(M;\bC) \leq \sum_{p+q=k} \dim E_2^{p,q}
\end{equation*}
with equality on Sasakian manifolds (cf.\ \cite{Frolicher1955}).
This gives another example of Sasakian manifolds as odd-dimensional analogues of K\"ahler manifolds (cf.\ \cite{BoyerGalicki2008}).
Third, we give new proofs, with some improvements in the conclusions, of certain topological properties~\cites{Rukimbira1993,Itoh1997, CappellettiMontanoDeNicolaYudin2015,Bungart1992,Nozawa2014} of closed Sasakian manifolds.
These results indicate the utility of the bigraded Rumin complex for studying Sasakian manifolds, and especially for obstructions to the existence of a torsion-free contact form on a given CR manifold.
Fourth, we define the cohomology class $\mL\in H^1(M;\sP)$ and use it to discuss Lee's conjecture~\cite{Lee1988} that, under suitable hypotheses, a pseudo-Einstein contact form exists if and only if the real first Chern class $c_1(T^{1,0}) \in H^2(M;\bR)$ vanishes.

\input{bg/bg}
\input{bg/pseudohermitian}
\input{bg/linear-algebra}

%% file: bg/bg.tex
\section{CR manifolds}
\label{sec:bg}

The fundamental object of study in this article is a CR manifold.

\begin{definition}
 \label{defn:cr}
 A \defn{(nondegenerate) CR manifold} is a pair $(M^{2n+1},T^{1,0})$ consisting of a real, connected $(2n+1)$-dimensional manifold $M^{2n+1}$ and a complex rank $n$ distribution $T^{1,0}\subset \CTM := TM \otimes \bC$ such that
 \begin{enumerate}
  \item $T^{1,0} \cap T^{0,1} = \{0\}$, where $T^{0,1} := \overline{T^{1,0}}$;
  \item $[T^{1,0},T^{1,0}] \subset T^{1,0}$; and
  \item if $\theta$ is a real one-form on an open set $U\subset M$ and $\ker\theta = \Real (T^{1,0}\oplus T^{0,1})\rv_U$, then
  \begin{equation}
   \label{eqn:levi-form-invariant}
   L(Z,W) := -i\,d\theta(Z,\overline{W})
  \end{equation}
  defines a nondegenerate hermitian inner product on $T^{1,0}U$.
 \end{enumerate}
\end{definition}

The assumption that $M$ is connected is included only for technical convenience.
We always assume nondegeneracy, so that our CR manifolds are contact manifolds.
This is because the Rumin complex is in fact a contact invariant complex~\cite{Rumin1990}.

One source of examples of CR manifolds are boundaries of a strictly pseudoconvex domain in $\bC^{n+1}$.
We later explicitly describe two examples of quotients involving the flat Heisenberg manifold $\bH^{n+1}$:
\Cref{heisenberg-quotient-pseudo-einstein,heisenberg-quotient,cup-length} discuss properties of a compact quotient of $\bH^{n+1}$, and \cref{cr-hopf-manifold} discusses properties of a compact quotient of $\bH^{n+1}\setminus\{0\}$.

We give names to the important objects which appear in \cref{defn:cr}.

\begin{definition}
 Let $(M^{2n+1},T^{1,0})$ be a CR manifold.
 \begin{enumerate}
  \item The \defn{contact distribution} of $(M^{2n+1},T^{1,0})$ is the real hyperplane bundle $H:=\Real (T^{1,0}\oplus T^{0,1})$.
  \item A \defn{contact form} on $(M^{2n+1},T^{1,0})$ is a real one-form $\theta$ such that $\ker\theta = H$.
  \item A \defn{local contact form} for $(M^{2n+1},T^{1,0})$ is a contact form on $(U,T^{1,0}\rv_U)$ for some open set $U\subset M$.
  \item The \defn{Levi form} of a contact form $\theta$ is the hermitian inner product~\eqref{eqn:levi-form-invariant}.
 \end{enumerate}
\end{definition}

We emphasize two features of our definition of a CR manifold.

First, we do not assume the existence of a global contact form on $(M^{2n+1},T^{1,0})$.
Note that if $\theta$ is a local contact form for $(M^{2n+1},T^{1,0})$ defined on $U$, then another one-form $\htheta$ on $U$ is also a local contact form if and only if $\htheta=u\theta$ for some nowhere zero smooth function $u$ on $U$.
It follows that if the Levi form of $\theta$ has signature $(p,q)$ and $u>0$, then the Levi form of $u\theta$ has signature $(p,q)$ and the Levi form of $-u\theta$ has signature $(q,p)$.
This justifies the following definition:

\begin{definition}
 \label{defn:signature}
 A CR manifold $(M^{2n+1},T^{1,0})$ has \defn{signature $(p,q)$} if for each point $x\in M$, there is a local contact form for $(M^{2n+1},T^{1,0})$ with Levi form of signature $(p,q)$ defined on a neighborhood of $x$.
\end{definition}

In general, the existence of a global contact form on a CR manifold is equivalent to the assumption of orientability:
If $\theta$ is a global contact form, then $\theta\wedge d\theta^n$ is a nowhere-vanishing volume form;
conversely, if $(M^{2n+1},T^{1,0})$ is orientable, then the orientation on the contact distribution $H$ induced by $T^{1,0}$ determines a co-orientation on $H$.
However, if $(M^{2n+1},T^{1,0})$ is not of split signature, then it admits a global contact form.

\begin{lemma}
 \label{orientable}
 Let $(M^{2n+1},T^{1,0})$ be a CR manifold of signature $(p,q)$, $p\not=q$.
 Then there exists a global contact form $\theta$ for $(M^{2n+1},T^{1,0})$.
\end{lemma}

\begin{proof}
 Let $\{U_\alpha\}_{\alpha\in A}$ be a locally finite cover of $M$ such that for each $\alpha\in A$ there is a contact form $\theta_\alpha$ in $(U_\alpha, T^{1,0}\rv_{U_\alpha})$ such that the Levi form induced by $\theta_\alpha$ has signature $(p,q)$.
 Let $\{\eta_\alpha\}_{\alpha\in A}$ be a partition of unity subordinate to $\{U_\alpha\}$.
 Set
 \begin{equation}
  \label{eqn:global-contact-form}
  \theta := \sum_{\alpha\in  A} \eta_\alpha \theta_\alpha .
 \end{equation}
 Since $p\not=q$, if $U_\alpha\cap U_\beta\not=0$ then there is a positive $u_{\alpha\beta}\in C^\infty(U_\alpha\cap U_\beta)$ such that $\theta_\alpha = u_{\alpha\beta}\theta_\beta$.
 We conclude that Equation~\eqref{eqn:global-contact-form} defines a global, nowhere vanishing contact form with Levi form of signature $(p,q)$.
\end{proof}

Second, we do not make an assumption on the signature of the Levi form.  However, when discussing Hodge theorems in \cref{part:hodge} and applications thereof in \cref{part:applications}, we restrict to strictly pseudoconvex CR manifolds.

\begin{definition}
 \label{defn:strictly-pseudoconvex}
 A CR manifold $(M^{2n+1},T^{1,0})$ is \defn{strictly pseudoconvex} if there exists a global contact form $\theta$ such that the Levi form is positive definite.
\end{definition}

\Cref{orientable} implies that a CR manifold is strictly pseudoconvex if and only if around each point there exists a local contact form with positive definite Levi form.

%% file: bg/pseudohermitian.tex
\section{Pseudohermitian manifolds}
\label{sec:pseudohermitian}

On an orientable CR manifold, a choice of contact form determines a pseudohermitian structure.

\begin{definition}
 A \defn{pseudohermitian manifold} is a triple $(M^{2n+1},T^{1,0},\theta)$ consisting of an orientable CR manifold $(M^{2n+1},T^{1,0})$ and a contact form $\theta$.
 We say that $(M^{2n+1},T^{1,0},\theta)$  is \defn{strictly pseudoconvex} if the Levi form of $\theta$ is positive definite.
\end{definition}

There is a canonical vector field defined on a pseudohermitian manifold.

\begin{definition}
 Let $(M^{2n+1},T^{1,0},\theta)$ be a pseudohermitian manifold.
 The \defn{Reeb vector field} is the unique vector field $T$ such that $\theta(T)=1$ and $d\theta(T,\cdot)=0$.
\end{definition}

Given our focus on differential forms, it is convenient to perform computations in terms of local coframes which are well-adapted to the pseudohermitian structure.

\begin{definition}
 Let $(M^{2n+1},T^{1,0},\theta)$ be a pseudohermitian manifold.
 An \defn{admissible coframe} is a set $\{\theta^\alpha\}_{\alpha=1}^n$ of complex-valued one-forms on an open set $U\subset M$ such that
 \begin{enumerate}
  \item $\theta^\alpha$ annihilates $T$ and $T^{0,1}$ for each $\alpha\in\{1,\dotsc,n\}$; and
  \item $\{\theta,\theta^1,\dotsc,\theta^n,\theta^{\bar1},\dotsc,\theta^{\bar n}\}$ forms a basis for the space of sections of $\CT^\ast U$, where $\theta^{\bar\beta}:=\overline{\theta^\beta}$.
 \end{enumerate}
\end{definition}

Given an admissible coframe, we identify the Levi form with the hermitian matrix $h_{\alpha\bar\beta}$ of functions determined by
\begin{equation*}
 d\theta = ih_{\alpha\bar\beta}\,\theta^\alpha\wedge\theta^{\bar\beta} .
\end{equation*}
Here and throughout we employ Einstein summation notation.
When an admissible coframe has been fixed, we use the Levi form $h_{\alpha\bar\beta}$ to raise and lower indices without further comment;
e.g.\ given $\{\tau^\alpha\}_{\alpha=1}^n$, we set
\[ \tau_{\bar\beta} := h_{\mu\bar\beta}\tau^\mu . \]

Webster constructed~\cite{Webster1978}*{Theorem~1.1} a canonical connection on a pseudohermitian manifold:
If $(M^{2n+1},T^{1,0},\theta)$ is a pseudohermitian manifold, then for any admissible coframe there are unique one-forms $\{\omega_\alpha{}^\gamma\}_{\alpha,\gamma=1}^n$ and $\{\tau^\gamma\}_{\gamma=1}^n$ such that
\begin{equation}
 \label{eqn:tanaka-webster-connection}
 \begin{aligned}
  d\theta^\alpha & = \theta^\mu \wedge \omega_\mu{}^\alpha + \theta \wedge \tau^\alpha, \\
  \tau^\alpha & \equiv 0 \mod \theta^{\bar\beta}, \\
  dh_{\alpha\bar\beta} & = \omega_{\alpha\bar\beta} + \omega_{\bar\beta\alpha}, & \omega_{\bar\beta\alpha} & := \overline{\omega_{\beta\bar\alpha}} , \\
  \theta^\mu\wedge\tau_\mu & = 0 .
 \end{aligned}
\end{equation}
The one-forms $\omega_\alpha{}^\beta$ are the connection one-forms of the Tanaka--Webster connection and the one-forms $\tau^\alpha$ determine the nontrivial components of the torsion of the Tanaka--Webster connection (cf.\ \cite{Tanaka1975}*{Proposition~3.1}):

\begin{definition}
 Let $(M^{2n+1},T^{1,0},\theta)$ be a pseudohermitian manifold.
 The \defn{Tanaka--Webster connection} is the unique connection $\nabla$ on $\CTM$ determined by
 \[ \nabla T = 0, \qquad \nabla Z_\alpha = \omega_\alpha{}^\mu \otimes Z_\mu , \]
 and conjugation, where $\{Z_\alpha\}$ is the local frame for $T^{1,0}$ dual to $\{\theta^\alpha\}$.
 The \defn{pseudohermitian torsion} is
 \[ A := A_{\alpha\gamma}\,\theta^\alpha\otimes\theta^\gamma , \]
 where we use Display~\eqref{eqn:tanaka-webster-connection} to write $\tau^{\bar\beta}=A_\alpha{}^{\bar\beta}\,\theta^\alpha$.
\end{definition}

Display~\eqref{eqn:tanaka-webster-connection} implies that $A_{\alpha\gamma}=A_{\gamma\alpha}$.
Note that the pseudohermitian torsion is globally defined.
This motivates the following definition.

\begin{definition}
 A pseudohermitian manifold $(M^{2n+1},T^{1,0},\theta)$ is \defn{torsion-free} if the pseudohermitian torsion vanishes.
\end{definition}

Strictly pseudoconvex, torsion-free pseudohermitian manifolds are in one-to-one correspondence with Sasaki manifolds~\cite{BoyerGalicki2008}*{Chapter 6}.

Webster showed~\cite{Webster1978}*{Theorem~1.1a and Equation~(2.2)} that the curvature forms
\begin{equation}
 \label{eqn:curvature-forms}
 \Omega_\alpha{}^\gamma := d\omega_\alpha{}^\gamma - \omega_\alpha{}^\mu \wedge \omega_\mu{}^\gamma ,
\end{equation}
regarded as sections of $\Lambda^2\CT^\ast M \otimes \End(T^{1,0})$, are globally defined and skew-hermitian with respect to the Levi form.
He also showed that (cf.\ \cite{Lee1988}*{Equation~(2.4)})
\begin{equation}
 \label{eqn:curvature-form}
 \Omega_{\alpha\bar\beta} = R_{\alpha\bar\beta\rho\bar\sigma}\,\theta^\rho\wedge\theta^{\bar\sigma} - \nabla_{\bar\beta}A_{\alpha\rho}\,\theta\wedge\theta^\rho + \nabla_\alpha A_{\bar\beta\bar\sigma}\,\theta\wedge\theta^{\bar\sigma} + i\theta_\alpha\wedge\tau_{\bar\beta} - i\tau_\alpha\wedge\theta_{\bar\beta} ,
\end{equation}
where $R_{\alpha\bar\beta\rho\bar\sigma}$ satisfies
\begin{equation*}
 R_{\alpha\bar\beta\rho\bar\sigma} = R_{\alpha\bar\sigma\rho\bar\beta} = R_{\rho\bar\sigma\alpha\bar\beta}
\end{equation*}
and $\nabla_{\bar\beta}A_{\alpha\rho}$ denotes the obvious component of the covariant derivative $\nabla A$ of the pseudohermitian torsion.
Taking the exterior derivative of Equation~\eqref{eqn:curvature-forms} yields the Bianchi identity
\begin{equation}
 \label{eqn:bianchi}
 d\Omega_\alpha{}^\gamma = \omega_\alpha{}^\mu \wedge \Omega_\mu{}^\gamma - \Omega_\alpha{}^\mu \wedge \omega_\mu{}^\gamma 
\end{equation}
(cf.\ \cite{Lee1988}*{Lemma~2.2}).

The functions $R_{\alpha\bar\beta\gamma\bar\sigma}$ determine the pseudohermitian curvature, the pseudohermitian Ricci curvature, and the pseudohermitian scalar curvature.

\begin{definition}
 \label{defn:curvature}
 Let $(M^{2n+1},T^{1,0},\theta)$ be a pseudohermitian manifold.
 The \defn{pseudohermitian curvature} is the tensor
 \[ R_{\alpha\bar\beta\gamma\bar\sigma}\,\theta^\alpha\otimes\theta^{\bar\beta}\otimes\theta^{\gamma}\otimes\theta^{\bar\sigma} . \]
 The \defn{pseudohermitian Ricci curvature} is the tensor
 \[ R_{\alpha\bar\beta}\,\theta^\alpha\otimes\theta^{\bar\beta}, \]
 where $R_{\alpha\bar\beta}:=R_{\alpha\bar\beta\mu}{}^\mu$.
 The \defn{pseudohermitian scalar curvature} is $R:=R_\mu{}^\mu$.
\end{definition}

Both the pseudohermitian torsion and the pseudohermitian curvature appear in commutator formulas involving the Tanaka--Webster connection.
All such commutator formulas can be deduced from the following result of Lee~\cite{Lee1988}.

\begin{lemma}[\cite{Lee1988}*{Lemma~2.3}]
 \label{commutators}
 Let $(M^{2n+1},T^{1,0},\theta)$ be a pseudohermitian manifold.
 For every complex-valued function $f\in C^\infty(M;\bC)$, it holds that
 \begin{align*}
  [\nabla_{\bar\beta},\nabla_\alpha] f & = ih_{\alpha\bar\beta}\nabla_0f, \\
  [\nabla_\gamma,\nabla_\alpha]f & = 0, \\
  [\nabla_\alpha,\nabla_0]f & = A_{\alpha\mu}f^\mu .
 \end{align*}
 For every section $\omega_\alpha$ of $(T^{1,0})^\ast$, it holds that
 \begin{align*}
  [\nabla_{\bar\beta},\nabla_\alpha]\omega_\gamma & = ih_{\alpha\bar\beta}\nabla_0\omega_\gamma + R_{\alpha\bar\beta\gamma}{}^\mu\omega_\mu, \\
  [\nabla_\alpha,\nabla_\gamma]\omega_\rho & = iA_{\alpha\rho}\omega_\gamma - iA_{\gamma\rho}\omega_\alpha, \\
  [\nabla_{\bar\beta},\nabla_{\bar\sigma}]\omega_\alpha & = ih_{\alpha\bar\sigma}A_{\bar\beta}{}^\mu\omega_\mu - ih_{\alpha\bar\beta}A_{\bar\sigma}{}^\mu\omega_\mu , \\
  [\nabla_\gamma,\nabla_0]\omega_\alpha & = A_\gamma{}^{\bar\nu}\nabla_{\bar\nu}\omega_\alpha - \omega_\mu\nabla^\mu A_{\alpha\gamma}, \\
  [\nabla_{\bar\beta},\nabla_0]\omega_\alpha & = A_{\bar\beta}{}^\mu\nabla_\mu\omega_\alpha + \omega_\mu\nabla_\alpha A_{\bar\beta}{}^\mu .
 \end{align*}
\end{lemma}

Here the unbarred (resp.\ barred) subscripts denote components in $(T^{1,0})^\ast$ (resp.\ $(T^{0,1})^\ast$) of the associated tensor.
The subscript zero denotes covariant differentiation in the direction of the Reeb vector field.

We define the type of a differential form using the bundles $\Lambda^{p,q}$.

\begin{definition}
 Let $(M^{2n+1},T^{1,0})$ be a CR manifold.  We denote by $\Lambda^{p,q}$ the vector bundle
 \[ \Lambda^{p,q} := \Lambda^p(T^{1,0})^\ast \otimes \Lambda^q(T^{0,1})^\ast . \]
 and denote by $\Omega^{p,q}$ its space of smooth sections.
\end{definition}

We use multi-indices to compute with (local) sections of $\Lambda^{p,q}$:
Given $\omega\in\Omega^{p,q}$, we write
\begin{equation*}
 \omega = \frac{1}{p!q!}\omega_{\Alpha\bar\Beta}\,\theta^\Alpha\wedge\theta^{\bar\Beta} ,
\end{equation*}
where $\Alpha=(\alpha_1,\dotsc,\alpha_p)$ and $\Beta=(\beta_1,\dotsc,\beta_q)$ are multi-indices of length $p$ and $q$, respectively,
\begin{equation*}
 \theta^\Alpha := \theta^{\alpha_1} \wedge \dotsm \wedge \theta^{\alpha_p} ,
\end{equation*}
and $\{\theta^\alpha\}_{\alpha=1}^n$ is a local coframe for $(T^{1,0})^\ast$.
Given a multi-index $\Alpha$ of length $p$, we denote by $\Alpha^\prime \subset\Alpha$ a sub-multi-index of length $(p-1)$;
e.g.\ given a contact form, we express the contraction $\Lambda^{0,1}\otimes\Lambda^{p,q}\to\Lambda^{p-1,q}$ by
\begin{equation*}
 \tau_{\bar\sigma}\,\theta^{\bar\sigma} \otimes \frac{1}{p!q!}\omega_{\Alpha\bar\Beta}\,\theta^{\Alpha} \wedge \theta^{\bar\beta} \mapsto \frac{1}{(p-1)!q!}\tau^{\mu}\omega_{\mu\Alpha^\prime\bar\Beta}\,\theta^{\Alpha^\prime}\wedge\theta^{\bar\Beta} .
\end{equation*}
We use square brackets to denote skew symmetrization;
e.g.\ if $\tau=\tau_\alpha\,\theta^\alpha\in\Lambda^{1,0}$ and $\omega=\frac{1}{p!q!}\omega_{\Alpha\bar\Beta}\,\theta^{\Alpha}\wedge\theta^{\bar\Beta}\in\Lambda^{p,q}$, then the exterior product $\tau\wedge\omega=\frac{1}{(p+1)!q!}\Omega_{\alpha\Alpha\bar\Beta}\,\theta^{\alpha\Alpha}\wedge\theta^{\bar\Beta}$ has components
\begin{multline}
 \label{eqn:skew-model}
 \Omega_{\alpha\Alpha\bar\Beta} = (p+1)\tau_{[\alpha}\omega_{\Alpha\bar\Beta]} := \bigl( \tau_\alpha\omega_{\alpha_1\dotsm\alpha_p\bar\Beta} - \tau_{\alpha_1}\omega_{\alpha\alpha_2\dotsm\alpha_p\bar\Beta} \\
  - \tau_{\alpha_2}\omega_{\alpha_1\alpha\alpha_3\dotsm\alpha_p\bar\Beta} - \dotsm - \tau_{\alpha_p}\omega_{\alpha_1\dotsm\alpha_{p-1}\alpha\bar\Beta} \bigr) .
\end{multline}
With this notation, we only skew symmetrize over indices of the same type;
e.g.\ in Equation~\eqref{eqn:skew-model} we skew over all $p+1$ appearances of an alpha and over all $q$ appearances of a beta.
When clear from context, we identify $\Alpha$ and $(\alpha,\Alpha^\prime)$;
e.g.\ we express the action of $P_\alpha{}^\gamma\in\End(T^{1,0})$ as a derivation on $\Lambda^{p,q}$ by $P\hash\omega:=\frac{1}{p!q!}\Omega_{\Alpha\bar\Beta}\,\theta^{\Alpha}\wedge\theta^{\bar\Beta}$, where
\begin{equation*}
 \Omega_{\Alpha\bar\Beta} = -pP_{[\alpha}{}^\mu\omega_{\mu\Alpha^\prime\bar\Beta]}
\end{equation*}
and $\omega=\frac{1}{p!q!}\omega_{\Alpha\bar\Beta}\,\theta^{\Alpha}\wedge\theta^{\bar\Beta}$.
By our above convention, the index $\mu$ is fixed;
i.e.
\begin{equation*}
 pP_{[\alpha}{}^\mu\omega_{\mu\Alpha^\prime\bar\Beta]} := P_{\alpha_1}{}^\mu\omega_{\mu\alpha_2\dotsm\alpha_p\bar\Beta} + P_{\alpha_2}{}^\mu\omega_{\alpha_1\mu\alpha_3\dotsm\alpha_p\bar\Beta} + \dotsm + P_{\alpha_p}{}^\mu\omega_{\alpha_1\dotsm\alpha_{p-1}\mu\bar\Beta} .
\end{equation*}

While the bundle $\Lambda^{p,q}$ is CR invariant, the embedding $\Lambda^{p,q}\hookrightarrow\Lambda^{p+q}\CT^\ast M$ determined by admissible coframes is not.
This follows from the transformation law for admissible coframes under change of contact form (cf.\ \cite{Lee1988}*{Lemma~2.4}).

\begin{lemma}
 \label{coframe-transformation}
 Let $\{\theta^\alpha\}$ be an admissible coframe for a pseudohermitian manifold $(M^{2n+1},T^{1,0},\theta)$.
 If $\htheta=e^\Upsilon\theta$ is another contact form on $(M^{2n+1},T^{1,0})$, then
 \[ \htheta^\alpha := \theta^\alpha + i\Upsilon^\alpha\theta \]
 determines an admissible coframe $\{\htheta^\alpha\}$ for $(M^{2n+1},T^{1,0},\htheta)$.
\end{lemma}

\begin{remark}
 Our choice of admissible coframe $\{\htheta^\alpha\}$ is such that the dual frames to $\{\theta^\alpha\}$ and $\{\htheta^\alpha\}$ are the same.
\end{remark}

\begin{proof}
 Note that
 \[ d\htheta = e^\Upsilon\left( d\theta - \Upsilon_\alpha\,\theta\wedge\theta^\alpha - \Upsilon_{\bar\beta}\,\theta\wedge\theta^{\bar\beta} \right) . \]
 We readily deduce that $\hT:=e^{-\Upsilon}(T - i\Upsilon^\alpha\,Z_\alpha + i\Upsilon^{\bar\beta}\,Z_{\bar\beta})$ is the Reeb vector field associated to $\htheta$, where $\{T,Z_\alpha,Z_{\bar\beta}\}$ is the frame dual to $\{\theta,\theta^\alpha,\theta^{\bar\beta}\}$.
 It follows that $\{\htheta^\alpha\}$ is an admissible coframe for $(M^{2n+1},T^{1,0},\htheta)$.
\end{proof}

A choice of contact form determines an embedding $\Lambda^{p,q}\hookrightarrow\Lambda^{p+q}\CT^\ast M$ of vector bundles by using an admissible coframe to regard elements of $\Lambda^{p,q}$ as elements of $\Lambda^{p+q}\CT^\ast M$.
Indeed, on a pseudohermitian manifold there is a canonical splitting of $\Lambda^k\CT^\ast M$ into the bundle generated by the choice of contact form and the bundle $\bigoplus_{p+q=k}\Lambda^{p,q}$.

\begin{lemma}
 \label{Lambdapq-embedding}
 Let $(M^{2n+1},T^{1,0},\theta)$ be a pseudohermitian manifold.
 There is a canonical embedding $\iota\colon\Omega^{p,q}\hookrightarrow \COmega^{p+q}M$ such that if $\comega=\iota(\omega)$, then $\comega\rv_{\CH}=\omega$ and $\comega(T,\cdot)=0$.
 This embedding induces an isomorphism
 \begin{equation}
  \label{eqn:Lambdak-isomorphism}
  \COmega^k M \cong \Omega^{k} \oplus (\theta \wedge \Omega^{k-1})
 \end{equation}
 for each $k\in\bN_0:=\bN\cup\{0\}$, where $\COmega^kM$ is the space of complex-valued differential $k$-forms and
 \begin{equation*}
  \Omega^{k} := \bigoplus_{p+q=k} \Omega^{p,q} .
 \end{equation*}
\end{lemma}

\begin{remark}
 These statements are also true for local sections of the respective bundles.
\end{remark}

\begin{proof}
 Let $\omega\in\Omega^{p,q}$ and let $\{\theta^\alpha\}$ be an admissible coframe for $(M^{2n+1},T^{1,0},\theta)$.
 By abuse of notation, regard $\{\theta^\alpha\}$ as a coframe for $(T^{1,0})^\ast$.
 Write $\omega=\frac{1}{p!q!}\omega_{\Alpha\bar\Beta}\,\theta^\Alpha\wedge\theta^{\bar\Beta}$.
 Then $\iota(\omega)$ is defined by regarding the right-hand side as an element of $\COmega^{p+q}M$.
 It is clear that $\iota$ is injective and that $\iota(\omega)(T,\cdot)=0$.
 Moreover, $\iota$ extends by linearity to an embedding $\iota\colon\Omega^k\hookrightarrow\COmega^kM$ with image
 \begin{equation*}
  \iota(\Omega^k) = \left\{ \omega \in \COmega^kM \suchthatcolon \omega(T,\cdot) = 0 \right\} .
 \end{equation*}
 We use the embedding $\iota$ to identify $\Omega^k$ with its image.
 
 It is clear that $\Omega^k \cap (\theta \wedge\Omega^{k-1})=\{0\}$.
 Hence $\Omega^k+(\theta\wedge\Omega^{k-1})\subset\COmega^kM$ is a direct sum.
 Now let $\omega\in\COmega^kM$.
 Set $\eta:=\theta\wedge\omega(T,\cdot)$.
 Then $\omega(T,\cdot)$ and $\omega-\eta$ both annihilate $T$.
 In particular, $\omega-\eta\in\Omega^k$ and $\eta \in \theta\wedge\Omega^{k-1}$.
 Therefore Equation~\eqref{eqn:Lambdak-isomorphism} holds.
\end{proof}

When a choice of pseudohermitian structure has been fixed, we use \cref{Lambdapq-embedding} to define the subspace $\Omega^{p,q}M\subset\COmega^{p+q}M$.

\begin{definition}
 \label{defn:Omegapq}
 Let $(M^{2n+1},T^{1,0},\theta)$ be a pseudohermitian manifold.
 Given nonnegative integers $p$ and $q$, we denote by $\Omega^{p,q}M$ the image of $\Omega^{p,q}$ in $\COmega^{p+q}M$ with respect to the embedding  described by \cref{Lambdapq-embedding}.
\end{definition}

Importantly, the $C^\infty(M;\bC)$-module $\theta\wedge\Omega^{p,q}M$ is CR invariant.

\begin{lemma}
 \label{Omegapq-CR-invariant}
 Let $(M^{2n+1},T^{ 1,0},\theta)$ be a pseudohermitian manifold.
 If $\htheta=e^\Upsilon\theta$ is another choice of contact form for $(M^{2n+1},T^{1,0})$, then, as $C^\infty(M;\bC)$-modules,
 \begin{equation*}
  \theta \wedge \Omega^{p,q}M = \htheta \wedge \hOmega^{p,q}M
 \end{equation*}
 where $\Omega^{p,q}M$ and $\hOmega^{p,q}M$ are determined from $\theta$ and $\htheta$, respectively, via \cref{Lambdapq-embedding}.
\end{lemma}

\begin{proof}
 \Cref{coframe-transformation} implies that
 \begin{equation*}
  \frac{1}{p!q!}\omega_{\Alpha\bar\Beta}\,\htheta\wedge\htheta^{\Alpha}\wedge\htheta^{\bar\Beta} = \frac{1}{p!q!}e^\Upsilon\omega_{\Alpha\bar\Beta}\,\theta\wedge\theta^{\Alpha}\wedge\theta^{\bar\Beta} .
 \end{equation*}
 In particular, $\htheta\wedge\hiota(\omega)=\theta\wedge\iota(e^\Upsilon\omega)$.
 The conclusion readily follows.
\end{proof}

We conclude this section by discussing how the Tanaka--Webster connection, the pseudohermitian torsion, and the pseudohermitian curvature transform under change of contact form.
The latter is more clearly stated in terms of the CR Schouten tensor and the Chern tensor.

\begin{definition}
 Let $(M^{2n+1},T^{1,0},\theta)$ be a pseudohermitian manifold.
 The \defn{CR Schouten tensor} is determined by
 \[ P_{\alpha\bar\beta} := \frac{1}{n+2}\left( R_{\alpha\bar\beta} - \frac{R}{2(n+1)}h_{\alpha\bar\beta}\right) . \]
 The \defn{Chern tensor} is determined by
 \[ S_{\alpha\bar\beta\gamma\bar\sigma} := R_{\alpha\bar\beta\gamma\bar\sigma} - P_{\alpha\bar\beta}h_{\gamma\bar\sigma} - P_{\alpha\bar\sigma}h_{\gamma\bar\beta} - P_{\gamma\bar\beta}h_{\alpha\bar\sigma} - P_{\gamma\bar\sigma}h_{\alpha\bar\beta} . \]
\end{definition}

Note that the Chern tensor is totally trace-free and has the same symmetries as the pseudohermitian curvature tensor;
i.e.\ $S_\mu{}^\mu{}_{\alpha\bar\beta}=0$ and $S_{\alpha\bar\beta\gamma\bar\sigma}=S_{\alpha\bar\sigma\gamma\bar\beta}=S_{\gamma\bar\sigma\alpha\bar\beta}$.
The following lemma collects the well-known transformation formulas for the CR Schouten tensor, the Chern tensor, the pseudohermitian torsion, and the Tanaka--Webster connection under change of contact form.

\begin{lemma}[\citelist{\cite{Lee1988}*{Lemma~2.4} \cite{GoverGraham2005}*{Equation (2.7)}}]
 \label{transformation}
 Let $(M^{2n+1},T^{1,0},\theta)$ be a pseudohermitian manifold.
 If $\htheta=e^\Upsilon\theta$, $\Upsilon\in C^\infty(M)$, then
 \begin{align*}
  \hS_{\alpha\bar\beta\gamma\bar\sigma} & = e^\Upsilon S_{\alpha\bar\beta\gamma\bar\sigma} , \\
  \hP_{\alpha\bar\beta} & = P_{\alpha\bar\beta} - \frac{1}{2}(\Upsilon_{\alpha\bar\beta}+\Upsilon_{\bar\beta\alpha}) - \frac{1}{2}\Upsilon_\mu\Upsilon^\mu h_{\alpha\bar\beta} , \\
  \hA_{\alpha\gamma} & = A_{\alpha\gamma} + i\Upsilon_{\alpha\gamma} -  i\Upsilon_\alpha\Upsilon_\gamma .
 \end{align*}
 where admissible coframes for $\theta$ and $\htheta$ are related as in \cref{coframe-transformation}.
 Additionally, if $\omega_\alpha\,\theta^\alpha\in\Omega^{1,0}M$, then
 \begin{align*}
  \hnabla_\alpha\omega_\gamma & = \nabla_\alpha\omega_\gamma - \Upsilon_\alpha\omega_\gamma - \Upsilon_\gamma\omega_\alpha , \\
  \hnabla_{\bar\beta}\omega_\gamma & = \nabla_{\bar\beta}\omega_\alpha + \Upsilon^\mu\omega_\mu h_{\gamma\bar\beta} .
 \end{align*}
\end{lemma}

By \cref{transformation}, the vanishing of the Chern tensor is a CR invariant condition.
Indeed, if $n\geq2$, then a CR manifold $(M^{2n+1},T^{1,0})$ is locally CR equivalent to the standard CR sphere if and only if the Chern tensor vanishes~\cite{ChernMoser1974}.

\begin{definition}
 A CR manifold $(M^{2n+1},T^{1,0})$, $n\geq2$, is \defn{locally spherical} if the Chern tensor of any local contact form vanishes.
\end{definition}

When $n=1$, the role of the Chern tensor is played instead by the Cartan tensor~\cites{Cartan1932a,Cartan1932b}.

%% file: bg/linear-algebra.tex
\section{The Lefschetz operator}
\label{sec:linear-algebra}

The Levi form of a pseudohermitian manifold determines a symplectic structure on the fibers of the bundles $\Lambda^{p,q}$.
In this section we describe some consequences of this structure which are needed for our constructions of the Rumin complex and the bigraded Rumin complex.
Note that all of the results of this section are local;
we only describe them globally to avoid the introduction of new notation.

We begin by introducing three objects determined by a choice of contact form.
First, we define the Lefschetz operator on $\theta\wedge\Omega^{p,q}M$.

\begin{definition}
 Let $(M^{2n+1},T^{1,0},\theta)$ be a pseudohermitian manifold.
 The \defn{Lefschetz operator} $\mL \colon \theta \wedge \Omega^{p,q}M \to \theta\wedge \Omega^{p+1,q+1}M$ is given by
 \begin{equation*}
  \mL( \theta \wedge \omega ) := \theta \wedge \omega \wedge d\theta .
 \end{equation*}
\end{definition}

Second, we define the hermitian form on $\theta\wedge\Omega^{p,q}M$ induced by the Levi form.

\begin{definition}
 Let $(M^{2n+1},T^{1,0},\theta)$ be a pseudohermitian manifold.
 The \defn{hermitian form on $\theta\wedge\Omega^{p,q}M$} is
 \begin{equation*}
  \lp \theta \wedge \omega, \theta \wedge \tau \rp := \frac{1}{p!q!}\omega_{\Alpha\bar\Beta}\otau^{\bar\Beta\Alpha} ,
 \end{equation*}
 where $\omega=\frac{1}{p!q!}\omega_{\Alpha\bar\Beta}\,\theta^\Alpha\wedge\theta^{\bar\Beta}$ and $\tau=\frac{1}{p!q!}\tau_{\Alpha\bar\Beta}\,\theta^\Alpha\wedge\theta^{\bar\Beta}$, and we set $\otau_{\Beta\bar\Alpha}:=\overline{\tau_{\Alpha\bar\Beta}}$.
\end{definition}

Third, we define the adjoint Lefschetz operator.

\begin{definition}
 Let $(M^{2n+1},T^{1,0},\theta)$ be a pseudohermitian manifold.
 The \defn{adjoint Lefschetz operator} $\trace \colon \theta \wedge \Omega^{p,q}M \to \theta \wedge \Omega^{p-1,q-1}M$ is the adjoint of the Lefschetz operator;
 i.e.\ $\trace$ is defined by
 \begin{equation*}
  \lp \trace(\theta\wedge\omega), \theta\wedge\tau \rp = \lp \theta\wedge\omega, \mL(\theta\wedge\tau) \rp
 \end{equation*}
 for all $\omega\in\Omega^{p,q}M$ and all $\tau\in\Omega^{p-1,q-1}M$.
\end{definition}

For later computations, the following coordinate expressions of the Lefschetz operator and the adjoint Lefschetz operator are useful.

\begin{lemma}
 \label{lefschetz-computation}
 Let $(M^{2n+1},T^{1,0},\theta)$ be a pseudohermitian manifold.
 Then
 \begin{align*}
  \mL(\theta\wedge\omega) & = \frac{(-1)^p}{p!q!}ih_{\alpha\bar\beta}\omega_{\Alpha\bar\Beta}\,\theta \wedge \theta^{\alpha\Alpha}\wedge\theta^{\bar\beta\bar\Beta}, \\
  \trace(\theta\wedge\omega) & = \frac{(-1)^p}{(p-1)!(q-1)!}i\omega_{\mu\Alpha^\prime}{}^\mu{}_{\bar\Beta^\prime}\,\theta \wedge \theta^{\Alpha^\prime} \wedge \theta^{\bar\Beta^\prime}
 \end{align*}
 for all $\omega=\frac{1}{p!q!}\omega_{\Alpha\bar\Beta}\,\theta^{\Alpha}\wedge\theta^{\bar\Beta}\in\Omega^{p,q}M$.
\end{lemma}

\begin{proof}
 This is a straightforward computation.
\end{proof}

The Lefschetz decomposition is stated in terms of primitive elements of $\theta\wedge\Omega^{p,q}M$.

\begin{definition}
 \label{defn:theta-wedge-primitive}
 Let $(M^{2n+1},T^{1,0},\theta)$ be a pseudohermitian manifold.
 The space of \defn{primitive elements} of $\theta\wedge\Omega^{p,q}M$ is
 \begin{equation*}
  \theta \wedge P^{p,q}M := \ker \left( \trace \colon \theta\wedge\Omega^{p,q}M \to \theta\wedge\Omega^{p-1,q-1}M \right) .
 \end{equation*}
\end{definition}

\Cref{lefschetz-computation} implies that $\omega\in\theta\wedge\Omega^{p,q}M$ is primitive if and only if $\omega$ is trace-free.

The following theorem collects the properties of $\mL$ that we require.

\begin{proposition}
 \label{lefschetz-decomposition}
 Let $(M^{2n+1},T^{1,0},\theta)$ be a pseudohermitian manifold.
 If $\htheta=e^\Upsilon\theta$, then $\mL^{\htheta}=e^\Upsilon \mL^\theta$, where $\mL^\theta$ and $\mL^{\htheta}$ denote the Lefschetz operators determined by $\theta$ and $\htheta$, respectively.
 Moreover, if $p,q\in\bN_0$ are such that $p+q\leq n$, then
 \begin{enumerate}
  \item $\mL^{n-p-q} \colon \theta \wedge \Omega^{p,q}M \to \theta \wedge \Omega^{n-q,n-p}M$ is an isomorphism; and
  \item $\ker \left( \mL^{n+1-p-q} \colon \theta\wedge \Omega^{p,q}M \to \theta \wedge \Omega^{n+1-q,n+1-p}M\right) = \theta \wedge P^{p,q}M$.
 \end{enumerate}
\end{proposition}

\begin{proof}
 Let $\omega\in\Omega^{p,q}M$.
 \Cref{Omegapq-CR-invariant} yields an $\homega\in\hOmega^{p,q}M$ such that $\theta\wedge\omega=\htheta\wedge\homega$.
 A direct computation yields
 \begin{equation*}
  \mL^{\htheta}(\htheta\wedge\homega) = e^\Upsilon\,\theta\wedge\omega\wedge d\theta .
 \end{equation*}
 Therefore $\mL^{\htheta}=e^\Upsilon \mL^\theta$.
 
 Since the Levi form is nondegenerate, the restriction $d\theta\rv_{H}$ of $d\theta$ to the contact distribution $H$ is a symplectic form on $H$.
 Let $p,q\in\bN_0$ be such that $p+q\leq n$.
 Since $d\theta\rv_H\in\Lambda^{1,1}$, it holds that
 \begin{equation*}
  (d\theta\rv_H \wedge )^{n-p-q} \colon \Lambda^{p,q} \to \Lambda^{n-q,n-p}
 \end{equation*}
 is an isomorphism~\citelist{ \cite{Huybrechts2005}*{Proposition~1.2.30(iv)} }.
 The last claim follows from the alternative characterization~\cite{Huybrechts2005}*{Proposition~1.2.30(v)} of primitive elements of $\Lambda^{p,q}$.
\end{proof}

We can use the Lefschetz decomposition to identify, in a CR invariant way, the primitive part of a complex-valued differential form; see \cref{projection-to-R} below.
The key step is to identify, in a CR invariant way, the non-primitive part.

\begin{theorem}
 \label{inverse-lefschetz}
 Let $(M^{2n+1},T^{1,0},\theta)$ be a pseudohermitian manifold.
 If $k \leq n$, then for each $\omega \in \COmega^kM$ there is a unique $\theta \wedge \xi \in \theta \wedge \COmega^{k-2}M$ such that
 \begin{equation}
  \label{eqn:defn-Gamma-1}
  \theta \wedge \omega \wedge d\theta^{n+1-k} = \theta \wedge \xi \wedge d\theta^{n+2-k} .
 \end{equation}
 If $k \geq n+1$, then for each $\omega \in \COmega^kM$ there is a unique $\theta \wedge \xi \in \theta \wedge \COmega^{2n-k}M$ such that
 \begin{equation}
  \label{eqn:defn-Gamma-2}
  \theta \wedge \omega = \theta \wedge \xi \wedge d\theta^{k-n} .
 \end{equation}
 Moreover, the linear map $\Gamma \colon \COmega^kM \to \theta \wedge \COmega^{k-2}M$,
 \begin{equation*}
  \Gamma\omega :=
  \begin{cases}
  	\theta \wedge \xi , & \text{if $0 \leq k \leq n$}, \\
  	\theta \wedge \xi \wedge d\theta^{k-n-1} , & \text{if $n+1 \leq k \leq 2n+1$} ,
  \end{cases}
 \end{equation*}
 is such that
 \begin{enumerate}
  \item if $\htheta = e^\Upsilon\theta$, then $\Gamma^{\htheta} = \Gamma^\theta$;
  \item $\Gamma(\theta \wedge \omega) = 0$ for all $\omega \in \COmega^kM$;
  \item $\Gamma(\omega \wedge d\theta) = \theta \wedge \omega$ for all $\omega \in \COmega^{k-2}M$, $k \leq n+1$; and
  \item if $\omega\rv_H \in \Lambda^{p,q}$, then $\Gamma\omega \in \theta \wedge \Omega^{p-1,q-1}M$.
 \end{enumerate}
\end{theorem}

\begin{proof}
 We first show that $\Gamma$ is well-defined.
 Let $\omega \in \COmega^kM$.
 
 Suppose first that $k \leq n$.
 \Cref{lefschetz-decomposition} implies that
 \begin{equation*}
  \mL^{n+2-k} \colon \theta \wedge \COmega^{k-2}M \to \theta \wedge \COmega^{2n+2-k}M
 \end{equation*}
 is an isomorphism.
 Thus there is a unique $\theta \wedge \xi \in \theta \wedge \COmega^{k-2}M$ such that Equation~\eqref{eqn:defn-Gamma-1} holds.
 
 Suppose next that $k \geq n+1$.
 \Cref{lefschetz-decomposition} implies that
 \begin{equation*}
  \mL^{k-n} \colon \theta \wedge \COmega^{2n-k}M \to \theta \wedge \COmega^kM
 \end{equation*}
 is an isomorphism.
 Thus there is a unique $\theta \wedge \xi \in \COmega^{2n-k}M$ such that Equation~\eqref{eqn:defn-Gamma-2} holds.
 
 Since the above cases are exhaustive, $\Gamma$ is well-defined.
 
 It follows immediately from Equations~\eqref{eqn:defn-Gamma-1} and~\eqref{eqn:defn-Gamma-2} that $\Gamma$ is CR invariant and that $\Gamma(\theta \wedge \omega) = 0$ for all $\omega \in \COmega^kM$.
 Moreover, Equations~\eqref{eqn:defn-Gamma-1} and~\eqref{eqn:defn-Gamma-2} readily imply that if $\omega \in \COmega^{k-2}M$, $k \leq n+1$, then $\Gamma(\omega \wedge d\theta) = \theta \wedge \omega$.
 The final claim follows directly from the application of \cref{lefschetz-decomposition}.
\end{proof}

%% file: part-intrinsic.tex
\part{The bigraded Rumin complex and cohomology}
\label{part:intrinsic}

In this part we construct the bigraded Rumin complex using subsheaves of the sheaf of complex-valued differential forms.
This approach gives concrete characterizations of the de Rham cohomology groups, the Kohn--Rossi cohomology groups, and the cohomology groups with coefficients in the sheaf of CR pluriharmonic functions, as well as an explicitly-described long exact sequence relating these groups.
This approach also identifies balanced $A_\infty$-structures on the Rumin and bigraded Rumin complexes.
These recover the usual algebra structure on de Rham cohomology and establish a natural algebra structure on Kohn--Rossi cohomology.
We stress that our constructions are CR invariant and apply to all CR manifolds.
A more specific description of these results, as well as the structure of this part of this article, is as follows:

In \cref{sec:rumin_complex} we define a complex of subsheaves of the sheaf of differential forms and show that it is equivalent to Rumin's complex~\cite{Rumin1990}.
Specifically, given a CR manifold $(M^{2n+1},T^{1,0})$, we define
\begin{align*}
 \sR^k & := \left\{ \omega \in \sA^k \suchthatcolon \theta \wedge \omega \wedge d\theta^{n+1-k}=0, \theta \wedge d\omega \wedge d\theta^{n-k}=0 \right\}, && \text{if $k\leq n$}, \\
 \sR^k & := \left\{ \omega \in \sA^k \suchthatcolon \theta \wedge \omega = 0, \theta \wedge d\omega = 0 \right\}, && \text{if $k\geq n+1$} .
\end{align*}
where $\theta$ is a local contact form and $\sA^k$ is the sheaf of differential $k$-forms.
We also show that the \defn{Rumin complex}
\begin{equation*}
 0 \longrightarrow \sR^0 \overset{d}{\longrightarrow} \sR^1 \overset{d}{\longrightarrow} \dotsm \overset{d}{\longrightarrow} \sR^{2n+1} \overset{d}{\longrightarrow} 0
\end{equation*}
is a complex.
When $k\geq n+1$, the sheaf $\sR^k$ is identical to the corresponding sheaf defined by Rumin;
when $k\leq n$, we construct a CR invariant isomorphism between $\sR^k$ and the corresponding sheaf defined by Rumin.
We define $H_R^k(M;\bR)$ as the $k$-th cohomology group of this complex.
Like Rumin's original construction~\cite{Rumin1990}, the results of \cref{sec:rumin_complex} are applicable on contact manifolds.

In \cref{sec:complex} we define a bigraded complex of subsheaves of the sheaf of complex-valued differential forms and show that it is equivalent to Garfield and Lee's complex~\cites{GarfieldLee1998,Garfield2001}.
Specifically, given a CR manifold $(M^{2n+1},T^{1,0})$, we define
\begin{align*}
 \sR^{p,q} & := \left\{ \omega \in \CsR^{p+q} \suchthatcolon \omega\rv_{\CH} \in \Lambda^{p,q} \right\}, && \text{if $p+q\leq n$}, \\
 \sR^{p,q} & := \left\{ \omega \in \CsR^{p+q} \suchthatcolon (T \contr \omega)\rv_{\CH} \in \Lambda^{p-1,q} \right\}, && \text{if $p+q\geq n+1$} ,
\end{align*}
where $\CsR^k:=\sR^k\otimes\bC$ and $T$ is a local Reeb vector field.
By splitting the exterior derivative according to types, we obtain the \defn{bigraded Rumin complex}
\begin{equation*}
 \begin{tikzpicture}[baseline=(current bounding box.center),yscale=0.8,xscale=1.0]
  \node (0_0) at (0,0) {$\sR^{0,0}$};
  \node (1_0) at (1,1) {$\sR^{1,0}$};
  \node (1_1) at (2,0) {$\sR^{1,1}$};
  \node (0_1) at (1,-1) {$\sR^{0,1}$};
  \node (n-1_0) at (3,3) {$\sR^{n-1,0}$};
  \node (n_0) at (4,4) {$\sR^{n,0}$};
  \node (n-1_1) at (4,2) {$\sR^{n-1,1}$};
  \node (n+1_1) at (7,3) {$\sR^{n+1,1}$};
  \node (n+1_0) at (6,4) {$\sR^{n+1,0}$};
  \node (n_1) at (6,2) {$\sR^{n,1}$};
  \node (0_n-1) at (3,-3) {$\sR^{0,n-1}$};
  \node (2_n) at (7,-3) {$\sR^{2,n}$};
  \node (2_n-1) at (6,-2) {$\sR^{2,n-1}$};
  \node (1_n) at (6,-4) {$\sR^{1,n}$};
  \node (0_n) at (4,-4) {$\sR^{0,n}$};
  \node (1_n-1) at (4,-2) {$\sR^{1,n-1}$};
  \node (n_n-1) at (8,0) {$\sR^{n,n-1}$};
  \node (n+1_n) at (10,0) {$\sR^{n+1,n}$.};
  \node (n+1_n-1) at (9,1) {$\sR^{n+1,n-1}$};
  \node (n_n) at (9,-1) {$\sR^{n,n}$};
  \draw [->] (0_0) to node[pos=0.8, left] {\tiny $\db$} (1_0);
  \draw [->] (1_0) to node[pos=0.8, left] {\tiny $\db$} (1.7,1.7) ;
  \draw [loosely dotted, line width=0.3mm] (1.9,1.9) -- (2.2,2.2);
  \draw [->] (2.3,2.3) to node[pos=0.8, left] {\tiny $\db$} (n-1_0);
  \draw [->] (n-1_0) to node[pos=0.8, left] {\tiny $\db$} (n_0);
  \draw [->] (0_1) to node[pos=0.8, left] {\tiny $\db$} (1_1);
  \draw [->] (1_1) to node[pos=0.8, left] {\tiny $\db$} (2.7, 0.7) ;
  \draw [loosely dotted, line width=0.3mm] (2.9,0.9) -- (3.2, 1.2);
  \draw [->] (3.3, 1.3) to node[pos=0.8, left] {\tiny $\db$} (n-1_1);
  \draw [->] (n-1_1) to node[pos=0.2, left] {\tiny $\db$} (n+1_0);
  \draw [loosely dotted, line width=0.3mm] (1.9,-1.9) -- (2.2,-2.2);
  \draw [->] (5.5,1.5) -- (n_1);
  \draw [->] (n_1) to node[pos=0.8, left] {\tiny $\db$} (n+1_1);
  \draw [->] (0_n-1) to node[pos=0.8, left] {\tiny $\db$} (1_n-1);
  \draw [->] (1_n-1) -- (4.5,-1.5);
  \draw [loosely dotted, line width=0.3mm] (4.6,-1.4) -- (4.9,-1.1);
  \draw [->] (0_n) to node[pos=0.2, left] {\tiny $\db$} (2_n-1);
  \draw [->] (2_n-1) to node[pos=0.8, left] {\tiny $\db$} (6.7,-1.3);
  \draw [loosely dotted, line width=0.3mm] (6.9,-1.1) -- (7.2,-0.8);
  \draw [->] (7.3,-0.7) to node[pos=0.8, left] {\tiny $\db$} (n_n-1);
  \draw [->] (n_n-1) to node[pos=0.8, left] {\tiny $\db$} (n+1_n-1);
  \draw [->] (1_n) to node[pos=0.8, left] {\tiny $\db$} (2_n);
  \draw [->] (2_n) to node[pos=0.8, left] {\tiny $\db$} (7.7,-2.3);
  \draw [loosely dotted, line width=0.3mm] (7.9,-2.1) -- (8.2,-1.8);
  \draw [->] (8.3,-1.7) to node[pos=0.8, left] {\tiny $\db$} (n_n);
  \draw [->] (n_n) to node[pos=0.8, left] {\tiny $\db$} (n+1_n);
  \draw [->] (0_0) to node[left] {\tiny $\dbbar$} (0_1);
  \draw [->] (0_1) to node[left] {\tiny $\dbbar$} (1.7,-1.7);
  \draw [loosely dotted, line width=0.3mm] (5.15,1.15) -- (5.5,1.5);
  \draw [->] (2.3,-2.3) to node[left] {\tiny $\dbbar$} (0_n-1);
  \draw [->] (0_n-1) to node[left] {\tiny $\dbbar$} (0_n);
  \draw [->] (1_0) to node[left] {\tiny $\dbbar$} (1_1);
  \draw [->] (1_1) to node[left] {\tiny $\dbbar$} (2.7,-0.7);
  \draw [loosely dotted, line width=0.3mm] (2.9,-0.9) -- (3.2,-1.2);
  \draw [->] (3.3,-1.3) to node[left] {\tiny $\dbbar$} (1_n-1);
  \draw [->] (1_n-1) to node[pos=0.2, left] {\tiny $\dbbar$} (1_n);
  \draw [->] (5.5,-1.5) -- (2_n-1);
  \draw [loosely dotted, line width=0.3mm] (5.15,-1.15) -- (5.5,-1.5);
  \draw [->] (2_n-1) to node[left] {\tiny $\dbbar$} (2_n);
  \draw [->] (n-1_0) to node[pos=0.2, left] {\tiny $\dbbar$} (n-1_1);
  \draw [->] (n-1_1) -- (4.5,1.5);
  \draw [loosely dotted, line width=0.3mm] (4.6,1.4) -- (4.9,1.1) ;
  \draw [->] (n_0) to node[pos=0.2, left] {\tiny $\dbbar$} (n_1);
  \draw [->] (n_1) to node[left] {\tiny $\dbbar$} (6.7,1.3);
  \draw [loosely dotted, line width=0.3mm] (6.9,1.1) -- (7.2,0.8);
  \draw [->] (7.3,0.7) to node[left] {\tiny $\dbbar$} (n_n-1);
  \draw [->] (n_n-1) to node[left] {\tiny $\dbbar$} (n_n);
  \draw [->] (n+1_0) to node[left] {\tiny $\dbbar$} (n+1_1);
  \draw [->] (n+1_1) to node[left] {\tiny $\dbbar$} (7.7,2.3);
  \draw [loosely dotted, line width=0.3mm] (7.9,2.1) -- (8.2,1.8);
  \draw [->] (8.3,1.7) to node[left] {\tiny $\dbbar$} (n+1_n-1);
  \draw [->] (n+1_n-1) to node[left] {\tiny $\dbbar$} (n+1_n);
  \draw [->] (n_0) to node[above] {\tiny $\dhor$} (n+1_0);
  \draw [->] (n-1_1) to node[above] {\tiny $\dhor$} (n_1);
  \draw [->] (1_n-1) to node[above] {\tiny $\dhor$} (2_n-1);
  \draw [->] (0_n) to node[above] {\tiny $\dhor$} (1_n);
 \end{tikzpicture}
\end{equation*}
That this is a bigraded complex means that the sum of all compositions of two maps vanishes;
see \cref{justification-of-bigraded-complex} below for a more explicit statement.
It follows that for each $p\in\{0,\dotsc,n+1\}$, the \defn{Kohn--Rossi complex}
\begin{equation*}
 0 \longrightarrow \sR^{p,0} \overset{\dbbar}{\longrightarrow} \sR^{p,1} \overset{\dbbar}{\longrightarrow} \dotsm \overset{\dbbar}{\longrightarrow} \sR^{p,n} \overset{\dbbar}{\longrightarrow} 0 
\end{equation*}
is a complex.
We define $H_R^{p,q}(M)$ as the $q$-th cohomology group of this complex.
In particular, $H_R^{p,0}(M)$ is the space of global sections of the sheaf $\sO^p:=\ker\left( \dbbar \colon \sR^{p,0} \to \sR^{p,1} \right)$ of CR holomorphic $(p,0)$-forms.

There are two important comments to make about the bigraded Rumin complex relative to the literature.
First, we have made a deliberate change of notation from the prior literature~\cites{Garfield2001,GarfieldLee1998,AkahoriMiyajima1993,AkahoriGarfieldLee2002}, where $d^{\prime\prime}$ denotes $\dbbar$ and $d^\prime$ denotes the operator mapping $\sR^{p,q}$ to $\sR^{p+1,q}$.
On the one hand, our change of notation is such that one \emph{also} obtains a complex by following the operators $\db$ along a diagonal;
by contrast, it is not true that $d^\prime \circ d^\prime = 0$ in general.
On the other hand, the operators $\dbbar$ and $\db$ are always conjugates of one another, yielding an obvious isomorphism between the cohomology groups $H_R^{p,q}(M)$ and the corresponding cohomology groups defined in terms of the $\db$-complex.
Note that this change of notation complicates some notation for the spectral sequence associated to the bigrading $\sR^{p,q}$; see \cref{sec:frolicher}.
Second, our sheaves $\sR^{p,q}$, $p+q\geq n+1$, are identical to the corresponding sheaves defined by Garfield and Lee~\cites{Garfield2001,GarfieldLee1998}; and there is a CR invariant isomorphism between our sheaves $\sR^{p,q}$, $p+q\leq n$, and the corresponding sheaves defined by Garfield and Lee.

In \cref{sec:cohomology} we define the cohomology groups $H_R^k(M;\sP)$.
Specifically, given a CR manifold $(M^{2n+1},T^{1,0})$, we define sheaves $\sS^k$ by
\begin{equation*}
 \sS^k :=
 \begin{cases}
  \sR^0, & \text{if $k=0$}, \\
  \Real(\sR^{k,1} \oplus \sR^{k-1,2} \oplus \dotsm \oplus \sR^{1,k}), & \text{if $1\leq k\leq n-1$}, \\
  \sR^{k+1} , & \text{if $k\geq n$} .
 \end{cases}
\end{equation*}
It follows immediately that the \defn{CR pluriharmonic complex}
\begin{equation*}
 0 \longrightarrow \sS^0 \overset{id\dbbar}{\longrightarrow} \sS^1 \overset{d}{\longrightarrow} \sS^1 \overset{d}{\longrightarrow} \dotsm \overset{d}{\longrightarrow} \sS^{2n} \overset{d}{\longrightarrow} 0
\end{equation*}
is a complex.
Note that $\sP:=\ker\bigl( id\dbbar \colon \sS^0 \to \sS^1 \bigr)$ is the sheaf of CR pluriharmonic functions~\cite{Lee1988}.
We define $H_R^k(M;\sP)$ as the $k$-th cohomology group of this complex.

In \cref{sec:les} we give an explicit description of the long exact sequence
\[ \dotsm \longrightarrow H_R^k(M;\bR) \longrightarrow H_R^{0,k}(M) \longrightarrow H_R^k(M;\sP) \longrightarrow H_R^{k+1}(M;\bR) \longrightarrow \dotsm \]
using the bigraded Rumin complex.

In \cref{sec:wedge} we define balanced $A_\infty$-structures~\cites{Keller2001,Markl1992} on the Rumin and bigraded Rumin complexes.
The $A_\infty$-structure $(m_j)_{j\in\bN}$ on the Rumin complex has $m_1=d$ and is such that $m_2$ is induced by the exterior product on differential forms, $m_3$ is nonzero, but $m_j=0$ for $j\geq4$.
In particular, the Rumin complex with our $A_\infty$-structure is not a differential graded algebra.
The $A_\infty$-structure $(m_j^{\dbbar})_{j\in\bN}$ for the bigraded Rumin complex is similar, except that $m_1^{\dbbar}=\dbbar$.
These structures are similar to the $A_\infty$-structure on the differential forms of a symplectic manifold found by Tsai, Tseng and Yau~\cite{TsaiTsengYau2016}.
When combined with the isomorphisms in \cref{sec:resolution}, our $A_\infty$-structures on the Rumin and bigraded Rumin complex induce the usual algebra structure on de Rham cohomology and a new algebra structure on Kohn--Rossi cohomology, respectively.

In \cref{sec:resolution} we relate $H_R^k(M;\bR)$, $H_R^{p,q}(M)$, and $H_R^k(M;\sP)$ to the classical de Rham cohomology groups, Kohn--Rossi cohomology groups, and cohomology groups with coefficients in the sheaf of CR pluriharmonic functions, respectively.
For the de Rham groups, we adapt an argument of Rumin~\cite{Rumin1990} to show that the Rumin complex gives an acyclic resolution of the constant sheaf $\underline{\bR}$, and hence $H_R^k(M;\bR)$ is isomorphic to the $k$-th de Rham cohomology group.
For the Kohn--Rossi groups, we exhibit an isomorphism between $H_R^{p,q}(M)$ and the corresponding cohomology group determined by Tanaka's intrinsic characterization~\cite{Tanaka1975} of the Kohn--Rossi complex~\cite{KohnRossi1965}.
We also show that the Kohn--Rossi and CR pluriharmonic complexes give partial acyclic resolutions of the sheaves of CR holomorphic $(p,0)$-forms and CR pluriharmonic functions, respectively, on strictly pseudoconvex, locally embeddable CR manifolds.

\input{intrinsic/vector_space}
\input{intrinsic/complex}
\input{intrinsic/cohomology}
\input{intrinsic/les}
\input{intrinsic/wedge}
\input{intrinsic/resolution}

%% file: intrinsic/vector_space.tex
\section{The Rumin complex}
\label{sec:rumin_complex}

In this section we construct the Rumin complex as a CR invariant complex of subsheaves of the sheaf of real-valued differential forms.
Moreover, we show that our complex is isomorphic to Rumin's original construction~\cite{Rumin1990}.

The sheaves in our construction of the Rumin complex are defined as follows:

\begin{definition}
 Let $(M^{2n+1},T^{1,0})$ be a CR manifold.
 Given $k\in\{0,\dotsc,2n+1\}$, set
 \begin{align*}
  \sR^k & := \left\{ \omega\in\sA^k \suchthatcolon \theta\wedge\omega\wedge d\theta^{n+1-k}=0, \theta\wedge d\omega \wedge d\theta^{n-k}=0 \right\} , && \text{if $k \leq n$}, \\
  \sR^k & := \left\{ \omega \in \sA^k \suchthatcolon \theta \wedge \omega = 0, \theta \wedge d\omega = 0 \right\} , && \text{if $k \geq n+1$},
 \end{align*}
 where $\theta$ is a local contact form and $\sA^k$ is the sheaf of differential $k$-forms. 
\end{definition}

It is clear that $\sR^k$ is CR invariant and that $d(\sR^k) \subset \sR^{k+1}$ for all $k\in\{0,\dotsc,n\}$.
We thus obtain the Rumin complex:

\begin{definition}
 \label{defn:rumin-complex}
 Let $(M^{2n+1},T^{1,0})$ be a CR manifold.
 The \defn{Rumin complex} is
 \[ 0 \longrightarrow \sR^0 \overset{d}{\longrightarrow} \sR^1 \overset{d}{\longrightarrow} \dotsm \overset{d}{\longrightarrow} \sR^{2n+1} \overset{d}{\longrightarrow} 0 . \]
 Given $k\in\{0,\dotsc,2n+1\}$, we set
 \begin{equation*}
  H_R^k(M;\bR) := \frac{\ker\left(d\colon\mR^k\to\mR^{k+1}\right)}{\im\left(d\colon\mR^{k-1}\to\mR^k\right)} ,
 \end{equation*}
 where $\mR^k$ is the space of global sections of $\sR^k$, and $\mR^{-1}:=0$ and $\mR^{2n+2}:=0$.
\end{definition}

\begin{remark}
 Throughout this article we will use the convention that script fonts (e.g.\ $\sR$) denote sheaves and calligraphic fonts (e.g.\ $\mR$) denote their corresponding spaces of global sections.
\end{remark}

In \cref{sec:resolution}, we show that $H_R^k(M;\bR)$ is isomorphic to the $k$-th de Rham cohomology group.

The linear map $\Gamma$ from \cref{inverse-lefschetz} determines a CR invariant projection $\pi \colon \sA^k \to \sR^k$ which should be regarded as mapping $\omega \in \sA^k$ to its primitive part $\pi\omega \in \sR^k$.

\begin{lemma}
 \label{projection-to-R}
 Let $(M^{2n+1},T^{1,0})$ be a CR manifold.
 Given $\omega \in \sA^k$, the formula
 \begin{equation}
  \label{eqn:projection-to-R}
  \pi\omega := \omega - d\Gamma\omega - \Gamma d\omega
 \end{equation}
 defines a CR invariant projection $\pi \colon \sA^k \to \sR^k$.
 Moreover, if $k \leq n$, then
 \begin{equation*}
  \ker\pi = \left\{ \theta \wedge \alpha + \beta \wedge d\theta \suchthatcolon \alpha \in \sA^{k-1}, \beta \in \sA^{k-2} \right\} .
 \end{equation*}
\end{lemma}

\begin{proof}
 It is clear that $\pi$ is linear.
 We begin by showing that if $\omega \in \sA^k$, then $\pi\omega \in \sR^k$.
 
 Suppose first that $k \leq n$.
 Since $\Gamma$ takes values in $\theta \wedge \COmega^{k-2}M$, it holds that $\theta \wedge d\Gamma\omega = \Gamma\omega \wedge d\theta$.
 Therefore
 \begin{align*}
  \theta \wedge \pi\omega \wedge d\theta^{n+1-k} & = \theta \wedge \omega \wedge d\theta^{n+1-k} - \Gamma\omega \wedge d\theta^{n+2-k} , \\
  \theta \wedge d\pi\omega \wedge d\theta^{n-k} & = \theta \wedge d\omega \wedge d\theta^{n-k} - \Gamma d\omega \wedge d\theta^{n+1-k} .
 \end{align*}
 \Cref{inverse-lefschetz} now implies that $\pi\omega \in \sR^k$.
 
 Suppose next that $k \geq n+1$.
 We compute that
 \begin{align*}
  \theta \wedge \pi\omega & = \theta \wedge \omega - \Gamma \omega \wedge d\theta , \\
  \theta \wedge d\pi\omega & = \theta \wedge d\omega - \Gamma d\omega \wedge d\theta .
 \end{align*}
 \Cref{inverse-lefschetz} now implies that $\pi\omega \in \sR^k$.
 
 Next, \cref{inverse-lefschetz} implies that if $\omega \in \sR^k$, then $\Gamma\omega = 0$.
 Since $d(\sR^k) \subset \sR^{k+1}$, we conclude that $\pi$ is a projection.
 
 Suppose now that $\alpha \in \sA^{k-1}$ and $\beta \in \sA^{k-2}$, $k \leq n$.
 \Cref{inverse-lefschetz} implies that $\Gamma(\theta \wedge \alpha) = 0$ and $\Gamma d(\theta \wedge \alpha) = \theta \wedge \alpha$.
 Hence $\pi(\theta \wedge \alpha) = 0$.
 \Cref{inverse-lefschetz} also implies that $\Gamma(\beta \wedge d\theta) = \theta \wedge \beta$ and $\Gamma(d\beta \wedge d\theta) = \theta \wedge d\beta$.
 Hence $\pi(\beta \wedge d\theta) = 0$.
 
 Finally, suppose that $\omega \in \sA^k$, $k \leq n$, is such that $\pi\omega = 0$.
 Then $\omega = \Gamma d\omega + d\Gamma\omega$.
 Since $\Gamma(\sA^k) \subset \theta \wedge \sA^{k-2}$, we conclude that $\omega \in \theta \wedge \sA^{k-1} + d\theta \wedge \sA^{k-2}$.
\end{proof}

We now show that $(\sR^k,d)$ is isomorphic to the complex constructed by Rumin~\cite{Rumin1990}.
In fact, it is clear that if $k \geq n+1$, then our space $\sR^k$ coincides with the space $\sJ^k$ defined by Rumin.
However, when $k \leq n$, the $k$-th space in Rumin's formulation of his complex is the equivalence class $\sA^k / \sI^k$, where $\sI^k$ is the intersection of $\sA^k$ with the ideal generated by $\theta$ and $d\theta$.
The projection $\pi \colon \sA^k \to \sR^k$ induces an isomorphism between $\sR^k$ and Rumin's $k$-th space.

\begin{corollary}
 \label{explicit-rumin}
 Let $(M^{2n+1},T^{1,0})$ be a CR manifold and let $k\in\{0,\dotsc,n\}$.
 Set
 \[ \sI^k := \left\{ \theta\wedge\alpha + \beta\wedge d\theta \suchthatcolon \alpha \in \sA^{k-1} , \beta\in\sA^{k-2} \right\} , \]
 where $\theta$ is a local contact form.
 Then $\pi \colon \sA^k \to \sR^k$ induces a CR invariant isomorphism $\sA^k / \sI^k \cong \sR^k$.
\end{corollary}

\begin{proof}
 It is clear that $\sI^k$ and $\sA^k/\sI^k$ are CR invariant.
 We conclude from \cref{projection-to-R} that $\pi$ induces a CR invariant isomorphism $\sA^k / \sI^k \to \sR^k$.
\end{proof}

We conclude this section with the useful observation that the projection $\pi$ and the exterior derivative $d$ commute.

\begin{lemma}
 \label{projection-and-d-commute}
 Let $(M^{2n+1},T^{1,0})$ be a CR manifold.
 Then $\pi d = d\pi$.
\end{lemma}

\begin{proof}
 This follows immediately from Equation~\eqref{eqn:projection-to-R}.
\end{proof}

%% file: intrinsic/complex.tex
\section{The bigraded Rumin complex}
\label{sec:complex}

In this section we construct the bigraded Rumin complex as a CR invariant bigraded complex of subsheaves of the sheaf of complex-valued differential forms.
We also include local explicit formulas for the operators in this complex.

The sheaves in the bigraded Rumin complex are obtained from the sheaves in the Rumin complex using the bigrading on $\Omega^{p,q}$, which in this section denotes the sheaf of local sections of $\Lambda^{p,q}$.

\begin{definition}
 \label{defn:sRpq}
 Let $(M^{2n+1},T^{1,0})$ be a CR manifold and let $p\in\{0,\dotsc,n+1\}$ and $q\in\{0,\dotsc,n\}$.
 Set
 \begin{align*}
  \sR^{p,q} & := \left\{ \omega\in\CsR^{p+q} \suchthatcolon \omega\rv_{\CH} \in \Omega^{p,q} \right\} , && \text{if $p+q\leq n$}, \\
  \sR^{p,q} & := \left\{ \omega\in\CsR^{p+q} \suchthatcolon (T\contr\omega)\rv_{\CH} \in \Omega^{p-1,q} \right\} , && \text{if $p+q\geq n+1$} ,
 \end{align*}
 where $\CsR^k:=\sR^k\otimes\bC$ and $T$ is a local Reeb vector field.
\end{definition}

Note that
\begin{equation*}
 \CsR^k =
 \begin{cases}
  \sR^{k,0} \oplus \dotsm \oplus \sR^{0,k} , & \text{if $k\leq n$}, \\
  \sR^{n+1,k-n-1} \oplus \dotsm \oplus \sR^{k-n,n}, & \text{if $k\geq n+1$} .
 \end{cases}
\end{equation*}
We denote by $\pi^{p,q}\colon\CsR^{p+q}\to\sR^{p,q}$ the canonical projection.
These projections are CR invariant.

We begin by recording the mapping properties of conjugation.

\begin{lemma}
 \label{conjugate-rumin-bundles}
 Let $(M^{2n+1},T^{1,0})$ be a CR manifold and let $p\in\{0,\dotsc,n+1\}$ and $q\in\{0,\dotsc,n\}$.
 \begin{enumerate}
  \item If $p+q\leq n$, then conjugation defines an isomorphism $\sR^{p,q} \cong \sR^{q,p}$.
  \item If $p+q\geq n+1$, then conjugation defines an isomorphism $\sR^{p,q} \cong \sR^{q+1,p-1}$.
 \end{enumerate}
\end{lemma}

\begin{proof}
 It is clear that conjugation defines an isomorphism $\Omega^{p,q}\cong\Omega^{q,p}$.
 The conclusion readily follows from \cref{defn:sRpq}.
\end{proof}

It is useful to have a local expression for the projection of \cref{projection-to-R}.
For clarity, we separately handle the cases $k\leq n$ and $k \geq n+1$.
The general formula for $\pi$ follows from the following lemmas by bilinearity.

\begin{lemma}
 \label{projection-to-E-expression}
 Let $(M^{2n+1},T^{1,0})$ be a CR manifold and let $\omega\in\CsA^k$, $k\leq n$, be such that $\omega\rv_{\CH}\in\Omega^{p,q}$.
 Then there is a unique $\frac{1}{p!q!}\omega_{\Alpha\bar\Beta}\,\theta^{\Alpha}\wedge\theta^{\bar\Beta}\in\Omega^{p,q}$ such that $\omega_{\mu\Alpha^\prime}{}^{\mu}{}_{\bar\Beta^\prime}=0$ and
 \[ \omega\rv_{\CH} \equiv \frac{1}{p!q!}\omega_{\Alpha\bar\Beta}\,\theta^{\Alpha}\wedge\theta^{\bar\Beta} \mod \im\left( (d\theta\rv_{\CH}\wedge) \colon \Omega^{p-1,q-1} \to \Omega^{p,q} \right) . \]
 Moreover, for each local contact form $\theta$ for $(M^{2n+1},T^{1,0})$, it holds that
 \begin{multline}
  \label{eqn:projection-to-E-expression}
  \pi\omega := \frac{1}{p!q!}\Bigl(\omega_{A\bar B}\,\theta^{A}\wedge\theta^{\bar\Beta} - \frac{pi}{n-p-q+1}\nabla^\mu\omega_{\mu A^\prime\bar B}\,\theta\wedge\theta^{A^\prime}\wedge\theta^{\bar\Beta} \\
   + \frac{(-1)^pqi}{n-p-q+1}\nabla^{\bar\nu}\omega_{A\bar\nu\bar B^\prime}\,\theta\wedge\theta^{A}\wedge\theta^{\bar B^\prime}\Bigr)  ,
 \end{multline}
 where the right-hand side is computed with respect to an admissible coframe.
\end{lemma}

\begin{proof}
 The existence and uniqueness of $\frac{1}{p!q!}\omega_{A\bar B}\,\theta^{\Alpha}\wedge\theta^{\bar\Beta}$ is guaranteed by the Lefschetz decomposition~\cite{Huybrechts2005}*{Proposition~1.2.30(i)} applied to $\Omega^{p,q}$.
 
 Let $\theta$ be a local contact form for $(M^{2n+1},T^{1,0})$ and let $\comega$ be given by the right-hand side of Equation~\eqref{eqn:projection-to-E-expression}.
 Clearly $\comega\equiv\omega\mod\theta,d\theta$.
 Since $\omega_{\mu\Alpha^\prime}{}^{\mu}{}_{\bar\Beta^\prime}=0$, we conclude from \cref{lefschetz-computation,lefschetz-decomposition} that $\theta\wedge\comega\wedge d\theta^{n+1-k}=0$.
 A direct computation yields
 \begin{align*}
  d\comega & \equiv \frac{1}{p!q!}\left(\nabla_{\alpha}\omega_{A\bar B} - \frac{q}{n-p-q+1}h_{\alpha\bar\beta}\nabla^{\bar\nu}\omega_{A\bar\nu\bar B^\prime}\right)\,\theta^{\alpha A}\wedge\theta^{\bar B} \\
   & \quad + \frac{(-1)^p}{p!q!}\left(\nabla_{\bar\beta}\omega_{A\bar B} - \frac{p}{n-p-q+1}h_{\alpha\bar\beta}\nabla^\mu\omega_{\mu A^\prime\bar B}\right)\,\theta^{A}\wedge\theta^{\bar\beta\bar B} \mod \theta .
 \end{align*}
 We readily compute that
 \begin{equation}
  \label{eqn:tracefree-trick}
  \begin{split}
   0 & = h^{\alpha\bar\beta} \left(\nabla_{[\alpha}\omega_{\Alpha\bar\Beta]} - \frac{q}{n-p-q+1}h_{[\alpha\bar\beta}\nabla^{\bar\nu}\omega_{\Alpha\bar\nu\bar\Beta^\prime]}\right), \\
   0 & = h^{\alpha\bar\beta} \left( \nabla_{[\bar\beta}\omega_{\Alpha\bar\Beta]} - \frac{p}{n-p-q+1}h_{[\alpha\bar\beta}\nabla^\mu\omega_{\mu\Alpha^\prime\bar\Beta]} \right) .
  \end{split}
 \end{equation}
 \Cref{lefschetz-computation,lefschetz-decomposition} then imply that $\theta\wedge d\comega\wedge d\theta^{n-k}=0$.
 Therefore $\comega\in\sR^{p,q}$.
 The uniqueness and CR invariance of $\comega$ follows from \cref{projection-to-R}.
\end{proof}

\begin{lemma}
 \label{projection-to-F-expression}
 Let $(M^{2n+1},T^{1,0})$ be a CR manifold and let $\omega\in\sR^{p,q}$, $p+q\geq n+1$.
 For each local contact form $\theta$, there is a unique $\tau\in\Omega^{n-q,n+1-p}$ such that
 \[ \tau = \frac{1}{(n+1-p)!(n-q)!}\tau_{\Alpha\bar\Beta}\,\theta^{\Alpha}\wedge\theta^{\bar\Beta}, \qquad \tau_{\mu\Alpha^\prime}{}^\mu{}_{\bar\Beta^\prime} = 0, \]
 and
 \begin{equation}
  \label{eqn:projection-to-F-expression}
  \omega = \frac{1}{(p+q-n-1)!}\theta \wedge \tau \wedge d\theta^{p+q-n-1} .
 \end{equation}
 Moreover, if $\htheta=e^\Upsilon\theta$, then $\htau=e^{(n-p-q)\Upsilon}\tau$, where $\htau$ is determined by $\htheta$ as above.
\end{lemma}

\begin{proof}
 Let $\omega\in\sR^{p,q}$ and let $\theta$ be a local contact form for $(M^{2n+1},T^{1,0})$.
 Since $\theta\wedge\omega=0$, we conclude from \cref{Lambdapq-embedding} that $\omega\in\theta\wedge\Omega^{p-1,q}M$.
 Since $\omega\wedge d\theta=0$, we conclude from \cref{lefschetz-decomposition} that there is a unique $\theta \wedge \tau \in \theta \wedge P^{n-q,n+1-p}$ such that $\omega = \frac{1}{(p+q-n-1)!}L^{p+q-n-1}(\theta \wedge \tau)$.
 This yields Equation~\eqref{eqn:projection-to-F-expression}.
 \Cref{Lambdapq-embedding} determines the corresponding element $\tau\in\Omega^{n-q,n+1-p}$.
 
 It follows readily from uniqueness that $\htau = e^{(n-p-q)\Upsilon}\tau$ if $\htheta=e^\Upsilon\theta$.
\end{proof}

\Cref{projection-to-F-expression} provides a CR covariant bijection $\sR^{p,q}\cong\sR^{n-q,n+1-p}$, $p+q\geq n+1$.

\begin{corollary}
 \label{sF-to-sE}
 Let $(M^{2n+1},T^{1,0},\theta)$ be a pseudohermitian manifold.
 For each $\omega\in\sR^{p,q}$, $p+q\geq n+1$, there is a unique $\tau\in\sR^{n-q,n+1-p}$ such that
 \[ \omega = \frac{1}{(p+q-n-1)!}\theta \wedge \tau \wedge d\theta^{p+q-n-1} . \]
 Moreover, if $\htheta=e^\Upsilon\theta$, then $\htau=e^{(n-p-q)\Upsilon}\tau$.
\end{corollary}

\begin{proof}
 Let $\omega\in\sR^{p,q}$ and let $\ctau$ be the unique element of $\Omega^{n-q,n+1-p}$ determined by \cref{projection-to-F-expression}.
 Let $\tau \in \sR^{n-q,n+1-p}$ be determined by $\ctau$ and \cref{projection-to-E-expression}.
 Then $\tau \equiv \ctau \mod \theta$.
 The conclusion readily follows.
\end{proof}

The operators in the bigraded Rumin complex are obtained from the exterior derivative via the projections $\pi^{p,q}$.
Since $T^{1,0}$ is integrable, $d$ splits into two operators on $\sR^{p,q}$, $p+q\not=n$, and into three operators on $\sR^{p+q}$, $p+q=n$.
A direct computation yields local formulas for these operators.

\begin{lemma}
 \label{formula-for-exterior-derivative}
 Let $(M^{2n+1},T^{1,0})$ be a CR manifold, let $\omega\in\sR^{p,q}$, and let $\theta$ be a local contact form.
 \begin{enumerate}
  \item If $p+q\leq n-1$, then
  \begin{align*}
   d\omega & \equiv \frac{1}{p!q!}\left( \nabla_{\alpha}\omega_{\Alpha\bar\Beta} - \frac{q}{n-p-q+1}h_{\alpha\bar\beta}\nabla^{\bar\nu}\omega_{\Alpha\bar\nu\bar\Beta^\prime} \right)\,\theta^{\alpha\Alpha}\wedge\theta^{\bar\Beta} \\
    & \quad + \frac{(-1)^p}{p!q!}\left( \nabla_{\bar\beta}\omega_{\Alpha\bar\Beta} - \frac{p}{n-p-q+1}h_{\alpha\bar\beta}\nabla^\mu\omega_{\mu\Alpha^\prime\bar\Beta} \right)\,\theta^{\Alpha}\wedge\theta^{\bar\beta\bar\Beta} \mod\theta ,
  \end{align*}
  where $\omega\equiv\frac{1}{p!q!}\omega_{\Alpha\bar\Beta}\,\theta^{\Alpha}\wedge\theta^{\bar\Beta}\mod\theta$.
  \item If $p+q=n$, then
  \begin{align*}
   d\omega & = \frac{(-1)^{p-1}i}{p!(q-1)!}\left( \nabla_{\alpha}\nabla^{\bar\nu}\omega_{\Alpha\bar\nu\bar\Beta^\prime} + iA_{\alpha}{}^{\bar\nu}\omega_{\Alpha\bar\nu\bar\Beta^\prime} \right)\,\theta\wedge\theta^{\alpha\Alpha}\wedge\theta^{\bar\Beta^\prime} \\
    & \quad + \frac{1}{p!q!}\left( \nabla_0\omega_{\Alpha\bar\Beta} + pi\nabla_{\alpha}\nabla^\mu\omega_{\mu \Alpha^\prime\bar\Beta} - qi\nabla_{\bar\beta}\nabla^{\bar\nu}\omega_{\Alpha\bar\nu\bar\Beta^\prime} \right)\,\theta\wedge\theta^{\Alpha}\wedge\theta^{\bar\Beta} \\
    & \quad + \frac{(-1)^{p-1}i}{(p-1)!q!}\left( \nabla_{\bar\beta}\nabla^\mu\omega_{\mu\Alpha^\prime\bar\Beta} - iA_{\bar\beta}{}^\mu\omega_{\mu\Alpha^\prime\bar\Beta} \right)\,\theta\wedge\theta^{\Alpha^\prime}\wedge\theta^{\bar\beta\bar\Beta} ,
  \end{align*}
  where $\omega\equiv\frac{1}{p!q!}\omega_{\Alpha\bar\Beta}\,\theta^{\Alpha}\wedge\theta^{\bar\Beta}\mod\theta$.
  \item If $p+q\geq n+1$, then
  \begin{align*}
   d\omega & = \frac{(-1)^{n-q}i}{(n-p)!(n-q)!(p+q-n)!}\nabla^{\bar\nu}\tau_{\Alpha\bar\nu\bar\Beta^\prime}\,\theta\wedge\theta^{\Alpha}\wedge\theta^{\bar\Beta^\prime}\wedge d\theta^{p+q-n} \\
    & \quad - \frac{i}{(n+1-p)!(n-q-1)!(p+q-n)!}\nabla^\mu\tau_{\mu\Alpha^\prime\bar\Beta}\,\theta\wedge\theta^{\Alpha^\prime}\wedge\theta^{\bar\Beta}\wedge d\theta^{p+q-n} ,
  \end{align*}
  where $\omega=\frac{1}{(p+q-n-1)!}\theta\wedge\tau\wedge d\theta^{p+q-n-1}$ for $\tau\equiv\frac{1}{(n+1-p)!(n-q)!}\tau_{\Alpha\bar\Beta}\,\theta^\Alpha\wedge\theta^{\bar\Beta}\mod\theta$ as in \cref{sF-to-sE}.
 \end{enumerate}
\end{lemma}

\begin{proof}
 First suppose that $p+q\leq n-1$.
 Then the direct computation in the proof of \cref{projection-to-E-expression} yields the claimed formula.
 
 Next suppose that $p+q=n$.
 Let $\omega\in\sR^{p,q}$.
 Write $\omega\equiv\frac{1}{p!q!}\omega_{A\bar B}\,\theta^A\wedge\theta^{\bar B}\mod\theta$.
 \Cref{projection-to-E-expression} implies that
 \begin{multline}
  \label{eqn:compute-domega-middle}
  \omega = \frac{1}{p!q!}\Bigl(\omega_{\Alpha\bar\Beta}\,\theta^{\Alpha}\wedge\theta^{\bar\Beta} - pi\nabla^\mu\omega_{\mu\Alpha^\prime\bar\Beta}\,\theta\wedge\theta^{\Alpha^\prime}\wedge\theta^{\bar\Beta} \\
   + (-1)^pqi\nabla^{\bar\nu}\omega_{\Alpha\bar\nu\bar\Beta^\prime}\,\theta\wedge\theta^{\Alpha}\wedge\theta^{\bar\Beta^\prime}\Bigr) .
 \end{multline}
 Since $d\omega\in\CsR^{n+1}$, it holds that $\theta\wedge d\omega=0$.
 In particular, $d\omega$ is in the ideal generated by $\theta$.
 Direct computation then implies the claimed formula.

 Finally suppose that $p+q\geq n+1$.
 Direct computation gives
 \begin{multline*}
  d\omega = -\frac{1}{(n+1-p)!(n-q)!(p+q-n-1)!}\Bigl( \nabla_{\alpha}\tau_{\Alpha\bar\Beta}\,\theta\wedge\theta^{\alpha\Alpha}\wedge\theta^{\bar\Beta} \\
   + (-1)^{n-q}\nabla_{\bar\beta}\tau_{\Alpha\bar\Beta}\,\theta\wedge\theta^{\Alpha}\wedge\theta^{\bar\beta\bar\Beta}\Bigr) \wedge d\theta^{p+q-n-1} .
 \end{multline*}
 Combining Equation~\eqref{eqn:tracefree-trick} with \cref{lefschetz-computation,lefschetz-decomposition} yields
 \begin{align*}
  0 & = \left(\nabla_{\alpha}\tau_{\Alpha\bar\Beta} - \frac{n+1-p}{p+q-n}h_{\alpha\bar\beta}\nabla^{\bar\nu}\tau_{\Alpha\bar\nu\bar\Beta^\prime} \right)\,\theta\wedge\theta^{\alpha\Alpha}\wedge\theta^{\bar\Beta}\wedge d\theta^{p+q-n-1}, \\
  0 & = \left( \nabla_{\bar\beta}\tau_{\Alpha\bar\Beta} - \frac{n-q}{p+q-n}h_{\alpha\bar\beta}\nabla^\mu\tau_{\mu\Alpha^\prime\bar\Beta} \right)\,\theta\wedge\theta^{\Alpha}\wedge\theta^{\bar\beta\bar\Beta}\wedge d\theta^{p+q-n-1} .
 \end{align*}
 Combining the previous two displays yields the desired formula.
\end{proof}

\Cref{formula-for-exterior-derivative} implies that if $\omega\in\sR^{p,q}$, then
\begin{equation*}
 d\omega \in
  \begin{cases}
   \sR^{p+1,q} \oplus \sR^{p,q+1}, & \text{if $p+q\not=n$}, \\
   \sR^{p+2,q-1} \oplus \sR^{p+1,q} \oplus \sR^{p,q+1}, & \text{if $p+q=n$} . 
  \end{cases}
\end{equation*}
This justifies the following definition.

\begin{definition}
 \label{defn:dbbar}
 Let $(M^{2n+1},T^{1,0})$ be a CR manifold.
 The operators in the bigraded Rumin complex are defined as follows:
 \begin{enumerate}
  \item If $p+q\not=n$, then
  \begin{align*}
   \db & := \pi^{p+1,q}\circ d \colon \sR^{p,q} \to \sR^{p+1,q}, \\
   \dbbar & := \pi^{p,q+1}\circ d \colon \sR^{p,q} \to \sR^{p,q+1} .
  \end{align*}
  \item If $p+q=n$, then
  \begin{align*}
   \db & := \pi^{p+2,q-1}\circ d \colon \sR^{p,q} \to \sR^{p+2,q-1}, \\
   \dhor & := \pi^{p+1,q}\circ d \colon \sR^{p,q} \to \sR^{p+1,q} , \\
   \dbbar & := \pi^{p,q+1}\circ d \colon \sR^{p,q} \to \sR^{p,q+1} .
  \end{align*}
 \end{enumerate}
\end{definition}

Our notation is chosen for compatibility with conjugation.

\begin{lemma}
 \label{conjugate-rumin-operators}
 Let $(M^{2n+1},T^{1,0})$ be a CR manifold.
 Then
 \[ \overline{\db\omega} = \dbbar\oomega \]
 for all $\omega\in\sR^{p,q}$.
 Moreover, if $\omega\in\sR^{p,q}$ and $p+q=n$, then
 \[ \overline{\dhor\omega} = \dhor\oomega . \]
\end{lemma}

\begin{proof}
 Since the exterior derivative is a real operator, it holds that $\overline{d\omega}=d\oomega$.
 The conclusion follows from \cref{conjugate-rumin-bundles,defn:dbbar}.
\end{proof}

\Cref{defn:sRpq,defn:dbbar} give the bigraded Rumin complex the structure of a bigraded complex.
This is encoded in the following proposition:

\begin{proposition}
 \label{justification-of-bigraded-complex}
 Let $(M^{2n+1},T^{1,0})$ be a CR manifold and let $p\in\{0,\dotsc,n+1\}$ and $q\in\{0,\dotsc,n\}$.
 It holds that
 \begin{enumerate}
  \item $\db^2=0$ and $\dbbar^2=0$ on $\sR^{p,q}$;
  \item if $p+q\not\in\{n-1,n\}$, then $\db\dbbar+\dbbar\db=0$ on $\sR^{p,q}$;
  \item if $p+q=n-1$, then $\dhor\db+\db\dbbar=0$ and $\dbbar\db+\dhor\dbbar=0$ on $\sR^{p,q}$; and
  \item if $p+q=n$, then $\dbbar\db+\db\dhor=0$ and $\dbbar\dhor+\db\dbbar=0$ on $\sR^{p,q}$.
 \end{enumerate}
\end{proposition}

\begin{proof}
 This follows immediately from the identity $d^2=0$. 
\end{proof}

In particular, following downward-pointing diagonal arrows yields a complex.

\begin{definition}
 \label{defn:kohn-rossi}
 Let $(M^{2n+1},T^{1,0})$ be a CR manifold and let $p\in\{0,\dotsc,n+1\}$.
 The \defn{$p$-th Kohn--Rossi complex} is
 \begin{equation*}
  0 \longrightarrow \sR^{p,0} \overset{\dbbar}{\longrightarrow} \sR^{p,1} \overset{\dbbar}{\longrightarrow} \dotsm \overset{\dbbar}{\longrightarrow} \sR^{p,n} \overset{\dbbar}{\longrightarrow} 0 .
 \end{equation*}
 Given an integer $q\in[0,n]$, we set
 \begin{equation*}
  H_R^{p,q}(M) := \frac{\ker\left(\dbbar\colon\mR^{p,q}\to\mR^{p,q+1}\right)}{\im\left(\dbbar\colon\mR^{p,q-1}\to\mR^{p,q}\right)} ,
 \end{equation*}
 where $\mR^{p,q}$ is the space of global sections of $\sR^{p,q}$, and $\mR^{p,-1}:=0$ and $\mR^{p,n+1}:=0$.
\end{definition}

In \cref{sec:resolution}, we show that $H_R^{p,q}(M)$ is isomorphic to the Kohn--Rossi cohomology group~\cites{KohnRossi1965,Tanaka1975} of bidegree $(p,q)$.

We conclude this section by recording explicit local formulas for the operators in the bigraded Rumin complex.

\begin{proposition}
 \label{bigraded-operators}
 Let $(M^{2n+1},T^{1,0})$ be a CR manifold and let $\omega\in\sR^{p,q}$.
 \begin{enumerate}
  \item If $p+q\leq n-1$, then
  \begin{align*}
   \db\omega & \equiv \frac{1}{(p+1)!q!}\Omega^{(1)}_{\alpha\Alpha\bar\Beta}\,\theta^{\alpha\Alpha}\wedge\theta^{\bar\Beta} \mod \theta, \\
   \dbbar\omega & \equiv \frac{1}{p!(q+1)!}\Omega^{(2)}_{\Alpha\bar\beta\bar\Beta}\,\theta^{\Alpha}\wedge\theta^{\bar\beta\bar\Beta} \mod \theta ,
  \end{align*}
  where
  \begin{align*}
   \Omega^{(1)}_{\alpha\Alpha\bar\Beta} & :=  (p+1)\left( \nabla_{[\alpha}\omega_{\Alpha\bar\Beta]} - \frac{q}{n-p-q+1}h_{[\alpha\bar\beta}\nabla^{\bar\nu}\omega_{\Alpha\bar\nu\bar\Beta^\prime]} \right) , \\
   \Omega^{(2)}_{\Alpha\bar\beta\bar\Beta} & := (q+1)(-1)^p\left( \nabla_{[\bar\beta}\omega_{\Alpha\bar\Beta]} - \frac{p}{n-p-q+1}h_{[\alpha\bar\beta}\nabla^\mu\omega_{\mu\Alpha^\prime\bar\Beta]} \right) ,
  \end{align*}
  and $\omega \equiv \frac{1}{p!q!}\omega_{\Alpha\bar\Beta}\,\theta^{\Alpha}\wedge\theta^{\bar\Beta}\mod\theta$.
  \item If $p+q=n$, then
  \begin{align*}
   \db\omega & = \frac{1}{(n-p-1)!(n+1-q)!}\Omega^{(1)}_{\alpha\Alpha\bar\Beta^\prime}\,\theta\wedge\theta^{\alpha\Alpha}\wedge\theta^{\bar\Beta^\prime}, \\
   \dhor\omega & = \frac{1}{(n-p)!(n-q)!}\Omega^{(2)}_{\Alpha\bar\Beta}\,\theta\wedge\theta^{\Alpha}\wedge\theta^{\bar\Beta}, \\
   \dbbar\omega & = \frac{1}{(n+1-p)!(n-q-1)!}\Omega^{(3)}_{\Alpha^\prime\bar\beta\bar\Beta}\,\theta\wedge\theta^{\Alpha^\prime}\wedge\theta^{\bar\beta\bar\Beta} ,
  \end{align*}
  where
  \begin{align*}
   \Omega^{(1)}_{\alpha\Alpha\bar\Beta^\prime} & := (-1)^{p-1}(p+1)i\left( \nabla_{[\alpha}\nabla^{\bar\nu}\omega_{\Alpha\bar\nu\bar\Beta^\prime]} + iA_{[\alpha}{}^{\bar\nu}\omega_{\Alpha\bar\nu\bar\Beta^\prime]}\right) , \\
   \Omega^{(2)}_{\Alpha\bar\Beta} & := \nabla_0\omega_{\Alpha\bar\Beta} + pi\nabla_{[\alpha}\nabla^\mu\omega_{\mu\Alpha^\prime\bar\Beta]} - qi\nabla_{[\bar\beta}\nabla^{\bar\nu}\omega_{\Alpha\bar\nu\bar\Beta^\prime]} , \\
   \Omega^{(3)}_{\Alpha^\prime\bar\beta\bar\Beta} & := (-1)^{p-1}(q+1)i\left( \nabla_{[\bar\beta}\nabla^\mu\omega_{\mu\Alpha^\prime\bar\Beta]} - iA_{[\bar\beta}{}^\mu\omega_{\mu\Alpha^\prime\bar\Beta]} \right) ,
  \end{align*}
  and $\omega\equiv\frac{1}{p!q!}\omega_{\Alpha\bar\Beta}\,\theta^{\Alpha}\wedge\theta^{\bar\Beta}\mod\theta$.
  \item If $p+q\geq n+1$, then
  \begin{align*}
   \db\omega & = \frac{1}{(n-p)!(n-q)!(p+q-n)!}\Omega^{(1)}_{\Alpha\bar\Beta^\prime}\,\theta\wedge\theta^{\Alpha}\wedge\theta^{\bar\Beta^\prime}\wedge d\theta^{p+q-n} , \\
   \dbbar\omega & = \frac{1}{(n+1-p)!(n-q-1)!(p+q-n)!}\Omega^{(2)}_{\Alpha^\prime\bar\Beta}\,\theta\wedge\theta^{\Alpha^\prime}\wedge\theta^{\bar\Beta}\wedge d\theta^{p+q-n} ,
  \end{align*}
  where
  \begin{align*}
   \Omega^{(1)}_{\Alpha\bar\Beta^\prime} & := (-1)^{n-q}i\nabla^{\bar\nu}\omega_{\Alpha\bar\nu\bar\Beta^\prime} , \\
   \Omega^{(2)}_{\Alpha^\prime\bar\Beta} & := -i\nabla^\mu\omega_{\mu\Alpha^\prime\bar\Beta} ,
  \end{align*}
  and $\omega=\frac{1}{(n+1-p)!(n-q)!(p+q-n-1)!}\omega_{\Alpha\bar\Beta}\,\theta\wedge\theta^{\Alpha}\wedge\theta^{\bar\Beta}\wedge d\theta^{p+q-n-1}$.
 \end{enumerate}
\end{proposition}

\begin{proof}
 This follows directly from \cref{formula-for-exterior-derivative,defn:dbbar}.
\end{proof}

%% file: intrinsic/cohomology.tex
\section{The CR pluriharmonic complex}
\label{sec:cohomology}

In this section we construct the CR pluriharmonic complex as a CR invariant complex of subsheaves of the sheaf of real-valued differential forms.

\begin{definition}
 \label{defn:pluriharmonic-complex}
 Let $(M^{2n+1},T^{1,0})$ be a CR manifold.
 For each $k\in\{0,\dotsc,2n\}$, set
 \[ \sS^k := \begin{cases}
  \sR^0, & \text{if $k=0$}, \\
  \Real\bigoplus_{j=1}^k \sR^{j,k+1-j}, & \text{if $1\leq k\leq n-1$}, \\
  \sR^{k+1}, & \text{if $k\geq n$} .
 \end{cases} \]
 Define $D\colon\sS^k\to\sS^{k+1}$ by
 \[ D = \begin{cases}
  id\dbbar, & \text{if $k=0$}, \\
  d, & \text{if $k\geq1$} .
 \end{cases} \]
 The \defn{CR pluriharmonic complex} is
 \[ 0 \longrightarrow \sS^0 \overset{D}{\longrightarrow} \sS^1 \overset{D}{\longrightarrow}  \dotsb \overset{D}{\longrightarrow} \sS^{2n} \longrightarrow 0 . \]
\end{definition}

It is clear that the CR pluriharmonic complex is a complex.
Our terminology is partially justified by the definition of the sheaf of CR pluriharmonic functions.

\begin{definition}
 \label{defn:sP}
 Let $(M^{2n+1},T^{1,0})$ be a CR manifold.
 The \defn{sheaf $\sP$ of CR pluriharmonic functions} is
 \[ \sP := \ker \left( D\colon\sS^0\to\sS^1 \right) . \]
\end{definition}

Lee observed~\cite{Lee1988} that $\sP=\Real\sO$, where $\sO := \ker \left( \dbbar \colon \sR^{0,0} \to \sR^{0,1} \right)$ is the sheaf of CR functions:

\begin{lemma}[\cite{Lee1988}*{Lemma~3.1}]
 \label{sP-characterization}
 Let $(M^{2n+1},T^{1,0})$ be a CR manifold.
 It holds that $\sP=\Real\sO$.
\end{lemma}

We define the cohomology groups $H^k(M;\sP)$ via the CR pluriharmonic complex in the usual way.

\begin{definition}
 \label{defn:pluriharmonic-cohomology}
 Let $(M^{2n+1},T^{1,0})$ be a CR manifold.
 Given $k\in\{0,\dotsc,2n\}$, set
 \[ H_R^k(M;\sP) := \frac{\ker\left(D\colon\mS^k\to\mS^{k+1}\right)}{\im\left(D\colon\mS^{k-1}\to\mS^k\right)} . \]
 where $\mS^k$ is the space of global sections of $\sS^k$, and $\mS^{-1}:=0$ and $\mS^{2n+1}:=0$.
\end{definition}

In \cref{sec:resolution}, we show that the CR pluriharmonic complex yields a partial acyclic resolution of $\sP$ on strictly pseudoconvex, locally embeddable CR manifolds.
In particular, $H_R^k(M^{2n+1};\sP) = H^k(M^{2n+1};\sP)$ for $k < n$.

%% file: intrinsic/les.tex
\section{The long exact sequence in cohomology}
\label{sec:les}

In this section we construct the long exact sequence connecting the cohomology groups $H_R^k(M;\bR)$, $H_R^{0,k}(M)$, and $H_R^k(M;\sP)$.
We begin by defining the maps in this long exact sequence.

\begin{definition}
 \label{defn:cohomology_maps}
 Let $(M^{2n+1},T^{1,0})$ be a CR manifold and let $k\in\bN_0$.
 \begin{enumerate}
  \item We define the morphism $H_R^k(M;\bR)\to H_R^{0,k}(M)$ by
  \begin{equation*}
   H_R^k(M;\bR) \ni [\omega] \mapsto [i\pi^{0,k}\omega] \in H_R^{0,k}(M) .
  \end{equation*}
  \item We define the morphism $H_R^{0,k}(M)\to H_R^k(M;\sP)$ by
  \begin{equation*}
   \begin{cases}
    H_R^{0,0}(M) \ni f \mapsto \Real f \in H_R^0(M;\sP), & \text{if $k=0$}, \\
    H_R^{0,k}(M) \ni [\omega] \mapsto [-\Imaginary d\omega] \in H_R^k(M;\sP), & \text{if $k\geq1$} .
   \end{cases} 
  \end{equation*}
  \item We define the morphism $H_R^k(M;\sP) \to H_R^{k+1}(M;\bR)$ by
  \begin{equation*}
   \begin{cases}
    H_R^0(M;\sP) \ni u \mapsto [-\Imaginary\dbbar u] \in H_R^1(M;\bR), & \text{if $k=0$}, \\
    H_R^k(M;\sP) \ni [\omega] \mapsto [\omega] \in H_R^{k+1}(M;\sP), & \text{if $k\geq1$}.
   \end{cases}
  \end{equation*}
 \end{enumerate}
\end{definition}

Since each morphism in \cref{defn:cohomology_maps} is defined by applying a CR invariant operator to a representative of the given cohomology class, we conclude that the morphisms themselves are CR invariant as soon as they are well-defined.
The next three lemmas show that these morphisms are well-defined.

\begin{lemma}
 \label{HkR-to-H0k-well-defined}
 Let $(M^{2n+1},T^{1,0})$ be a CR manifold.
 For each $k\in\bN_0$, the morphism $H_R^k(M;\bR)\to H_R^{0,k}(M)$ is well-defined and CR invariant. 
\end{lemma}

\begin{proof}
 Let $\tau\in\mR^{k-1}$, so that $[d\tau]=0$ in $H_R^k(M;\bR)$.
 Since $\pi^{0,k}d\tau=\dbbar\pi^{0,k-1}\tau$, we conclude that $H_R^k(M;\bR) \to H_R^{0,k}(M)$ is well-defined.
\end{proof}

\begin{lemma}
 \label{H0k-to-HkP-well-defined}
 Let $(M^{2n+1},T^{1,0})$ be a CR manifold.
 For each $k\in\bN_0$, the morphism $H_R^{0,k}(M)\to H_R^k(M;\sP)$ is well-defined and CR invariant. 
\end{lemma}

\begin{proof}
 Clearly $H_R^{0,0}(M)\to H_R^0(M;\sP)$ is well-defined.
 
 Let $\tau\in\mR^{0,k-1}$, $k\geq 1$, so that $[\dbbar\tau]=0$ in $H_R^{0,k}(M)$.
 Observe that
 \begin{align*}
  -\Imaginary d\dbbar\tau & = D\Real\tau , && \text{if $k=0$}, \\
  -\Imaginary d\dbbar\tau & = D\Imaginary\db\tau , && \text{if $1\leq k\leq n$} , \\
  -\Imaginary d\dbbar\tau & = 0 , && \text{if $k\geq n+1$} .
 \end{align*}
 Therefore $H_R^{0,k}(M) \to H_R^k(M;\sP)$ is well-defined.
\end{proof}

\begin{lemma}
 \label{HkP-to-HkR-well-defined}
 Let $(M^{2n+1},T^{1,0})$ be a CR manifold.
 For each $k\in\bN_0$, the morphism $H_R^k(M;\sP)\to H_R^{k+1}(M;\bR)$ is well-defined and CR invariant. 
\end{lemma}

\begin{proof}
 Clearly $H_R^0(M;\sP)\to H_R^1(M;\bR)$ is well-defined.
 
 Let $k\geq 1$.
 Since $\mS^k\subset\mR^{k+1}$ and $D\colon\mS^{k-1}\to\mS^k$ factors through the exterior derivative on the left, we conclude that $H_R^k(M;\sP) \to H_R^{k+1}(M;\bR)$ is well-defined.
\end{proof}

The main result of this section is that the morphisms of \cref{defn:cohomology_maps} determine a long exact sequence involving the groups $H_R^k(M;\bR)$, $H_R^{0,k}(M)$, and $H_R^k(M;\sP)$.

\begin{theorem}
 \label{long-exact-sequence}
 Let $(M^{2n+1},T^{1,0})$ be a CR manifold.
 Then
 \begin{equation}
  \label{eqn:long-exact-sequence}
  0 \longrightarrow H_R^0(M;\bR) \longrightarrow H_R^{0,0}(M) \longrightarrow H_R^0(M;\sP) \longrightarrow H_R^1(M;\bR) \longrightarrow \dotsm
 \end{equation}
 is a CR invariant long exact sequence.
\end{theorem}

\begin{proof}
 Note that if $v\in C^\infty(M)$, then $dv=0$ if and only if $\dbbar v=0$.
 It follows readily that~\eqref{eqn:long-exact-sequence} is exact at $H_R^0(M;\bR)$ and at $H_R^{0,0}(M)$.
 
 Lee's characterization~\cite{Lee1988}*{Lemma~3.1} of CR pluriharmonic functions which are globally the real part of a CR function implies that~\eqref{eqn:long-exact-sequence} is exact at $H_R^0(M;\sP)$.
 
 Clearly the composition $H_R^0(M;\sP)\to H_R^1(M;\bR)\to H_R^{0,1}(M)$ is the zero map.
 Let $[\omega]\in\ker\bigl(H_R^1(M;\bR)\to H_R^{0,1}(M)\bigr)$.
 Then there is an $f=u+iv\in\mR^{0,0}$ such that $\pi^{0,1}\omega=-i\dbbar f$.
 Hence
 \begin{equation}
  \label{eqn:exact-at-H1}
  \omega = 2\Imaginary \dbbar u + dv .
 \end{equation}
 Since $d\omega=0$, we conclude from Equation~\eqref{eqn:exact-at-H1} that $u \in H_R^0(M;\sP)$.
 Therefore~\eqref{eqn:long-exact-sequence} is exact at $H_R^1(M;\bR)$.
 
 Observe that if $\omega\in\mR^1$, then
 \[ -\Imaginary d(i\pi^{0,1}\omega) = -\frac{1}{2}d\omega .  \]
 Hence the composition $H_R^1(M;\bR)\to H_R^{0,1}(M)\to H_R^1(M;\sP)$ is the zero map.
 Suppose now that $[\omega]\in\ker\bigl(H_R^{0,1}(M)\to H_R^1(M;\sP)\bigr)$.
 Then there is a $u\in\mR^0$ such that
 \[ -\Imaginary d\omega = -\Imaginary d\dbbar u . \]
 Set $\tau:=2\Imaginary(\omega-\dbbar u)$.
 Then $\tau$ is closed and $\omega = i\pi^{0,1}\tau + \dbbar u$.
 Therefore~\eqref{eqn:long-exact-sequence} is exact at $H_R^{0,1}(M)$.

 Let $k\geq 1$.
 Clearly the composition $H_R^{0,k}(M)\to H_R^k(M;\sP)\to H_R^{k+1}(M;\bR)$ is the zero map.
 Suppose that $[\omega]\in\ker\bigl(H_R^k(M;\sP)\to H_R^{k+1}(M;\bR)\bigr)$.
 Then there is a $\tau\in\mR^k$ such that $\omega=d\tau$.
 Note that $\dbbar\pi^{0,k}\tau=0$.
 Set $\rho:=\tau-2\Real\pi^{0,k}\tau\in\mS^{k-1}$, so that $\omega=d\rho+2\Real d\pi^{0,k}\tau$;
 note that $\rho=0$ if $k=1$.
 Therefore~\eqref{eqn:long-exact-sequence} is exact at $H_R^k(M;\sP)$.

 Let $k\geq2$.
 Clearly the composition $H_R^{k-1}(M;\sP)\to H_R^k(M;\bR)\to H_R^{0,k}(M)$ is the zero map.
 Suppose that $[\omega]\in\ker\bigl(H_R^k(M;\bR)\to H_R^{0,k}(M)\bigr)$.
 Then there is a $\tau\in\mR^{0,k-1}$ such that $\pi^{0,k}\omega=\dbbar\tau$.
 Since $\omega$ is real-valued, $\pi^{k,0}\omega=\db\otau$.
 Therefore $\omega - d(2\Real\tau)\in\mS^{k-1}$.
 In particular, \eqref{eqn:long-exact-sequence} is exact at $H_R^k(M;\bR)$.
 
 Let $k\geq2$.
 First note that if $\omega\in\mR^k$, then
 \begin{equation*}
  D \left( \sum_{j=1}^{k-1} \pi^{j,k-j}\omega \right) = d\omega - 2\Imaginary di\pi^{0,k}\omega .
 \end{equation*}
 Therefore $H_R^k(M;\bR) \to H_R^{0,k}(M) \to H_R^k(M;\sP)$ is the zero map.
 Suppose now that $[\omega] \in \ker \bigl( H_R^{0,k}(M) \to H_R^k(M;\sP) \bigr)$.
 Then there is a $\tau\in\mS^{k-1}$ such that $-\Imaginary d\omega = d\tau$.
 Thus $\tau+\Imaginary\omega$ is closed and $2i\pi^{0,k}(\tau+\Imaginary\omega)=\omega$.
 In particular, \eqref{eqn:long-exact-sequence} is exact at $H_R^{0,k}(M)$ for all $k\geq2$.
\end{proof}

%% file: intrinsic/wedge.tex
\section{The $A_\infty$-structures on the bigraded Rumin complex}
\label{sec:wedge}

In this section we define operators which give $\sR^\bullet:=\bigoplus\sR^k$ and $\sR^{\bullet,\bullet}:=\bigoplus\sR^{p,q}$ the structure of sheaves of balanced $A_\infty$-algebras.
The induced multiplication on $H_R^\bullet(M;\bR)$ recovers the cup product on de Rham cohomology;
the induced multiplication on $H_R^{\bullet,\bullet}(M)$ is new.

We begin by recalling the definition of an $A_\infty$-algebra~\cite{Keller2001}.

\begin{definition}
 An \defn{$A_\infty$-algebra} is a pair $(A,d)$ consisting of a $\bZ$-graded vector space $A = \bigoplus_{n\in\bZ} A^n$ and a family $(m_k \colon A^{\otimes k} \to A)_{k\in\bN}$ of linear maps such that
 \begin{enumerate}
  \item $m_k$ is homogeneous of degree $2-k$ for each $k \in \bN$, in the sense that $m_k(A^{i_1} \otimes \dotsm \otimes A^{i_k}) \subset A^{i_1 + \dotsm + i_k + 2 - k}$ for all $i_1, \dotsc, i_k \in \bZ$; and
  \item for each $k \in \bN$, it holds that
  \begin{equation}
   \label{eqn:a-infinity-requirement}
   \sum_{r+s+t = k} (-1)^{r+st}m_{r+t+1}( 1^{\otimes r} \otimes m_s \otimes 1^{\otimes t}) = 0 ,
  \end{equation}
  where $1$ denotes the identity map, the tensor product of two homogeneous maps $f$ and $g$ is defined by the Koszul sign convention
  \begin{equation*}
  	(f \otimes g)(x \otimes y) := (-1)^{\lv g\rv \lv x\rv}f(x) \otimes g(y) ,
  \end{equation*}
  and vertical bars denote the degree.
 \end{enumerate}
\end{definition}

The cases $n=1$, $n=2$, and $n=3$ of Equation~\eqref{eqn:a-infinity-requirement} are
\begin{align}
 \label{eqn:A-infinity-first-product} m_1m_1 & = 0 , \\
 \label{eqn:A-infinity-second-product} m_1m_2 & = m_2(m_1 \otimes 1 + 1 \otimes m_1) , \\
 \label{eqn:A-infinity-third-product} m_2(1 \otimes m_2 - m_2 \otimes 1) & = m_1m_3 + m_3(m_1 \otimes 1^{\otimes 2} + 1 \otimes m_1 \otimes 1 + 1^{\otimes2} \otimes m_1) ,
\end{align}
respectively.
In particular, $(A,m_1)$ is a differential complex and $m_1$ is a graded derivation with respect to $m_2$.
The final equation states that $m_2$ is associative up to homotopy~\cite{Keller2001}.
In particular, the induced product on the cohomology $H^\bullet(A)$ of $(A,m_1)$ gives $H^\bullet(A)$ the structure of an associative algebra.

\begin{lemma}
 \label{a-infinity-cohomology}
 Let $(A,m)$ be an $A_\infty$-algebra.
 For each $k \in \bZ$, the cohomology group
 \begin{equation*}
  H^k(A) := \frac{\ker \left( m_1 \colon A^k \to A^{k+1} \right)}{\im \left( m_1 \colon A^{k-1} \to A^k \right)}
 \end{equation*}
 is well-defined.
 Moreover, $m_2$ induces an associative product on $\bigoplus_{k\in\bZ} H^k(A)$.
\end{lemma}

\begin{proof}
 Equation~\eqref{eqn:A-infinity-first-product} implies that $H^k(A)$ is well-defined for all $k \in \bZ$.
 Equation~\eqref{eqn:A-infinity-second-product} implies that $m_2$ induces a product $H^k(A) \otimes H^\ell(A) \to H^{k+\ell}(A)$.
 Equation~\eqref{eqn:A-infinity-third-product} implies that this product is associative. 
\end{proof}

We also require the notion of a balanced $A_\infty$-algebra~\cite{Markl1992}.

\begin{definition}
 \label{defn:balanced-a-infinity-algebra}
 An $A_\infty$-algebra $(A,d)$ is \defn{balanced} if
 \begin{equation*}
  m_{p+q} \circ \mu_{p,q} = 0
 \end{equation*}
 for all $p,q \in \bN$, where
 \begin{equation*}
  \mu_{p,q}(x_1 \otimes \dotsm \otimes x_n) := \sum_{\sigma \in Sh_{p,q}} \sgn(\sigma)\epsilon(\sigma;x_1,\dotsc,x_n) x_{\sigma^{-1}(1)} \otimes \dotsm \otimes x_{\sigma^{-1}(p+q)}
 \end{equation*}
 for $\epsilon(\sigma;x_1, \dotsc, x_n)$ the Koszul sign determined by
 \begin{equation*}
  x_1 \wedge \dotsm \wedge x_n = \sum_{\sigma \in S_n} \epsilon( \sigma; x_1, \dotsc, x_n) x_{\sigma^{-1}(1)} \otimes \dotsm \otimes x_{\sigma^{-1}(n)} ,
 \end{equation*}
 and $Sh_{p,q} \subset S_{p+q}$ is the subgroup
 \begin{equation*}
  Sh_{p,q} := \left\{ \sigma \in S_{p+q} \suchthatcolon \sigma(1) < \dotsm < \sigma(p), \quad \sigma(p+1) < \dotsm < \sigma(p+q) \right\}
 \end{equation*}
 of $(p,q)$-shuffles.
\end{definition}

We now turn to defining the $A_\infty$-structures on the Rumin complex.
Throughout this section we denote by $\lv\omega\rv$ the degree of $\omega \in \sR^\bullet$;
i.e.\ if $\omega \in \sR^k$, then $\lv\omega\rv := k$.
We first define the degree zero multiplication.

\begin{definition}
 \label{defn:cdot-product}
 Let $(M^{2n+1},T^{1,0})$ be a CR manifold.
 The \defn{exterior product on $\sR^\bullet$} is
 \begin{equation*}
  \omega \rwedge \tau := \pi ( \omega \wedge \tau ) ,
 \end{equation*}
 where $\pi$ is as in Equation~\eqref{eqn:projection-to-R}.
\end{definition}

Clearly the exterior product on $\sR^{\bullet}$ is a graded commutative $\bR$-linear map of degree $0$ on $\sR^{\bullet}$.
It also holds that $d$ is a graded derivation with respect to $\rwedge$.

\begin{lemma}
 \label{real-product-rule}
 Let $(M^{2n+1},T^{1,0})$ be a CR manifold.
 Then
 \begin{equation*}
  d(\omega\rwedge\tau) = d\omega \rwedge \tau + (-1)^{\lv\omega\rv}\omega \rwedge d\tau .
 \end{equation*}
\end{lemma}

\begin{proof}
 \Cref{projection-and-d-commute} implies that
 \begin{equation*}
  d(\omega \rwedge \tau) = \pi d(\omega \wedge \tau) = \pi \left( d\omega \wedge \tau + (-1)^{\lv\omega\rv} \omega \wedge d\tau \right) .
 \end{equation*}
 The conclusion follows from the linearity of $\pi$.
\end{proof}

The exterior product on $\sR^\bullet$ is not associative.
As we see below, its associator can be expressed in terms of the following operator:

\begin{definition}
 Let $(M^{2n+1},T^{1,0})$ be a CR manifold.
 The \defn{$m_3$-operator} is
 \begin{equation*}
  m_3(\omega \otimes \tau \otimes \eta) := \pi \left( \Gamma(\omega \wedge \tau) \wedge \eta - (-1)^{\lv\omega\rv} \omega \wedge \Gamma(\tau \wedge \eta) \right) ,
 \end{equation*}
 where $\Gamma$ is as in \cref{inverse-lefschetz}.
\end{definition}

The $m_3$-operator is clearly homogeneous of degree $-1$.
As our notation suggests, it is the third operator in the balanced $A_\infty$-structure on the Rumin complex.
To that end, we first note that the $m_3$-operator vanishes for a large range of degrees (cf.\ \cite{TsaiTsengYau2016}*{Equation~5.38}).

\begin{lemma}
 \label{A-infinity-m3-vanishing}
 Let $(M^{2n+1},T^{1,0})$ be a CR manifold.
 Let $\omega,\tau,\eta \in \sR^\bullet$ be homogeneous.
 If
 \begin{equation*}
 \lv\omega\rv + \lv\tau\rv + \lv\eta\rv \leq n+1 \quad\text{or}\quad \max \{ \lv\omega\rv, \lv\tau\rv, \lv\eta\rv \} \geq n+1 ,
 \end{equation*}
 then
 \begin{equation}
  \label{eqn:A-infinity-m3-vanishing}
  m_3(\omega \otimes \tau \otimes \eta) = 0 .
 \end{equation}
\end{lemma}

\begin{proof}
 Suppose first that $\lv\omega\rv + \lv\tau\rv + \lv\eta\rv \leq n+1 $.
 Note that
 \begin{equation*}
  \Gamma(\omega \wedge \tau) \wedge \eta - (-1)^{\lv\omega\rv} \omega \wedge \Gamma(\tau \wedge \eta) \in \theta \wedge \COmega^{\lv\omega\rv + \lv\tau\rv + \lv\eta\rv - 2} .
 \end{equation*}
 Since $\lv\omega\rv + \lv\tau\rv + \lv\omega\rv - 2 \leq n-1$, we conclude from \cref{projection-to-R} that Equation~\eqref{eqn:A-infinity-m3-vanishing} holds.
 
 Suppose second that $\max \{ \lv\omega\rv, \lv\eta\rv \} \geq n+1$.
 We present the case $\lv\eta\rv \geq n+1$;
 the case $\lv\omega\rv \geq n+1$ is similar.
 Since $\theta \wedge \eta = 0$, we conclude from \cref{inverse-lefschetz} that
 \begin{equation*}
  m_3(\omega \otimes \tau \otimes \eta) = \pi \left( \Gamma(\omega \wedge \tau) \wedge \eta \right) .
 \end{equation*}
 Since $\Gamma(\omega \wedge \tau) \in \theta \wedge \COmega^{\lv\omega\rv + \lv\tau\rv - 2}$ and $\theta \wedge \eta = 0$, we conclude that Equation~\eqref{eqn:A-infinity-m3-vanishing} holds.
 
 Suppose finally that $\lv\tau\rv \geq n+1$.
 Since $\theta \wedge \tau = 0$, we conclude from \cref{inverse-lefschetz} that $\Gamma(\omega \wedge \tau) = \Gamma(\tau \wedge \eta) = 0$.
 Therefore Equation~\eqref{eqn:A-infinity-m3-vanishing} holds.
\end{proof}

We now show that the exterior derivative, the exterior product on $\sR^\bullet$, and the $m_3$-operator together give the Rumin complex the structure of a balanced $A_\infty$-algebra.

\begin{theorem}
 \label{rumin-a-infinity}
 Let $(M^{2n+1},T^{1,0})$ be a CR manifold.
 Set $m_1 := d$ and $m_2 := \rwedge$ and $m_j := 0$ for all $j\geq4$.
 Then $(\sR^\bullet,m)$ is a balanced $A_\infty$-algebra.
\end{theorem}

\begin{proof}
 We first show that $(\sR^\bullet,m)$ is an $A_\infty$-algebra.
 It is immediate that Equation~\eqref{eqn:a-infinity-requirement} holds when $k=1$.
 \Cref{real-product-rule} implies that Equation~\eqref{eqn:a-infinity-requirement} holds when $k=2$.
 Since $m_j=0$ for $j\geq4$, we immediately see that Equation~\eqref{eqn:a-infinity-requirement} holds for $k\geq6$.
 It remains to consider the cases $k \in \{ 3, 4, 5 \}$.
 
 Suppose first that $k=3$.
 On the one hand, \cref{inverse-lefschetz} and Equation~\eqref{eqn:projection-to-R} imply that
 \begin{equation}
  \label{eqn:check-m3-m2m2}
  \begin{split}
   \MoveEqLeft m_2(1 \otimes m_2 - m_2 \otimes 1)(\omega \otimes \tau \otimes \eta) \\
    & = \pi \bigl( d\Gamma(\omega \wedge \tau) \wedge \eta + \Gamma d(\omega \wedge \tau) \wedge \eta - \omega \wedge d\Gamma(\tau \wedge \eta) - \omega \wedge \Gamma d(\tau \wedge \eta) \bigr) \\
    & = \pi \bigl( \Gamma d(\omega \wedge \tau) \wedge \eta - \omega \wedge \Gamma d(\tau \wedge \eta) \bigr) + d\Gamma(\omega \wedge \tau) \wedge \eta  \\
     & \quad - \omega \wedge d\Gamma(\tau \wedge \eta) - d\Gamma \left(d\Gamma(\omega \wedge \tau) \wedge \eta \right) + d\Gamma \left(\omega \wedge d\Gamma(\tau \wedge \eta) \right) \\
     & \quad - (-1)^{\lv\omega\rv + \lv\tau\rv}\Gamma \left( d\Gamma(\omega \wedge \tau) \wedge d\eta \right) + \Gamma \left( d\omega \wedge d\Gamma(\tau \wedge \eta) \right) .
  \end{split}
 \end{equation}
 On the other hand, direct computation using \cref{inverse-lefschetz} yields
 \begin{multline}
  \label{eqn:check-m3-m1m3}
  m_1m_3(\omega \otimes \tau \otimes \eta) = d\left(\Gamma(\omega \wedge \tau) \wedge \eta\right) - (-1)^{\lv\omega\rv}d\left(\omega \wedge \Gamma(\tau \wedge \eta) \right) \\
   - d\Gamma \left( d\Gamma(\omega \wedge \tau) \wedge \eta \right) + d\Gamma \left( \omega \wedge d\Gamma(\tau \wedge \eta) \right) 
 \end{multline}
 and
 \begin{align*}
  \MoveEqLeft m_3(m_1 \otimes 1^{\otimes 2} + 1 \otimes m_1 \otimes 1 + 1^{\otimes 2} \otimes m_1)(\omega \otimes \tau \otimes \eta) \\
   & = (-1)^{\lv\omega\rv}\pi \left( d\omega \wedge \Gamma(\tau \wedge \eta) + (-1)^{\lv\tau\rv}\Gamma(\omega \wedge \tau) \wedge d\eta \right) \\
    & \quad + \pi \bigl( \Gamma d(\omega \wedge \tau) \wedge \eta - \omega \wedge \Gamma d(\tau \wedge \eta) \bigr) .
 \end{align*}
 Therefore
 \begin{equation}
  \label{eqn:check-m3-m3m1}
  \begin{split}
  \MoveEqLeft m_3(m_1 \otimes 1^{\otimes 2} + 1 \otimes m_1 \otimes 1 + 1^{\otimes 2} \otimes m_1)(\omega \otimes \tau \otimes \eta) \\
   & = \pi\bigl( \Gamma d(\omega \wedge \tau) \wedge \eta - \omega \wedge \Gamma d(\tau \wedge \eta) \bigr) \\
    & \quad + (-1)^{\lv\omega\rv} d\omega \wedge \Gamma(\tau \wedge \eta) + (-1)^{\lv\omega\rv + \lv\tau\rv}\Gamma(\omega \wedge \tau) \wedge d\eta \\
    & \quad + \Gamma (d\omega \wedge d\Gamma(\tau \wedge \eta) ) - (-1)^{\lv\omega\rv + \lv\tau\rv}\Gamma (d\Gamma(\omega \wedge \tau) \wedge d\eta) .
  \end{split}
 \end{equation}
 Combining Equations~\eqref{eqn:check-m3-m2m2}, \eqref{eqn:check-m3-m1m3}, and~\eqref{eqn:check-m3-m3m1} yields Equation~\eqref{eqn:A-infinity-third-product}.
 
 Suppose next that $k=4$.
 Since $m_4=0$, Equation~\eqref{eqn:a-infinity-requirement} is equivalent to
 \begin{equation}
  \label{eqn:A-infinity-fourth-product}
  m_3\left( m_2 \otimes 1^{\otimes 2} - 1 \otimes m_2 \otimes 1 + 1^{\otimes 2} \otimes m_2 \right) = m_2\left( 1 \otimes m_3 + m_3 \otimes 1\right) .
 \end{equation}
 On the one hand, we compute using \cref{inverse-lefschetz} that
 \begin{align*}
  \MoveEqLeft m_3\left( m_2 \otimes 1^{\otimes 2} - 1 \otimes m_2 \otimes 1 + 1^{\otimes 2} \otimes m_2 \right)(\omega_1 \otimes \omega_2 \otimes \omega_3 \otimes \omega_4) \\
   & = \pi\Bigl( \Gamma\bigl(\pi(\omega_1\wedge\omega_2)\wedge\omega_3\bigr) \wedge \omega_4 - (-1)^{\lv\omega_1\rv+\lv\omega_2\rv}\pi(\omega_1\wedge\omega_2) \wedge \Gamma(\omega_3\wedge\omega_4) \\
    & \quad - \Gamma\bigl(\omega_1 \wedge \pi(\omega_2\wedge\omega_3)\bigr) \wedge \omega_4 + (-1)^{\lv\omega_1\rv}\omega_1 \wedge \Gamma\bigl(\pi(\omega_2\wedge\omega_3)\wedge\omega_4\bigr) \\
    & \quad + \Gamma(\omega_1\wedge\omega_2) \wedge \pi(\omega_3\wedge\omega_4) - (-1)^{\lv\omega_1\rv}\omega_1 \wedge \Gamma\bigl(\omega_2\wedge\pi(\omega_3\wedge\omega_4)\bigr) \Bigr) \\
   & = \pi\Bigl( \Gamma(\omega_1 \wedge \omega_2) \wedge \omega_3 \wedge \omega_4 - (-1)^{\lv\omega_1\rv + \lv\omega_2\rv}\omega_1 \wedge \omega_2 \wedge \Gamma(\omega_3 \wedge \omega_4) \\ 
    & \quad - (-1)^{\lv\omega_1\rv}\omega_1 \wedge \Gamma d\bigl( \Gamma(\omega_2 \wedge \omega_3) \wedge \omega_4\bigr) + (-1)^{\lv\omega_1\rv}\Gamma d\bigl(\omega_1 \wedge \Gamma(\omega_2 \wedge \omega_3)\bigr) \wedge \omega_4 \\
    & \quad - \Gamma d\bigl(\Gamma(\omega_1 \wedge \omega_2) \wedge \omega_3\bigr) \wedge \omega_4 + (-1)^{\lv\omega_1\rv + \lv\omega_2\rv}\omega_1 \wedge \Gamma d\bigl(\omega_2 \wedge \Gamma(\omega_3 \wedge \omega_4)\bigr) \Bigr) .
 \end{align*}
 On the other hand, we compute using \cref{inverse-lefschetz} that
 \begin{align*}
  \MoveEqLeft m_2(1 \otimes m_3 + m_3 \otimes 1)(\omega_1 \otimes \omega_2 \otimes \omega_3 \otimes \omega_4) \\
   & = \pi\Bigl( \Gamma(\omega_1 \wedge \omega_2) \wedge \omega_3 \wedge \omega_4 - (-1)^{\lv\omega_1\rv + \lv\omega_2\rv}\omega_1 \wedge \omega_2 \wedge \Gamma(\omega_3 \wedge \omega_4) \\ 
    & \quad - (-1)^{\lv\omega_1\rv}\omega_1 \wedge \Gamma d\bigl( \Gamma(\omega_2 \wedge \omega_3) \wedge \omega_4\bigr) + (-1)^{\lv\omega_1\rv}\Gamma d\bigl(\omega_1 \wedge \Gamma(\omega_2 \wedge \omega_3)\bigr) \wedge \omega_4 \\
    & \quad - \Gamma d\bigl(\Gamma(\omega_1 \wedge \omega_2) \wedge \omega_3\bigr) \wedge \omega_4 + (-1)^{\lv\omega_1\rv + \lv\omega_2\rv}\omega_1 \wedge \Gamma d\bigl(\omega_2 \wedge \Gamma(\omega_3 \wedge \omega_4)\bigr) \Bigr) .
 \end{align*}
 Therefore Equation~\eqref{eqn:A-infinity-fourth-product} holds.

 Suppose next that $k=5$.
 The first condition in \cref{A-infinity-m3-vanishing} implies that
 \begin{equation*}
  \im m_3 \subset \bigoplus_{j \geq n+1} \sR^j .
 \end{equation*}
 The second condition in \cref{A-infinity-m3-vanishing} then implies that
 \begin{equation*}
  m_3(m_3 \otimes 1^{\otimes 2}) = m_3(1 \otimes m_3 \otimes 1) = m_3(1^{\otimes 2} \otimes m_3) = 0 .
 \end{equation*}
 Since $m_4=m_5=0$, we conclude that Equation~\eqref{eqn:a-infinity-requirement} holds.
 
 Since the above cases are exhaustive, we conclude that $(\sR^\bullet,m)$ is an $A_\infty$-algebra.
 
 We conclude by showing that $(\sR^\bullet,m)$ is balanced.
 When $p+q=1$, the balanced requirement is vacuous.
 Since
 \begin{equation*}
  \mu_{1,1}(\omega \otimes \tau) = \omega \otimes \tau - (-1)^{\lv\omega\rv \lv\tau\rv}\tau \otimes \omega ,
 \end{equation*}
 we see that
 \begin{equation*}
  (m_2 \circ \mu_{1,1})(\omega \otimes \tau) = \pi \left( \omega \wedge \tau - (-1)^{\lv\omega\rv \lv\tau\rv} \tau \wedge \omega \right) = 0 .
 \end{equation*}
 Therefore the balanced requirement holds when $p+q=2$.
 Finally, note that
 \begin{align*}
  \mu_{1,2}(\omega \otimes \tau \otimes \eta) & = \omega \otimes \tau \otimes \eta - (-1)^{\lv\omega\rv\lv\tau\rv} \tau \otimes \omega \otimes \eta + (-1)^{\lv\omega\rv ( \lv\tau\rv + \lv \eta\rv)} \tau \otimes \eta \otimes \omega , \\
  \mu_{2,1}(\omega \otimes \tau \otimes \eta) & = \omega \otimes \tau \otimes \eta - (-1)^{\lv\tau\rv\lv\eta\rv} \omega \otimes \eta \otimes \tau + (-1)^{\lv\eta\rv(\lv\omega\rv + \lv\tau\rv)} \eta \otimes \omega \otimes \tau .
 \end{align*}
 We directly compute that
 \begin{align*}
  \MoveEqLeft (m_3 \circ \mu_{1,2})(\omega \otimes \tau \otimes \eta) = \pi \Bigl( \Gamma(\omega \wedge \tau) \wedge \eta - (-1)^{\lv\omega\rv} \omega \wedge \Gamma(\tau \wedge \eta) \\
   & - (-1)^{\lv\omega\rv\lv\tau\rv}\Gamma(\tau \wedge \omega) \wedge \eta + (-1)^{\lv\tau\rv(\lv\omega\rv + 1)}\tau \wedge \Gamma(\omega \wedge \eta) \\
   & + (-1)^{\lv\omega\rv(\lv\tau\rv + \lv\eta\rv)}\Gamma(\tau \wedge \eta) \wedge \omega - (-1)^{\lv\tau\rv(\lv\omega\rv + 1) + \lv\omega\rv\lv\eta\rv} \tau \wedge \Gamma(\eta \wedge \omega) \Bigr) .
 \end{align*}
 Therefore $m_3 \circ \mu_{1,2} = 0$.
 Similarly, $m_3 \circ \mu_{2,1} = 0$.
 Therefore $(\sR^\bullet,m)$ is balanced.
\end{proof}

We require two corollaries of \cref{rumin-a-infinity}.
First, the $A_\infty$-structure on $\sR^\bullet$ gives $H_R^\bullet(M;\bR)$ the structure of a graded commutative algebra.

\begin{corollary}
 \label{de-rham-cup}
 Let $(M^{2n+1},T^{1,0})$ be a CR manifold.
 If $k,\ell\in\{0,\dotsc,2n+1\}$, then the operator $\cup \colon H_R^k(M;\bR) \times H_R^\ell(M;\bR) \to H_R^{k+\ell}(M;\bR)$,
 \begin{equation*}
  [\omega] \cup [\tau] := [ \omega \rwedge \tau ] ,
 \end{equation*}
 is well-defined and associative.
\end{corollary}

\begin{proof}
 This follows immediately from \cref{a-infinity-cohomology,rumin-a-infinity}.
\end{proof}

\begin{remark}
 \label{rk:easy-recover-derham}
 If $\omega\in\Omega^kM$ is $d$-closed, then $\pi\omega \equiv \omega \mod \im d$.
 \Cref{projection-and-d-commute} then implies that $H_R^k(M;\bR)$ is isomorphic to the usual $k$-th de Rham cohomology group (cf.\ \cref{r-resolution}) and that our cup product is the usual one.
\end{remark}

Second, the exterior product gives $\sR^\bullet$ the structure of a $C^\infty(M)$-module.
More generally:

\begin{corollary}
 \label{A-infinity-m3-vanishing-module}
 Let $(M^{2n+1},T^{1,0})$ be a CR manifold.
 If $\lv\omega\rv + \lv\tau\rv + \lv\eta\rv \leq n$ or $\max \{ \lv\omega\rv , \lv\tau\rv, \lv\eta\rv \} \geq n+1$, then
 \begin{equation*}
  \left( \omega \rwedge \tau \right) \rwedge \eta = \omega \rwedge \left( \tau \rwedge \eta \right) .
 \end{equation*}
\end{corollary}

\begin{proof}
 If $\lv\omega\rv + \lv\tau\rv + \lv\eta\rv \leq n$ or $\max \{ \lv\omega\rv , \lv\tau\rv, \lv\eta\rv \} \geq n+1$, then \cref{A-infinity-m3-vanishing} implies that
 \begin{equation*}
  \left( m_1m_3 + m_3(m_1 \otimes 1^{\otimes 2} + 1 \otimes m_1 \otimes 1 + 1^{\otimes 2} \otimes m_2) \right)(\omega \otimes \tau \otimes \eta) = 0 .
 \end{equation*}
 The conclusion now follows from \cref{rumin-a-infinity}.
\end{proof}

We now turn to the $A_\infty$-structure on the bigraded Rumin complex.
To begin, we determine the extent to which the exterior product on $\sR^\bullet$ does not preserve the bigrading on $\sR^{\bullet,\bullet}$.

\begin{lemma}
 \label{rwedge-image}
 Let $(M^{2n+1},T^{1,0})$ be a CR manifold and let $\omega\in\sR^{p,q}$ and $\tau\in\sR^{r,s}$.
 If $\max\{p+q,r+s\}\geq n+1$ or $p+q+r+s\leq n$, then
 \[ \omega\rwedge\tau \in \sR^{p+r,q+s} ; \]
 if $p+q,r+s\leq n$ and $p+q+r+s\geq n+1$, then
 \[ \omega\rwedge\tau \in \sR^{p+r+1,q+s-1} \oplus \sR^{p+r,q+s} . \]
\end{lemma}

\begin{proof}
 Let $\omega\in\sR^{p,q}$ and $\tau\in\sR^{r,s}$.

 Suppose that $\max\{p+q,r+s\}\geq n+1$.
 We may assume that $p+q\geq n+1$ and $r+s\leq n$.
 Then $\omega \rwedge \tau = \omega \wedge \tau$.
 Since $\omega \in \theta \wedge \Omega^{p-1,q}M$ and $\tau\rv_{\CH} \in \Omega^{r,s}$, we conclude that $\omega\rwedge\tau\in\sR^{p+r,q+s}$.
 
 Now suppose that $p+q+r+s\leq n$.
 Then $\omega\rv_{\CH}\in\Omega^{p,q}$ and $\tau\rv_{\CH}\in\Omega^{r,s}$.
 Hence $(\omega\wedge\tau)\rv_{\CH}\in\Omega^{p+r,q+s}$.
 We conclude from \cref{projection-to-E-expression} that $\omega\rwedge\tau\in\sR^{p+q,r+s}$.
 
 Finally, suppose that $p+q,r+s\leq n$ and $p+q+r+s\geq n+1$.
 Therefore $\omega \in \Omega^{p,q}M \oplus \theta \wedge (\Omega^{p,q-1}M \oplus \Omega^{p-1,q}M)$ and similarly for $\tau$.
 We readily conclude from \cref{inverse-lefschetz,projection-to-R} that $\omega\rwedge\tau\in\sR^{p+r+1,q+s-1}\oplus\sR^{p+r,q+s}$.
\end{proof}

\Cref{rwedge-image} indicates the natural way to define the exterior product on $\sR^{\bullet,\bullet}$.

\begin{definition}
 \label{defn:bullet-product}
 Let $(M^{2n+1},T^{1,0})$ be a CR manifold.
 The \defn{exterior product on $\sR^{\bullet,\bullet}$} is
 \[ \omega \krwedge \tau := \pi^{p+r,q+s}\left(\omega\rwedge\tau\right) . \]
 for all $\omega\in\sR^{p,q}$ and all $\tau\in\sR^{r,s}$.
\end{definition}

Clearly the exterior product on $\sR^{\bullet,\bullet}$ is a graded commutative $\bC$-linear map of degree $0$ on $\sR^{\bullet,\bullet}$.
It also holds that $\dbbar$ is a graded derivation with respect to $\krwedge$.

\begin{lemma}
 \label{kr-product-rule}
 Let $(M^{2n+1},T^{1,0})$ be a CR manifold.
 Then
 \begin{equation}
  \label{eqn:kr-product-rule}
  \dbbar(\omega\krwedge\tau) = \dbbar\omega \krwedge \tau + (-1)^{p+q}\omega \krwedge \dbbar\tau
 \end{equation}
 for all $\omega\in\sR^{p,q}$ and all $\tau\in\sR^{r,s}$.
\end{lemma}

\begin{proof}
 \Cref{rwedge-image,defn:dbbar,defn:bullet-product} imply that
 \[ \dbbar(\omega\krwedge\tau) = \pi^{p+r,q+s+1} d(\omega\rwedge\tau) . \]
 Combining this with \cref{real-product-rule,rwedge-image} yields Equation~\eqref{eqn:kr-product-rule}.
\end{proof}

Unlike the exterior product on the Rumin complex, the $m_3$-operator does preserve the bigrading on $\sR^{\bullet,\bullet}$.

\begin{lemma}
 \label{m3-image}
 Let $(M^{2n+1},T^{1,0})$ be a CR manifold.
 Let $\omega \in \sR^{p,q}$ and $\tau \in \sR^{r,s}$ and $\eta \in \sR^{t,u}$.
 Then
 \begin{equation*}
  m_3(\omega \otimes \tau \otimes \eta) \in \sR^{p+r+t, q+s+u - 1} .
 \end{equation*}
\end{lemma}

\begin{proof}
 \Cref{A-infinity-m3-vanishing} implies that we may assume that $p+q, r+s, t+u \leq n$ and that $p+q+r+s+t+u \geq n+1$.
 \Cref{inverse-lefschetz} then implies that $\Gamma(\omega \wedge \tau) \wedge \eta$ and $\omega \wedge \Gamma(\eta \wedge \tau)$ are both elements of $\theta \wedge \Omega^{p+r+t-1,q+s+u-1}M$.
 The conclusion readily follows from \cref{projection-to-R}.
\end{proof}

As a consequence, the operators $\dbbar$, $\krwedge$, and $m_3$ give the bigraded Rumin complex the structure of an $A_\infty$-algebra.

\begin{theorem}
 \label{bigraded-rumin-a-infinity}
 Let $(M^{2n+1},T^{1,0})$ be a CR manifold.
 Set $m_1^{\dbbar} := \dbbar$ and $m_2^{\dbbar} := \krwedge$ and $m_3^{\dbbar} := m_3$ and $m_j^{\dbbar} := 0$ for $j \geq 4$.
 Then $(\sR^{\bullet,\bullet},m^{\dbbar})$ is a balanced $A_\infty$-algebra.
\end{theorem}

\begin{proof}
 Note that if $\omega_\ell \in \sR^{p_\ell,q_\ell}$, $\ell \in \{ 1 , \dotsc, j \}$, then
 \begin{equation*}
  m_j^{\dbbar}(\omega_1 \otimes \dotsm \otimes \omega_j) = \pi^{\lv P\rv, \lv Q\rv + 2 - j}m_j(\omega_1 \otimes \dotsm \otimes \omega_j) ,
 \end{equation*}
 where $\lv P\rv := p_1 + \dotsm + p_j$ and $\lv Q\rv := q_1 + \dotsm + q_j$.
 On the one hand, it immediately follows from \cref{rumin-a-infinity} that $m^{\dbbar}$ is balanced.
 On the other hand, \cref{rwedge-image,m3-image} yield
 \begin{multline}
  \label{eqn:m3-to-m3dbbar}
  \sum_{r+s+t=k} (-1)^{r+st} m_{r+t+1}^{\dbbar} (1^{\otimes r} \otimes m_s^{\dbbar} \otimes 1^{\otimes t}) \\
   = \sum_{r+s+t=k} (-1)^{r+st} \pi^{\lv P\rv,\lv Q\rv + 2 - k} m_{r+t+1} (1^{\otimes r} \otimes m_s \otimes 1^{\otimes t})
 \end{multline}
 when restricted to $\sR^{p_1,q_1} \otimes \dotsm \otimes \sR^{p_k,q_k}$.
 In particular, \cref{rumin-a-infinity} implies that $m^{\dbbar}$ is an $A_\infty$-structure on $\sR^{\bullet,\bullet}$.
\end{proof}

We require two corollaries of \cref{bigraded-rumin-a-infinity}.
First, the $A_\infty$-structure on $\sR^{\bullet,\bullet}$ gives $H_R^{\bullet,\bullet}(M)$ the structure of a graded commutative algebra.

\begin{corollary}
 \label{kohn-rossi-cup}
 Let $(M^{2n+1},T^{1,0})$ be a CR manifold.
 If $p,r\in\{0,\dotsc,n+1\}$ and $q,s \in \{ 0, \dotsc, n\}$, then the operator $\sqcup \colon H_R^{p,q}(M) \times H_R^{r,s}(M) \to H_R^{p+r,q+s}(M)$,
 \begin{equation*}
  [\omega] \sqcup [\tau] := [ \omega \krwedge \tau ] ,
 \end{equation*}
 is well-defined and associative.
\end{corollary}

\begin{proof}
 This follows immediately from \cref{a-infinity-cohomology,bigraded-rumin-a-infinity}.
\end{proof}

Second, the exterior product gives $\sR^{\bullet,\bullet}$ the structure of a $C^\infty(M;\bC)$-module.
More generally:

\begin{corollary}
 \label{bigraded-A-infinity-m3-vanishing-module}
 Let $(M^{2n+1},T^{1,0})$ be a CR manifold.
 If $\lv\omega\rv + \lv\tau\rv + \lv\eta\rv \leq n$ or $\max \{ \lv\omega\rv , \lv\tau\rv, \lv\eta\rv \} \geq n+1$, then
 \begin{equation*}
  \left( \omega \krwedge \tau \right) \krwedge \eta = \omega \krwedge \left( \tau \krwedge \eta \right) .
 \end{equation*}
\end{corollary}

\begin{proof}
 This follows immediately from \cref{A-infinity-m3-vanishing-module} and Equation~\eqref{eqn:m3-to-m3dbbar}.
\end{proof}

%% file: intrinsic/resolution.tex
\section{Recovering the usual de Rham and Kohn--Rossi groups}
\label{sec:resolution}

We conclude this part by relating $H_R^k(M;\bR)$, $H_R^{p,q}(M)$, and $H_R^k(M;\sP)$ to the usual de Rham cohomology groups, Kohn--Rossi cohomology groups, and cohomology groups with coefficients in the sheaf of CR pluriharmonic functions, respectively.
Our main strategy is to construct an acyclic resolution of the relevant sheaf.

We first show that the Rumin complex is an acyclic resolution of the constant sheaf $\underline{\bR}$, and hence $H_R^k(M;\bR)$ is isomorphic to the usual $k$-th de Rham cohomology group.
This is similar to the corresponding proof by Rumin~\cite{Rumin1990}.

\begin{theorem}
 \label{r-resolution}
 Let $(M^{2n+1},T^{1,0})$ be a CR manifold and let $k\in\{0,\dotsc,2n+1\}$.
 Then $H_R^k(M;\bR)$ is isomorphic to the $k$-th de Rham cohomology group.
\end{theorem}

\begin{proof}
 It suffices to prove that
 \begin{equation}
  \label{eqn:rumin-resolution}
  0 \longrightarrow \underline{\bR} \lhook\joinrel\longrightarrow \sR^0 \overset{d}{\longrightarrow} \sR^1 \overset{d}{\longrightarrow} \dotsm \overset{d}{\longrightarrow} \sR^{2n+1} \overset{d}{\longrightarrow} 0
 \end{equation}
 is an acyclic resolution of the constant sheaf $\underline{\bR}=\ker\left(d\colon\sR^0\to\sR^1\right)$.
 
 \Cref{A-infinity-m3-vanishing-module} implies that the $A_\infty$-structure on $\sR^\bullet$ gives each sheaf $\sR^k$ the structure of a $C^\infty(M)$-module.
 Thus the sheaves $\sR^k$ are fine, and hence acyclic.
 
 We now show that~\eqref{eqn:rumin-resolution} is a resolution.
 By the Poincar\'e Lemma, it suffices to show that if $\omega\in\sR^k$ is exact as an element of $\sA^k$, then it is exact as an element of $\sR^k$.
 Let $\tau\in\sA^{k-1}$ be such that $\omega=d\tau$.
 Therefore $\omega = \pi d\tau$.
 We conclude from \cref{projection-and-d-commute} that $\omega = d\pi\tau$, as desired.
\end{proof}

We next show that $H_R^{p,q}(M)$ is isomorphic to the Kohn--Rossi cohomology group of bidegree $(p,q)$.
The original definition by Kohn and Rossi~\cite{KohnRossi1965} of these groups is given extrinsically for embedded CR manifolds.
Tanaka~\cite{Tanaka1975} later gave an intrinsic characterization of these groups which generalizes to all CR manifolds.
We show that our groups $H_R^{p,q}(M)$ are isomorphic to those introduced by Tanaka.

\begin{theorem}
 \label{kr-resolution}
 Let $(M^{2n+1},T^{1,0})$ be a CR manifold.
 Then $H_R^{p,q}(M)$ is isomorphic to the Kohn--Rossi cohomology group of bidegree $(p,q)$.
\end{theorem}

\begin{proof}
 We first recall Tanaka's definition~\cite{Tanaka1975} of the Kohn--Rossi cohomology group $H^{p,q}(M)$ of bidegree $(p,q)$.
 Define
 \[ F^p\COmega^kM := \left\{ \omega\in\COmega^kM \suchthatcolon \omega(\bar Z_1,\dotsm,\bar Z_{k+1-p}, \cdot,\dotsc,\cdot)=0, \bar Z_1,\dotsc,\bar Z_{k+1-p}\in T^{0,1} \right\} . \]
 Note that
 \[ \COmega^kM = F^0\COmega^kM \supset F^1\COmega^{k-1}M \supset \dotsm \supset F^k\COmega^k M \supset F^{k+1}\COmega^kM = 0 . \]
 Define
 \[ \mC^{p,q} := F^p\COmega^{p+q}M / F^{p+1}\COmega^{p+q}M . \]
 Since $d\left(F^p\COmega^kM\right)\subset F^p\COmega^{k+1}M$, the operator $\dbbar\colon\mC^{p,q}\to\mC^{p,q+1}$, $\dbbar[\omega]:=[d\omega]$, is well-defined.
 Moreover, $\dbbar^2=0$.
 The Kohn--Rossi cohomology groups~\cites{KohnRossi1965,Tanaka1975} are
 \[ H^{p,q}(M) := \frac{\ker\left(\dbbar\colon\mC^{p,q}\to\mC^{p,q+1}\right)}{\im\left(\dbbar\colon\mC^{p,q-1}\to\mC^{p,q}\right)} . \]
 
 We now define a morphism $H_R^{p,q}(M)\to H^{p,q}(M)$.
 Since $\mR^{p,q}\subset F^p\COmega^{p+q}M$, there is a canonical projection $\Pi\colon\mR^{p,q}\to\mC^{p,q}$.
 \Cref{projection-to-E-expression,projection-to-F-expression} imply that $\Pi$ is injective.
 By definition, $\dbbar\Pi(\omega) = \Pi(\dbbar\omega)$.
 Therefore $\Pi$ induces an injective morphism $H_R^{p,q}(M)\to H^{p,q}(M)$.
 
 Finally, we show that $H_R^{p,q}(M) \to H^{p,q}(M)$ is surjective.
 We prove this in the case $p+q\leq n$;
 the proof in the case $p+q\geq n+1$ is similar.
 Let $\omega\in F^p\COmega^{p+q}M$ be such that $d\omega\in F^{p+1}\COmega^{p+q+1}M$.
 Then $[\omega] \in H^{p,q}(M)$.
 \Cref{inverse-lefschetz} implies that $\Gamma\omega \in F^p\COmega^{p+q-1}M$ and $\Gamma d\omega \in F^{p+1}\COmega^{p+q}M$.
 On the one hand, \cref{projection-to-R} implies that $[\omega] = [\pi\omega] = [\pi^{p,q}\pi\omega]$ in $H^{p,q}(M)$.
 On the other hand, \cref{projection-and-d-commute} implies that $\dbbar\pi^{p,q}\pi\omega=0$.
 Therefore $[\omega]\in\im\bigl( H_R^{p,q}(M) \to H^{p,q}(M)\bigr)$.
\end{proof}

An alternative approach to \cref{kr-resolution}, based on the Kohn and Rossi's identification~\cite{KohnRossi1965} of $H^{p,q}(M)$ as the $q$-th cohomology group with coefficients in the sheaf of CR holomorphic $(p,0)$-forms goes roughly as follows.
Let
\begin{equation*}
 \sO^p := \ker \left( \dbbar \colon \sR^{p,0} \to \sR^{p,1} \right)
\end{equation*}
denote the sheaf of CR holomorphic $(p,0)$-forms.
The exterior product on $\sR^{\bullet,\bullet}$ gives each sheaf $\sR^{p,q}$ the structure of a $C^\infty(M;\bC)$-module, and hence these sheaves are fine.
Consider now the complex
\begin{equation*}
 0 \longrightarrow \sO^p \lhook\joinrel\longrightarrow \sR^{p,0} \overset{\dbbar}{\longrightarrow} \sR^{p,1} \overset{\dbbar}{\longrightarrow} \dotsm \overset{\dbbar}{\longrightarrow} \sR^{p,n} \overset{\dbbar}{\longrightarrow} 0 .
\end{equation*}
If $(M^{2n+1},T^{1,0})$ is strictly pseudoconvex and locally embeddable, then truncating this complex at $\sR^{p,n-1}$ yields a resolution of $\sO^p$.
In particular, if $q\leq n-1$, then $H^{p,q}(M) \cong H^q(M;\sO^p)$.
The isomorphism in the case $q=n$ can be exhibited directly~\cite{GarfieldLee1998}.

While we do not require the isomorphism $H_R^{p,q}(M) \cong H^q(M;\sO^p)$ in this article, we do require the analogous isomorphism between $H_R^1(M;\sP)$ and the usual cohomology group with coefficients in the sheaf of CR pluriharmonic functions.
This is obtained by the argument outlined above.

\begin{theorem}
 \label{sP-resolution}
 Let $(M^{2n+1},T^{1,0})$ be a strictly pseudoconvex, locally embeddable CR manifold.
 Then
 \[ 0 \longrightarrow \sP \lhook\joinrel\longrightarrow \sS^0 \overset{D}{\longrightarrow} \sS^1 \overset{D}{\longrightarrow} \sS^2 \overset{D}{\longrightarrow} \dotsm \overset{D}{\longrightarrow} \sS^{n} \]
 is an acyclic resolution of $\sP$.
\end{theorem}

\begin{remark}
 If $n\geq3$, then every strictly pseudoconvex CR manifold is locally embeddable~\cite{Akahori1987}.
 If $n=2$, then every closed, strictly pseudoconvex CR manifold is locally embeddable~\cite{Boutet1975}.
\end{remark}

\begin{proof}
 Clearly $\sP\hookrightarrow\sS^0$ is injective and $\sP=\ker\left(D\colon\sS^0\to\sS^1\right)$.
 The exterior product on $\sR^\bullet$ gives each sheaf $\sS^k$ the structure of a $C^\infty(M)$-module.
 Thus $\sS^k$ is fine, and hence acyclic.
 
 Now let $k \in \{1,\dotsc,n-1\}$.
 Since $M$ is strictly pseudoconvex and locally embeddable, for each point $x \in M$ there is an open set $U \subset M$ containing $x$ such that the Kohn--Rossi group $H^{0,k}(U)$ vanishes~\cite{AndreottiHill1972b}*{Conclusion on pg.\ 802}.
 \Cref{kr-resolution} then implies that $H_R^{0,k}(U)=0$.
 \Cref{r-resolution} and the Poincar\'e Lemma imply that, by shrinking $U$ if necessary, $H_R^{k+1}(U)=0$.
 We conclude from \cref{long-exact-sequence} that $H^k(U;\sP)=0$.
\end{proof}

%% file: part-hodge.tex
\part{Hodge theory}
\label{part:hodge}

There are a number of different Hodge isomorphism theorems for pseudohermitian manifolds.
Rumin~\cite{Rumin1994} used his formulation of the Rumin complex to establish an isomorphism between the de Rham cohomology group $H^k(M;\bR)$ and a suitably-defined space of harmonic $k$-forms.
Kohn and Rossi~\cite{KohnRossi1965} and Tanaka~\cite{Tanaka1975} used their respective formulations of the Kohn--Rossi complex to establish isomorphisms between the Kohn--Rossi cohomology group $H^{p,q}(M)$ and a suitably-defined space of harmonic $(p,q)$-forms.
Garfield and Lee~\cites{Garfield2001,GarfieldLee1998} used their formulation of the bigraded Rumin complex to establish another version of this latter isomorphism.
In this part, we establish similar isomorphisms in terms of our formulation of the bigraded Rumin complex;
our proofs closely parallel the proofs of Rumin~\cite{Rumin1994} and of Garfield and Lee~\cites{Garfield2001,GarfieldLee1998}.
A more precise outline of this part is as follows:

In \cref{sec:hodge_star} we define the Hodge star operator on $\mR^{\bullet,\bullet}$.
We then define the adjoints of the $\dbbar$-, $\db$-, and $\dhor$-operators and derive local formulas.

In \cref{sec:setup} we introduce some new families of operators on pseudohermitian manifolds which facilitate computations.
The first family of operators, $\onablas$, is obtained from $\nablabbar\omega$ by skew symmetrization.
While this operator does not in general agree with $\dbbar$ --- indeed, it does not even preserve $\mR^{\bullet,\bullet}$ --- there is a simple relationship between these two operators.
The second family of operators are the actions of the pseudohermitian torsion, the pseudohermitian curvature, and the pseudohermitian Ricci curvature on complex-valued differential forms.
These formulas enable concise Weitzenb\"ock-type formulas.
Both families of operators are present in the approaches of Rumin~\cite{Rumin1994} and of Garfield and Lee~\cites{Garfield2001,GarfieldLee1998} to pseudohermitian Hodge theorems.
We have changed notation in an effort to more clearly distinguish the operators $\dbbar$ and $\onablas$, and thereby clarify their respective roles in the proof.

In \cref{sec:hodge/folland-stein} we collect some essential facts about Folland--Stein spaces~\cite{FollandStein1974} and Heisenberg pseudodifferential operators~\cites{BealsGreiner1988,Ponge2008}.
In particular, we include relevant sufficient conditions for a Heisenberg pseudodifferential operator to be maximally hypoelliptic.

In \cref{sec:weitzenbock} we define the Kohn and Rumin Laplacians.
Our definition of the Kohn Laplacian differs from other definitions in the literature, but retains the properties that it is formally self-adjoint and its kernel is the intersection of $\ker\dbbar$ and $\ker\dbbar^\ast$;
our definition of the Rumin Laplacian coincides with Rumin's original definition~\cite{Rumin1994} via the identification of \cref{explicit-rumin}.
We also derive some useful relations involving these operators.
On the one hand, we derive some Weitzenb\"ock-type and factorization formulas for the Kohn Laplacian;
these results are later used both to prove that the Kohn Laplacian is maximally hypoelliptic under appropriate hypotheses and to derive vanishing theorems for certain Kohn--Rossi cohomology groups.
On the other hand, we show that there is a close relationship between the Kohn and Rumin Laplacians;
for example, on torsion-free pseudohermitian manifolds, the Rumin Laplacian is twice the real part of the Kohn Laplacian (plus $\dhor^\ast\dhor$ or $\dhor\dhor^\ast$ in the middle degrees), generalizing the well-known relationship~\cite{Lee1986}*{Theorem~2.3} between the sublaplacian and the Kohn Laplacian on functions.
This relationship both justifies our definition of the Kohn Laplacian and provides a new perspective on Rumin's proof~\cite{Rumin1994} of the maximal hypoellipticity of the Rumin Laplacian.

In \cref{sec:hypoelliptic} we show that --- with the exception of the Kohn Laplacian on $\mR^{p,q}$, $q\in\{0,n\}$ --- the Kohn and Rumin Laplacians are maximally hypoelliptic on closed, strictly pseudoconvex, $(2n+1)$-dimensional pseudohermitian manifolds.

In \cref{sec:hodge_thm} we deduce Hodge decomposition theorems involving the Kohn and Rumin Laplacians.
As an application, we recover Serre duality~\cites{Garfield2001,Tanaka1975} for the Kohn--Rossi cohomology groups and Poincar\'e duality for the de Rham cohomology groups.

In \cref{sec:popovici} we adapt a construction of Popovici~\cite{Popovici2016} in complex geometry to prove a new Hodge decomposition theorem for $\mR^{p,q}$ with no constraints on $p$ and $q$.
This result is obtained by studying a maximally hypoelliptic, Heisenberg pseudodifferential operator which is closely related to the Kohn and Rumin Laplacians.
We use this theorem in \cref{part:applications} to prove a Hodge isomorphism theorem for the second page of the natural spectral sequence associated to the bigraded Rumin complex.

\input{hodge/hodge_star}
\input{hodge/setup}
\input{hodge/folland-stein}
\input{hodge/weitzenbock}
\input{hodge/subelliptic}
\input{hodge/hodge_thm}
\input{hodge/popovici}

%% file: hodge/hodge_star.tex
\section{The Hodge star operator}
\label{sec:hodge_star}

In this section  we define the Hodge star operator on $\mR^{\bullet,\bullet}$ and use it to define the formal adjoints of the operators in the Rumin and bigraded Rumin complexes.
These definitions are not CR invariant.

The Levi form induces an hermitian form on $\mR^{p,q}$.

\begin{definition}
 \label{defn:sRpq-inner-product}
 Let $(M^{2n+1},T^{1,0},\theta)$ be a pseudohermitian manifold.
 The \defn{hermitian form on $\mR^{p,q}$} is
 defined by
 \[ \lp \omega,\tau\rp := \lp \omega\rv_{\CH}, \tau\rv_{\CH} \rp, \]
 for all $\omega,\tau\in\mR^{p,q}$, $p+q\leq n$, and by
 \[ \lp \omega,\tau\rp := \lp (T\contr\omega)\rv_{\CH}, (T\contr\tau)\rv_{\CH} \rp, \]
 for all $\omega,\tau\in\mR^{p,q}$, $p+q\geq n+1$, where $\lp\cdot,\cdot\rp$ is the hermitian inner product on $\Lambda^{p,q}$ induced by the Levi form.
\end{definition}

Our terminology is justified by the following lemma.

\begin{lemma}
 \label{sR-inner-product}
 Let $(M^{2n+1},T^{1,0},\theta)$ be a pseudohermitian manifold.
 Then $\lp\cdot,\cdot\rp$ defines an hermitian form on $\mR^{p,q}$.
\end{lemma}

\begin{proof}
 It is clear that $\lp\cdot,\cdot\rp$ is sesquilinear.
 It remains to show nondegeneracy.
 
 Suppose first that $p+q\leq n$.
 Let $\omega\in\mR^{p,q}$ be such that $\lp\omega,\tau\rp=0$ for all $\tau\in\mR^{p,q}$.
 Then $\omega\rv_{\CH}=0$, and hence $\omega\equiv0\mod\theta$.  \Cref{projection-to-R} implies that $\omega=0$.
 
 Suppose now that $p=q\geq n+1$.
 Let $\omega\in\mR^{p,q}$ be such that $\lp\omega,\tau\rp=0$ for all $\tau\in\mR^{p,q}$.
 Then $(T\contr\omega)\rv_{\CH}=0$.
 Since $\theta\wedge\omega=0$, we conclude that $\omega=0$.
\end{proof}

We now define the Hodge star operator on $\mR^{\bullet,\bullet}$.

\begin{definition}
 \label{defn:hodge}
 Let $(M^{2n+1},T^{1,0},\theta)$ be a pseudohermitian manifold.
 The \defn{Hodge star operator} $\hodge\colon\mR^{\bullet,\bullet}\to\mR^{\bullet,\bullet}$ is defined by
 \[ \omega \wedge \hodge \otau = \frac{1}{n!}\lp \omega,\tau\rp\,\theta\wedge d\theta^n \]
 for all $\omega,\tau\in\mR^{\bullet,\bullet}$.
\end{definition}

We begin by recording some basic properties of the Hodge star operator.

\begin{lemma}
 \label{hodge-mapping-properties}
 Let $(M^{2n+1},T^{1,0},\theta)$ be a pseudohermitian manifold.
 Then
 \begin{enumerate}
  \item $\hodge$ is $C^\infty(M;\bC)$-linear;
  \item if $p+q\leq n$, then $\hodge\left(\mR^{p,q}\right)\subset\mR^{n+1-q,n-p}$;
  \item if $p+q\geq n+1$, then $\hodge\left(\mR^{p,q}\right)\subset\mR^{n-q,n+1-p}$;
  \item $\overline{\hodge\omega}=\hodge\oomega$ for all $\omega\in\mR^{p,q}$.
 \end{enumerate}
\end{lemma}

\begin{proof}
 That $\hodge$ is $C^\infty(M;\bC)$-linear follows from the fact that the inner product on $\mR^{p,q}$ is $C^\infty(M;\bC)$-antilinear in its second component.
 
 That $\hodge\left(\mR^{p,q}\right) \subset \mR^{n+1-q,n-p}$ when $p+q\leq n$ and $\hodge \left(\mR^{p,q}\right) \subset \mR^{n-q,n+1-p}$ when $p+q\geq n+1$ follows from \cref{conjugate-rumin-bundles} and the fact that $\theta\wedge d\theta^n\in\mR^{n+1,n}$.
 
 That $\overline{\hodge\omega}=\hodge\oomega$ follows from the identity $\overline{\lp\omega,\tau\rp}=\lp\oomega,\otau\rp$.
\end{proof}

There are two isomorphisms $\mR^{p,q}\cong\mR^{n+1-q,n-p}$, $p+q\leq n$, one from \cref{sF-to-sE} and one from the Hodge star operator.
These isomorphisms differ by the composition with the almost complex structure on $\mR^{\bullet,\bullet}$.

\begin{lemma}
 \label{hodge-star-sE}
 Let $(M^{2n+1},T^{1,0},\theta)$ be a pseudohermitian manifold.
 Then
 \[ \hodge \omega = \frac{(-1)^{\frac{(p+q)(p+q+1)}{2}}i^{q-p}}{(n-p-q)!}\theta \wedge \omega \wedge d\theta^{n-p-q} \]
 for all $\omega\in\mR^{p,q}$, $p+q\leq n$. 
\end{lemma}

\begin{proof}
 Let $\omega,\tau\in\mR^{p,q}$, $p+q\leq n$.
 By abuse of notation, let $\hodge\colon\Lambda^{p,q}\to\Lambda^{n-q,n-p}$ denote the Hodge star operator on $\Lambda^{p,q}$ defined in each fiber via the volume element $\frac{1}{n!}(d\theta\rv_H)^n$.
 Using \cref{defn:sRpq-inner-product}, we compute that
 \begin{equation*}
  \tau \wedge \hodge \oomega = \theta \wedge \tau\rv_{\CH} \wedge \hodge \overline{\omega\rv_{\CH}} = (-1)^{p+q}\tau \wedge \overline{\theta \wedge \hodge \omega\rv_{\CH}} .
 \end{equation*}
 It is well-known~\cite{Huybrechts2005}*{Proposition~1.2.31} that
 \begin{equation*}
  \hodge \overline{\omega\rv_{\CH}} = \frac{(-1)^{\frac{(p+q)(p+q+1)}{2}}i^{q-p}}{(n-p-q)!} \oomega\rv_H \wedge d\theta\rv_{\CH}^{n-p-q} .
 \end{equation*}
 The conclusion readily follows.
\end{proof}

\begin{lemma}
 \label{hodge-star-sF}
 Let $(M^{2n+1},T^{1,0},\theta)$ be a pseudohermitian manifold.
 Then
 \begin{equation*}
  \hodge \omega = (-1)^n (-1)^{\frac{(p+q)(p+q+1)}{2}} i^{p-q+1}\xi
 \end{equation*}
 for all $\omega\in\mR^{p,q}$, $p+q\geq n+1$, where $\xi$ is the unique element of $\mR^{n-q,n+1-p}$ such that
 \begin{equation}
  \label{eqn:hodge-star-sF-choose-xi}
  \omega = \frac{1}{(p+q-n-1)!}\theta\wedge\xi\wedge d\theta^{p+q-n-1} .
 \end{equation}
\end{lemma}

\begin{proof}
 Let $\omega,\tau\in\mR^{p,q}$, $p+q\geq n+1$.
 \Cref{sF-to-sE} implies that there is a unique $\xi\in\mR^{n-q,n+1-p}$ such that Equation~\eqref{eqn:hodge-star-sF-choose-xi} holds.
 By abuse of notation, let $\hodge\colon\Lambda^{p-1,q}\to\Lambda^{n-q,n+1-p}$ denote the Hodge star operator on $\Lambda^{p-1,q}$ defined in each fiber via the volume element $\frac{1}{n!}(d\theta\rv_{\CH})^n$.
 Using \cref{defn:sRpq-inner-product}, we compute that
 \begin{equation*}
  \tau \wedge \hodge \oomega = \frac{1}{(p+q-n-1)!}\tau \wedge \hodge \overline{ \xi\rv_{\CH} \wedge d\theta\rv_{\CH}^{p+q-n-1} } .
 \end{equation*}
 It is well-known~\cite{Huybrechts2005}*{Proposition~1.2.31} that
 \begin{equation*}
  \hodge \xi\rv_{\CH} \wedge d\theta\rv_{\CH}^{p+q-n-1} = (-1)^{\frac{(2n+1-p-q)(2n+2-p-q)}{2}}(p+q-n-1)!i^{q+1-p}\overline{\xi\rv_{\CH}} .
 \end{equation*}
 The conclusion readily follows.
\end{proof}

In particular, we recover the well-known fact $\hodge^2=1$.

\begin{corollary}
 \label{hodge-star-squared}
 Let $(M^{2n+1},T^{1,0},\theta)$ be a pseudohermitian manifold.
 For any $\omega\in\mR^{p,q}$, it holds that $\hodge^2\omega=\omega$.
\end{corollary}

\begin{proof}
 This follows directly from \cref{hodge-star-sE,hodge-star-sF}.
\end{proof}

The volume element $\frac{1}{n!}\,\theta\wedge d\theta^n$ and the hermitian form on $\mR^{\bullet,\bullet}$ give rise to an $L^2$-inner product on the space of compactly supported sections of $\mR^{\bullet,\bullet}$:

\begin{definition}
 Let $(M^{2n+1},T^{1,0},\theta)$ be a pseudohermitian manifold.
 The \defn{$L^2$-inner product on $\mR^{p,q}$} is defined by
 \begin{equation}
  \label{eqn:sR-inner-product}
  \llp \omega,\tau \rrp := \frac{1}{n!}\int_M \lp \omega,\tau\rp\,\theta\wedge d\theta^n
 \end{equation}
 for all compactly supported $\omega,\tau\in\mR^{p,q}$.
\end{definition}

We now define the formal adjoints of the operators in the bigraded Rumin complex.

\begin{definition}
 \label{defn:dbbar-adjoint}
 Let $(M^{2n+1},T^{1,0},\theta)$ be a pseudohermitian manifold.
 For each pair $p\in\{0,\dotsc,n+1\}$ and $q\in\{0,\dotsc,n\}$, we define $\dbbar^\ast\colon\mR^{p,q}\to\mR^{p,q-1}$ by
 \[ \dbbar^\ast\omega := (-1)^{p+q}\hodge\db\hodge\omega .\]
 If $p+q\not=n+1$, then we define $\db^\ast\colon\mR^{p,q}\to\mR^{p-1,q}$ by
 \[ \db^\ast\omega := (-1)^{p+q}\hodge\dbbar\hodge\omega . \]
 If $p+q=n+1$, then we define $\db^\ast\colon\mR^{p,q}\to\mR^{p-2,q+1}$ and $\dhor^\ast\colon\mR^{p,q}\to\mR^{p-1,q}$ by
 \begin{align*}
  \db^\ast\omega & := (-1)^{n+1}\hodge\dbbar\hodge\omega , \\
  \dhor^\ast\omega & := (-1)^{n+1}\hodge\dhor\hodge\omega .
 \end{align*}
\end{definition}

These operators are formal adjoints with respect to the $L^2$-inner product.

\begin{lemma}
 \label{formal-adjoint}
 Let $(M^{2n+1},T^{1,0},\theta)$ be a pseudohermitian manifold.
 Let $\omega\in\mR^{p,q}$ be compactly supported.
 \begin{enumerate}
  \item If $\tau\in\mR^{p,q-1}$ is compactly supported, then
  \[ \llp \dbbar\tau, \omega \rrp = \llp \tau, \dbbar^\ast\omega \rrp . \]
  \item If $p+q\not=n+1$ and $\tau\in\mR^{p-1,q}$ is compactly supported, then
  \[ \llp \db\tau, \omega \rrp = \llp \tau, \db^\ast\omega \rrp . \]
  \item If $p+q=n+1$ and $\tau\in\mR^{p-2,q+1}$ is compactly supported, then
  \[ \llp \db\tau, \omega \rrp = \llp \tau, \db^\ast\omega \rrp . \]
  \item If $p+q=n+1$ and $\tau\in\mR^{p-1,q}$ is compactly supported, then
  \[ \llp \dhor\tau, \omega \rrp = \llp \tau, \dhor^\ast\omega \rrp . \]
 \end{enumerate}
\end{lemma}

\begin{proof}
 We prove that $\dbbar^\ast$ is the formal adjoint of $\dbbar$;
 the remaining cases are similar.
 
 The definition of the Hodge star operator implies that
 \[ \llp \dbbar\tau, \omega \rrp = \int_M \dbbar\tau \wedge\hodge\oomega . \]
 Type considerations imply that $\dbbar\tau\wedge\hodge\oomega=d\tau\wedge\hodge\oomega$.
 Stokes' Theorem and type considerations then imply that
 \[ \llp \dbbar\tau, \omega \rrp = \int_M (-1)^{p+q}\tau\wedge\dbbar\hodge\oomega . \]
 The final conclusion follows from \cref{hodge-star-squared}.
\end{proof}

Using the Hodge star operator to dualize the bigraded Rumin complex yields another bigraded complex.
We record the resulting identities involving compositions of two of $\db^\ast$, $\dbbar^\ast$, and $\dhor^\ast$.

\begin{proposition}
 \label{dual-justification}
 Let $(M^{2n+1},T^{1,0},\theta)$ be a pseudohermitian manifold and let $p \in \{ 0 , \dotsc , n+1 \}$ and $q \in \{ 0 , \dotsc , n \}$.
 It holds that
 \begin{enumerate}
  \item $(\db^\ast)^2=0$ and $(\dbbar^\ast)^2=0$ on $\mR^{p,q}$;
  \item if $p+q \not\in \{ n+1, n+2\}$, then $\db^\ast\dbbar^\ast + \dbbar^\ast\db^\ast = 0$ on $\mR^{p,q}$;
  \item if $p+q = n+1$, then $\dbbar^\ast\db^\ast + \db^\ast\dhor^\ast = 0$ and $\dbbar^\ast\dhor^\ast + \db^\ast\dbbar^\ast = 0$ on $\mR^{p,q}$; and
  \item if $p+q = n+2$, then $\dhor^\ast\db^\ast + \db^\ast\dbbar^\ast = 0$ and $\dbbar^\ast\db^\ast + \dhor^\ast\dbbar^\ast = 0$ on $\mR^{p,q}$.
 \end{enumerate}
\end{proposition}

\begin{proof}
 This follows immediately from \cref{justification-of-bigraded-complex,hodge-star-squared}.
\end{proof}

We conclude with local formulas for $\db^\ast$, $\dhor^\ast$, and $\dbbar^\ast$.

\begin{lemma}
 \label{divergence-formula}
 Let $(M^{2n+1},T^{1,0},\theta)$ be a pseudohermitian manifold.
 Let $\omega\in\mR^{p,q}$.
 \begin{enumerate}
  \item If $p+q\leq n$, then
  \begin{align*}
   \db^\ast\omega & \equiv -\frac{1}{(p-1)!q!}\nabla^\mu\omega_{\mu\Alpha^\prime\bar\Beta}\,\theta^{\Alpha^\prime}\wedge\theta^{\bar\Beta} \mod\theta , \\
   \dbbar^\ast\omega & \equiv -\frac{(-1)^p}{p!(q-1)!}\nabla^{\bar\nu}\omega_{\Alpha\bar\nu\bar\Beta^\prime}\,\theta^{\Alpha}\wedge\theta^{\bar\Beta^\prime} \mod\theta ,
  \end{align*}
  where $\omega\equiv\frac{1}{p!q!}\omega_{\Alpha\bar\Beta}\,\theta^{\Alpha}\wedge\theta^{\bar\Beta}\mod\theta$.
  \item If $p+q=n+1$, then
  \begin{align*}
   \db^\ast\omega & \equiv \frac{(-1)^p(q+1)i}{(p-2)!(q+1)!}\left(\nabla_{\bar\beta}\nabla^\mu\omega_{\mu\Alpha^\prime\bar\Beta} - iA_{\bar\beta}{}^\mu\omega_{\mu\Alpha^\prime\bar\Beta}\right)\,\theta^{\Alpha^\prime}\wedge\theta^{\bar\beta\bar\Beta} \mod\theta , \\
   \dhor^\ast\omega & \equiv -\frac{1}{(p-1)!q!}\Bigl(\nabla_0\omega_{\Alpha\bar\Beta} \\
    & \qquad + (p-1)i\nabla_{\alpha}\nabla^\mu\omega_{\mu\Alpha^\prime\bar\Beta} - qi\nabla_{\bar\beta}\nabla^{\bar\nu}\omega_{\Alpha\bar\nu\bar\Beta^\prime}\Bigr)\,\theta^{\Alpha}\wedge\theta^{\bar\Beta} \mod\theta , \\
   \dbbar^\ast\omega & \equiv \frac{(-1)^ppi}{p!(q-1)!}\left(\nabla_{\alpha}\nabla^{\bar\nu}\omega_{\Alpha\bar\nu\bar\Beta^\prime} + iA_{\alpha}{}^{\bar\nu}\omega_{\Alpha\bar\nu\bar\Beta^\prime}\right)\,\theta^{\alpha\Alpha}\wedge\theta^{\bar\Beta^\prime} \mod\theta ,
  \end{align*}
  where $\omega=\frac{1}{(p-1)!q!}\omega_{\Alpha\bar\Beta}\,\theta\wedge\theta^{\Alpha}\wedge\theta^{\bar\Beta}$.
  \item If $p+q \geq n+2$, then
  \begin{align*}
   \db^\ast \omega & = \frac{1}{(n+2-p)!(n-q)!(p+q-n-2)!} \Omega_{\Alpha\bar\beta\bar\Beta}^{(1)} \, \theta \wedge \theta^\Alpha \wedge \theta^{\bar\beta\bar\Beta} \wedge d\theta^{p+q-n-2} , \\
   \dbbar^\ast \omega & = \frac{1}{(n+1-p)!(n+1-q)!(p+q-n-2)!}\Omega_{\Alpha\bar\Beta}^{(2)} \, \theta \wedge \theta^{\alpha\Alpha} \wedge \theta^{\bar\Beta} \wedge d\theta^{p+q-n-2} ,
  \end{align*}
  where $\omega = \frac{1}{(n+1-p)!(n-q)!(p+q-n-1)!}\omega_{\Alpha\bar\Beta}\,\theta \wedge \theta^\Alpha \wedge \theta^{\bar\Beta} \wedge d\theta^{p+q-n-1}$ and
  \begin{align*}
   \Omega_{\bar\beta\Alpha\bar\Beta}^{(1)} & := (-1)^{n-q}(n+2-p)i \left( \nabla_{[\bar\beta}\omega_{\Alpha\bar\Beta]} - \frac{n-q}{p+q-n}h_{[\alpha\bar\beta}\nabla^\mu\omega_{\mu\Alpha^\prime\bar\Beta]} \right) , \\
   \Omega_{\alpha\Alpha\bar\Beta}^{(2)} & := -(n+1-q)i\left( \nabla_{[\alpha}\omega_{\Alpha\bar\Beta]} - \frac{n+1-p}{p+q-n}h_{[\alpha\bar\beta}\nabla^{\bar\nu}\omega_{\Alpha\bar\nu\bar\Beta^\prime]} \right) .
  \end{align*}
 \end{enumerate}
\end{lemma}

\begin{proof}
 We prove the case $p+q\leq n$;
 the remaining cases are similar.

 Let $\omega\in\mR^{p,q}$.
 Set $\tau:=\hodge\omega$.
 \Cref{bigraded-operators,hodge-star-sE} imply that
 \begin{align*}
  \db\tau & = \frac{(-1)^{\frac{(p+q)(p+q+1)}{2}}i^{q-p+1}(-1)^{p}}{p!(q-1)!(n+1-p-q)!}\nabla^{\bar\nu}\omega_{\Alpha\bar\nu\bar\Beta^\prime}\,\theta\wedge\theta^{\Alpha}\wedge\theta^{\bar\Beta^\prime}\wedge d\theta^{n+1-p-q} , \\
  \dbbar\tau & = \frac{(-1)^{\frac{(p+q)(p+q+1)}{2}}i^{q-p-1}}{(p-1)!q!(n+1-p-q)!}\nabla^\mu\omega_{\mu\Alpha^\prime\bar\Beta}\,\theta\wedge\theta^{\Alpha^\prime}\wedge\theta^{\bar\Beta}\wedge d\theta^{n+1-p-q} .
 \end{align*}
 The claimed formulas follow from \cref{hodge-star-sF,defn:dbbar-adjoint}.
\end{proof}

%% file: hodge/setup.tex
\section{Some convenient operators}
\label{sec:setup}

Recall that if $(M^{2n+1},T^{1,0},\theta)$ is a pseudohermitian manifold, then $\Omega^{p,q}M$ is the space of complex-valued differential $(p+q)$-forms $\omega$ such that $\omega(T,\cdot)=0$ and $\omega\rv_{\CH}\in\Lambda^{p,q}$.
This allows us to define primitive $(p,q)$-forms.

\begin{definition}
 Let $(M^{2n+1},T^{1,0},\theta)$ be a pseudohermitian manifold.
 The space of \defn{primitive $(p,q)$-forms $P^{p,q}M$} is
 \begin{equation*}
  P^{p,q}M = \left\{ \frac{1}{p!q!}\omega_{\Alpha\bar\Beta}\,\theta^{\Alpha}\wedge\theta^{\bar\Beta} \in \Omega^{p,q}M \suchthatcolon \omega_{\mu\Alpha^\prime}{}^\mu{}_{\bar\Beta^\prime} = 0 \right\} .
 \end{equation*}
\end{definition}

\Cref{lefschetz-computation} implies that this is consistent with \cref{defn:theta-wedge-primitive}.

Given a choice of contact form, \cref{projection-to-E-expression} provides an isomorphism between $\mR^{p,q}$ and $P^{p,q}M$ when $p+q\leq n$.
In this section we introduce new operators on $\Omega^{p,q}M$ and $P^{p,q}M$ which help us systematically study the operators $\dbbar$, $\db$, $\dhor$, and their adjoints.

We begin by splitting the Tanaka--Webster connection according to types.

\begin{definition}
 \label{defn:nablab}
 Let $(M^{2n+1},T^{1,0},\theta)$ be a pseudohermitian manifold.
 We define $\nablab\colon\Omega^{p,q}M\to\Omega^{1,0}M\otimes\Omega^{p,q}M$ by
 \[ \nablab\omega := \frac{1}{p!q!}\nabla_\gamma\omega_{\Alpha\bar\Beta}\,\theta^\gamma\otimes\theta^{\Alpha}\wedge\theta^{\bar\Beta} ; \]
 we define $\nablabbar\colon\Omega^{p,q}M\to\Omega^{0,1}M\otimes\Omega^{p,q}M$ by
 \[ \nablabbar\omega := \frac{1}{p!q!}\nabla_{\bar\sigma}\omega_{\Alpha\bar\Beta}\,\theta^{\bar\sigma}\otimes\theta^{\Alpha}\wedge\theta^{\bar\Beta} ; \]
 and we define $\nabla_0\colon\Omega^{p,q}M\to\Omega^{p,q}M$ by
 \[ \nabla_0\omega = \frac{1}{p!q!}\nabla_0\omega_{\Alpha\bar\Beta}\,\theta^{\Alpha}\wedge\theta^{\bar\Beta} , \]
 where $\omega=\frac{1}{p!q!}\omega_{\Alpha\bar\Beta}\,\theta^{\Alpha}\wedge\theta^{\bar\Beta}$.
\end{definition}

Note that $\nabla_0$ is a real operator and that $\nablab$ and $\nablabbar$ are conjugates of one another.

Skew symmetrizing $\nablab$ and $\nablabbar$ yields useful operators on $\Omega^{\bullet,\bullet}M:=\bigoplus_{p,q}\Omega^{p,q}M$ (cf.\ \citelist{ \cite{GarfieldLee1998}*{Equation~(2)} \cite{Rumin1990}*{Equation~(2)} }).

\begin{definition}
 \label{defn:nablas}
 Let $(M^{2n+1},T^{1,0},\theta)$ be a pseudohermitian manifold.
 We define
 \begin{align*}
  \nablas\omega & := \frac{1}{p!q!}\nabla_{\alpha}\omega_{\Alpha\bar\Beta}\,\theta^{\alpha\Alpha}\wedge\theta^{\bar\Beta}, \\
  \onablas\omega & := \frac{(-1)^p}{p!q!}\nabla_{\bar\beta}\omega_{\Alpha\bar\Beta}\,\theta^{\Alpha}\wedge\theta^{\bar\beta\bar\Beta}
 \end{align*}
 for all $\omega=\frac{1}{p!q!}\omega_{\Alpha\bar\Beta}\,\theta^{\Alpha}\wedge\theta^{\bar\Beta}\in\Omega^{p,q}M$.
\end{definition}

Note that $\nablas$ and $\onablas$ are conjugates of one another.

The pseudohermitian curvature tensor, the pseudohermitian Ricci tensor, and the pseudohermitian torsion all act on $\Omega^{p,q}M$ in a natural way.

\begin{definition}
 \label{defn:curvature-actions}
 Let $(M^{2n+1},T^{1,0},\theta)$ be a pseudohermitian manifold.
 We define
 \begin{align*}
  R\ddbarhash\omega & := \frac{pq}{p!q!}R_{\alpha\bar\beta}{}^{\bar\nu\mu}\omega_{\mu\Alpha^\prime\bar\nu\bar\Beta^\prime}\,\theta^{\Alpha}\wedge\theta^{\bar\Beta}, \\
  \Ric\hash\omega & := -\frac{p}{p!q!}R_{\alpha}{}^\mu\omega_{\mu\Alpha^\prime\bar\Beta}\,\theta^{\Alpha}\wedge\theta^{\bar\Beta}, \\
  \Ric\ohash\omega & := -\frac{q}{p!q!}R^{\bar\nu}{}_{\bar\beta}\omega_{\Alpha\bar\nu\bar\Beta^\prime}\,\theta^{\Alpha}\wedge\theta^{\bar\Beta}, \\
  A\hash\omega & := -\frac{(-1)^p(p+1)}{(p+1)!(q-1)!}A_{\alpha}{}^{\bar\nu}\omega_{\Alpha\bar\nu\bar\Beta^\prime}\,\theta^{\alpha\Alpha}\wedge\theta^{\bar\Beta^\prime} , \\
  A\ohash\omega & := -\frac{(-1)^{p-1}(q+1)}{(p-1)!(q+1)!}A_{\bar\beta}{}^\mu\omega_{\mu\Alpha^\prime\bar\Beta}\,\theta^{\Alpha^\prime}\wedge\theta^{\bar\beta\bar\Beta}
 \end{align*}
 for all $\omega=\frac{1}{p!q!}\omega_{\Alpha\bar\Beta}\,\theta^{\Alpha}\wedge\theta^{\bar\Beta}\in\Omega^{p,q}M$.
\end{definition}

In general, only $A\hash$ and $A\ohash$ preserve primitive forms.

Note that $R\hash\ohash$ is a real operator and that $\Ric\hash$ and $\Ric\ohash$ (resp.\ $A\hash$ and $A\ohash$) are conjugates of one another.

The pseudohermitian structure induces an hermitian inner product on $\Omega^{p,q}M$.

\begin{definition}
 \label{defn:Omega-inner-product}
 Let $(M^{2n+1},T^{1,0},\theta)$ be a pseudohermitian manifold.
 We define
 \[ \lp \omega, \tau \rp := \frac{1}{p!q!}\omega_{\Alpha\bar\Beta}\otau^{\bar\Beta\Alpha} \]
 for all $\omega=\frac{1}{p!q!}\omega_{\Alpha\bar\Beta}\,\theta^{\Alpha}\wedge\theta^{\bar\Beta}$ and $\tau=\frac{1}{p!q!}\tau_{\Alpha\bar\Beta}\,\theta^{\Alpha}\wedge\theta^{\bar\Beta}$ in $\Omega^{p,q}M$, where $\otau_{\Beta\bar\Alpha}:=\overline{\tau_{\Alpha\bar\Beta}}$.
 This induces the \defn{$L^2$-inner product}
 \begin{equation}
  \label{eqn:Omega-inner-product}
  ( \omega, \tau ) := \frac{1}{n!}\int_M \lp \omega, \tau\rp\,\theta\wedge d\theta^n
 \end{equation}
 on compactly supported elements of $\Omega^{p,q}M$.
\end{definition}

In particular, we have a Laplace-type operator associated to $\onablas$.

\begin{definition}
 Let $(M^{2n+1},T^{1,0},\theta)$ be a pseudohermitian manifold.
 We define
 \[ \Boxs\omega := \onablas^\ast\onablas \omega + \onablas\onablas^\ast \omega \]
 for all $\omega\in\Omega^{p,q}M$, where $\onablas^\ast$ is the $L^2$-adjoint of $\onablas$.
\end{definition}

We conclude this section with some useful formulas involving these operators.

First, we give expressions for $\nablas^\ast$, $\nablab^\ast\nablab$, and their conjugates.

\begin{lemma}
 \label{nablab-and-nablas-adjoints}
 Let $(M^{2n+1},T^{1,0},\theta)$ be a pseudohermitian manifold.
 Then
 \begin{align}
  \label{eqn:nablasast} \nablas^\ast\omega & = -\frac{1}{(p-1)!q!}\nabla^\mu\omega_{\mu\Alpha^\prime\bar\Beta}\,\theta^{\Alpha^\prime}\wedge\theta^{\bar\Beta}, \\
  \label{eqn:onablasast} \onablas^\ast\omega & = -\frac{(-1)^p}{p!(q-1)!}\nabla^{\bar\nu}\omega_{\Alpha\bar\nu\bar\Beta^\prime}\,\theta^{\Alpha}\wedge\theta^{\bar\Beta^\prime} , \\
  \label{eqn:rough-laplacian} \nablab^\ast\nablab\omega & = -\frac{1}{p!q!}\nabla^\mu\nabla_\mu\omega_{\Alpha\bar\Beta}\,\theta^{\Alpha}\wedge\theta^{\bar\Beta}, \\
  \label{eqn:rough-laplacian-bar} \nablabbar^\ast\nablabbar\omega & = -\frac{1}{p!q!}\nabla^{\bar\nu}\nabla_{\bar\nu}\omega_{\Alpha\bar\Beta}\,\theta^{\Alpha}\wedge\theta^{\bar\Beta} ,
 \end{align}
 for all $\omega=\frac{1}{p!q!}\omega_{\Alpha\bar\Beta}\,\theta^{\Alpha}\wedge\theta^{\bar\Beta}\in\Omega^{p,q}M$.
\end{lemma}

\begin{proof}
 Let $\omega=\frac{1}{(p-1)!q!}\omega_{\Alpha\bar\Beta}\,\theta^{\Alpha}\wedge\theta^{\bar\Beta}\in\Omega^{p-1,q}M$ and $\tau=\frac{1}{p!q!}\tau_{\Gamma\bar\Sigma}\,\theta^{\Gamma}\wedge\theta^{\bar\Sigma}\in\Omega^{p,q}M$ be compactly supported.
 We compute from \cref{defn:nablas,defn:Omega-inner-product} that
 \begin{align*}
  (\omega, \nablas^\ast\tau) := (\nablas\omega, \tau) & = \frac{1}{n!(p-1)!q!}\int_M \otau^{\bar\Beta\alpha\Alpha}\nabla_\alpha\omega_{\Alpha\bar\Beta}\,\theta\wedge d\theta^n \\
  & = -\frac{1}{n!(p-1)!q!}\int_M \omega_{\Alpha\bar\Beta}\nabla_\mu\tau^{\bar\Beta\mu\Alpha}\,\theta\wedge d\theta^n , 
 \end{align*}
 where the last equality uses the divergence theorem~\cite{Lee1988}*{Equation~(2.18)}.
 Equation~\eqref{eqn:nablasast} readily follows.
 Equations~\eqref{eqn:onablasast}, \eqref{eqn:rough-laplacian}, and~\eqref{eqn:rough-laplacian-bar} follow similarly.
\end{proof}

Second, we record the canonical isomorphism between $\mR^{p,q}$ and $P^{p,q}M$ and the relations, in terms of this isomorphism, between the operators in the bigraded Rumin complex and our newly defined operators (cf.\ \citelist{ \cite{Garfield2001}*{Propositions~2.9 and~3.10} }).

\begin{lemma}
 \label{sR-to-P}
 Let $(M^{2n+1},T^{1,0},\theta)$ be a pseudohermitian manifold.
 Let $p,q\in\bN_0$ be such that $p+q\leq n$.
 Define $I\colon\mR^{p,q}\to P^{p,q}M$ by
 \[ I(\omega) := \frac{1}{p!q!}\omega_{\Alpha\bar\Beta}\,\theta^{\Alpha}\wedge\theta^{\bar\Beta}, \]
 where $\omega\equiv\frac{1}{p!q!}\omega_{\Alpha\bar\Beta}\,\theta^{\Alpha}\wedge\theta^{\bar\Beta}\mod\theta$.
 Then $I$ is an isomorphism.
 With respect to this isomorphism, it holds that
 \begin{align}
  \label{eqn:dbast-to-nablasast} \db^\ast\omega & = \nablas^\ast \omega , \\
  \label{eqn:dbbarast-to-onablasast} \dbbar^\ast\omega & = \onablas^\ast \omega .
 \end{align}
 Moreover,
 \begin{equation}
  \label{eqn:dbbar-to-onablas}
  \dbbar\omega = \onablas \omega + \frac{i}{n-p-q+1}d\theta \wedge \nablas^\ast \omega
 \end{equation}
 if $p+q\leq n-1$; and
 \begin{align}
  \label{eqn:db-to-nablas-middle} \hodge\db\omega & = (-1)^pi^{n^2}\left(-i\nablas\onablas^\ast + A\hash\right)\omega , \\
  \label{eqn:dhor-to-onablas-middle} \hodge\dhor\omega & = (-1)^pi^{n^2}\left( \nabla_0 - i\nablas\nablas^\ast + i\onablas\onablas^\ast\right) \omega , \\
  \label{eqn:dbbar-to-onablas-middle} \hodge\dbbar\omega & = (-1)^pi^{n^2}\left(i\onablas\nablas^\ast + A\ohash\right)\omega
 \end{align}
 if $p+q=n$.
\end{lemma}

\begin{proof}
 \Cref{projection-to-E-expression} implies that $I$ is an isomorphism.
 
 Combining \cref{divergence-formula,nablab-and-nablas-adjoints} yields Equations~\eqref{eqn:dbast-to-nablasast} and~\eqref{eqn:dbbarast-to-onablasast}.
 
 Combining \cref{bigraded-operators,nablab-and-nablas-adjoints} yields Equation~\eqref{eqn:dbbar-to-onablas}.
 
 Combining \cref{bigraded-operators,hodge-star-sF,nablab-and-nablas-adjoints} yields Equations~\eqref{eqn:db-to-nablas-middle}, \eqref{eqn:dhor-to-onablas-middle}, and~\eqref{eqn:dbbar-to-onablas-middle}.
\end{proof}

In particular, the operator $\nablas^\ast + \onablas^\ast$ squares to zero on primitive forms.

\begin{corollary}
 \label{nablasast-complex}
 Let $(M^{2n+1},T^{1,0},\theta)$ be a pseudohermitian manifold.
 Let $p,q \in \bN_0$ be such that $p+q \leq n$.
 Then
 \begin{align*}
  \nablas^\ast\nablas^\ast & = 0 , \\
  \onablas^\ast\onablas^\ast & = 0 , \\
  \nablas^\ast\onablas^\ast + \onablas^\ast\nablas^\ast & = 0
 \end{align*}
 on $P^{p,q}M$.
\end{corollary}

\begin{proof}
 This follows immediately from \cref{dual-justification,sR-to-P}.
\end{proof}

Third, we compute the commutator of $\onablas^\ast$ and the Lefschetz operator.

\begin{lemma}
 \label{onablas-simple-identities}
 Let $(M^{2n+1},T^{1,0},\theta)$ be a pseudohermitian manifold.
 Then
 \begin{align}
  \label{eqn:onablasonablah} \onablas^\ast( d\theta \wedge \omega ) & = d\theta \wedge \onablas^\ast\omega + i\nablas\omega , \\
  \label{eqn:nablasnablah} \nablas^\ast( d\theta \wedge \omega ) & = d\theta \wedge \nablas^\ast\omega - i\onablas\omega 
 \end{align}
 for all $\omega\in\Omega^{p,q}M$.
\end{lemma}

\begin{proof}
 Let $\omega=\frac{1}{p!q!}\omega_{\Alpha\bar\Beta}\,\theta^{\Alpha}\wedge\theta^{\bar\Beta}\in\Omega^{p,q}M$.
 \Cref{nablab-and-nablas-adjoints} implies that
 \begin{equation*}
  \onablas^\ast( d\theta \wedge \omega) = \frac{i}{p!q!}\left( \nabla_\alpha\omega_{\Alpha\bar\Beta} - qh_{\alpha\bar\beta}\nabla^{\bar\nu}\omega_{\Alpha\bar\nu\bar\Beta^\prime} \right)\,\theta^{\alpha\Alpha}\wedge\theta^{\bar\Beta} ,
 \end{equation*}
 which is equivalent to Equation~\eqref{eqn:onablasonablah}.
 Equation~\eqref{eqn:nablasnablah} follows by conjugation.
\end{proof}

Fourth, we derive a Weitzenb\"ock-type formula for $\Boxs$.

\begin{lemma}
 \label{nablas-weitzenbock}
 Let $(M^{2n+1},T^{1,0},\theta)$ be a pseudohermitian manifold.
 Then
 \[ \Boxs\omega = \frac{q}{n}\nablab^\ast\nablab\omega + \frac{n-q}{n}\nablabbar^\ast\nablabbar\omega - R\hash\ohash\omega - \frac{q}{n}\Ric\hash\omega - \frac{n-q}{n}\Ric\ohash\omega \]
 for all $\omega\in\Omega^{p,q}M$.
\end{lemma}

\begin{proof}
 Let $\omega=\frac{1}{p!q!}\omega_{\Alpha\bar\Beta}\,\theta^{\Alpha}\wedge\theta^{\bar\Beta}\in\Omega^{p,q}M$.
 \Cref{nablab-and-nablas-adjoints} implies that
 \begin{align}
  \label{eqn:onablasonablasast} \onablas\onablas^\ast\omega & = -\frac{1}{p!(q-1)!}\nabla_{\bar\beta}\nabla^{\bar\nu}\omega_{\Alpha\bar\nu\bar\Beta^\prime}\,\theta^{\Alpha}\wedge\theta^{\bar\Beta} , \\
  \label{eqn:onablasastonablas} \onablas^\ast\onablas\omega & = -\frac{1}{p!q!}\left(\nabla^{\bar\nu}\nabla_{\bar\nu}\omega_{\Alpha\bar\Beta} - q\nabla^{\bar\nu}\nabla_{\bar\beta}\omega_{\Alpha\bar\nu\bar\Beta^\prime} \right)\,\theta^{\Alpha}\wedge\theta^{\bar\Beta} .
 \end{align}
 On the one hand, using \cref{commutators} to add Equations~\eqref{eqn:onablasonablasast} and~\eqref{eqn:onablasastonablas} yields
 \begin{equation}
  \label{eqn:Boxs-with-0}
  \Boxs\omega = \nablabbar^\ast\nablabbar\omega - qi\nabla_0\omega - R\hash\ohash\omega - \Ric\ohash\omega .
 \end{equation}
 On the other hand, \cref{commutators,nablab-and-nablas-adjoints} imply that
 \begin{equation}
  \label{eqn:nabla0-fancy-commutator}
  \nablabbar^\ast\nablabbar\omega - \nablab^\ast\nablab\omega = ni\nabla_0\omega - \Ric\hash\omega + \Ric\ohash\omega .
 \end{equation}
 Combining Equations~\eqref{eqn:Boxs-with-0} and~\eqref{eqn:nabla0-fancy-commutator} yields the desired conclusion.
\end{proof}

Fifth, we derive a formula for $\onablas\nablas^\ast + \nablas^\ast\onablas \colon\Omega^{p,q}M \to \Omega^{p-1,q+1}M$.

\begin{lemma}
 \label{slash-laplacian-add-one-to-q}
 Let $(M^{2n+1},T^{1,0},\theta)$ be a pseudohermitian manifold.
 Then
 \begin{align}
  \label{eqn:slash-laplacian-down} \onablas\nablas^\ast\omega + \nablas^\ast\onablas\omega & = (n-p-q)iA\ohash\omega , \\
  \label{eqn:slash-laplacian-up} \nablas\onablas^\ast\omega + \onablas^\ast\nablas\omega & = -(n-p-q)iA\hash\omega
 \end{align}
 for all $\omega\in\Omega^{p,q}M$.
\end{lemma}

\begin{proof}
 Let $\omega=\frac{1}{p!q!}\omega_{\Alpha\bar\Beta}\,\theta^{\Alpha}\wedge\theta^{\bar\Beta}\in\Omega^{p,q}M$.
 \Cref{nablab-and-nablas-adjoints} implies that
 \begin{align*}
  \onablas\nablas^\ast \omega & = \frac{(-1)^p}{(p-1)!q!}\nabla_{\bar\beta}\nabla^\mu\omega_{\mu\Alpha^\prime\bar\Beta}\,\theta^{\Alpha^\prime}\wedge\theta^{\bar\beta\bar\Beta} , \\
  \nablas^\ast\onablas\omega & = -\frac{(-1)^p}{(p-1)!q!}\nabla^\mu\nabla_{\bar\beta}\omega_{\mu\Alpha^\prime\bar\Beta}\,\theta^{\Alpha^\prime}\wedge\theta^{\bar\beta\bar\Beta} .
 \end{align*}
 Adding these equations using \cref{commutators} yields Equation~\eqref{eqn:slash-laplacian-down}.
 Equation~\eqref{eqn:slash-laplacian-up} follows by conjugation.
\end{proof}

Finally, we derive a partial analogue of the identity $\dbbar^\ast\db^\ast + \db^\ast\dbbar^\ast=0$.

\begin{lemma}
 \label{triple-commute}
 Let $(M^{2n+1},T^{1,0},\theta)$ be a pseudohermitian manifold.
 If $\omega\in\Omega^{p,0}M$, $p\in\{0,\dotsc,n+1\}$, then
 \begin{align}
  \label{eqn:triple-down} \onablas^\ast\nablas^\ast\onablas\omega + \nablas^\ast\onablas^\ast\onablas\omega + i\nabla_0\nablas^\ast\omega & = 0 , \\
  \label{eqn:triple-up} \onablas^\ast\nablas^\ast\onablas\omega + \nablas^\ast\onablas^\ast\onablas\omega + i\nabla_0\nablas^\ast\omega & = 0 .
 \end{align}
\end{lemma}

\begin{proof}
 Let $\omega=\frac{1}{p!}\omega_{\Alpha}\,\theta^{\Alpha}$.
 Using \cref{commutators,nablab-and-nablas-adjoints}, we compute that
 \begin{align*}
  \onablas^\ast\nablas^\ast\onablas\omega & = -\frac{1}{(p-1)!}\nabla^{\bar\nu}\nabla^\mu\nabla_{\bar\nu}\omega_{\mu\Alpha^\prime}\,\theta^{\Alpha^\prime} \\
  & = -\frac{1}{(p-1)!}\left( \nabla^\mu\nabla^{\bar\nu}\nabla_{\bar\nu}\omega_{\mu\Alpha^\prime} - i\nabla_0\nabla^\mu\omega_{\mu\Alpha^\prime} \right)\,\theta^{\Alpha^\prime} \\
  & = -\nablas^\ast\onablas^\ast\onablas\omega - i\nabla_0\nablas^\ast\omega . \qedhere
 \end{align*}
 Hence Equation~\eqref{eqn:triple-down} holds.
 Equation~\eqref{eqn:triple-up} follows by conjugation.
\end{proof}

%% file: hodge/folland-stein.tex
\section{Folland--Stein spaces}
\label{sec:hodge/folland-stein}

In this section we collect some necessary background about Folland--Stein spaces on closed, strictly pseudoconvex, pseudohermitian manifolds.

Let $(M^{2n+1},T^{1,0},\theta)$ be a closed, strictly pseudoconvex manifold and let $E$ be a \defn{CR tensor bundle (over $M$)}; i.e.\ a subbundle of $\bigoplus_{a,b \in \bN_0} (T^{1,0})^{\otimes a} \otimes (T^{0,1})^{\otimes b}$.
Let $( \cdot , \cdot )$ denote the $L^2$-inner product on the space $\Gamma(E)$ of smooth sections of $E$ induced by the Levi form (cf.\ \eqref{eqn:Omega-inner-product}).
Let $\nablab$ and $\nablabbar$ and $\nabla_0$ be the obvious components of the Tanaka--Webster connection on $\Gamma(E)$, and let $\nablab^\ast$ and $\nablabbar^\ast$ and $\nabla_0^\ast$ denote their respective adjoints.
Given $\alpha \in \bC$, we define the \defn{Folland--Stein operator} $\mL_\alpha \colon \Gamma(E) \to \Gamma(E)$ by
\begin{equation*}
 \mL_\alpha := \nablab^\ast\nablab + \nablabbar^\ast\nablabbar + i\alpha\nabla_0 .
\end{equation*}
Given $s \in \bZ$, the \defn{Folland--Stein space $S^{s}(M;E)$} is the Hilbert space obtained as the completion of $\Gamma(E)$ with respect to the inner product
\begin{equation}
 \label{eqn:fs-inner-product}
 ( \omega, \tau )_s := ( \omega , (1 + \mL_0)^s\tau ) .
\end{equation}
It is known~\cites{Ponge2008,ChengLee1990} that these spaces are equivalent to the original spaces introduced by Folland and Stein~\cite{FollandStein1974} when $s \in \bN_0$.
When the CR tensor bundle $E$ is clear from context, we shall simply write $S^{s}(M)$ for the corresponding Folland--Stein space;
when $E$ is the vector bundle whose sections are the space $P^{p,q}$ of primitive $(p,q)$-forms, we sometimes write the hermitian inner product as $\llp \cdot , \cdot \rrp$ to emphasize the equivalence with~\eqref{eqn:sR-inner-product}.
In all cases, we denote the squared norm on $S^s(M)$ by
\begin{equation*}
 \lV \omega \rV_{s}^2 := ( \omega , \omega )_s .
\end{equation*}

We first record some basic mapping properties involving Folland--Stein spaces.

\begin{lemma}
 \label{embeddings}
 Let $(M^{2n+1},T^{1,0},\theta)$ be a closed, strictly pseudoconvex manifold and let $E$ be a CR tensor bundle.
 Then
 \begin{enumerate}
  \item $\Gamma(E) = \bigcap_{s\in\bZ} S^{s}(M;E)$;
  \item if $s,t \in \bN_0$ are such that $s < t$, then the inclusion map $S^{t}(M) \hookrightarrow S^{s}(M)$ is a compact operator; and
  \item if $r < s < t$ are integers and $\varepsilon>0$, then there is a constant $C>0$ such that
  \begin{equation*}
   \lV \omega \rV_s \leq \varepsilon \lV \omega \rV_t + C \lV \omega \rV_r
  \end{equation*}
  for all $\omega \in S^t(M)$.
 \end{enumerate}
 Moreover, in the case $E = \Lambda^{p,q}$, it holds that
 \begin{enumerate}
  \setcounter{enumi}{3}
  \item $\hodge \colon S^s(M) \to S^s(M)$ is an isomorphism; and
  \item conjugation induces an isomorphism $S^s(M;\Lambda^{p,q}) \cong S^s(M;\Lambda^{q,p})$.
 \end{enumerate}
\end{lemma}

\begin{proof}
 Denote by $W^{\ell,2}(M;E)$ the usual fractional Sobolev space defined with respect to some fixed Riemannian metric on $M$.
 It is known (cf.\ \citelist{ \cite{Ponge2008}*{Proposition~5.5.7} \cite{ChengLee1990}*{Lemma~4.1} }) that if $s \geq 0$, then the embeddings
 \begin{equation*}
  W^{s,2}(M;E) \hookrightarrow S^s(M;E) \hookrightarrow W^{s/2,2}(M;E)
 \end{equation*}
 are continuous.
 
 Statement~(i) follows immediately from the embedding $S^s(M;E) \hookrightarrow W^{s/2}(M;E)$ when $s\geq0$.
 
 Let $s>0$.
 The embedding $S^s(M;E) \hookrightarrow W^{s/2}(M;E)$ implies that the inclusion $S^s(M;E) \hookrightarrow S^0(M;E)$ is compact.
 Now let $s<t$ be integers and let $( \omega_j )_{j \in \bN}$ be a bounded sequence in $S^t(M)$.
 Then $\bigl( (1 + \mL_0)^{-s/2}\omega_j \bigr)_{j\in\bN}$ has a convergent subsequence.
 Statement~(ii) readily follows.
 
 Statement~(iii) follows easily from the compactness of $S^t(M) \hookrightarrow S^s(M) \hookrightarrow S^r(M)$.
 
 Statement~(iv) follows from the fact that $\theta \wedge d\theta^n$ is parallel with respect to the Tanaka--Webster connection.
 
 Statement~(v) follows from the fact that conjugation is a fiberwise isomorphism.
\end{proof}

Given two CR tensor bundles $E$ and $F$, a \defn{differential operator} is a linear operator $D \colon \Gamma(E) \to \Gamma(F)$ of the form
\begin{equation*}
 D\omega = \sum_{j=0}^k \sum_{r+s+2t = j} \lcpcontr \left( A \otimes \nablab^r \nablabbar^s\nabla_0^t \omega \right)
\end{equation*}
where $\lcpcontr ( A \otimes \nablab^r \nablabbar^s\nabla_0^t \omega )$ denotes a linear combinations of partial contractions of tensors of the form $A \otimes \nablab^r \nablabbar^s \nabla_0^t \omega$.
If $k > 0$, then we require that the sum $\sum_{r+s+2t=k} \lcpcontr ( A \otimes \nablab^r \nablabbar^s \nabla_0^t \omega )$ is nonzero;
in this case we call $k$ the \defn{order} of $D$.
If $k=0$, then we say that $D$ has \defn{order zero}.
We denote by \defn{$\DiffOperator{k}$} the set of differential operators of order $k$.

More generally, we denote by \defn{$\PseudoDiffOperator{k}$}, $k \in \bZ$, the set of Heisenberg pseudodifferential operators $D \colon \Gamma(E) \to \Gamma(F)$ of order $k$.
Recall that
\begin{equation*}
 \PseudoDiffOperator{-\infty} = \bigcap_{k\in\bZ} \PseudoDiffOperator{k}
\end{equation*}
is the space of \defn{smoothing operators}.
Here we only collect important properties of Heisenberg pseudodifferential operators, and refer to the monographs of Beals and Greiner~\cite{BealsGreiner1988} and Ponge~\cite{Ponge2008} for definitions and details.
Note that $\DiffOperator{k} \subset \PseudoDiffOperator{k}$.

The following lemma collects some basic properties of Heisenberg pseudodifferential operators.

\begin{lemma}
 \label{compositions-in-pseudodiffoperator}
 Let $(M^{2n+1},T^{1,0},\theta)$ be a closed, strictly pseudoconvex manifold.
 \begin{enumerate}
  \item If $D \in \PseudoDiffOperator{k}$ and $s \in \bZ$, then $D$ extends to a bounded linear map $D \colon S^{s+k}(M) \to S^s(M)$.
  \item If $D_1 \in \PseudoDiffOperator{k_1}$ and $D_2 \in \PseudoDiffOperator{k_2}$ are properly supported and the composition $D_1D_2$ makes sense, then $D_1D_2 \in \PseudoDiffOperator{k_1+k_2}$.
 \end{enumerate}
\end{lemma}

\begin{remark}
 Differential operators are always properly supported.
 More generally, Heisenberg pseudodifferential operators are properly supported modulo smoothing operators~\cite{BealsGreiner1988}*{pg.\ 83}.
\end{remark}

\begin{proof}
 Statement~(i) is~\cite{Ponge2008}*{Proposition~5.5.8}.
 
 Statement~(ii) is~\cite{BealsGreiner1988}*{Theorem~14.1}.
\end{proof}

The remainder of this section concerns necessary and sufficient conditions for a Heisenberg pseudodifferential operator to admit a parametrix.

\begin{definition}
 Let $(M^{2n+1},T^{1,0},\theta)$ be a closed, strictly pseudoconvex manifold, and let $D \in \PseudoDiffOperator{k}$, $k \in \bN$.
 A \defn{parametrix} for $D$ is a $Q \in \PseudoDiffOperator{-k}$ such that $DQ - I$ and $QD - I$ are smoothing operators.
\end{definition}

One can determine whether a Heisenberg pseudodifferential operator admits a parametrix by studying its principal symbol.

\begin{proposition}
 \label{invertibility-characterizations}
 Let $(M^{2n+1},T^{1,0},\theta)$ be a closed, strictly pseudoconvex manifold and let $E$ be a CR tensor bundle.
 Let $D \in \PseudoDiffOperator{k}$, $k \in \bN$, be an endomorphism of $\Gamma(E)$.
 The following are equivalent:
 \begin{enumerate}
  \item $D$ admits a parametrix.
  \item The principal symbol $\sigma_k(D)$ of $D$ is invertible with respect to the convolution product~\citelist{ \cite{BealsGreiner1988}*{Definition~13.10} \cite{Ponge2008}*{Proposition~3.2.8} }.
  \item $D$ and $D^\ast$ satisfy the Rockland condition~\cite{Ponge2008}*{Definition~3.3.8}.
 \end{enumerate}
 Moreover, if any of the equivalent conditions hold, then $D$ is \defn{maximally hypoelliptic}:
 For each $s \in \bZ$, there is a constant $C>0$ such that
 \begin{equation*}
  \lV \omega \rV_{s+k}^2 \leq C \left( \lV D\omega \rV_s^2 + \lV \omega \rV_s^2 \right) .
 \end{equation*}
\end{proposition}

\begin{remark}
 The Rockland condition depends only on the principal symbol of $D$.
 It has recently been shown~\cite{AndroulidakisMohsenYuncken2022} that if $D \in \DiffOperator{k}$, $k \in \bN$, then $D$ admits a parametrix if and only if it is maximally hypoelliptic.
\end{remark}

\begin{proof}
 The equivalence of (i) and (ii) is~\cite{Ponge2008}*{Proposition~3.3.1}.
 
 The equivalence of (ii) and (iii) is~\cite{Ponge2008}*{Theorem~3.3.18}.

 The final conclusion is~\cite{Ponge2008}*{Proposition~5.5.9}.
\end{proof}

One can often show that a given Heisenberg pseudodifferential operator admits a parametrix by comparing it to an appropriate Folland--Stein operator.
Executing this strategy requires two observations.

First, the Folland--Stein operators which admit a parametrix are classified.

\begin{lemma}
 \label{invert-folland-stein}
 Let $(M^{2n+1},T^{1,0},\theta)$ be a closed, strictly pseudoconvex manifold and let $E$ be a CR tensor bundle.
 The Folland--Stein operator $\mL_\alpha$ admits a parametrix if and only if $\alpha \not\in \{ \pm n, \pm (n+2), \dotsc \}$.
\end{lemma}

\begin{proof}
 This follows from~\cite{FollandStein1974}*{Theorem~6.2};
 see also~\cite{Ponge2008}*{pg.\ 62}.
\end{proof}

Second, there are two ways to favorably compare Heisenberg pseudodifferential operators.

\begin{definition}
 \label{defn:comparisons}
 Let $(M^{2n+1},T^{1,0},\theta)$ be a closed, strictly pseudoconvex manifold and let $E$ be a CR tensor bundle.
 Let $A , B \in \PseudoDiffOperator{k}$, $k \in \bN$, be endomorphisms of $\Gamma(E)$.
 \begin{enumerate}
  \item We write \defn{$A \simeq B$} if $A-B \in \PseudoDiffOperator{k-1}$.
  \item We write \defn{$A \gtrsim B$} if there are $P_1 , \dotsc, P_j \in \PseudoDiffOperator{k/2}$ such that
  \begin{equation}
   \label{eqn:simeq-condition}
   A \simeq B + P_1^\ast P_1 + \dotsm + P_j^\ast P_j .
  \end{equation}
 \end{enumerate}
\end{definition}

For example, \cref{commutators} implies that
\begin{equation}
 \label{eqn:folland-stein-no-0}
 \mL_\alpha \equiv \frac{n-\alpha}{n}\nablab^\ast\nablab + \frac{n+\alpha}{n}\nablabbar^\ast\nablabbar \mod \DiffOperator{0} .
\end{equation}
In particular, if $\alpha \in (-n,n)$, then $\mL_\alpha \gtrsim C\mL_0$ for some constant $C>0$.

\Cref{defn:comparisons} is used as follows (cf.\ \cite{Rumin1994}*{Corollaire~3 and Proposition~6}):

\begin{proposition}
 \label{comparisons}
 Let $(M^{2n+1},T^{1,0},\theta)$ be a closed, strictly pseudoconvex manifold and let $E$ be a CR tensor bundle.
 Let $A,B \in \PseudoDiffOperator{k}$, $k \in \bN$, be endomorphisms.
 \begin{enumerate}
  \item If $A \simeq B$, then $A$ admits a parametrix if and only if $B$ admits a parametrix.
  \item If $A \gtrsim B \gtrsim 0$ and $B$ admits a parametrix, then $A$ admits a parametrix.
 \end{enumerate}
\end{proposition}

\begin{proof}
 (i).
 If $A \simeq B$, then $A$ and $B$ have the same principal symbol.
 The conclusion follows from \cref{invertibility-characterizations}.
 
 (ii).
 We prove this using the Rockland condition~\cite{Ponge2008}*{Definition~3.3.8}.
 Let $P_1 , \dotsc , P_j$ be such that~\eqref{eqn:simeq-condition} holds.
 Let $\pi$ be a nontrivial unitary representation of the Heisenberg group and let $x \in M$.
 Let $A^x$, $B^x$, and $P_\ell^x$, $1 \leq \ell \leq j$, be model operators~\cite{Ponge2008}*{Section~3.2} for $A$, $B$, and $P_\ell$, respectively.
 It is known~\cite{Ponge2008}*{Proposition~3.3.6} that $\overline{\pi_{B^x}} \geq 0$ and
 \begin{equation*}
  \overline{\pi_{A^x}} = \overline{\pi_{B^x}} + \overline{\pi_{P_1^x}}^\ast \overline{\pi_{P_1^x}} + \dotsm + \overline{\pi_{P_j^x}}^\ast \overline{\pi_{P_j^x}} .
 \end{equation*}
 (See~\cite{Ponge2008}*{Section~3.3.2} for definitions; note that $\overline{\pi_P}$ can be defined for Heisenberg pseudodifferential operators $P \colon \Gamma(E) \to \Gamma(F)$.)
 \Cref{invertibility-characterizations} implies that $B$ satisfies the Rockland condition;
 in particular, $\overline{\pi_{B^x}}$ is injective.
 Therefore $\overline{\pi_{A^x}}$ is injective.
 Since $A$ is formally self-adjoint, it follows that $A$ satisfies the Rockland condition.
 The conclusion follows from \cref{invertibility-characterizations}.
\end{proof}

We also require a related comparison result.

\begin{proposition}
 \label{interior-composition-comparison}
 Let $(M^{2n+1},T^{1,0},\theta)$ be a closed, strictly pseudoconvex manifold and let $E$ be a CR tensor bundle.
 Let $A \in \PseudoDiffOperator{2k}$ be an endomorphism of $\Gamma(E)$ such that
 \begin{equation*}
  A \simeq \sum_{j=1}^m P_j^\ast P_j 
 \end{equation*}
 and let $B_j^\ast B_j \in \PseudoDiffOperator{2\ell}$, $1 \leq j \leq m$, be endomorphisms of $\Gamma(E)$.
 If $A$ and $B_j^\ast B_j$ admit parametrices, then
 \begin{equation*}
  D := \sum_{j=1}^m P_j^\ast B_j^\ast B_j P_j
 \end{equation*}
 admits a parametrix.
\end{proposition}

\begin{proof}
 Let $\pi$ be a nontrivial unitary representation of the Heisenberg group and let $x \in M$.
 Let $P_j^x$, $B_j^x$, and $D^x$ be model operators~\cite{Ponge2008}*{Section~3.2} for $P_j$, $B_j$, and $D$, respectively.
 It is known~\cite{Ponge2008}*{Proposition~3.3.6} that
 \begin{equation*}
  \overline{\pi_{D^x}} = \sum_{j=1}^\ell \overline{\pi_{P_j^x}}^\ast \overline{\pi_{B_j^x}}^\ast \overline{\pi_{B_j^x}} \overline{\pi_{P_j^x}} .
 \end{equation*}
 Moreover, \cref{invertibility-characterizations} implies that $A$ and $B_j^\ast B_j$ satisfy the Rockland condition;
 in particular, $\sum \overline{\pi_{P_j^x}}^\ast \overline{\pi_{P_j^x}}$ and each of $\overline{\pi_{B_j^x}}^\ast \overline{\pi_{B_j^x}}$ are injective.
 Since $D$ is formally self-adjoint, we deduce that $D$ satisfies the Rockland condition.
 The conclusion follows from \cref{invertibility-characterizations}.
\end{proof}

We conclude with two general theorems needed to prove our Hodge theorems.

The first theorem is a collection of important facts about the partial inverses of Heisenberg pseudodifferential operators which admit a parametrix.

\begin{theorem}[\cite{BealsGreiner1988}*{Theorem~19.16}]
 \label{invert-maximally-hypoelliptic}
 Let $(M^{2n+1},T^{1,0},\theta)$ be a closed, strictly pseudoconvex manifold.
 Let $D \in \PseudoDiffOperator{k}$, $k \in \bN$, be formally self-adjoint and admit a parametrix.
 Then $\ker D$ is finite-dimensional.
 Moreover, there are $H \in \PseudoDiffOperator{-\infty}$ and $N \in \PseudoDiffOperator{-k}$ such that
 \begin{equation}
  \label{eqn:partial-inverse}
  \begin{aligned}
  I & = DN + H , & I & = ND + H , \\
  H^\ast & = H = H^2 , & N^\ast & = N , \\
  HN & = NH = 0 ,
  \end{aligned}
 \end{equation}
 where $I$ is the identity map.
\end{theorem}

The second theorem constructs a partial inverse for certain non-hypoelliptic pseudodifferential operators.

\begin{theorem}
 \label{invert-almost-invertible}
 Let $(M^{2n+1},T^{1,0},\theta)$ be a closed, strictly pseudoconvex manifold.
 Let $D \in \PseudoDiffOperator{k}$, $k \in \bN$, be formally self-adjoint.
 Suppose additionally that the $L^2$-closure of $D$ has closed range.
 If there are $Q \in \PseudoDiffOperator{-k}$ and $S \in \PseudoDiffOperator{0}$ and $R \in \PseudoDiffOperator{-1}$ such that
 \begin{align}
  \label{eqn:DQ-assumption} DQ + S & = I - R , \\
  \label{eqn:DS-assumption} DS & \equiv 0 \mod \PseudoDiffOperator{-\infty} ,
 \end{align}
 then there are $N \in \PseudoDiffOperator{-k}$ and $H \in \PseudoDiffOperator{0}$ such that~\eqref{eqn:partial-inverse} holds.
 Moreover, if $S \in \PseudoDiffOperator{-\infty}$, then $H \in \PseudoDiffOperator{-\infty}$.
\end{theorem}

\begin{proof}
 The assumptions that $D$ is formally self-adjoint and its $L^2$-closure has closed range implies that there are operators $H$ and $N$ on $L^2(M)$ such that~\eqref{eqn:partial-inverse} holds.
 
 Let $A$ be a parametrix for $I-R$.
 Set
 \begin{align*}
  \cS & := SA , \\
  \cQ & := (I-\cS)QA .
 \end{align*}
 Then $\cS \in \PseudoDiffOperator{0}$ and $\cQ \in \PseudoDiffOperator{-k}$;
 moreover, if $S \in \PseudoDiffOperator{-\infty}$, then $\cS \in \PseudoDiffOperator{-\infty}$.
 Following computations~\cite{BealsGreiner1988}*{Proof of Proposition~25.4} of Beals and Greiner, we deduce that
 \begin{align*}
  I & \sim D\cQ + \cS , & I & \sim \cQ D + \cS , \\
  \cS^\ast & \sim \cS \sim \cS^2 , & \cQ^\ast & \sim \cQ , \\
  \cQ\cS & \sim \cS\cQ \sim 0 ,
 \end{align*}
 where $A \sim B$ means $A \equiv B \mod \PseudoDiffOperator{-\infty}$.
 The final conclusion follows from another argument~\cite{BealsGreiner1988}*{Proof of Theorem~25.20} of Beals and Greiner.
\end{proof}

%% file: hodge/weitzenbock.tex
\section{Some identities involving the Kohn and Rumin Laplacians}
\label{sec:weitzenbock}

In this section we derive a number of useful identities involving the Kohn and Rumin Laplacians.
These identities will be used both to prove maximal hypoellipticity and to study harmonic differential forms.

We begin by defining the Kohn and Rumin Laplacians.

\begin{definition}
 \label{defn:kohn-laplacian}
 Let $(M^{2n+1},T^{1,0},\theta)$ be a pseudohermitian manifold.
 The \defn{Kohn Laplacian} on $\mR^{p,q}$ is the operator $\Box_b\colon\mR^{p,q}\to\mR^{p,q}$ given by
 \begin{align*}
  \Box_b & := \frac{n-p-q}{n-p-q+1}\dbbar\dbbar^\ast + \dbbar^\ast\dbbar, && \text{if $p+q\leq n-1$;} \\
  \Box_b & := \dbbar^\ast\dbbar + \dbbar\db^\ast\db\dbbar^\ast + \frac{1}{2}\dbbar\db\db^\ast\dbbar^\ast + \dbbar\dbbar^\ast\dbbar\dbbar^\ast, && \text{if $p+q=n$;} \\
  \Box_b & := \dbbar\dbbar^\ast + \dbbar^\ast\db\db^\ast\dbbar + \frac{1}{2}\dbbar^\ast\db^\ast\db\dbbar + \dbbar^\ast\dbbar\dbbar^\ast\dbbar, && \text{if $p+q=n+1$;} \\
  \Box_b & := \frac{p+q-n-1}{p+q-n}\dbbar^\ast\dbbar + \dbbar\dbbar^\ast, && \text{if $p+q\geq n+2$.}
 \end{align*}
\end{definition}

\begin{definition}
 \label{defn:rumin-laplacian}
 Let $(M^{2n+1},T^{1,0},\theta)$ be a pseudohermitian manifold.
 The \defn{Rumin Laplacian} on $\mR^k$ is the operator $\Box_b\colon\mR^k\to\mR^k$ given by
 \begin{align*}
  \Delta_b & := d^\ast d + \frac{n-k}{n-k+1}dd^\ast, && \text{if $k\leq n-1$;} \\
  \Delta_b & := d^\ast d + dd^\ast dd^\ast, && \text{if $k=n$;} \\
  \Delta_b & := dd^\ast + d^\ast dd^\ast d, && \text{if $k=n+1$;} \\
  \Delta_b & := dd^\ast + \frac{k-n-1}{k-n}d^\ast d, && \text{if $k\geq n+2$.}
 \end{align*}
\end{definition}

Note that if $p+q=k\in\{n,n+1\}$, then the Kohn and Rumin Laplacians are fourth-order operators;
otherwise they are second-order operators.

Our definition of the Rumin Laplacian coincides with Rumin's original definition~\cite{Rumin1994}*{pg.\ 290} via \cref{explicit-rumin}.
However, our definition of the Kohn Laplacian differs from existing definitions in the literature (e.g.\ \cites{Kohn1965,FollandKohn1972,FollandStein1974,Tanaka1975,Garfield2001,GarfieldLee1998}).
Like other definitions, our Kohn Laplacian is a formally self-adjoint operator which has kernel equal to $\ker\dbbar\cap\ker\dbbar^\ast$ and which commutes with the conjugate Hodge star operator.
Our definition has the extra property that, modulo lower order terms, $\Delta_b=\Box_b+\oBox_b$, plus $\dhor^\ast\dhor$ or $\dhor\dhor^\ast$ if $p+q=n$ or $p+q=n+1$, respectively (cf.\ \cite{Rumin1994}*{Proposition~5(ii)}).
This leads to a uniform presentation of subelliptic estimates for the Kohn and Rumin Laplacians (see \cref{sec:hypoelliptic}) as well as to K\"ahler-like identities for these operators on Sasaki manifolds (see \cref{sec:frolicher,sec:sasakian}).

We begin to make these comments precise by recording the fact that the Kohn and Rumin Laplacians are nonnegative, formally self-adjoint operators and their kernels are the expected ones from the point of view of the bigraded Rumin  complex.

\begin{lemma}
 \label{kernel-kohn-laplacian}
 Let $(M^{2n+1},T^{1,0},\theta)$ be a closed pseudohermitian manifold.
 Given $p\in\{0,\dotsc,n+1\}$ and $q\in\{0,\dotsc,n\}$, the Kohn Laplacian $\Box_b\colon\mR^{p,q}\to\mR^{p,q}$ is a nonnegative, formally self-adjoint operator with
 \[ \ker \left( \Box_b \colon \mR^{p,q} \to \mR^{p,q}\right) = \left\{ \omega \in \mR^{p,q} \suchthatcolon \dbbar\omega = 0 = \dbbar^\ast \omega \right\} . \]
\end{lemma}

\begin{proof}
 This is immediate from the definition of the Kohn Laplacian.
\end{proof}

\begin{lemma}
 \label{kernel-rumin-laplacian}
 Let $(M^{2n+1},T^{1,0},\theta)$ be a closed pseudohermitian manifold.
 Given $k\in\{0,\dotsc,2n+1\}$, the Rumin Laplacian $\Delta_b\colon\mR^k\to\mR^k$ is a nonnegative, formally self-adjoint operator with
 \[ \ker \left( \Delta_b \colon \mR^k \to \mR^k \right) = \left\{ \omega \in \mR^k \suchthatcolon d\omega = 0 = d^\ast\omega \right \} . \]
\end{lemma}

\begin{proof}
 This is immediate from the definition of the Rumin Laplacian.
\end{proof}

Next we show that the Kohn Laplacian (resp.\ Rumin Laplacian) commutes with the conjugate Hodge star operator (resp.\ Hodge star operator).

\begin{lemma}
 \label{kohn-laplacian-hodge-star}
 Let $(M^{2n+1},T^{1,0},\theta)$ be a pseudohermitian manifold.
 Then
 \[ \Box_b\hodge\oomega = \hodge\overline{\Box_b\omega}  \]
 for all $\omega\in\mR^{p,q}$.
\end{lemma}

\begin{proof}
 \Cref{hodge-star-squared,defn:dbbar-adjoint} imply that $\Box_b\hodge = \hodge\oBox_b$ as operators on $\mR^{p,q}$.  The conclusion follows from \cref{conjugate-rumin-operators,hodge-mapping-properties}.
\end{proof}

\begin{lemma}
 \label{rumin-hodge-star}
 Let $(M^{2n+1},T^{1,0},\theta)$ be a pseudohermitian manifold.
 Then
 \[ \Delta_b\hodge\omega = \hodge\Delta_b\omega \]
 for all $\omega\in\mR^k$.
\end{lemma}

\begin{proof}
 \Cref{defn:dbbar-adjoint,formal-adjoint} imply that $d^\ast=(-1)^k\hodge d\hodge$ on $\mR^k$.
 The conclusion readily follows.
\end{proof}

Our next objective is to establish the aforementioned identity relating the Kohn and Rumin Laplacians.
This requires three lemmas which relate different compositions of the operators $\db$, $\dbbar$, $\dhor$, and their adjoints which preserve $\mR^k$ but not $\mR^{p,q}$.
Each lemma identifies two differential operators of order $s$ modulo a differential operator of order $s-2$ which vanishes on torsion-free manifolds.
This is enough for our purposes, though explicit formulas for the lower-order terms are readily derived.

\begin{lemma}
 \label{dbdbbarstar}
 Let $(M^{2n+1},T^{1,0},\theta)$ be a pseudohermitian manifold.
 Suppose that $k\in\{0,\dotsc,n-1\} \cup \{ n+2,\dotsc,2n+1\}$.
 Then
 \begin{subequations}
  \label{eqn:dbdbbarstar-sub}
  \begin{align}
   \label{eqn:dbdbbarstar-sub-noconj} \frac{n-k}{n-k+1}\db\dbbar^\ast + \dbbar^\ast\db & \equiv 0 \mod \DiffOperator{0} , \\
   \label{eqn:dbdbbarstar-sub-conj} \frac{n-k}{n-k+1}\dbbar\db^\ast + \db^\ast\dbbar & \equiv 0 \mod \DiffOperator{0} 
  \end{align}
 \end{subequations}
 on $\mR^k$, $k \leq n-1$; and
 \begin{subequations}
  \label{eqn:dbdbbarstar-sup}
  \begin{align}
   \label{eqn:dbdbbarstar-sup-conj} \frac{k-n-1}{k-n}\db^\ast\dbbar + \dbbar\db^\ast & \equiv 0 \mod \DiffOperator{0} , \\
   \label{eqn:dbdbbarstar-sup-noconj} \frac{k-n-1}{k-n}\dbbar^\ast\db + \db\dbbar^\ast & \equiv 0 \mod \DiffOperator{0} 
  \end{align}
 \end{subequations}
 on $\mR^k$, $k \geq n+2$.
 Moreover, if $\theta$ is torsion-free, then Equations~\eqref{eqn:dbdbbarstar-sub} and~\eqref{eqn:dbdbbarstar-sup} hold exactly.
\end{lemma}

\begin{proof}
 By linearity, it suffices to prove Equations~\eqref{eqn:dbdbbarstar-sub} and~\eqref{eqn:dbdbbarstar-sup} on $\mR^{p,q}$.
 Set $k := p+q$.
 
 Suppose that $k\leq n-1$.
 On the one hand, \cref{onablas-simple-identities,nablasast-complex,sR-to-P} imply that
 \begin{equation*}
  \frac{n-k}{n-k+1}\db\dbbar^\ast + \dbbar^\ast\db = \nablas\onablas^\ast + \onablas^\ast\nablas .
 \end{equation*}
 On the other hand, \cref{slash-laplacian-add-one-to-q} asserts that
 \begin{equation*}
  \nablas\onablas^\ast + \onablas^\ast\nablas \equiv 0 \mod \DiffOperator{0} ,
 \end{equation*}
 with exact equality if $\theta$ is torsion-free.
 Equation~\eqref{eqn:dbdbbarstar-sub-noconj} readily follows.
 Equation~\eqref{eqn:dbdbbarstar-sub-conj} follows by conjugation.
 
 Suppose now that $k\geq n+2$.
 Applying Equation~\eqref{eqn:dbdbbarstar-sub} to $\hodge\oomega$ yields Equation~\eqref{eqn:dbdbbarstar-sup}.
\end{proof}

\begin{lemma}
 \label{dhordbbarstar}
 Let $(M^{2n+1}, T^{1,0},\theta)$ be a pseudohermitian manifold.
 Then
 \begin{subequations}
  \label{eqn:dhordbbarstar-sub}
  \begin{align}
   \label{eqn:dhordbbarstar-sub-noconj} \db^\ast\dhor + \dhor^\ast\dbbar & \equiv -\dbbar\db^\ast\db\db^\ast - \dbbar\dbbar^\ast\dbbar\db^\ast \mod \DiffOperator{2} , \\
   \label{eqn:dhordbbarstar-sub-conj} \dbbar^\ast\dhor + \dhor^\ast\db & \equiv -\db\dbbar^\ast\dbbar\dbbar^\ast - \db\db^\ast\db\dbbar^\ast \mod \DiffOperator{2}
  \end{align}
 \end{subequations}
 on $\mR^n$;
 and
 \begin{subequations}
  \label{eqn:dhordbbarstar-sup}
  \begin{align}
   \label{eqn:dhordbbarstar-sup-noconj} \db\dhor^\ast + \dhor\dbbar^\ast \equiv -\dbbar^\ast\db\db^\ast\db - \dbbar^\ast\dbbar\dbbar^\ast\db \mod \DiffOperator{2} \\
   \label{eqn:dhordbbarstar-sup-conj} \dbbar\dhor^\ast + \dhor\db^\ast \equiv -\db^\ast\dbbar\dbbar^\ast\dbbar - \db^\ast\db\db^\ast\dbbar \mod \DiffOperator{2}
  \end{align}
 \end{subequations}
 on $\mR^{n+1}$.
 Moreover, if $\theta$ is torsion-free, then Equations~\eqref{eqn:dhordbbarstar-sub} and~\eqref{eqn:dhordbbarstar-sup} hold exactly.
\end{lemma}

\begin{proof}
 By linearity, it suffices to verify Equations~\eqref{eqn:dhordbbarstar-sub} and~\eqref{eqn:dhordbbarstar-sup} on $\mR^{p,q}$ for $p+q\in\{n,n+1\}$.
 
 Suppose that $p+q=n$.
 We deduce from \cref{sR-to-P,nablasast-complex} that
 \begin{align}
  \label{eqn:n-hodge-dbbar} \hodge\dbbar & \equiv (-1)^pi^{n^2+1}\dbbar\db^\ast \mod \DiffOperator{0} , \\
  \label{eqn:n-hodge-db} \hodge\db & \equiv (-1)^pi^{n^2-1}\db\dbbar^\ast \mod \DiffOperator{0} , \\
  \label{eqn:n-hodge-dhor} \hodge\dhor & = (-1)^pi^{n^2}\left( \nabla_0 - i\db\db^\ast + i\dbbar\dbbar^\ast\right) ;
 \end{align}
 moreover, if $\theta$ is torsion-free, then Equations~\eqref{eqn:n-hodge-dbbar} and~\eqref{eqn:n-hodge-db} hold exactly.
 On the one hand, combining \cref{justification-of-bigraded-complex,sR-to-P,dbdbbarstar} with Equations~\eqref{eqn:n-hodge-dbbar}, \eqref{eqn:n-hodge-db}, and~\eqref{eqn:n-hodge-dhor} yields
 \begin{align}
  \label{eqn:I-dbbarast-dhor} \db^\ast\dhor & \equiv -i\onablas\nablas^\ast\nabla_0 - \dbbar\db^\ast\db\db^\ast \mod \DiffOperator{2} , \\
  \label{eqn:I-dhorast-dbbar} \dhor^\ast\dbbar & \equiv i\nabla_0\onablas\nablas^\ast - \dbbar\dbbar^\ast\dbbar\db^\ast \mod \DiffOperator{2} ;
 \end{align}
 moreover, if $\theta$ is torsion-free, then Equations~\eqref{eqn:I-dbbarast-dhor} and~\eqref{eqn:I-dhorast-dbbar} hold exactly.
 On the other hand, \cref{commutators,nablab-and-nablas-adjoints} imply that
 \begin{equation}
  \label{eqn:nablas-nabla0-commutator}
  \onablas\nablas^\ast\nabla_0 - \nabla_0\onablas\nablas^\ast \equiv 0 \mod \DiffOperator{2};
 \end{equation}
 moreover, if $\theta$ is torsion-free, then Equation~\eqref{eqn:nablas-nabla0-commutator} holds exactly.
 Combining Equations~\eqref{eqn:I-dbbarast-dhor}, \eqref{eqn:I-dhorast-dbbar}, and~\eqref{eqn:nablas-nabla0-commutator} yields Equation~\eqref{eqn:dhordbbarstar-sub-noconj}.
 Equation~\eqref{eqn:dhordbbarstar-sub-conj} follows by conjugation.
 
 Suppose now that $p+q=n+1$.
 Applying Equation~\eqref{eqn:dhordbbarstar-sub} to $\hodge\oomega$ yields Equation~\eqref{eqn:dhordbbarstar-sup}. 
\end{proof}

\begin{lemma}
 \label{critical-dbdbbarstar}
 Let $(M^{2n+1}, T^{1,0},\theta)$ be a pseudohermitian manifold.
 Then
 \begin{subequations}
  \label{eqn:critical-dbdbbarstar-sub}
  \begin{align}
   \label{eqn:critical-dbdbbarstar-sub-noconj} \dbbar^\ast\db & \equiv 0 \mod \DiffOperator{2} , \\
   \label{eqn:critical-dbdbbarstar-sub-conj} \db^\ast\dbbar & \equiv 0 \mod \DiffOperator{2}
  \end{align}
 \end{subequations}
 on $\mR^n$;
 and
 \begin{subequations}
  \label{eqn:critical-dbdbbarstar-sup}
  \begin{align}
   \label{eqn:critical-dbdbbarstar-sup-conj} \dbbar\db^\ast & \equiv 0 \mod \DiffOperator{2} , \\
   \label{eqn:critical-dbdbbarstar-sup-noconj} \db\dbbar^\ast & \equiv 0 \mod \DiffOperator{2}
  \end{align}
 \end{subequations}
 on $\mR^{n+1}$.
 Moreover, if $\theta$ is torsion-free, then Equations~\eqref{eqn:critical-dbdbbarstar-sub} and~\eqref{eqn:critical-dbdbbarstar-sup} hold exactly.
\end{lemma}

\begin{proof}
 By linearity, it suffices to verify Equations~\eqref{eqn:critical-dbdbbarstar-sub} and~\eqref{eqn:critical-dbdbbarstar-sup} on $\mR^{p,q}$ for $p+q\in\{n,n+1\}$.
 
 Suppose that $p+q=n$.
 By definition, $\dbbar^\ast\db = (-1)^{n+1}\hodge\db\hodge\db$ on $\mR^{p,q}$.
 Combining this with Equation~\eqref{eqn:n-hodge-db} yields Equation~\eqref{eqn:critical-dbdbbarstar-sub-noconj}.
 Equation~\eqref{eqn:critical-dbdbbarstar-sub-conj} follows by conjugation.
 
 Suppose now that $p+q=n+1$.
 Applying Equation~\eqref{eqn:critical-dbdbbarstar-sub} to $\hodge\oomega$ yields Equation~\eqref{eqn:critical-dbdbbarstar-sup}.
\end{proof}

We now prove the claimed relationship between the Kohn and Rumin Laplacians.
To state this clearly, define $L_b\colon\mR^{p,q}\to\mR^{p,q}$ by
\begin{equation}
 \label{eqn:Lb}
 L_b :=
 \begin{cases}
  \Box_b + \oBox_b , & \text{if $p+q\not\in\{n,n+1\}$}, \\
  \Box_b + \oBox_b + \dhor^\ast\dhor , & \text{if $p+q=n$}, \\
  \Box_b + \oBox_b + \dhor\dhor^\ast, & \text{if $p+q=n+1$} .
 \end{cases}
\end{equation}
We show that $\Delta_b \equiv L_b$ modulo lower-order terms, recovering an observation of Rumin~\cite{Rumin1994}*{Proposition~5}.

\begin{proposition}
 \label{kohn-laplacian-to-rumin-laplacian}
 Let $(M^{2n+1},T^{1,0},\theta)$ be a pseudohermitian manifold.
 Then
 \begin{subequations}
  \label{eqn:kohn-laplacian-to-rumin-laplacian}
  \begin{align}
   \label{eqn:kohn-laplacian-to-rumin-laplacian-non-middle} \Delta_b & \equiv L_b \mod \DiffOperator{0}, && \text{on $\mR^k$, $k\not\in\{n,n+1\}$}, \\
   \label{eqn:kohn-laplacian-to-rumin-laplacian-n} \Delta_b & \equiv L_b \mod \DiffOperator{2} , && \text{on $\mR^k$, $k\in\{n,n+1\}$} .
  \end{align}
 \end{subequations}
 Moreover, if $\theta$ is torsion-free, then Equation~\eqref{eqn:kohn-laplacian-to-rumin-laplacian} holds exactly.
\end{proposition}

\begin{proof}
 By \cref{kohn-laplacian-hodge-star,rumin-hodge-star}, it suffices to consider $\Delta_b$ and $L_b$ on $\mR^k$, $k\leq n$.
 
 Suppose first that $k\leq n-1$.
 By \cref{defn:kohn-laplacian,defn:rumin-laplacian}, if $\omega\in\mR^k$, then
 \[ \Delta_b\omega = \Box_b\omega + \oBox_b\omega + \dbbar^\ast\db\omega + \frac{n-k}{n-k+1}\db\dbbar^\ast \omega + \db^\ast\dbbar\omega + \frac{n-k}{n-k+1}\dbbar\db^\ast \omega . \]
 Equation~\eqref{eqn:kohn-laplacian-to-rumin-laplacian-non-middle} readily follows from \cref{dbdbbarstar}.
 
 Suppose next that $k=n$.
 On the one hand, \cref{critical-dbdbbarstar} implies that
 \begin{equation*}
  d^\ast d \equiv \dbbar^\ast\dbbar + \db^\ast\db + \dhor^\ast\dhor + \db^\ast\dhor + \dhor^\ast\dbbar + \dhor^\ast\db + \dbbar^\ast\dhor \mod \DiffOperator{2} ;
 \end{equation*}
 moreover, if $\theta$ is torsion-free, then equality holds exactly.
 On the other hand, expanding $(dd^\ast)^2$ and applying \cref{dbdbbarstar} yields
 \begin{align*}
  (dd^\ast)^2 & \equiv \dbbar\dbbar^\ast\dbbar\dbbar^\ast +  \db\db^\ast\db\db^\ast + \db\db^\ast\dbbar\dbbar^\ast + \dbbar\dbbar^\ast\db\db^\ast + \dbbar\db^\ast\db\dbbar^\ast + \db\dbbar^\ast\dbbar\db^\ast \\
   & \quad + \dbbar\dbbar^\ast\dbbar\db^\ast + \dbbar\db^\ast\db\db^\ast + \db\db^\ast\db\dbbar^\ast + \db\dbbar^\ast\dbbar\dbbar^\ast \mod \DiffOperator{2} ;
 \end{align*}
 moreover, if $\theta$ is torsion-free, then equality holds exactly.
 Combining the previous two displays with \cref{dhordbbarstar} yields
 \begin{equation}
  \label{eqn:n-rumin-partial-simplification}
  \begin{split}
  \Delta_b & \equiv \dbbar^\ast\dbbar + \dbbar\db^\ast\db\dbbar^\ast + \db\db^\ast\dbbar\dbbar^\ast + \dbbar\dbbar^\ast\dbbar\dbbar^\ast  \\
   & \quad + \db^\ast\db + \db\dbbar^\ast\dbbar\db^\ast + \dbbar\dbbar^\ast\db\db^\ast + \db\db^\ast\db\db^\ast \\
   & \quad + \dhor^\ast\dhor \mod \DiffOperator{2} ;
   \end{split}
 \end{equation}
 moreover, if $\theta$ is torsion-free, then equality holds exactly.
 Finally, \cref{justification-of-bigraded-complex,dbdbbarstar} imply that
 \begin{equation}
  \label{eqn:n-rumin-last-observation}
  \db\db^\ast\dbbar\dbbar^\ast \equiv \frac{1}{2}\dbbar\db\db^\ast\dbbar^\ast \mod \DiffOperator{2} ;
 \end{equation}
 moreover, if $\theta$ is torsion-free, then equality holds exactly.
 Equation~\eqref{eqn:kohn-laplacian-to-rumin-laplacian-n} follows immediately from Equations~\eqref{eqn:n-rumin-partial-simplification} and~\eqref{eqn:n-rumin-last-observation}.
\end{proof}

The basic ingredient in proving that the Kohn and Rumin Laplacians are maximally hypoelliptic is the following Weitzenb\"ock-type formula for the Kohn Laplacian.

\begin{proposition}
 \label{bochner-order2}
 Let $(M^{2n+1},T^{1,0},\theta)$ be a pseudohermitian manifold.
 Then
 \begin{multline}
  \label{eqn:bochner-order2}
  \Box_b\omega = \frac{q}{n}\nablab^\ast\nablab\omega + \frac{n-q}{n}\nablabbar^\ast\nablabbar\omega - \frac{1}{n-p-q+1}(\dbbar\dbbar^\ast\omega + \db\db^\ast\omega) \\
   - R\hash\ohash\omega - \frac{q}{n}\Ric\hash\omega - \frac{n-q}{n}\Ric\ohash\omega
 \end{multline}
 for all $\omega\in\mR^{p,q}$, $p+q\leq n-1$.
\end{proposition}

\begin{proof}
 Let $\omega\in\mR^{p,q}M$, $p+q\leq n-1$.
 \Cref{sR-to-P,onablas-simple-identities,nablasast-complex} imply that
 \begin{equation}
  \label{eqn:Boxb-to-Boxs}
  \dbbar^\ast\dbbar + \dbbar\dbbar^\ast = \Boxs\omega - \frac{1}{n-p-q+1}\db\db^\ast\omega .
 \end{equation}
 Combining \cref{nablas-weitzenbock,defn:kohn-laplacian} with Equation~\eqref{eqn:Boxb-to-Boxs} yields Equation~\eqref{eqn:bochner-order2}.
\end{proof}

\Cref{kohn-laplacian-to-rumin-laplacian,bochner-order2} easily yield a Weitzenb\"ock formula for the Rumin Laplacian (cf.\ \cite{Rumin1994}*{Lemma~10}).
The following expression for $L_b$ is sufficient for our purposes.

\begin{corollary}
 \label{rumin-order2}
 Let $(M^{2n+1},T^{1,0},\theta)$ be a pseudohermitian manifold.
 Then
 \begin{align*}
  \frac{n-p-q+2}{n-p-q}L_b\omega & = \frac{n-p+q}{n}\nablab^\ast\nablab\omega + \frac{n-q+p}{n}\nablabbar^\ast\nablabbar\omega \\
   & \quad + \frac{2}{n-p-q}\left( \db^\ast\db + \dbbar^\ast\dbbar \right)\omega - 2 R\hash\ohash \omega \\
   & \quad - \frac{n-p+q}{n}\Ric\hash\omega - \frac{n-q+p}{n}\Ric\ohash\omega 
 \end{align*}
 for all $\omega\in\mR^{p,q}$, $p+q\leq n-1$.
\end{corollary}

\begin{proof}
 On the one hand, adding \cref{bochner-order2} to its conjugate yields
 \begin{align*}
  L_b\omega & = \frac{n-p+q}{n}\nablab^\ast\nablab\omega + \frac{n-q+p}{n}\nablabbar^\ast\nablabbar\omega - \frac{2}{n-p-q+1}\left( \dbbar\dbbar^\ast + \db\db^\ast \right)\omega \\
   & \quad - 2R\hash\ohash\omega - \frac{n-p+q}{n}\Ric\hash\omega - \frac{n-q+p}{n}\Ric\ohash\omega .
 \end{align*}
 On the other hand,
 \begin{equation}
  \label{eqn:simplify-db-dbast-term}
  \frac{2}{n-p-q+1}\left(\db\db^\ast + \dbbar\dbbar^\ast\right)\omega = \frac{2}{n-p-q}\left(L_b - \db^\ast\db - \dbbar^\ast\dbbar\right)\omega .
 \end{equation}
 The conclusion readily follows.
\end{proof}

\cref{bochner-order2} is sufficient to prove that the Kohn Laplacian is maximally hypoelliptic on $\mR^{0,q}$, $q\in\{1,\dotsc,n-1\}$, and to deduce a sufficient condition for $H^{0,q}(M)$ to vanish.
The corresponding statements for $\mR^{p,q}$ and $H^{p,q}(M)$ require a more complicated formula to control the term $\db\db^\ast\omega$.

\begin{corollary}
 \label{bochner-order2-better-sign}
 Let $(M^{2n+1},T^{1,0},\theta)$ be a pseudohermitian manifold.  Then
 \begin{equation}
  \label{eqn:bochner-order2-better-sign}
  \begin{split}
   \Box_b\omega & = \frac{(q-1)(n-p-q)}{n(n-p-q+2)}\nablab^\ast\nablab\omega  + \frac{(n-q+1)(n-p-q)}{n(n-p-q+2)}\nablabbar^\ast\nablabbar\omega \\
    & \quad + \frac{1}{n-p-q+2}\left( \db^\ast\db + \dbbar^\ast\dbbar \right)\omega \\
    & \quad - \frac{n-p-q}{n-p-q+2}R\hash\ohash\omega - \frac{(q-1)(n-p-q)}{n(n-p-q+2)} \Ric\hash\omega \\
    & \quad - \frac{(n-q+1)(n-p-q)}{n(n-p-q+2)} \Ric\ohash\omega 
  \end{split}
 \end{equation}
 for all $\omega\in\mR^{p,q}$, $p+q\leq n-1$.
\end{corollary}

\begin{proof}
 \Cref{bochner-order2} and Equation~\eqref{eqn:simplify-db-dbast-term} imply that
 \begin{equation*}
  \begin{split}
   \Box_b\omega & = \frac{q}{n}\nablab^\ast\nablab\omega + \frac{n-q}{n}\nablabbar^\ast\nablabbar\omega - \frac{1}{n-p-q}L_b\omega \\
    & \quad + \frac{1}{n-p-q}\left( \dbbar^\ast\dbbar + \db^\ast\db \right)\omega - R \hash\hash \omega - \frac{q}{n}\Ric\ohash\omega - \frac{n-q}{n}\Ric\ohash\omega .
  \end{split}
 \end{equation*}
 Combining this with \cref{rumin-order2} yields Equation~\eqref{eqn:bochner-order2-better-sign}.
\end{proof}

When $p+q\in\{n,n+1\}$, it is a much more tedious task to derive suitable Weitzenb\"ock-type formulas for the Kohn and Rumin Laplacians.
For the purposes of establishing maximal hypoellipticity, we only need to understand the leading-order terms of these operators.
This is mostly done using \cref{sR-to-P,onablas-simple-identities}.
The remaining ingredient, needed for a more precise description of the partial inverse of these operators in the cases $q\in\{0,n\}$, requires the following fact about the commutators $[L_b,\dbbar\dbbar^\ast]$ and $[L_b,\dbbar^\ast\dbbar]$.

\begin{lemma}
 \label{L-dbbardbbarast-commutator}
 Let $(M^{2n+1},T^{1,0},\theta)$ be a pseudohermitian manifold.
 Suppose that $p\in\{0,\dotsc,n+1\}$ and $q\in\{0,\dotsc,n\}$ are such that $p+q\not\in\{n,n+1\}$.
 Then, as operators on $\mR^{p,q}$,
 \[ [ L_b, \dbbar\dbbar^\ast ], [ L_b, \dbbar^\ast\dbbar ] \in \DiffOperator{2} . \]
 Moreover, if $\theta$ is torsion-free, then $[ L_b, \dbbar\dbbar^\ast] , [L_b, \dbbar^\ast\dbbar] = 0$.
\end{lemma}

\begin{proof}
 By \cref{kohn-laplacian-hodge-star}, it suffices to consider the case $p+q \leq n-1$.
 
 First, observe that \cref{justification-of-bigraded-complex,dual-justification,dbdbbarstar} imply that
 \begin{equation}
  \label{eqn:L-dbdbarddbarast-commutator-easy-case}
  [ \db^\ast\db , \dbbar\dbbar^\ast ] , [ \db\db^\ast , \dbbar\dbbar^\ast ] , [ \db\db^\ast , \dbbar^\ast\dbbar ] \equiv 0 \mod \DiffOperator{2} ;
 \end{equation}
 moreover, if $\theta$ is torsion-free, then equality holds exactly.
 
 Second, observe that if $p+q \leq n-2$, then \cref{justification-of-bigraded-complex,dual-justification,dbdbbarstar} imply that
 \begin{equation}
  \label{eqn:L-dbdbarddbarast-commutator-hard-case}
  [ \db^\ast\db , \dbbar^\ast\dbbar ] \equiv 0 \mod \DiffOperator{2} ;
 \end{equation}
 moreover, if $\theta$ is torsion-free, then equality holds exactly.
 If instead $p+q = n-1$, then \cref{sR-to-P} implies that
 \begin{align*}
  \db^\ast\db\dbbar^\ast\dbbar & \equiv (-1)^pi^{-n^2+1} \db^\ast \hodge \db \dbbar \mod \DiffOperator{2} , \\
  \dbbar^\ast\dbbar\db^\ast\db & \equiv (-1)^pi^{-n^2+1} \dbbar^\ast \hodge \dbbar\db \mod \DiffOperator{2} ;
 \end{align*}
 moreover, if $\theta$ is torsion-free, then equality holds exactly.
 Combining this with the definitions of $\db^\ast$ and $\dbbar^\ast$ yields
 \begin{align*}
  \db^\ast\db\dbbar^\ast\dbbar & \equiv (-1)^pi^{n^2+1} \hodge \dbbar \db \dbbar \mod \DiffOperator{2} , \\
  \dbbar^\ast\dbbar\db^\ast\db & \equiv (-1)^pi^{n^2+1} \hodge \db \dbbar\db \mod \DiffOperator{2} ;
 \end{align*}
 moreover, if $\theta$ is torsion-free, then equality holds exactly.
 \Cref{justification-of-bigraded-complex} implies that $\dbbar\db\dbbar = \db\dbbar\db$ on $\mR^{n-1}$, and hence Equation~\eqref{eqn:L-dbdbarddbarast-commutator-hard-case} again holds.
 
 Combining Equations~\eqref{eqn:L-dbdbarddbarast-commutator-easy-case} and~\eqref{eqn:L-dbdbarddbarast-commutator-hard-case} yields the final conclusion.
\end{proof}

%% file: hodge/subelliptic.tex
\section{Maximal hypoellipticity of the Kohn and Rumin Laplacians}
\label{sec:hypoelliptic}

In this section we show that the Kohn Laplacian is maximally hypoelliptic.
More precisely, we show that on a closed $(2n+1)$-dimensional strictly pseudoconvex manifold, if $q\not\in\{0,n\}$ then
\begin{enumerate}
 \item the second-order Kohn Laplacian is bounded below by a Folland--Stein operator $\mL_\alpha$, $\alpha\in(-n,n)$, modulo a zeroth-order differential operator; and
 \item the fourth-order Kohn Laplacian is bounded below by a product of two such Folland--Stein operators modulo a second-order differential operator.
\end{enumerate} 
In particular, the Kohn Laplacian admits a parametrix if $q\not\in\{0,n\}$.
This is consistent with known results~\cites{Garfield2001,GarfieldLee1998,Tanaka1975,FollandKohn1972,FollandStein1974,BealsGreiner1988} in other contexts.

We also show that $L_b\colon\mR^{p,q}\to\mR^{p,q}$ is maximally hypoelliptic.
More precisely, we show that on a closed $(2n+1)$-dimensional strictly pseudoconvex manifold,
\begin{enumerate}
 \item if $p+q\not\in\{n,n+1\}$, then $L_b$ is bounded below by a Folland--Stein operator $\mL_\alpha$, $\alpha\in(-n,n)$, modulo a zeroth-order differential operator;
 \item if $p+q\in\{n,n+1\}$ and $q\not\in\{0,n\}$, then $L_b$ is bounded below by a product of two such Folland--Stein operators modulo a second-order differential operator; and
 \item if $p+q\in\{n,n+1\}$ and $q\in\{0,n\}$, then $L_b$ is the square of $\mL_{\pm(n+1)}$ modulo a second-order differential operator.
\end{enumerate}
In particular, the Rumin Laplacian admits a parametrix.
This fact was previously observed by Rumin~\cite{Rumin1994}.

We begin by showing that the second-order Kohn Laplacians admit a parametrix except for in the endpoint cases $q\in\{0,n\}$.

\begin{proposition}
 \label{hypoelliptic2}
 Let $(M^{2n+1},T^{1,0},\theta)$ be a closed, strictly pseudoconvex manifold.
 Let $p\in\{0,\dotsc,n+1\}$ and $q\in\{1,\dotsc,n-1\}$ be such that $p+q\not\in\{n,n+1\}$.
 Then $\Box_b\colon\mR^{p,q}\to\mR^{p,q}$ admits a parametrix.
\end{proposition}

\begin{proof}
 By \cref{embeddings,kohn-laplacian-hodge-star}, it suffices to consider the case $p+q\leq n-1$.
 
 First observe that, by \cref{bochner-order2-better-sign} and Equation~\eqref{eqn:folland-stein-no-0},
 \begin{equation}
  \label{eqn:Boxb-general-estimate}
  \Box_b \gtrsim  \frac{n-p-q}{2(n-p-q+2)}\mL_{n-2q+2} .
 \end{equation}
 We conclude from \cref{invertibility-characterizations,comparisons,invert-folland-stein} that if $q>1$, then $\Box_b$ admits a parametrix.
 
 Second observe that, by \cref{bochner-order2} and Equation~\eqref{eqn:folland-stein-no-0}, if $q=1$, then
 \begin{equation*}
  \frac{n-p}{n-p-1}\Box_b \gtrsim \frac{1}{2}\mL_{n-2} - \frac{1}{n-p}\db\db^\ast .
 \end{equation*}
 Since $\db^\ast\omega$ is a partial contraction of $\nablabbar\omega$, there is a constant $K>0$ such that $K\nablabbar^\ast\nablabbar \gtrsim \frac{2}{n-p}\db\db^\ast$.
 In particular,
 \begin{equation}
  \label{eqn:Boxb-q=1-estimate}
  \Box_b \gtrsim \frac{n-p-1}{2(n-p)}\left( \mL_{n-2} - K\mL_{n}\right) .
 \end{equation}
 By taking a suitable convex combination of Inequalities~\eqref{eqn:Boxb-general-estimate} and~\eqref{eqn:Boxb-q=1-estimate}, we see that there is a constant $C>0$ such that
 \begin{equation*}
  \Box_b \gtrsim C\mL_{n-2} .
 \end{equation*}
 We again conclude from \cref{invertibility-characterizations,comparisons,invert-folland-stein} that $\Box_b$ admits a parametrix.
\end{proof}

The estimates in the proof of \cref{hypoelliptic2} imply that, away from the middle degrees, the operator $L_b$ of Equation~\eqref{eqn:Lb} admits a parametrix.

\begin{proposition}
 \label{hypoelliptic-lb2}
 Let $(M^{2n+1},T^{1,0},\theta)$ be a closed, strictly pseudoconvex manifold.
 Let $p\in\{0,\dotsc,n+1\}$ and $q\in\{1,\dotsc,n-1\}$ be such that $p+q\not\in\{n,n+1\}$.
 Then $L_b\colon\mR^{p,q}\to\mR^{p,q}$ admits a parametrix.
\end{proposition}

\begin{proof}
 By \cref{embeddings,kohn-laplacian-hodge-star}, it suffices to consider the case $p+q\leq n-1$.
 
 Inequality~\eqref{eqn:Boxb-general-estimate} and its conjugate imply that
 \begin{equation*}
  L_b \gtrsim \frac{n-p-q}{n-p-q+2}\mL_{p-q} .
 \end{equation*}
 Since $\lv p-q\rv<n$, we conclude from \cref{invertibility-characterizations,comparisons,invert-folland-stein} that $L_b\colon\mR^{p,q}\to\mR^{p,q}$ admits a parametrix.
\end{proof}

We now turn to the maximal hypoellipticity of the fourth-order Kohn Laplacian.
Our approach is inspired by Rumin's proof~\cite{Rumin1994} of the maximal hypoellipticity of the fourth-order Rumin Laplacian.
However, we slightly simplify his argument by means of the following consequence of the Lefschetz decomposition of a symplectic vector space.

\begin{lemma}
 \label{lefschetz-consequence}
 Let $(M^{2n+1},T^{1,0},\theta)$ be a pseudohermitian manifold.
 Suppose that $p,q\in\{0,\dotsc,n\}$ are such that $p+q=n$.
 Then
 \begin{equation}
  \label{eqn:lefschetz-consequence}
  \onablas^\ast\onablas\omega = \nablas\nablas^\ast\omega - i\,d\theta \wedge \onablas^\ast\nablas^\ast \omega
 \end{equation}
 for all $\omega\in P^{p,q}M$.
\end{lemma}

\begin{proof}
 Equation~\eqref{eqn:tracefree-trick} implies that if $\omega\in P^{p,q}M$, $p+q=n$, then
 \[ \onablas\omega + i\,d\theta \wedge \nablas^\ast\omega \in P^{p,q+1}M . \]
 Since $p+q+1>n$, it holds~\cite{Huybrechts2005}*{Proposition~1.2.30(ii)} that $P^{p,q+1}M=0$.
 Therefore
 \begin{equation}
  \label{eqn:critical-lefschetz-consequence}
  \onablas\omega = -i\,d\theta \wedge \nablas^\ast\omega .
 \end{equation}
 Combining this with \cref{onablas-simple-identities} yields Equation~\eqref{eqn:lefschetz-consequence}.
\end{proof}

That the fourth-order Kohn Laplacian on $\mR^{p,q}$, $q \not\in \{ 0, n \}$, admits a parametrix follows by comparing it to an operator of the form $P_1^\ast D_1 P_1 + P_2^\ast D_2 P_2$ with each of $P_1^\ast P_1 + P_2^\ast P_2$ and $D_1$ and $D_2$ bounded below by a Folland--Stein operator (cf.\ \cite{Garfield2001}*{Theorem~3.22}).

\begin{proposition}
 \label{hypoelliptic4}
 Let $(M^{2n+1},T^{1,0},\theta)$ be a closed, strictly pseudoconvex manifold.
 Let $p\in\{0,\dotsc,n\}$ and $q\in\{1,\dotsc,n-1\}$ be such that $p+q\in\{n,n+1\}$.
 Then $\Box_b\colon\mR^{p,q}\to\mR^{p,q}$ admits a parametrix.
\end{proposition}

\begin{proof}
 By \cref{embeddings,kohn-laplacian-hodge-star}, it suffices to consider the case $p+q=n$.
 
 On the one hand, \cref{sR-to-P} implies that
 \begin{equation*}
  \dbbar^\ast\dbbar \simeq \db\dbbar^\ast\dbbar\db^\ast .
 \end{equation*}
 Inserting this into \cref{defn:kohn-laplacian} yields
 \begin{equation}
  \label{eqn:middle-kohn-to-Boxs}
  \Box_b \simeq \dbbar\dbbar^\ast\dbbar\dbbar^\ast + \db\dbbar^\ast\dbbar\db^\ast + \dbbar\db^\ast\db\dbbar^\ast + \frac{1}{2}\dbbar\db\db^\ast\dbbar^\ast .
 \end{equation}
 In particular,
 \begin{equation}
  \label{eqn:middle-kohn-to-Boxs-estimate}
  \Box_b \gtrsim \frac{1}{2}\db\Box_b\db^\ast + \frac{1}{2}\dbbar\oBox_b\dbbar^\ast .
 \end{equation}
 On the other hand, \cref{lefschetz-consequence,sR-to-P,nablasast-complex} imply that
 \begin{equation}
  \label{eqn:middle-Boxs-formula}
  \Boxs \simeq \db\db^\ast + \dbbar\dbbar^\ast .
 \end{equation}
 
 Now, \cref{hypoelliptic-lb2} implies that $\Box_b$ and $\oBox_b$ admit parametrices on $\mR^{p-1,q}$ and $\mR^{p,q-1}$, respectively.
 Additionally, \cref{nablas-weitzenbock,invert-folland-stein,comparisons,invertibility-characterizations} imply that $\Boxs$ admits a parametrix on $\mR^{p,q}$.
 Combining Equations~\eqref{eqn:middle-kohn-to-Boxs-estimate} and~\eqref{eqn:middle-Boxs-formula} with \cref{interior-composition-comparison} yields the final conclusion.
\end{proof}

\Cref{hypoelliptic4} implies that the operator $L_b\colon\mR^{p,q}\to\mR^{p,q}$, $p+q\in\{n,n+1\}$, admits a parametrix if $q\not\in\{0,n\}$.
A more delicate argument shows that this is also true in the remaining cases $q \in \{ 0 , n \}$.

\begin{proposition}
 \label{hypoelliptic-lb4}
 Let $(M^{2n+1},T^{1,0},\theta)$ be a closed, strictly pseudoconvex manifold.
 Let $p\in\{0,\dotsc,n+1\}$ and $q\in\{0,\dotsc,n\}$ be such that $p+q\in\{n,n+1\}$.
 Then $L_b\colon\mR^{p,q}\to\mR^{p,q}$ admits a parametrix.
\end{proposition}

\begin{proof}
 It suffices to consider the case $p+q=n$.
 
 Suppose first that $q\not\in\{0,n\}$.
 Then $p\not\in\{0,n\}$.
 Since
 \[ L_b = \Box_b+\oBox_b+\dhor^\ast\dhor \gtrsim \Box_b \]
 on $\mR^{p,q}$, the conclusion follows from \cref{hypoelliptic4,comparisons}.
 
 Suppose next that $q\in\{0,n\}$.
 By conjugating if necessary, we may assume that $q=0$.
 On the one hand, \cref{sR-to-P} implies that
 \begin{equation}
  \label{eqn:n0-dhor}
  \dhor^\ast\dhor = (\nablas\nablas^\ast + i\nabla_0)^2 .
 \end{equation}
 On the other hand, combining \cref{sR-to-P} with Equation~\eqref{eqn:middle-kohn-to-Boxs} and its conjugate imply that
 \begin{align}
  \label{eqn:n0-box} \Box_b & \simeq \nablas\onablas^\ast\onablas\nablas^\ast, \\
  \label{eqn:n0-obox} \oBox_b & \simeq \nablas\onablas^\ast\onablas\nablas^\ast + \nablas\nablas^\ast\nablas\nablas^\ast .
 \end{align}
 Combining \cref{commutators} with Equations~\eqref{eqn:n0-dhor}, \eqref{eqn:n0-box}, and~\eqref{eqn:n0-obox} yields
 \begin{equation}
  \label{eqn:n0-first-expression}
  L_b \simeq 2\nablas\nablas^\ast\nablas\nablas^\ast + 2\nablas\onablas^\ast\onablas\nablas^\ast + 2i\nabla_0\nablas\nablas^\ast - \nabla_0\nabla_0 .
 \end{equation}
 It follows readily from \cref{slash-laplacian-add-one-to-q,triple-commute} that
 \begin{equation}
  \label{eqn:partial-commute-middle}
  \nablas\onablas^\ast\onablas\nablas^\ast \simeq \nablas\nablas^\ast\onablas^\ast\onablas + i\nablas\nabla_0\nablas^\ast .
 \end{equation}
 Using \cref{commutators,lefschetz-consequence} to combine Equations~\eqref{eqn:n0-first-expression} and~\eqref{eqn:partial-commute-middle} yields
 \begin{equation}
  \label{eqn:n0-second-expression}
  L_b \simeq \left( 2\nablas\nablas^\ast + i\nabla_0 \right)^2 .
 \end{equation}
 Combining \cref{nablas-weitzenbock} with Equations~\eqref{eqn:folland-stein-no-0}, \eqref{eqn:middle-Boxs-formula}, and~\eqref{eqn:n0-second-expression} yields
 \begin{equation*}
  L_b \simeq \left( 2\Boxs\omega + i\nabla_0 \right)^2 \simeq \mL_{n+1}^2 .
 \end{equation*}
 The conclusion follows from \cref{invertibility-characterizations,comparisons,invert-folland-stein}. 
\end{proof}

In particular, the Rumin Laplacian admits a parametrix.

\begin{corollary}
 \label{hypoelliptic-rumin}
 Let $(M^{2n+1},T^{1,0},\theta)$ be a closed, strictly pseudoconvex manifold.
 Let $k\in\{0,\dotsc,2n+1\}$.
 Then $\Delta_b\colon\mR^k\to\mR^k$ admits a parametrix.
\end{corollary}

\begin{proof}
 This follows from \cref{kohn-laplacian-to-rumin-laplacian,hypoelliptic-lb2,hypoelliptic-lb4}.
\end{proof}

%% file: hodge/hodge_thm.tex
\section{The Hodge theorems and cohomological dualities}
\label{sec:hodge_thm}

In this section we prove Hodge decomposition theorems for $\mR^{p,q}$ and $\mR^k$ using the Kohn and Rumin Laplacians, respectively.
As applications, we prove Hodge isomorphism theorems for the Kohn--Rossi and de Rham cohomology groups, and then deduce Serre and Poincar\'e duality.

We begin by introducing notation for the spaces of $\dbbar$-harmonic $(p,q)$-forms and $d$-harmonic $k$-forms.

\begin{definition}
 Let $(M^{2n+1},T^{1,0},\theta)$ be a pseudohermitian manifold.
 For each $p\in\{0,\dotsc,n+1\}$ and $q\in\{0,\dotsc,n\}$, the \defn{space of $\dbbar$-harmonic $(p,q)$-forms} is
 \[ \mH^{p,q} := \ker\left( \Box_b\colon \mR^{p,q}\to\mR^{p,q}\right) . \]
\end{definition}

Recall from \cref{kernel-kohn-laplacian} that $\omega\in\mH^{p,q}$ if and only if $\dbbar\omega=0$ and $\dbbar^\ast\omega=0$.

\begin{definition}
 Let $(M^{2n+1},T^{1,0},\theta)$ be a pseudohermitian manifold.
 For each $k\in\{0,\dotsc,2n+1\}$, the \defn{space of $d$-harmonic $k$-forms} is
 \[ \mH^k := \ker \left( \Delta_b \colon \mR^k \to \mR^k \right) . \]
\end{definition}

Recall from \cref{kernel-rumin-laplacian} that $\omega\in\mH^k$ if and only if $d\omega=0$ and $d^\ast\omega=0$.

Our first Hodge decomposition theorem combines three statements with close analogues in the literature.
First, if $q\not\in\{0,n\}$, then there is an $L^2$-orthogonal splitting $\mR^{p,q} = \mH^{p,q} \oplus \im \Box_b$ and $\mH^{p,q}$ is finite-dimensional (cf.\ \cites{KohnRossi1965,Garfield2001,GarfieldLee1998}).
Second, if $n\geq2$ and $q\in\{0,n\}$, then the same splitting holds, though $\mH^{p,q}$ may be infinite-dimensional (cf.\ \cites{BealsGreiner1988,Garfield2001,GarfieldLee1998}).
Third, if $(M^3,T^{1,0})$ is embeddable, then the same splitting holds, though again $\mH^{p,q}$ may be infinite-dimensional (cf.\ \cites{Kohn1986,Miyajima1999}).

\begin{theorem}
 \label{hodge-decomposition}
 Let $(M^{2n+1},T^{1,0},\theta)$ be a closed, embeddable, strictly pseudoconvex manifold and let $p\in\{0,\dotsc,n+1\}$ and $q\in\{0,\dotsc,n\}$.
 Set
 \begin{equation*}
  s :=
  \begin{cases}
   2, & \text{if $p+q\not\in\{n,n+1\}$}, \\
   4, & \text{if $p+q\in\{n,n+1\}$} .
  \end{cases}
 \end{equation*}
 Then there is an $L^2$-orthogonal decomposition
 \begin{equation}
  \label{eqn:mRpq-closed-range}
  \mR^{p,q} = \mH^{p,q} \oplus \im \left( \Box_b \colon \mR^{p,q} \to \mR^{p,q} \right) .
 \end{equation}
 Moreover, the $L^2$-orthogonal projection $H^{p,q}\colon \mR^{p,q}\to\mH^{p,q}$ and the partial inverse $N^{p,q}\colon \mR^{p,q}\to \mR^{p,q}$ are Heisenberg pseudodifferential operators of order $0$ and $-s$, respectively.
 There is also an $L^2$-orthogonal decomposition
 \begin{equation}
  \label{eqn:hodge-decomposition}
  \mR^{p,q} = \mH^{p,q} \oplus \im \bigl( \dbbar \colon \mR^{p,q-1} \to \mR^{p,q} \bigr) \oplus \im \bigl( \dbbar^\ast \colon \mR^{p,q+1} \to \mR^{p,q} \bigr) .
 \end{equation}
 Additionally, if $q\not\in\{0,n\}$, then $\mH^{p,q}$ is finite-dimensional.
\end{theorem}

\begin{remark}
 Our proof does not require embeddability when $n\geq2$.
 Of course, closed, strictly pseudoconvex manifolds of dimension at least five are embeddable~\cite{Boutet1975}.
\end{remark}

\begin{proof}
 The proof splits into three cases.
 
 \textbf{Case 1}: $q\not\in\{0,n\}$.
 
 \Cref{invert-maximally-hypoelliptic} implies both the existence of $H^{p,q}$ and $N^{p,q}$ and the $L^2$-orthogonal decomposition~\eqref{eqn:mRpq-closed-range}.
 The $L^2$-orthogonal decomposition~\eqref{eqn:hodge-decomposition} follows from the fact $\dbbar^2=0$.
 
 \textbf{Case 2}: $q\in\{0,n\}$ and $n\geq 2$.
 
 By \cref{kohn-laplacian-hodge-star}, it suffices to consider the case $q=0$.
 The argument splits into three subcases.
 The first case, when $p\in\{0,\dotsc,n-2\}$ or $p=n$, is the case when the Kohn Laplacians on $\mR^{p,0}$ and $\mR^{p,1}$ have the same order.
 In this case, it is straightforward to write down explicit formulas for the operators $H^{p,0}$ and $N^{p,0}$ in terms of $H^{p,1}$ and $N^{p,1}$.
 The second and third cases, when $p=n-1$ and $p=n+1$, respectively, are the cases when the Kohn Laplacians on $\mR^{p,0}$ and $\mR^{p,1}$ have different orders.
 In these cases, we determine $H^{p,0}$ and $N^{p,0}$ modulo lower-order terms and apply \Cref{invert-almost-invertible}.
 Before considering these cases separately, note that since $n\geq2$, the resolution of Case 1 implies that
 \begin{equation}
  \label{eqn:apply-dbbar-trick}
  \dbbar = (H^{p,1} + \Box_b N^{p,1})\dbbar = \Box_b N^{p,1}\dbbar 
 \end{equation}
 on $\mR^{p,0}$.
 Our objective is to use Equation~\eqref{eqn:apply-dbbar-trick} and the fact that $\im\dbbar$ and $\im\dbbar^\ast$ are $L^2$-orthogonal to derive (approximate) formulas for $H^{p,0}$ and $N^{p,0}$.
 
 \textbf{Case 2a}: $p\in\{0,\dotsc,n-2\}$ or $p=n$.
 
 Combining Equation~\eqref{eqn:apply-dbbar-trick} with \cref{defn:kohn-laplacian} yields
 \begin{equation*}
  \dbbar =
  \begin{cases}
   \frac{n-p-1}{n-p}\dbbar\dbbar^\ast N^{p,1} \dbbar + \dbbar^\ast\dbbar N^{p,1}\dbbar , & \text{if $p\leq n-2$}, \\
   \dbbar\dbbar^\ast N^{p,1} \dbbar + \dbbar^\ast L_b \dbbar N^{p,1} \dbbar , & \text{if $p=n$}.
  \end{cases}
 \end{equation*}
 It readily follows that
 \begin{equation}
  \label{eqn:szego-same-order}
  H^{p,0} :=
  \begin{cases}
   I^{p,0} - \frac{n-p-1}{n-p}\dbbar^\ast N^{p,1} \dbbar , & \text{if $p\leq n-2$}, \\
   I^{p,0} - \dbbar^\ast N^{p,1} \dbbar , & \text{if $p=n$} ,
  \end{cases}
 \end{equation}
 defines an $L^2$-orthogonal projection $H^{p,0} \colon \mR^{p,0} \to \mH^{p,0}$.
 Indeed, if $\omega\in\mR^{n,0}$, then $\tau:=\dbbar^\ast N^{p,1}\dbbar\omega$ is such that $\dbbar\tau=\dbbar\omega$ and $\lV\tau\rV_2^2 \leq C\lV\dbbar\omega\rV_2^2$, where $C$ is independent of $\omega$.
 It follows that the $L^2$-closure of $\dbbar$, and hence of $\Box_b$, has closed range.
 A similar argument shows that the $L^2$-closure of $\Box_b$ has closed range on $\mR^{p,0}$, $p\leq n-2$.
 
 Now, since $\dbbar^\ast N^{p,1}=\dbbar^\ast\Box_b N^{p,1}N^{p,1}$, we deduce from \cref{dual-justification} and Equation~\eqref{eqn:szego-same-order} that
 \begin{equation*}
  I^{p,0} =
  \begin{cases}
   H^{p,0} + \bigl( \frac{n-p-1}{n-p} \bigr)^2 \Box_b \dbbar^\ast N^{p,1} N^{p,1} \dbbar, & \text{if $p\leq n-2$}, \\
   H^{p,0} + \Box_b \dbbar^\ast N^{p,1} N^{p,1} \dbbar , & \text{if $p=n$}.
  \end{cases}
 \end{equation*}
 We conclude that
 \begin{equation}
  \label{eqn:partial-inverse-same-order}
  N^{p,0} :=
  \begin{cases}
   \bigl( \frac{n-p-1}{n-p} \bigr)^2 \dbbar^\ast N^{p,1} N^{p,1} \dbbar , & \text{if $p\leq n-2$}, \\
   \dbbar^\ast N^{p,1} N^{p,1} \dbbar, & \text{if $p=n$} ,
  \end{cases}
 \end{equation}
 is the partial inverse of $\Box_b\colon\mR^{p,0}\to\mR^{p,0}$.
 This yields Equation~\eqref{eqn:mRpq-closed-range}.
 That $H^{p,0}$ (resp.\ $N^{p,0}$) is a Heisenberg pseudodifferential operator of order $0$ (resp.\ order $-s$) follows readily from Equation~\eqref{eqn:szego-same-order} (resp.\ Equation~\eqref{eqn:partial-inverse-same-order}) and \cref{compositions-in-pseudodiffoperator}.
 
 \textbf{Case 2b}: $p=n-1$.
 
 Combining Equation~\eqref{eqn:apply-dbbar-trick} with \cref{defn:kohn-laplacian} yields
 \begin{equation}
  \label{eqn:case2b-dbbar}
  \dbbar = \dbbar L_b \dbbar^\ast N^{n-1,1} \dbbar + \dbbar^\ast \dbbar N^{n-1,1} \dbbar
 \end{equation}
 on $\mR^{n-1,0}$.
 It follows that
 \begin{equation}
  \label{eqn:szego-n-1}
  S := I^{n-1,0} - L_b\dbbar^\ast N^{n-1,1}\dbbar
 \end{equation}
 maps $\mR^{n-1,0}$ to $\mH^{n-1,0}$.
 Arguing as in Case 2a, we see that the $L^2$-closure of $\Box_b$ has closed range.
 \Cref{dual-justification} and the identity $\dbbar^\ast N^{n-1,1}=\dbbar^\ast\Box_b N^{n-1,1} N^{n-1,1}$ imply that
 \begin{equation*}
   L_b \dbbar^\ast N^{n-1,1}\dbbar = L_b \dbbar^\ast \dbbar L_b \dbbar^\ast  N^{n-1,1} N^{n-1,1} \dbbar .
 \end{equation*}
 Hence, by \cref{L-dbbardbbarast-commutator},
 \begin{equation}
  \label{eqn:case2b-near-final}
  \begin{split}
   L_b \dbbar^\ast N^{n-1,1}\dbbar & = L_b\Box_b L_b \dbbar^\ast N^{n-1,1} N^{n-1,1} \dbbar \\
    & \equiv \Box_b L_b L_b \dbbar^\ast N^{n-1,1} N^{n-1,1} \dbbar \mod \PseudoDiffOperator{-2} .
  \end{split}
 \end{equation}
 Set $Q:=L_b L_b \dbbar^\ast N^{n-1,1} N^{n-1,1} \dbbar$.
 Combining Equations~\eqref{eqn:case2b-dbbar}, \eqref{eqn:szego-n-1}, and~\eqref{eqn:case2b-near-final} yields
 \begin{align*}
  I^{n-1,0} & \equiv S + \Box_b Q \mod \PseudoDiffOperator{-2} , \\
  \Box_b S & = 0 .
 \end{align*}
 Clearly $S\in\PseudoDiffOperator{0}$ and $Q\in\PseudoDiffOperator{-2}$.
 The conclusion now follows from \cref{invert-almost-invertible}.

 \textbf{Case 2c}: $p=n+1$.
 
 Combining Equation~\eqref{eqn:apply-dbbar-trick} with the definition of the Kohn Laplacian yields
 \begin{equation*}
  \dbbar = \dbbar\dbbar^\ast N^{n+1,1} \dbbar + \frac{1}{2}\dbbar^\ast\dbbar N^{n+1,1}\dbbar .
 \end{equation*}
 As in Case 2a, we see that the $L^2$-closure of $\Box_b$ has closed range and that
 \begin{equation}
  \label{eqn:szego-n+1}
  H^{n+1,0} := I^{n+1,0} - \dbbar^\ast N^{n+1,1}\dbbar
 \end{equation}
 is the $L^2$-orthogonal projection $H^{n+1,0}\colon\mR^{n+1,0}\to\mH^{n+1,0}$.
 As before, it holds that
 \begin{equation*}
  \dbbar^\ast N^{n+1,1}\dbbar = \dbbar^\ast \dbbar \dbbar^\ast N^{n+1,1} N^{n+1,1} \dbbar .
 \end{equation*}
 It follows from \cref{invert-maximally-hypoelliptic,hypoelliptic-lb2} that there is a $P^{n+1,1} \in \PseudoDiffOperator{-2}$ such that $L_bP^{n+1,1}\equiv I^{n+1,1}$ modulo a smoothing operator.
 Therefore
 \begin{align}
  \label{eqn:case2c-first-observation} \dbbar^\ast N^{n+1,1}\dbbar & \equiv \dbbar^\ast L_b P^{n+1,1} \dbbar \dbbar^\ast N^{n+1,1} N^{n+1,1} \dbbar \mod \PseudoDiffOperator{-\infty} , \\
  \label{eqn:M-commutator} [\dbbar\dbbar^\ast, P^{n+1,1}] & \equiv P^{n+1,1}[L_b,\dbbar\dbbar^\ast]P^{n+1,1} \mod \PseudoDiffOperator{-\infty} .
 \end{align}
 It follows from \cref{L-dbbardbbarast-commutator} and Equation~\eqref{eqn:M-commutator} that $[P^{n+1,1},\dbbar\dbbar^\ast]\in\PseudoDiffOperator{-2}$.
 Combining this with \cref{nablasast-complex} and Equation~\eqref{eqn:case2c-first-observation} yields
 \begin{equation}
  \label{eqn:case2c-second-observation}
  \begin{split}
   \dbbar^\ast N^{n+1,1}\dbbar & \equiv \dbbar^\ast L_b \dbbar\dbbar^\ast P^{n+1,1} N^{n+1,1} N^{n+1,1} \dbbar \mod \PseudoDiffOperator{-2} \\
   & = \Box_b \dbbar^\ast P^{n+1,1} N^{n+1,1} N^{n+1,1} \dbbar \mod \PseudoDiffOperator{-2} .
  \end{split}
 \end{equation}
 Combining Equations~\eqref{eqn:szego-n+1} and~\eqref{eqn:case2c-second-observation} implies that $\cN:=\dbbar^\ast P^{n+1,1} N^{n+1,1} N^{n+1,1} \dbbar$ is such that
 \begin{align*}
  I^{n+1,0} & \equiv H^{n+1,0} + \Box_b \cN \mod \PseudoDiffOperator{-2} , \\
  \Box_b H^{n+1,0} & = 0 .
 \end{align*}
 The conclusion again follows from \cref{invert-almost-invertible}.
 
 \textbf{Case 3}: $n=1$.
 Let $p\in\{0,1,2\}$.
 The existence of the operators $H^{p,0}$ and $N^{p,0}$ follows by applying a result of Miyajima~\cite{Miyajima1999}*{Proposition~4.1} to $F^p\COmega^pM$ and arguing as in the proof of \cref{kr-resolution}.
 The existence of the operators $H^{p,1}$ and $N^{p,1}$ follows from duality.
\end{proof}

Our second Hodge decomposition theorem is precisely the analogue of Rumin's Hodge decomposition theorem~\cite{Rumin1994}*{pg.\ 290} adapted to our formulation of the Rumin complex via differential forms.
Note that embeddability is not assumed.

\begin{theorem}
 \label{rumin-hodge-decomposition}
 Let $(M^{2n+1},T^{1,0},\theta)$ be a closed, strictly pseudoconvex manifold and let $k\in\{0,\dotsc,2n+1\}$.
 Set
 \begin{equation*}
  s :=
  \begin{cases}
   2, & \text{if $k\not\in\{n,n+1\}$}, \\
   4, & \text{if $k\in\{n,n+1\}$} .
  \end{cases}
 \end{equation*}
 Then $\mH^k$ is finite-dimensional and there is an $L^2$-orthogonal decomposition
 \begin{equation}
  \label{eqn:mRk-closed-range}
  \mR^k = \mH^k \oplus \im \left( \Delta_b \colon \mR^k \to \mR^k \right) .
 \end{equation}
 Moreover, the $L^2$-orthogonal projection $H^k\colon\mR^k\to\mH^k$ is a smoothing operator and the partial inverse $N^k\colon\mR^k\to\mR^k$ is a Heisenberg pseudodifferential operator of order $-s$.
 There is also an $L^2$-orthogonal decomposition
 \begin{equation}
  \label{eqn:rumin-hodge-decomposition}
  \mR^k = \mH^k \oplus \im \left( d \colon \mR^{k-1} \to \mR^k \right) \oplus \im \left( d^\ast \colon \mR^{k+1} \to \mR^k \right) .
 \end{equation}
\end{theorem}

\begin{proof}
 \Cref{invert-maximally-hypoelliptic,hypoelliptic-rumin} imply both the existence of $H^k$ and $N^k$ and the $L^2$-orthogonal decomposition~\eqref{eqn:mRk-closed-range}.
 The $L^2$-orthogonal decomposition~\eqref{eqn:rumin-hodge-decomposition} follows from the identity $d^2=0$.
\end{proof}

\Cref{hodge-decomposition,rumin-hodge-decomposition} immediately imply Hodge isomorphism theorems for the Kohn--Rossi and Rumin cohomology groups, respectively.

\begin{corollary}
 \label{hodge-isomorphism}
 Let $(M^{2n+1},T^{1,0},\theta)$ be a closed, embeddable, strictly pseudoconvex manifold and let $p\in\{0,\dotsc,n+1\}$ and $q\in\{0,\dotsc,n\}$.
 Then each class $[\omega]\in H^{p,q}(M)$ has a unique representative in $\mH^{p,q}$.
 In particular, there is a canonical isomorphism
 \[ H^{p,q}(M) \cong \mH^{p,q} . \]
\end{corollary}

\begin{proof}
 It follows immediately from \cref{kr-resolution,hodge-decomposition} that the $L^2$-orthogonal projection $H^{p,q}$ induces an isomorphism $H^{p,q}\colon H^{p,q}(M)\to\mH^{p,q}$.
\end{proof}

\begin{corollary}
 \label{rumin-hodge-isomorphism}
 Let $(M^{2n+1},T^{1,0},\theta)$ be a closed, strictly pseudoconvex manifold and let $k\in\{0,\dotsc,2n+1\}$.
 Then each class $[\omega]\in H^{k}(M;\bR)$ has a unique representative in $\mH^{k}$.
 In particular, there is a canonical isomorphism
 \[ H^{k}(M;\bR) \cong \mH^{k} . \]
\end{corollary}

\begin{proof}
 It follows immediately from \cref{r-resolution,rumin-hodge-decomposition} that the $L^2$-orthogonal projection $H^{k}$ induces an isomorphism $H^{k}\colon H^{k}(M;\bR)\to\mH^{k}$.
\end{proof}

We conclude with new proofs of Serre duality for the Kohn--Rossi cohomology groups (cf.\ \cites{Tanaka1975,GarfieldLee1998,Garfield2001}) and of Poincar\'e duality for the de Rham cohomology groups.

\begin{corollary}
 \label{serre}
 Let $(M^{2n+1},T^{1,0})$ be a closed, embeddable, strictly pseudoconvex CR manifold.
 Given $p\in\{0,\dotsc,n+1\}$ and $q\in\{0,\dotsc,n\}$, the conjugate Hodge star operator induces an isomorphism $H^{p,q}(M) \cong H^{n+1-p,n-q}(M)$.
\end{corollary}

\begin{proof}
 It follows immediately from \cref{hodge-mapping-properties,hodge-star-squared,kohn-laplacian-hodge-star} that
 \[ \mH^{p,q} \ni \omega \mapsto \hodge\oomega \in \mH^{n+1-p,n-q} \]
 is an isomorphism.
 The conclusion now follows from \cref{hodge-isomorphism}.
\end{proof}

\begin{corollary}
 \label{rumin-serre}
 Let $(M^{2n+1},T^{1,0})$ be a closed, strictly pseudoconvex CR manifold.
 Given $k\in\{0,\dotsc,2n+1\}$, the Hodge star operator induces an isomorphism $H^{k}(M;\bR) \cong H^{2n+1-k}(M;\bR)$.
\end{corollary}

\begin{proof}
 It follows immediately from \cref{hodge-star-squared,rumin-hodge-star} that
 \[ \mH^{k} \ni \omega \mapsto \hodge\omega \in \mH^{2n+1-k} \]
 is an isomorphism.
 The conclusion now follows from \cref{rumin-hodge-isomorphism}.
\end{proof}

%% file: hodge/popovici.tex
\section{A Hodge decomposition theorem via a nonlocal Laplacian}
\label{sec:popovici}

We conclude this part by proving a Hodge decomposition theorem for $\mR^{p,q}$ which will be used in \cref{sec:frolicher} to deduce a Hodge isomorphism theorem for the second page of the Garfield spectral sequence~\cite{Garfield2001}.
This decomposition theorem is obtained via a nonlocal Laplace-type operator on $\mR^{p,q}$ modeled on a similar operator appearing in work of Popovici~\cite{Popovici2016} on complex manifolds.

\begin{definition}
 \label{defn:popovici-laplacian}
 Let $(M^{2n+1},T^{1,0},\theta)$ be a closed, embeddable, strictly pseudoconvex manifold.
 Given $p\in\{0,\dotsc,n+1\}$ and $q\in\{0,\dotsc,n\}$, we define the \defn{Popovici Laplacian} $\cBoxb\colon\mR^{p,q}\to\mR^{p,q}$ by
 \begin{align*}
  \cBoxb & := \Box_b + \frac{n-p-q}{n-p-q+1}\db H^{p-1,q} \db^\ast + \db^\ast H^{p+1,q} \db, && \text{if $p+q\leq n-1$;} \\
  \cBoxb & := \Box_b + \dhor^\ast H^{p+1,q} \dhor + \db H^{p-1,q} \db^\ast \db H^{p-1,q} \db^\ast, && \text{if $p+q=n$;} \\
  \cBoxb & := \Box_b + \dhor H^{p-1,q} \dhor^\ast + \db^\ast H^{p+1,q} \db \db^\ast H^{p+1,q} \db, && \text{if $p+q=n+1$;} \\
  \cBoxb & := \Box_b + \frac{p+q-n-1}{p+q-n} \db^\ast H^{p+1,q} \db + \db H^{p-1,q} \db^\ast, && \text{if $p+q \geq n+2$.}
 \end{align*}
\end{definition}

Like the Kohn and Rumin Laplacians, the main properties of the Popovici Laplacian are that it is a nonnegative, formally self-adjoint, maximally hypoelliptic, Heisenberg pseudodifferential operator which commutes with the conjugate Hodge star operator.
As we will see in \cref{sec:frolicher}, the kernel of the Popovici Laplacian on $\mR^{p,q}$ is isomorphic to the space $E_2^{p,q}$ on the second page of the Garfield spectral sequence.
Unlike the Kohn and Rumin Laplacians, the Popovici Laplacian is a nonlocal operator.

We begin by verifying that the Popovici Laplacian is a nonnegative, formally self-adjoint operator and giving an alternative characterization of its kernel.
For notational simplicity, we often use $H$ to denote any of the $L^2$-orthogonal projections associated to the Kohn Laplacian, with the domain of $H$ determined by context.

\begin{lemma}
 \label{popovici-kernel}
 Let $(M^{2n+1},T^{1,0},\theta)$ be a closed, embeddable, strictly pseudoconvex manifold and let $p\in\{0,\dotsc,n+1\}$ and $q\in\{0,\dotsc,n\}$.
 Then $\cBoxb$ is a nonnegative, formally self-adjoint, Heisenberg pseudodifferential operator.
 Moreover,
 \begin{equation}
  \label{eqn:popovici-kernel}
  \ker \left( \cBoxb \colon \mR^{p,q} \to \mR^{p,q} \right) = \cmH^{p,q} ,
 \end{equation}
 where
 \begin{enumerate}
  \item $\cmH^{p,q} := \left\{ \omega \in \mR^{p,q} \suchthatcolon \Box_b\omega = H(\db\omega) = H(\db^\ast\omega) = 0 \right\}$, if $p+q\not\in\{n,n+1\}$;
  \item $\cmH^{p,q} := \left\{ \omega \in \mR^{p,q} \suchthatcolon \Box_b\omega = H(\dhor\omega) = H(\db^\ast\omega) = 0 \right\}$, if $p+q=n$;
  \item $\cmH^{p,q} := \left\{ \omega \in \mR^{p,q} \suchthatcolon \Box_b\omega = H(\db\omega) = H(\dhor^\ast\omega) = 0 \right\}$, if $p+q=n+1$ .
 \end{enumerate}
\end{lemma}

\begin{proof}
 \Cref{compositions-in-pseudodiffoperator} implies that $\cBoxb$ is a Heisenberg pseudodifferential operator.
 
 Using $H^2=H$ and $H=H^\ast$, we deduce that
 \begin{align*}
  \llp \db\omega, H\db\omega\rrp & = \lV H\db\omega\rV_2^2, & \llp \dhor\omega, H\dhor\omega\rrp & = \lV H\dhor\omega\rV_2^2, \\
  \llp \db^\ast\omega, H\db^\ast\omega\rrp & = \lV H\db^\ast\omega\rV_2^2, & \llp \dhor^\ast\omega, H\dhor^\ast\omega\rrp & = \lV H\dhor^\ast\omega\rV_2^2,
 \end{align*}
 when defined.
 Equation~\eqref{eqn:popovici-kernel} and the fact that $\cBoxb$ is nonnegative and formally self-adjoint readily follow.
\end{proof}

Next, we show that the Popovici Laplacian commutes with the conjugate Hodge star operator.

\begin{lemma}
 \label{popovici-hodge-star}
 Let $(M^{2n+1},T^{1,0},\theta)$ be a closed, embeddable, strictly pseudoconvex manifold.
 If $\omega\in\mR^{p,q}$, $p\in\{0,\dotsc,n+1\}$ and $q\in\{0,\dotsc,n\}$, then
 \[ \cBoxb \ohodge \omega = \ohodge \cBoxb \omega, \]
 where $\ohodge\omega := \overline{\hodge\omega}$.
\end{lemma}

\begin{proof}
 On the one hand, \cref{kohn-laplacian-hodge-star,hodge-decomposition} imply that
 \begin{align*}
  [\Box_b , \ohodge] & = 0 , \\
  [ H , \ohodge ] & = 0 .
 \end{align*}
 On the other hand, \cref{conjugate-rumin-operators,defn:dbbar-adjoint,hodge-star-squared} imply that
 \begin{align*}
  \db\ohodge & = (-1)^k\ohodge\db^\ast , & \text{on $\mR^{p,q}$, $p+q=k$} , \\
  \dhor\ohodge & = (-1)^n\ohodge\dhor^\ast , & \text{on $\mR^{p,q}$, $p+q=n$} .
 \end{align*}
 The conclusion readily follows.
\end{proof}

The remainder of this section is devoted to showing that the Popovici Laplacian is maximally hypoelliptic and deducing the corresponding Hodge decomposition theorem.
This requires some properties of the $L^2$-orthogonal projections $H$.

\begin{lemma}
 \label{hodge-projection-mapping-properties}
 Let $(M^{2n+1},T^{1,0},\theta)$ be a closed, embeddable, strictly pseudoconvex manifold and let $p\in\{0,\dotsc,n+1\}$ and $q\in\{0,\dotsc,n\}$.
 Then
 \begin{subequations}
  \begin{align}
   \label{eqn:hodge-projection-mapping-properties-db} \db H^{p,q} & \equiv H^{p+1,q}\db \mod \PseudoDiffOperator{-1} , && \text{if $p+q\not=n$} , \\
   \label{eqn:hodge-projection-mapping-properties-dbast} \db^\ast H^{p,q} & \equiv H^{p-1,q}\db^\ast \mod \PseudoDiffOperator{-1} , && \text{if $p+q\not=n+1$}, \\
   \label{eqn:hodge-projection-mapping-properties-dhor} \dhor H^{p,q} & \equiv H^{p+1,q}\dhor \mod \PseudoDiffOperator{0} , && \text{if $p+q=n$}, \\
   \label{eqn:hodge-projection-mapping-properties-dhorast} \dhor^\ast H^{p,q} & \equiv H^{p-1,q}\dhor^\ast \mod \PseudoDiffOperator{0}, && \text{if $p+q=n+1$} , \\ 
   \label{eqn:hodge-projection-mapping-properties-db-middle} \db H^{p,q} & \equiv 0 \mod \PseudoDiffOperator{0}, && \text{if $p+q=n$} , \\
   \label{eqn:hodge-projection-mapping-properties-dbast-middle} \db^\ast H^{p,q} & \equiv 0 \mod \PseudoDiffOperator{0} , && \text{if $p+q=n+1$} .
  \end{align}
 \end{subequations}
\end{lemma}

\begin{proof}
 By applying the conjugate Hodge star operator if necessary, it suffices to consider the case $p+q \leq n$.
 In particular, we need only verify Equations~\eqref{eqn:hodge-projection-mapping-properties-db}, \eqref{eqn:hodge-projection-mapping-properties-dbast}, \eqref{eqn:hodge-projection-mapping-properties-dhor}, and~\eqref{eqn:hodge-projection-mapping-properties-db-middle}.
 The first two equations are relevant when $p+q \leq n-1$ and the last three are relevant when $p+q = n$.
 We consider these equations and the cases $p+q \leq n-1$ or $p+q = n$ separately.
 
 \textbf{Case 1.} The proof of \cref{eqn:hodge-projection-mapping-properties-db}.
 
 Let $\rho \in \mH^{p,q}$, $p+q \leq n-1$.
 On the one hand, \cref{justification-of-bigraded-complex} implies that $\dbbar\db\rho = 0$.
 On the other hand, \cref{dbdbbarstar} implies that there is an $A \in \DiffOperator{0}$ such that $\dbbar^\ast\db\rho = A\rho$.
 Applying \cref{hodge-decomposition} to $A\rho$ yields
 \begin{equation*}
  \dbbar^\ast\db\rho = HA\rho + \Box_b NA\rho .
 \end{equation*}
 The $L^2$-orthogonal decomposition~\eqref{eqn:hodge-decomposition} implies that
 \begin{equation*}
  \dbbar^\ast \left( \db\rho - \dbbar NA\rho \right) = 0 . 
 \end{equation*}
 Therefore $H\db\rho = \db\rho - \dbbar NA\rho$.
 In particular,
 \begin{equation}
  \label{eqn:lower-HdbH}
  H\db\rho \equiv \db\rho \mod \PseudoDiffOperator{-1} .
 \end{equation}
 
 Now let $\omega \in \mR^{p,q}$.
 \Cref{hodge-decomposition} implies that
 \begin{equation*}
  \omega = H\omega + \frac{n-p-q}{n-p-q+1}\dbbar\dbbar^\ast N\omega + \dbbar^\ast\dbbar N\omega .
 \end{equation*}
 We deduce from \cref{justification-of-bigraded-complex} that
 \begin{equation}
  \label{eqn:lower-Hdb}
  H\db\omega = H\db H\omega + H\db\dbbar^\ast\dbbar N\omega .
 \end{equation}
 If $p+q < n-1$, then \cref{dbdbbarstar} implies that
 \begin{equation*}
  \db\dbbar^\ast\dbbar N\omega \equiv -\frac{n-p-q}{n-p-q-1}\dbbar^\ast\db\dbbar N\omega \mod \PseudoDiffOperator{-1} .
 \end{equation*}
 If instead $p+q = n-1$, then \cref{eqn:n-hodge-db} implies that
 \begin{equation*}
  \db\dbbar^\ast\dbbar N\omega \equiv (-1)^p i^{n^2-1}\dbbar^\ast \hodge \dbbar N\omega \mod \PseudoDiffOperator{-1} .
 \end{equation*}
 In either case, we conclude that $H\db\dbbar^\ast\db N\omega \equiv 0 \mod \PseudoDiffOperator{-1}$.
 Combining this with \cref{eqn:lower-Hdb} yields Equation~\eqref{eqn:hodge-projection-mapping-properties-db}.

 \textbf{Case 2.} The proof of \cref{eqn:hodge-projection-mapping-properties-dbast} when $p+q \leq n-1$.
 
 Let $\rho \in \mH^{p,q}$.
 On the one hand, \cref{dual-justification} implies that $\dbbar^\ast\db^\ast\rho = 0$.
 On the other hand, \cref{dbdbbarstar} implies that there is an $A \in \DiffOperator{0}$ such that $\dbbar\db^\ast\rho = A\rho$.
 Arguing as in Case~1, we deduce from \cref{hodge-decomposition} that
 \begin{equation}
  \label{eqn:lower-HdbastH}
  H\db^\ast\rho \equiv \db^\ast\rho \mod \PseudoDiffOperator{-1} .
 \end{equation}
 
 Now let $\omega \in \mR^{p,q}$.
 It follows from \cref{hodge-decomposition,dual-justification} that
 \begin{equation}
  \label{eqn:lower-Hdbast}
  H\db^\ast\omega = H\db^\ast H\omega + \frac{n-p-q}{n-p-q+1}H\db^\ast\dbbar\dbbar^\ast N\omega .
 \end{equation}
 Combining Equations~\eqref{eqn:lower-HdbastH} and~\eqref{eqn:lower-Hdbast} with \cref{dbdbbarstar} yields Equation~\eqref{eqn:hodge-projection-mapping-properties-dbast} when $p+q \leq n-1$.

 \textbf{Case 3.} The proof of \cref{eqn:hodge-projection-mapping-properties-dbast} when $p+q = n$.
 
 Let $\rho \in \mH^{p,q}$.
 On the one hand, Equation~\eqref{eqn:n-hodge-dbbar} yields an $A \in \DiffOperator{0}$ such that $\dbbar\db^\ast\rho = A\rho$.
 We deduce from \cref{hodge-decomposition} that
 \begin{equation}
  \label{eqn:get-case3-inverse-a}
  \dbbar\left( \db^\ast\rho - L_b\dbbar^\ast NA\rho \right) = 0 .
 \end{equation}
 On the other hand, the definition~\eqref{eqn:Lb} of $L_b$ and \cref{dual-justification} imply that
 \begin{align*}
  \dbbar^\ast L_b\dbbar^\ast NA\rho & = \dbbar^\ast\db^\ast\db\dbbar^\ast NA\rho + \frac{1}{2}\dbbar^\ast\db\db^\ast\dbbar^\ast NA\rho \\
   & = -\db^\ast\dbbar^\ast\db\dbbar^\ast NA\rho - \frac{1}{2}\dbbar^\ast\db\dbbar^\ast\db^\ast NA\rho \\
   & = -\db^\ast\left( \dbbar^\ast\db + \frac{1}{2}\db\dbbar^\ast \right)\dbbar^\ast NA\rho - \dbbar^\ast \left( \frac{1}{2}\db\dbbar^\ast + \dbbar^\ast\db \right)\db^\ast NA\rho .
 \end{align*}
 Hence, by \cref{dbdbbarstar}, there is a $B \in \PseudoDiffOperator{-2}$ such that
 \begin{equation}
  \label{eqn:define-B}
  \dbbar^\ast L_b \dbbar^\ast NA\rho = B\rho .
 \end{equation}
 \Cref{dual-justification} implies that $\dbbar^\ast\db^\ast\rho = 0$.
 Combining this with \cref{hodge-decomposition} and Equation~\eqref{eqn:define-B} yields
 \begin{equation}
  \label{eqn:get-case3-inverse-b}
  \dbbar^\ast \left( \db^\ast\rho - L_b\dbbar^\ast NA\rho + \dbbar NB\rho \right) = 0 .
 \end{equation}
 Combining Equations~\eqref{eqn:get-case3-inverse-a} and~\eqref{eqn:get-case3-inverse-b} yields
 \begin{equation}
  \label{eqn:middle-HdbastH}
  H\db^\ast\rho \equiv \db^\ast\rho \mod \PseudoDiffOperator{-1} .
 \end{equation}

 Now let $\omega \in \mR^{p,q}$.
 It follows from \cref{hodge-decomposition,dual-justification} that
 \begin{equation}
  \label{eqn:middle-Hdbast}
  H\db^\ast\omega = H\db^\ast H\omega + H\db^\ast\dbbar L_b \dbbar^\ast N\omega .
 \end{equation}
 Combining Equations~\eqref{eqn:middle-HdbastH} and~\eqref{eqn:middle-Hdbast} with \cref{dbdbbarstar} yields Equation~\eqref{eqn:hodge-projection-mapping-properties-dbast} when $p+q=n$.
 
 \textbf{Case 4.} The proof of \cref{eqn:hodge-projection-mapping-properties-dhor}.
 
 Let $\rho \in \mH^{p,q}$, $p+q = n$.
 \Cref{justification-of-bigraded-complex} implies that $\dbbar\dhor\rho = 0$.
 \Cref{dhordbbarstar} and Equation~\eqref{eqn:n-hodge-db} imply that there is an $A \in \DiffOperator{2}$ such that $\dbbar^\ast\dhor\rho = A\rho$.
 Arguing as in Case 1, we see that
 \begin{equation}
  \label{eqn:middle-HdbH}
  H\dhor\rho \equiv \dhor\rho \mod \PseudoDiffOperator{0} .
 \end{equation}
 
 Now let $\omega \in \mR^{p,q}$.
 \Cref{hodge-decomposition} implies that
 \begin{equation*}
  \omega = H\omega + \dbbar L_b \dbbar^\ast N\omega + \dbbar^\ast\dbbar N\omega .
 \end{equation*}
 We deduce from \cref{justification-of-bigraded-complex} that
 \begin{equation}
  \label{eqn:middle-Hdb}
  H\dhor\omega = H\dhor H\omega + H\dhor \dbbar^\ast \dbbar N\omega .
 \end{equation}
 On the one hand, \cref{dhordbbarstar} implies that
 \begin{equation}
  \label{eqn:middle-Hdb-simplify1}
  H (\dhor \dbbar^\ast + \db \dhor^\ast) \dbbar N\omega \equiv 0 \mod \PseudoDiffOperator{0} .
 \end{equation}
 On the other hand, \cref{eqn:n-hodge-db} implies that
 \begin{equation}
  \label{eqn:middle-Hdb-simplify2}
  \begin{split}
   H\db & \equiv (-1)^{p}i^{n^2+1} H \hodge \db \dbbar^\ast \mod \PseudoDiffOperator{0} \\
    & = (-1)^p i^{n^2+1} H \dbbar^\ast \db \hodge \\
    & = 0
  \end{split}
 \end{equation}
 on $\mR^{p-1,q+1}$.
 Combining Equations~\eqref{eqn:middle-HdbH}, \eqref{eqn:middle-Hdb}, \eqref{eqn:middle-Hdb-simplify1}, and~\eqref{eqn:middle-Hdb-simplify2} yields Equation~\eqref{eqn:hodge-projection-mapping-properties-dhor}.
 
 \textbf{Case 5.} The proof of \cref{eqn:hodge-projection-mapping-properties-db-middle}.
 
 Let $\rho \in \mH^{p,q}$, $p+q = n$.
 Equation~\eqref{eqn:n-hodge-db} implies that $\db\rho \equiv 0 \mod \DiffOperator{0}$.
 Equation~\eqref{eqn:hodge-projection-mapping-properties-db-middle} readily follows from \cref{hodge-decomposition}.
\end{proof}

We now show that the Popovici Laplacian $\cBoxb\colon\mR^{p,q}\to\mR^{p,q}$ is maximally hypoelliptic for all $p\in\{0,\dotsc,n+1\}$ and all $q\in\{0,\dotsc,n\}$.
Since $\cBoxb\gtrsim\Box_b$, maximal hypoellipticity is automatic when $q\not\in\{0,n\}$ (cf.\ \cite{Popovici2016}).
The cases $q\in\{0,n\}$ are more subtle.
We present a unified argument which is valid for all $q\in\{0,\dotsc,n\}$.

\begin{proposition}
 \label{popovici-hypoelliptic}
 Let $(M^{2n+1},T^{1,0},\theta))$ be a closed, embeddable, strictly pseudoconvex manifold.
 For each $p\in\{0,\dotsc,n+1\}$ and $q\in\{0,\dotsc,n\}$, the Popovici Laplacian $\cBoxb\colon\mR^{p,q}\to\mR^{p,q}$ admits a parametrix.
\end{proposition}

\begin{proof}
 Let $s$ denote the order of $\cBoxb$;
 i.e.\ $s=2$ if $p+q \not\in \{ n , n+1 \}$ and $s=4$ otherwise.
 
 Since $H^2=H$ and $H\Box_b = \Box_bH = 0$, we deduce from \cref{hodge-projection-mapping-properties} that
 \begin{equation*}
  \cBoxb \equiv \Box_b + H L_b H \mod \PseudoDiffOperator{s-2} .
 \end{equation*}
 In particular, it suffices to show that $\cL_b := \Box_b + HL_bH$ admits a parametrix.
 Note that $\cL_b$ is formally self-adjoint.

 Let $N$ and $H$ be the partial inverse for $\Box_b$ and projection to $\ker \Box_b$, respectively, the existence of which is guaranteed by \cref{hodge-decomposition}.
 \Cref{invert-maximally-hypoelliptic,hypoelliptic-lb2,hypoelliptic-lb4} imply that there is a $P \in \PseudoDiffOperator{-s}$ such that $L_bP \equiv I$ and $PL_b \equiv I$ modulo $\PseudoDiffOperator{-\infty}$.
 We compute that
 \begin{align*}
  \cL_b(N + HPH) & = \Box_bN + HL_bHPH = I - H + HL_bHPH , \\
  (N + HPH)\cL_b & = N\Box_b + HPHL_bH = I - H + HPHL_bH .
 \end{align*}
 \Cref{hodge-projection-mapping-properties} implies that $HL_bH \equiv HL_b \equiv L_bH \mod \PseudoDiffOperator{s-2}$.
 Combining this with the previous display and the definition of $P$ yields
 \begin{align*}
  \cL_b(N + HPH) & \equiv I \mod \PseudoDiffOperator{s-2} , \\
  (N + HPH)\cL_b & \equiv I \mod \PseudoDiffOperator{s-2} .
 \end{align*}
 The second equivalence implies that the $L^2$-closure of $\cL_b$ has closed range.
 The conclusion now follows from \cref{invertibility-characterizations,invert-almost-invertible}.
\end{proof}

As a consequence, we have the following Hodge decomposition theorem.

\begin{theorem}
 \label{popovici-hodge-decomposition}
 Let $(M^{2n+1},T^{1,0},\theta))$ be a closed, embeddable, strictly pseudoconvex manifold and let $p\in\{0,\dotsc,n+1\}$ and $q\in\{0,\dotsc,n\}$.
 Then $\cmH^{p,q}$ is finite-dimensional.
 Moreover, there is an $L^2$-orthogonal splitting
 \begin{align*}
  \mR^{p,q} & = \cmH^{p,q} \oplus \left( \im \dbbar + \im \db \rv_{\ker\dbbar} \right) \oplus \left( \im \left(\db^\ast\circ H\right) + \im \dbbar^\ast \right) , && \text{$p+q\not\in\{n,n+1\}$} , \\
  \mR^{p,q} & = \cmH^{p,q} \oplus \left( \im \dbbar + \im \db\rv_{\ker\dbbar} \right) \oplus \left( \im \left(\dhor^\ast\circ H\right) + \im \dbbar^\ast \right) , && \text{$p+q=n$}, \\
  \mR^{p,q} & = \cmH^{p,q} \oplus \left( \im \dbbar + \im \dhor\rv_{\ker\dbbar} \right) \oplus \left( \im \left(\db^\ast\circ H\right) + \im \dbbar^\ast \right) , && \text{$p+q=n+1$} ,
 \end{align*}
 where $\im\dbbar$ denotes the image of $\dbbar$ in $\mR^{p,q}$, and so on.
\end{theorem}

\begin{proof}
 Recall from \cref{popovici-kernel} that $\cmH^{p,q}=\ker\cBoxb$.
 It follows readily from \cref{invert-maximally-hypoelliptic,popovici-hypoelliptic} that $\cmH^{p,q}$ is finite-dimensional and that there is an $L^2$-orthogonal splitting
 \begin{equation}
  \label{eqn:popovici-subelliptic-consequence}
  \mR^{p,q} = \cmH^{p,q} \oplus \im \cBoxb .
 \end{equation}
 Since $H(\mR^{p,q})\subset\ker\dbbar$, we see that
 \begin{equation}
  \label{eqn:im-dbH}
  \begin{aligned}
   \im \left(\db\circ H\right) & \subset \im \db\rv_{\ker\dbbar}, && \text{if $p+q\not=n+1$}, \\
   \im \left(\dhor\circ H\right) & \subset \im \dhor\rv_{\ker\dbbar}, && \text{if $p+q=n+1$} .
  \end{aligned}
 \end{equation}
 
 We give the rest of the proof in the case $p+q=n$;
 the other cases are similar.
 \Cref{defn:popovici-laplacian} and Display~\eqref{eqn:im-dbH} imply that
 \begin{equation*}
  \im \cBoxb \subset \left( \im \dbbar + \im \db\rv_{\ker\dbbar} \right) + \left( \im \left(\dhor^\ast \circ H\right) + \im \dbbar^\ast \right) .
 \end{equation*}
 \Cref{justification-of-bigraded-complex} implies that $(\im\dbbar + \im\db\rv_{\ker\dbbar})$ and $(\im \left(\dhor^\ast\circ H\right) + \im\dbbar^\ast)$ are $L^2$-orthogonal;
 in particular, this sum is direct.
 Hence, by Equation~\eqref{eqn:popovici-subelliptic-consequence},
 \begin{equation*}
  \mR^{p,q} \subset \cmH^{p,q} \oplus \left( \im \dbbar + \im \db\rv_{\ker\dbbar} \right) \oplus \left( \im \left(\dhor^\ast\circ H\right) + \im \dbbar^\ast \right) .
 \end{equation*}
 The conclusion follows readily.
\end{proof}

%% file: part-applications.tex
\part{Applications}
\label{part:applications}

We conclude this article by discussing some applications of the bigraded Rumin complex and our Hodge theorems to cohomology groups and pseudo-Einstein contact forms on CR manifolds.

In \cref{sec:vanishing} we give two estimates on the dimensions of the Kohn--Rossi cohomology groups $H^{p,q}(M)$.
First, we show that if $(M^{2n+1},T^{1,0})$, $n\geq2$, is a closed, strictly pseudoconvex CR manifold which admits a contact form with nonnegative pseudohermitian Ricci curvature, then $\dim H^{0,q}(M)\leq\binom{n}{q}$ for all $q\in\{1,\dotsc,n-1\}$;
moreover, $\dim H^{0,q}(M)=0$ if the pseudohermitian Ricci curvature is positive at a point.
This result is sharp and refines results of Tanaka~\cite{Tanaka1975}*{Proposition~7.4} and Lee~\cite{Lee1988}*{Proposition~6.4}.
Second, we give a similar estimate for $\dim H^{p,q}(M)$, $p+q\not\in\{n,n+1\}$ and $q\not\in\{0,n\}$, for closed, locally spherical, strictly pseudoconvex manifolds which admit a pseudo-Einstein contact form with nonnegative pseudohermitian scalar curvature.
This result is also sharp.
Of course, the assumptions on the latter result are quite strong;
this result is primarily intended to motivate future studies of the Kohn--Rossi cohomology groups.

In \cref{sec:frolicher} we introduce the CR analogue of the Fr\"olicher spectral sequence of a complex manifold~\cite{Frolicher1955}.
This spectral sequence was introduced by Garfield~\cite{Garfield2001}, and so we call it the \defn{Garfield spectral sequence}.
One property of the Garfield spectral sequence is that its first and limiting pages recover the Kohn--Rossi and complexified de Rham cohomology groups, respectively.
We present two results which indicate that the second page $E_2^{p,q}$ of the Garfield spectral sequence should be regarded as the CR analogue of the Dolbeault cohomology groups.
First, we show that $\dim E_2^{p,q}<\infty$ for all closed, embeddable, strictly pseudoconvex CR manifolds.
Second, we prove the \defn{CR Fr\"olicher inequality}
\begin{equation*}
 \dim H^k(M;\bC) \leq \sum_{p+q=k} \dim E_2^{p,q}
\end{equation*}
(cf.\ \cite{Frolicher1955}).
Moreover, we show that the existence of a torsion-free contact form implies the Hodge decomposition
\begin{equation*}
 H^k(M;\bC) \cong \bigoplus_{p+q=k} E_2^{p,q} .
\end{equation*}
In particular, the CR Fr\"olicher inequality holds with equality.
This furthers the analogy of torsion-free pseudohermitian manifolds --- also known as \defn{Sasaki manifolds} --- as the odd-dimensional analogue of K\"ahler manifolds~\cite{BoyerGalicki2008}.
It also yields a new proof of the fact~\cites{BlairGoldberg1967,Fujitani1966} that if $(M^{2n+1},T^{1,0})$ is a closed \defn{Sasakian manifold} --- meaning it admits a torsion-free contact form --- and $k\leq n$ is odd, then $\dim H^k(M;\bR)$ is even.

In \cref{sec:sasakian} we present five additional applications of the bigraded Rumin complex to the topology of closed Sasakian manifolds.
Our first result is a vanishing result for the spaces $E_2^{p,0}$, $p \geq 1$, on closed Sasakian $(2n+1)$-manifolds with nonnegative Ricci curvature.
When $p\not\in\{n,n+1\}$, this follows from the isomorphism $E_2^{p,0} \cong \overline{E_2^{0,p}}$ on Sasaki manifolds and an aforementioned vanishing theorem in \cref{sec:vanishing};
when $p \in \{ n, n+1 \}$, additional properties of Sasaki manifolds are required.
This result is analogous to a vanishing result of Nozawa~\cite{Nozawa2014} for the basic Dolbeault cohomology groups.
Our second result is the vanishing of the cup product
\begin{equation*}
 \cup \colon H^k(M;\bC) \otimes H^\ell(M;\bC) \to H^{k+\ell}(M;\bC)
\end{equation*}
when $k,\ell \leq n$ but $k+\ell \geq n+1$.
This result, and its application as an obstruction to the existence of a torsion-free contact form on a given CR manifold, was first observed by Bungart~\cite{Bungart1992}.
Notably, this result implies that the cuplength of a closed Sasakian $(2n+1)$-manifold is at most $n+1$;
closed quotients of the Heisenberg group~\cite{Folland2004} show that this estimate is sharp.
A weaker estimate on the cuplength was previously observed by Rukimbira~\cite{Rukimbira1993} and Itoh~\cite{Itoh1997}, and the sharpness corrects a misstatement in Boyer and Galicki's book~\cite{BoyerGalicki2008}*{Example~8.1.13}.
Our third result is that all real Chern classes of degree at least $n+1$ vanish on a closed Sasakian $(2n+1)$-manifold.
Though we cannot find this statement in the literature, it readily follows from known facts~\cite{BoyerGalicki2008}*{Theorem~7.2.9 and Lemma~7.5.22}.
This estimate gives an obstruction to the existence of a torsion-free contact form on a given CR manifold~\cite{Takeuchi2018}*{Proposition~4.2}.
Our fourth result is a Sasakian analogue of the Hard Lefschetz Theorem:
If $(M^{2n+1},T^{1,0},\theta)$ is a closed Sasaki manifold and $H$ denotes the $L^2$-orthogonal projection onto the kernel of the Rumin Laplacian, then
\begin{equation*}
 H^k(M;\bC) \ni [\omega] \mapsto \Lef([\omega]) := [ \theta \wedge H\omega \wedge d\theta^{n-k}] \in H^{2n+1-k}(M;\bC)
\end{equation*}
is an isomorphism for all $k \in \{ 0, \dotsc, n \}$.
Moreover, if $\hT^{1,0}$ is another CR structure such that $(M^{2n+1},\hT^{1,0},\theta)$ is Sasaki, then $\widehat{\Lef}=\Lef$.
In other words, $\Lef$ is independent of the choice of Sasaki structure.
This fact was first observed by Cappelletti-Montano, De Nicola, and Yudin~\cite{CappellettiMontanoDeNicolaYudin2015};
a key point is that the bigraded Rumin complex makes the identity $\widehat{\Lef}=\Lef$ transparent.
Our fifth result is a pair of Sasaki analogues of the $\partial\overline{\partial}$-lemma.
These analogues include extra assumptions involving the adjoints $\db^\ast$ and $\dbbar^\ast$ which are chosen for their application to the Lee Conjecture in \cref{sec:lee-conjecture}.
Examples are given which show that some extra assumption is necessary.

In \cref{sec:lee-conjecture}, we introduce the CR invariant cohomology class
\begin{equation*}
 \mL := \left[ -i\bigl( P_{\alpha\bar\beta} - \frac{P}{n} h_{\alpha\bar\beta}\bigr) \,\theta^\alpha\wedge\theta^{\bar\beta} \right] \in H_R^1(M;\sP) ,
\end{equation*}
$P:=P_\mu{}^\mu$, on an orientable CR manifold $(M^{2n+1},T^{1,0})$ and prove that $\mL=0$ if and only if $(M^{2n+1},T^{1,0})$ admits a pseudo-Einstein contact form.
This generalizes a construction of Lee~\cite{Lee1988} for locally embeddable CR manifolds.
The image of $\mL$ under the morphism $H_R^1(M;\sP)\to H^2(M;\bR)$ is a dimensional multiple of the real first Chern class $c_1(T^{1,0})$, and hence we recover Lee's observation~\cite{Lee1988}*{Proposition~D} that the vanishing of $c_1(T^{1,0})$ is necessary for the existence of a pseudo-Einstein contact form.
This motivates the Lee Conjecture~\cite{Lee1988}:
If $(M^{2n+1},T^{1,0})$ is a closed, embeddable, strictly pseudoconvex manifold with $c_1(T^{1,0})=0$, then it admits a pseudo-Einstein contact form.
We use our Hodge theorems to recover Lee's partial results~\cite{Lee1988} towards the Lee Conjecture:
It is true on Sasakian manifolds and on CR manifolds of dimension at least five which admit a contact form with nonnegative pseudohermitian Ricci tensor.

\input{applications/vanishing}
\input{applications/frolicher}
\input{applications/sasakian}
\input{applications/lee-conjecture}

%% file: applications/vanishing.tex
\section{Dimension estimates for Kohn--Rossi groups}
\label{sec:vanishing}

In this section we prove dimension estimates for the Kohn--Rossi cohomology groups provided there is a contact form with nonnegative curvature.
For notational convenience, we make the following definition.

\begin{definition}
 \label{defn:betti}
 Let $(M^{2n+1},T^{1,0})$ be a closed, strictly pseudoconvex manifold.
 Given $p\in\{0,\dotsc,n+1\}$ and $q\in\{1,\dotsc,n-1\}$, we set
 \[ b^{p,q} := \dim H^{p,q}(M) . \]
\end{definition}

The assumptions in \cref{defn:betti} imply that $b^{p,q}$ is finite.

We first estimate $b^{0,q}$.
This improves results of Tanaka~\cite{Tanaka1975}*{Proposition~7.4} and Lee~\cite{Lee1988}*{Proposition~6.4} both by giving a sharp estimate on $b^{0,q}$ and by showing that $b^{0,q}=0$ if $R_{\alpha\bar\beta}>0$ at a point.

\begin{theorem}
 \label{H0q-vanishing}
 Let $(M^{2n+1},T^{1,0},\theta)$ be a closed, strictly pseudoconvex pseudohermitian manifold with nonnegative pseudohermitian Ricci curvature $R_{\alpha\bar\beta}$.
 If $\omega\in\mH^{0,q}$, $q\in\{1,\dotsc,n-1\}$, then $\omega$ is parallel and $\Ric\ohash \omega=0$.
 In particular,
 \begin{enumerate}
  \item $b^{0,q} \leq \binom{n}{q}$; and
  \item if there is a point at which $R_{\alpha\bar\beta}>0$, then $b^{0,q}=0$.
 \end{enumerate}
\end{theorem}

\begin{proof}
 Let $\omega\in\mH^{0,q}$, $q\in\{1,\dotsc,n-1\}$.
 \Cref{bochner-order2} implies that
 \[ \frac{n-q}{n}\nablabbar^\ast\nablabbar\omega + \frac{q}{n}\nablab^\ast\nablab\omega - \frac{n-q}{n}\Ric\ohash\omega = 0 . \]
 Write $\omega = \frac{1}{q!}\omega_{\bar\Beta}\,\theta^{\bar\Beta}$.
 Integrating over $M$ yields
 \[ \int_M \left( \frac{n-q}{n}\lv\nablabbar\omega\rv^2 + \frac{q}{n}\lv\nablab\omega\rv^2 + \frac{n-q}{n(q-1)!}R^{\bar\nu}{}_{\bar\sigma}\omega_{\bar\nu\bar\Beta^\prime}\oomega^{\bar\sigma\bar\Beta^\prime} \right)\,\theta\wedge d\theta^n = 0 . \]
 Since $R_{\alpha\bar\beta}\geq0$, we deduce that $\nablab\omega=\nablabbar\omega=0$ and that $\Ric\ohash\omega=0$.
 \Cref{commutators} then implies that $\omega$ is parallel.
 Since $\mR^{0,q}$ is isomorphic to $\Omega^{0,q}M$, we conclude that $\dim\mH^{0,q}\leq\binom{n}{q}$.
 
 Suppose now that $R_{\alpha\bar\beta}>0$ at a point $x\in M$.
 Let $\omega\in\mH^{0,q}$.
 Then $\omega$ is parallel and $\omega=0$ at $x$.
 Hence $\omega=0$.
 Therefore $\mH^{0,q}=0$.
 
 The final conclusion follows from \cref{hodge-isomorphism}.
\end{proof}

Note that the dimension bounds in \cref{H0q-vanishing} are sharp:

\begin{example}
 \label{heisenberg-quotient}
 Consider the Heisenberg $n$-manifold $\bH^n:=\bC^n\times\bR$ with $T^{1,0}$ generated by $\bigl\{Z_\alpha := \partial_{z^\alpha} + \frac{i\oz^\alpha}{2}\partial_t\bigr\}_{\alpha=1}^n$.
 It is straightforward to verify that
 \[ \theta := dt + \frac{i}{2}\sum_{\alpha=1}^n \left( z^\alpha\,d\oz^\alpha - \oz^\alpha\,dz^\alpha \right) \]
 is a flat, torsion-free contact form on $(\bH^n,T^{1,0})$.
 Moreover, one readily checks that
 \[ (w,s)\cdot(z,t) := (z+w, t+s-\Imaginary z\cdot\ow) \]
 defines a free, proper left action of $\Gamma := \bigl(\bZ[i]\bigr)^n\times\bZ$ on $\bH^n$ by CR maps which preserve $\theta$.
 In particular, $\theta$ descends to the quotient $I^{2n+1}:=\Gamma \backslash \bH^n$.
 Moreover, each $d\oz^\alpha$, $\alpha\in\{1,\dotsc,n\}$, descends to a nowhere-vanishing, $\dbbar$-harmonic $(0,1)$-form.
 Therefore $d\oz^{\alpha_1} \krwedge \dotsm \krwedge d\oz^{\alpha_q}$ is a $\dbbar$-harmonic $(0,q)$-form on $I^{2n+1}$.
 We conclude from \cref{hodge-isomorphism,H0q-vanishing} that $b^{0,q}=\binom{n}{q}$ (cf.\ \cite{Folland2004}*{Corollary~3.4}).
\end{example}

\Cref{H0q-vanishing} gives a sufficient condition for $H_R^k(M;\sP)\to H^{k+1}(M;\bR)$ to be injective.

\begin{corollary}
 \label{cohomology-map-trivial}
 Let $(M^{2n+1},T^{1,0},\theta)$ be a closed, strictly pseudoconvex manifold with nonnegative pseudohermitian Ricci tensor.
 Then $H_R^k(M;\sP)\to H^{k+1}(M;\bR)$ is injective for each $k\in\{1,\dotsc,n-1\}$.
\end{corollary}

\begin{proof}
 Let $k\in\{1,\dotsc,n-1\}$.
 \Cref{H0q-vanishing} implies that elements of $H_R^{0,k}(M)$ are represented by parallel $(0,k)$-forms.
 Therefore $H_R^{0,k}(M)\to H_R^k(M;\sP)$ is the zero map.
 The conclusion follows from \cref{long-exact-sequence,r-resolution}.
\end{proof}

Our next result implies that $b^{p,q}=0$ for locally spherical pseudo-Einstein manifolds with positive scalar curvature.
These assumptions are quite strong, and are mainly made to illustrate the Bochner method for $(p,q)$-forms on CR manifolds.
This result should be compared with the partial resolution by D.-C.\ Chang, S.-C.\ Chang and Tie~\cite{ChangChangTie2014}*{Theorem~1.3} of the CR Frankel Conjecture~\cite{ChangChangTie2014}*{Conjecture~1.2} (cf.\ \cite{HeSun2016}).

\begin{theorem}
 \label{Hpq-vanishing}
 Let $(M^{2n+1},T^{1,0})$ be a closed, locally spherical, strictly pseudoconvex CR manifold and let $p\in\{0,\dotsc,n+1\}$ and $q\in\{1,\dotsc,n-1\}$ be such that $p+q\not\in\{n,n+1\}$.
 If $(M^{2n+1},T^{1,0})$ admits a pseudo-Einstein contact form with nonnegative pseudohermitian scalar curvature $R$, then every $\dbbar$-harmonic $(p,q)$-form is parallel.
 Moreover, if $R\not=0$, then $b^{p,q}=0$.
\end{theorem}

\begin{proof}
 Since the conclusion of the theorem is vacuous if $n=1$, we may assume that $n\geq 2$.
 Moreover, \cref{serre} implies that we may assume that $p+q \leq n-1$.
 
 Let $\theta$ be a pseudo-Einstein~\cite{Lee1988} contact form on $(M^{2n+1},T^{1,0})$;
 i.e.
 \begin{equation*}
  R_{\alpha\bar\beta} = \frac{R}{n}h_{\alpha\bar\beta} .
 \end{equation*}
 Since $(M^{2n+1},T^{1,0})$ is locally spherical, it holds that
 \begin{equation*}
  R_{\alpha\bar\beta\gamma\bar\sigma} = \frac{R}{n(n+1)}\left(h_{\alpha\bar\beta}h_{\gamma\bar\sigma} + h_{\alpha\bar\sigma}h_{\gamma\bar\beta} \right) .
 \end{equation*}
 Let $\omega\in\mH^{p,q}$.
 On the one hand, \cref{bochner-order2} implies that
 \begin{multline}
  \label{eqn:Hpq-vanishing-low1}
  0 = q\nablab^\ast\nablab\omega + (n-q)\nablabbar^\ast\nablabbar\omega \\
   - \frac{n}{n-p-q+1}\db\db^\ast\omega + \frac{pq + (n+1)q(n-q)}{n(n+1)}R\omega .
 \end{multline}
 On the other hand, \cref{bochner-order2-better-sign} implies that
 \begin{multline}
  \label{eqn:Hpq-vanishing-low2}
  0 = (q-1)\nablab^\ast\nablab\omega + (n-q+1)\nablabbar^\ast\nablabbar\omega \\
   + \frac{n}{n-p-q}\db^\ast\db\omega + \frac{pq + (n+1)(n-p+(q-1)(n-q))}{n(n+1)}R\omega .
 \end{multline}
 Integrating Equation~\eqref{eqn:Hpq-vanishing-low2} on $M$ yields $\nablabbar\omega=\db\omega=R\omega=0$.
 Integrating Equation~\eqref{eqn:Hpq-vanishing-low1} then yields $\nablab\omega=0$.
 In particular, $\omega$ is parallel;
 moreover, if $R\not=0$, then $\omega=0$.
 The final conclusion follows from \cref{hodge-isomorphism}.
\end{proof}

\Cref{Hpq-vanishing} allows us to conclude that most of the Kohn--Rossi cohomology groups of the CR Hopf manifold~\citelist{\cite{Dragomir1994} \cite{BurnsShnider1976} \cite{Jacobowitz1990}*{pg.\ 185}} vanish.

\begin{example}
 \label{cr-hopf-manifold}
 Let $(\bH^n,T^{1,0},\theta)$ be the flat Heisenberg manifold as described in \cref{heisenberg-quotient}.
 Let $\phi\in C^\infty(M;\bC)$ be the function $\phi(z,t) := \lv z\rv^2 - 2it$.
 It is straightforward to check that $\phi$ is a CR function and that $Z_\alpha\phi=2\oz^\alpha$ for each $\alpha\in\{1,\dotsc,n\}$.
 Let $\rho:=\lv\phi\rv^{1/2}$ denote the Heisenberg pseudodistance.
 It is straightforward to check that the dilations
 \[ \delta_\lambda(z,t) := (\lambda z, \lambda^2t) \]
 are CR maps and that $\delta_\lambda^\ast(\rho^{-2}\theta)=\rho^{-2}\theta$.
 In particular, $\htheta:=\rho^{-2}\theta$ descends to a well-defined contact form $\vartheta$ on the \defn{CR Hopf manifold} $(\bH^n\setminus\{0\})/\bZ \cong S^1 \times S^{2n}$, where $\bZ$ denotes the group generated by the dilation $\delta_2$.
 A straightforward computation using \cref{transformation} shows that $\htheta$ is pseudo-Einstein and its pseudohermitian scalar curvature is
 \[ \hR = n(n+1)\lv z\rv^2\rho^{-2} . \]
 We conclude from \cref{Hpq-vanishing} that the CR Hopf manifold is such that $b^{p,q}=0$ for all $p\in\{0,\dotsc,n+1\}$ and $q\in\{1,\dotsc,n-1\}$ such that $p+q \not\in \{ n, n+1 \}$.

 Note that the CR Hopf manifold does not admit a pseudo-Einstein contact form with positive scalar curvature~\cite{ChengMarugameMatveevMontgomery2019}.
\end{example}

%% file: applications/frolicher.tex
\section{The CR Fr\"olicher inequalities}
\label{sec:frolicher}

In this section we construct the Garfield spectral sequence as a CR analogue of the Fr\"olicher spectral sequence~\cite{Frolicher1955} and present some results which indicate that the spaces on the \emph{second page} of the Garfield spectral sequence should be regarded as the CR analogues of the Dolbeault cohomology groups.
Our construction --- which is equivalent to Garfield's construction~\cite{Garfield2001} via \cref{explicit-rumin} --- begins with the natural filtration of $\CmR^k$ adapted to the $\dbbar$-operator.

\begin{definition}
 Let $(M^{2n+1},T^{1,0})$ be a CR manifold and let $p,q\in\bZ$.
 Set
 \begin{equation*}
  F^p\CmR^{p+q} := \bigoplus_{j\geq 0} \mR^{p+j,q-j} ,
 \end{equation*}
 with the convention $\mR^{p,q}=0$ if $p\not\in\{0,\dotsc,n+1\}$ or $q\not\in\{0,\dotsc,n\}$.
\end{definition}

Note that $\{ F^p\CmR^k\}_{p=0}^{k+1}$ is a descending filtration of $\CmR^k$ which is preserved by the exterior derivative;
i.e.\ $F^{p+1}\CmR^{p+q}\subset F^p\CmR^{p+q}$ and $d(F^p\CmR^{p+q})\subset F^p\CmR^{p+q+1}$.
The Garfield spectral sequence then arises from a standard construction (cf.\ \cite{Frolicher1955}):

\begin{definition}
 Let $(M^{2n+1},T^{1,0})$ be a CR manifold and let $p,q,r\in\bZ$.
 Set
 \begin{align*}
  Z_r^{p,q} & := \left\{ \omega \in F^p\CmR^{p+q} \suchthatcolon d\omega \in F^{p+r}\CmR^{p+q+1} \right\} , \\
  B_r^{p,q} & := F^p\CmR^{p+q} \cap \im\left( d\colon F^{p-r+1}\CmR^{p+q-1} \to \CmR^{p+q} \right) .
 \end{align*}
 The \defn{Garfield spectral sequence} is determined by the sequence $\{E_r^{p,q}\}_{r=0}^\infty$ of pages and the sequence $\{d_r \colon E_r^{p,q} \to E_r^{p+r,q-r+1} \}_{r=0}^\infty$ of differentials, where
 \begin{align*}
  E_r^{p,q} & := \frac{Z_r^{p,q}}{B_r^{p,q} + Z_{r-1}^{p+1,q-1}} , \\
  d_r([\omega]) & := [d\omega] .
 \end{align*}
\end{definition}

It is straightforward to verify that this is a spectral sequence;
i.e.\ the identity map on $Z_{r+1}^{p,q}$ induces an isomorphism
\begin{equation}
 \label{eqn:Er+1}
 E_{r+1}^{p,q} \cong \frac{\ker\left( d_r\colon E_r^{p,q} \to E_r^{p+r,q-r+1} \right)}{\im\left( d_r\colon E_r^{p-r,q+r-1} \to E_r^{p,q} \right)} .
\end{equation}
Realizing the Kohn--Rossi and de Rham cohomology groups in the Garfield spectral sequence requires two lemmas.

Our first lemma identifies the first page of the Garfield spectral sequence with the Kohn--Rossi cohomology groups and computes the differential $d_1$.

\begin{lemma}
 \label{first-page}
 Let $(M^{2n+1},T^{1,0})$ be a CR manifold.
 Given $p,q\in\bN_0$, the projection $\pi^{p,q}\colon F^p\CmR^{p+q}\to \CmR^{p,q}$ induces an isomorphism
 \begin{equation}
  \label{eqn:E1-isomorphism}
  E_1^{p,q} \cong H_R^{p,q}(M) ,
 \end{equation}
 with respect to which
 \begin{equation*}
  d_1([\omega]) = [d^\prime\omega] ,
 \end{equation*}
 where $d^\prime \colon \mR^{p,q} \to \mR^{p+1,q}$ is the operator
 \begin{equation}
  \label{eqn:defn-dprime}
  d^\prime :=
  \begin{cases}
   \db , & \text{if $p+q \not= n$} , \\
   \dhor , & \text{if $p+q = n$} .
  \end{cases}
 \end{equation}
\end{lemma}

\begin{proof}
 It is clear that $\pi^{p,q}$ induces an isometry $E_0^{p,q}\cong\mR^{p,q}$ with respect to which $d_0\omega=\dbbar\omega$.
 Equation~\eqref{eqn:E1-isomorphism} now follows from Equation~\eqref{eqn:Er+1}.
 
 Let $\omega\in\ker\left(\dbbar\colon\mR^{p,q}\to\mR^{p,q+1}\right)$ be a representative of $[\omega]\in H_R^{p,q}(M)$.
 By definition, $d_1\colon H_R^{p,q}(M)\to H_R^{p+1,q}(M)$ maps $[\omega]$ to $[\pi^{p+1,q}d\omega]$.
 The conclusion follows from \cref{defn:dbbar}.
\end{proof}

Our second lemma relates the limiting page of the Garfield spectral sequence with the complexified de Rham cohomology groups.

\begin{lemma}
 \label{spectral-sequence-Einfty}
 Let $(M^{2n+1},T^{1,0})$ be a CR manifold.
 Given $p,q\in\bZ$, set
 \[ F^pH_R^{p+q}(M;\bC) := \left\{ [\omega] \in H_R^{p+q}(M;\bC) \suchthatcolon \omega \in \ker \left( d \colon F^p\CmR^{p+q} \to \CmR^{p+q+1} \right) \right\} . \]
 If $r\geq n+2$, then $d_r=0$ and the identity map induces an isomorphism
 \begin{equation}
  \label{eqn:limiting-isomorphism}
  E_r^{p,q} \cong \frac{F^pH_R^{p+q}(M;\bC)}{F^{p+1}H_R^{p+q}(M;\bC)} .
 \end{equation}
\end{lemma}

\begin{proof}
 It suffices to consider the case $p\in\{0,\dotsc,n+1\}$ and $q\in\{0,\dotsc,n\}$.
 Clearly $F^{p+1}H_R^{p+q}(M;\bC)\subset F^pH_R^{p+q}(M;\bC)$.
 Let $r\geq n+2$.
 Using the facts $F^{p+r}\CmR^{p+q+1}=0$ and $F^{p-r}\CmR^{p+q-1}=\CmR^{p+q-1}$, we deduce that the identity map induces the isomorphism~\eqref{eqn:limiting-isomorphism}.
 Since $E_r^{p+r,q-r+1}=0$, it holds that $d_r=0$.
\end{proof}

We now present our results which establish an analogy between the spaces on the second page of the Garfield spectral sequence and the Dolbeault cohomology groups of a complex manifold.
To that end, we give an alternative characterization of the spaces $E_2^{p,q}$ (cf.\ \cite{Popovici2016}*{Proposition~3.1}).

\begin{proposition}
 \label{e2-from-forms}
 Let $(M^{2n+1},T^{1,0},\theta)$ be a closed, embeddable, strictly pseudoconvex manifold and let $p\in\{0,\dotsc,n+1\}$ and $q\in\{0,\dotsc,n\}$.
 Let $H\colon\mR^{p,q}\to\mH^{p,q}$ denote the $L^2$-orthogonal projection onto $\ker\Box_b$.
 Define
 \begin{align*}
  \cH^{p,q}(M) & := \frac{\ker \left(H \circ \db\right) \cap \ker \dbbar }{\im \dbbar + \im \db\rv_{\ker\dbbar}} , && \text{if $p+q\not\in\{n,n+1\}$} , \\
  \cH^{p,q}(M) & := \frac{\ker \left(H \circ \dhor\right) \cap \ker \dbbar}{\im \dbbar + \im \db\rv_{\ker\dbbar}} , && \text{if $p+q=n$} , \\
  \cH^{p,q}(M) & := \frac{\ker \left(H \circ \db\right) \cap \ker \dbbar}{\im \dbbar + \im \dhor\rv_{\ker\dbbar}} , && \text{if $p+q=n+1$} ,
 \end{align*}
 where all quotients are of subspaces of $\mR^{p,q}$.
 Then the identity map induces an isomorphism $\cH^{p,q}(M) \cong E_2^{p,q}$.
\end{proposition}

\begin{proof}
 We present the proof in the case $p+q=n$.
 The remaining cases are similar.
 
 First note that, by \cref{justification-of-bigraded-complex},
 \begin{equation*}
  \im \dbbar + \im \db\rv_{\ker\dbbar} \subset \ker \left(H \circ \db\right) \cap \ker \dbbar .
 \end{equation*}
 Therefore $\cH^{p,q}(M)$ is meaningful.
 
 Let $\omega \in \ker \left(H\circ\dhor\right) \cap \ker \dbbar$.
 Then $\omega$ determines a class $[\omega]\in H_R^{p,q}(M)$ and
 \begin{equation*}
  d_1([\omega]) = [\dhor\omega] = 0 \in H_R^{p+1,q}(M) .
 \end{equation*}
 We conclude from \cref{first-page} that $\omega$ determines a class $\bigl[[\omega]\bigr]\in E_2^{p,q}$.
 In particular, the identity map induces a homomorphism
 \begin{equation*}
  \ker \left(H\circ\dhor\right) \cap \ker \dbbar \ni \omega \mapsto \bigl[[\omega]\bigr] \in E_2^{p,q} .
 \end{equation*}
 If $\omega \in \im\dbbar$, then $[\omega]=0\in H_R^{p,q}(M)$, and hence $\bigl[[\omega]\bigr]=0$.
 If $\omega \in \im \db\rv_{\ker\dbbar}$, then $[\omega]=d_1([\tau])$ for some $\tau\in\ker\dbbar$.
 Therefore $\bigl[[\omega]\bigr]=0$.
 We conclude that the identity map induces a homomorphism $T\colon\cH^{p,q}(M)\to E_2^{p,q}$.
 
 We conclude by showing that $T$ is a bijection.
 
 Suppose that $\omega \in \ker \left(H\circ\dhor\right) \cap \ker \dbbar$ is such that $\bigl[[\omega]\bigr]=0$.
 Then there is a $\tau \in \ker\dbbar$ such that $[\omega]=d_1([\tau])$ in $H_R^{p,q}(M)$.
 Thus there is a $\xi\in\mR^{p,q-1}$ such that $\omega=\dbbar\xi+\db\tau$.
 We conclude that $T$ is injective.
 
 Conversely, suppose that $\Omega \in E_2^{p,q}$.
 There there is an $\omega \in \ker\dbbar$ such that $\dhor\omega \in \im \dbbar$ and $\Omega=\bigl[[\omega]\bigr]$.
 \Cref{hodge-decomposition} implies that $(H\circ\dhor)\omega=0$.
 Therefore $\Omega\in\im\bigl( T \colon \cH^{p,q}(M) \to E_2^{p,q} \bigr)$.
\end{proof}

Combining \cref{e2-from-forms} with the Hodge decomposition theorem for the Popovici Laplacian has two immediate corollaries.
First, we deduce the Hodge isomorphism $E_2^{p,q}\cong\cmH^{p,q}$.

\begin{corollary}
 \label{popovici-hodge-isomorphism}
 Let $(M^{2n+1},T^{1,0},\theta)$ be a closed, embeddable, strictly pseudoconvex manifold and let $p\in\{0,\dotsc,n+1\}$ and $q\in\{0,\dotsc,n\}$.
 There is a canonical isomorphism
 \[ E_2^{p,q} \cong \cmH^{p,q} . \]
 In particular, $E_2^{p,q}$ is finite-dimensional.
\end{corollary}

\begin{proof}
 Suppose that $p+q=n$.
 \Cref{popovici-hodge-decomposition} implies that
 \begin{equation*}
  \mR^{p,q} \cap \ker \left(H\circ\dhor\right) \cap \ker\dbbar = \cmH^{p,q} \oplus \left( \im\dbbar + \im \db\rv_{\ker\dbbar} \right) .
 \end{equation*}
 We conclude from \cref{e2-from-forms} that $E_2^{p,q}\cong\cmH^{p,q}$.
 
 The proof of the remaining cases is similar.
\end{proof}

Second, we deduce Serre duality for the second page of the Garfield spectral sequence.

\begin{corollary}
 \label{popovici-serre-duality}
 Let $(M^{2n+1},T^{1,0},\theta)$ be a closed, embeddable, strictly pseudoconvex manifold and let $p\in\{0,\dotsc,n+1\}$ and $q\in\{0,\dotsc,n\}$.
 Then
 \[ E_2^{p,q} \cong E_2^{n+1-p,n-q} . \]
\end{corollary}

\begin{proof}
 \Cref{popovici-hodge-star} implies that $\ohodge\colon\cmH^{p,q}\to\cmH^{n+1-p,n-q}$ is an isomorphism.
 The conclusion follows from \cref{popovici-hodge-isomorphism}.
\end{proof}

We emphasize again that \cref{popovici-hodge-isomorphism} implies that the spaces $E_2^{p,q}$ are finite-dimensional.
In particular, this yields a meaningful CR analogue of the Fr\"olicher inequality~\cite{Frolicher1955}.
More can be said for Sasakian manifolds.

\begin{definition}
 A \defn{Sasakian manifold} is a strictly pseudoconvex CR manifold $(M^{2n+1},T^{1,0})$ which admits a torsion-free contact form.
 
 A \defn{Sasaki manifold} $(M^{2n+1},T^{1,0},\theta)$ is a Sasakian manifold together with a choice of torsion-free contact form.
\end{definition}

We make the distinction between Sasakian and Sasaki manifolds to emphasize which results do not depend on the choice of torsion-free contact form.

The final result of this section is that Sasakian manifolds realize equality in the CR analogue of the Fr\"olicher inequality, and moreover, the complexified $k$-th de Rham cohomology group is isometric to the direct sum of the spaces $E_2^{p,q}$ with $p+q=k$ (cf.\ \cite{Huybrechts2005}*{Corollary~3.2.12}).
This provides another instance of Sasakian manifolds as odd-dimensional analogues of K\"ahler manifolds.

\begin{theorem}
 \label{cr-frolicher}
 Let $(M^{2n+1},T^{1,0})$ be a closed, embeddable, strictly pseudoconvex manifold.
 Then $\dim E_2^{p,q}<\infty$ for all $p\in\{0,\dotsc,n+1\}$ and all $q\in\{0,\dotsc,n\}$, and moreover,
 \begin{equation}
  \label{eqn:cr-frolicher}
  \dim H^k(M;\bC) \leq \sum_{p+q=k} \dim E_2^{p,q}
 \end{equation}
 for all $k\in\{0,\dotsc,2n+1\}$.
 Furthermore, if $(M^{2n+1},T^{1,0})$ is Sasakian, then
 \begin{equation}
  \label{eqn:sasakian-hodge-decomposition}
  H^k(M;\bC) \cong \bigoplus_{p+q=k} E_2^{p,q}
 \end{equation}
 and~\eqref{eqn:cr-frolicher} holds with equality for all $k\in\{0,\dotsc,2n+1\}$.
\end{theorem}

\begin{proof}
 \Cref{popovici-hodge-isomorphism} implies that $\dim E_2^{p,q}<\infty$ for all $p\in\{0,\dotsc,n+1\}$ and all $q\in\{0,\dotsc,n\}$.
 The definition of the Garfield spectral sequence implies that $E_\infty^{p,q}\subset E_2^{p,q}$.
 We conclude from \cref{spectral-sequence-Einfty} that
 \[ H^k(M;\bC) \cong \bigoplus_{p+q}^k E_\infty^{p,q} . \]
 Inequality~\eqref{eqn:cr-frolicher} readily follows.
 
 Suppose now that $\theta$ is a torsion-free contact form on $(M^{2n+1},T^{1,0})$.
 Clearly
 \begin{equation*}
  \ker \bigl( \cBoxb \colon \mR^{p,q} \to \mR^{p,q} \bigr) \supseteq \left\{ \omega \in \mR^{p,q} \suchthatcolon d\omega = d^\ast\omega = 0 \right\} .
 \end{equation*}
 Now let $\omega \in \ker \bigl( \cBoxb \colon \mR^{p,q} \to \mR^{p,q} \bigr)$.
 \Cref{justification-of-bigraded-complex,nablasast-complex} imply that $\dbbar d^\prime\omega = 0$ and $\dbbar^\ast (d^\prime)^\ast \omega = 0$, respectively, where $d^\prime$ is the operator~\eqref{eqn:defn-dprime}.
 
 If $p+q \not\in \{ n, n+1 \}$, then \cref{dbdbbarstar} implies that $\dbbar^\ast d^\prime\omega = 0$ and $\dbbar (d^\prime)^\ast\omega = 0$.
 We conclude from \cref{popovici-kernel} that $d\omega = d^\ast\omega = 0$.
 
 If instead $p+q = n$, then Equations~\eqref{eqn:n-hodge-dbbar} and~\eqref{eqn:n-hodge-db} imply that $\dbbar (d^\prime)^\ast \omega = 0$ and $\db\omega = 0$, respectively.
 We then deduce from \cref{dhordbbarstar} that $\dbbar^\ast d^\prime = 0$.
 We conclude from \cref{popovici-kernel} that $d\omega = d^\ast\omega = 0$.
 
 If instead $p+q = n+1$, applying the previous argument to $\ohodge\omega$ using \cref{popovici-hodge-star} again yields $d\omega = d^\ast\omega = 0$.
 
 We conclude from the previous four paragraphs that
 \begin{equation}
  \label{eqn:sasakian-kernel-popovici}
  \ker \left( \cBoxb \colon \mR^{p,q} \to \mR^{p,q} \right) = \left\{ \omega \in \mR^{p,q} \suchthatcolon d\omega = d^\ast\omega = 0 \right\} .
 \end{equation}
 Hence, by \cref{popovici-hodge-isomorphism}, the identity map induces an isomorphism
 \begin{equation}
  \label{eqn:e2-canonical}
  \left\{ \omega \in \mR^{p,q} \suchthatcolon d\omega = d^\ast\omega = 0 \right\} \cong E_2^{p,q} .
 \end{equation}
 
 Finally, \cref{kohn-laplacian-to-rumin-laplacian} implies that $\Delta_b(\mR^{p,q})\subset\mR^{p,q}$.
 In particular, if $\omega \in \mH^k$, then $\pi^{p,q}\omega \in \ker\Delta_b \cap \mR^{p,q}$ for all integers $p,q$ such that $p+q=k$.
 Combining these two observations with \cref{rumin-hodge-isomorphism} yields
 \begin{equation*}
  H^k(M;\bC) \cong \mH^k \cong \bigoplus_{p+q=k} E_2^{p,q} .
 \end{equation*}
 It immediately follows that~\eqref{eqn:cr-frolicher} holds with equality.
\end{proof}

In particular, we recover a topological obstruction~\cites{BlairGoldberg1967,Fujitani1966} to the existence of a Sasakian structure on a given manifold in terms of its Betti numbers.

\begin{corollary}
 \label{sasakian}
 Let $(M^{2n+1},T^{1,0})$ be a closed Sasakian manifold and let $k\in\{0,\dotsc,2n+1\}$.
 \begin{enumerate}
  \item If $k\leq n$ is odd, then $\dim H^k(M;\bC)$ is even.
  \item If $k\geq n+1$ is even, then $\dim H^k(M;\bC)$ is even.
 \end{enumerate}
\end{corollary}

\begin{proof}
 Note that $(M^{2n+1},T^{1,0})$ is embeddable~\citelist{ \cite{Boutet1975}*{Th\'eor\`eme} \cite{MarinescuYeganefar2007}*{Theorem~1.4}}.
 
 By Equation~\eqref{eqn:e2-canonical} and conjugation, if $p+q\leq n$, then $E_2^{p,q}\cong E_2^{q,p}$; and if $p+q\geq n+1$, then $E_2^{p,q}\cong E_2^{q+1,p-1}$.
 The conclusion readily follows from Equation~\eqref{eqn:sasakian-hodge-decomposition}.
\end{proof}

%% file: applications/sasakian.tex
\section{Additional consequences for Sasakian manifolds}
\label{sec:sasakian}

\Cref{cr-frolicher} shows that the bigraded Rumin complex is well-suited to studying topological properties of closed Sasakian manifolds.
In this section we prove a selection of such results which further support this assertion.

It is well-known~\cite{Ballmann2006}*{Corollary~5.65} that if $X$ is a closed K\"ahler manifold with positive Ricci curvature, then $H^{p,0}(X)=0$ for all $p\geq1$.
As a further indication of the analogy between the second page of the Garfield spectral sequence and the Dolbeault cohomology groups, we prove a CR analogue of this fact.
This result is similar in spirit to a vanishing result of Nozawa~\cite{Nozawa2014}*{Theorem~1.2} for the basic Dolbeault cohomology groups of a Sasaki manifold.

\begin{theorem}
 \label{sasakian-Hp0-vanishing}
 Let $(M^{2n+1},T^{1,0},\theta)$ be a closed Sasaki manifold with nonnegative pseudohermitian Ricci curvature $R_{\alpha\bar\beta}$.
 Then
 \begin{align*}
  \dim E_2^{p,0} & \leq \binom{n}{p} , && \text{if $p \leq n$} , \\
  \dim E_2^{n+1,0} & \leq 1 .
 \end{align*}
 Moreover, if there is a point at which $R_{\alpha\bar\beta} > 0$, then $E_2^{p,0} = \{ 0 \}$ for all $p \in \{ 0, \dotsc , n+1 \}$.
\end{theorem}

\begin{proof}
 Combining \cref{popovici-hodge-star,popovici-hodge-isomorphism} with Equation~\eqref{eqn:sasakian-kernel-popovici} yields $E_2^{n,0} \cong E_2^{n+1,0}$.
 Thus it suffices to prove the cases $p \leq n$.
 
 Let $\omega \in \cmH^{p,0}$.
 Since $\theta$ is torsion-free, Equation~\eqref{eqn:sasakian-kernel-popovici} implies that $d\omega = d^\ast\omega = 0$.

 Suppose first that $p \leq n-1$.
 \Cref{kohn-laplacian-to-rumin-laplacian,rumin-order2} imply that
 \begin{equation*}
  0 = \frac{n-p}{n}\nablab^\ast\nablab\omega + \frac{n+p}{n}\nablabbar^\ast\nablabbar\omega - \frac{n-p}{n}\Ric \hash \omega .
 \end{equation*}
 Suppose last that $p=n$.
 Write $\omega = \frac{1}{n!}\omega_\Alpha\,\theta^\Alpha$.
 \Cref{sR-to-P} and Equation~\eqref{eqn:critical-lefschetz-consequence} imply that $\nabla_0\omega_\Alpha = 0$ and $\nabla_{\bar\beta}\omega_\Alpha = 0$, respectively.
 Computing using \cref{commutators} then yields
 \begin{equation*}
  \nabla^{\mu}\nabla_{\mu}\lv\omega_\Alpha\rv^2 = \lv\nabla_\mu\omega_\Alpha\rv^2 + nR_\gamma{}^\mu \omega_{\mu\Alpha^\prime}\oomega^{\gamma\Alpha^\prime} .
 \end{equation*}
 In both cases, integrating over $M$ implies that $\omega$ is parallel and $\Ric \hash \omega = 0$.
 The conclusion readily follows.
\end{proof}

We now turn to special properties of the de Rham cohomology algebra of a closed Sasakian manifold.
These results rely on two lemmas.

First, we show that on closed Sasaki manifolds, if $\omega \in \im d \cap \mR^k$, $k \leq n$, then the $L^2$-orthogonal projection of $\theta \wedge \omega \wedge d\theta^{n-k}$ onto $\ker\Delta_b$ is zero.

\begin{lemma}
 \label{dbc-vanishing}
 Let $(M^{2n+1},T^{1,0},\theta)$ be a closed Sasaki manifold.
 Let $\omega \in \mR^{k-1}$, $k \leq n$.
 Then $\theta \wedge d\omega \wedge d\theta^{n-k}$ is $L^2$-orthogonal to $\ker\Delta_b$.
\end{lemma}

\begin{proof}
 Suppose first that $\omega \in \mR^{p,q}$, $p+q \leq n-1$.
 \Cref{hodge-star-sE} implies that
 \begin{equation*}
  \frac{1}{(n-p-q-1)!}\theta \wedge d\omega \wedge d\theta^{n-p-q-1} = (-1)^{\frac{(p+q+1)(p+q+2)}{2}}i^{p-q} \hodge (i\db\omega - i\dbbar\omega) .
 \end{equation*}
 In particular, $\theta \wedge d\omega \wedge d\theta^{n-p-q-1} \in \im \db^\ast + \im \dbbar^\ast$.
 
 Now let $\omega \in \mR^{k-1}$, $k \leq n$.
 It follows from linearity and the previous paragraph that $\tau := \theta \wedge d\omega \wedge d\theta^{n-k} \in \im \db^\ast + \im \dbbar^\ast$.
 \Cref{kohn-laplacian-to-rumin-laplacian} then implies that $\tau$ is $L^2$-orthogonal to $\ker\Delta_b$.
\end{proof}

Second, we show that if $\omega$ is a $d$-closed element of $\Omega^{p,q}M$, $p+q \geq n+1$, then $\pi\omega$ is $d$-exact.

\begin{lemma}
 \label{basic-vanishing}
 Let $(M^{2n+1},T^{1,0},\theta)$ be a closed Sasaki manifold.
 Let $\omega \in \Omega^{p,q}M$, $p+q \geq n+1$, be such that $d\omega = 0$.
 Then $\pi\omega \in \im \bigl( d \colon \CmR^{p+q-1} \to \CmR^{p+q} \bigr)$.
\end{lemma}

\begin{proof}
 Since $d\omega = 0$, we conclude from \cref{projection-and-d-commute} that $d\pi\omega = 0$.
 Combining the assumption $p+q \geq n+1$ with the Lefschetz decomposition~\cite{Huybrechts2005}*{Proposition~1.2.30} yields a $\xi \in \Omega^{n-q,n-p}M$ such that
 \begin{equation}
  \label{eqn:basic-vanishing-xi}
  \omega = \xi \wedge d\theta^{p+q-n} .
 \end{equation}
 The first consequence of Equation~\eqref{eqn:basic-vanishing-xi} is that
 \begin{equation*}
  0 = d\omega = d\xi \wedge d\theta^{p+q-n} .
 \end{equation*}
 Therefore $d\xi \in \CmR^{2n-p-q+1}$.
 \Cref{projection-and-d-commute} then implies that $d\xi = d\pi \xi$.
 The second consequence of Equation~\eqref{eqn:basic-vanishing-xi} is that $\Gamma\omega = \theta \wedge \xi \wedge d\theta^{p+q-n-1}$, where $\Gamma$ is as in \cref{inverse-lefschetz}.
 Therefore
 \begin{equation*}
  \pi\omega = \theta \wedge d\xi \wedge d\theta^{p+q-n-1} .
 \end{equation*}
 \Cref{rumin-hodge-decomposition,dbc-vanishing} now imply that $\pi\omega \in \im d$.
\end{proof}

Our first application of \cref{dbc-vanishing,basic-vanishing} is a new proof that if $(M^{2n+1},T^{1,0})$ is a closed Sasakian manifold and if $J \in \bN^\ell$ is a multi-index such that $\lv J \rv \geq n+1$ and $j_k \leq n$ for all $k \in \{ 1 , \dotsc, \ell\}$, then the cup product
\begin{equation}
 \label{eqn:bungart-cup}
 \cup \, \colon \bigotimes_{j=1}^\ell H^{j_\ell}(M;\bC) \to H^{\lv J\rv}(M;\bC)
\end{equation}
vanishes.
This fact was first observed by Bungart~\cite{Bungart1992} using a Hodge theorem of Tanaka~\cite{Tanaka1975}*{Theorem~13.1}.
Note that this conclusion is not true in the more general setting of closed, strictly pseudconvex manifolds~\cite{Bungart1992}*{Section~3}, though it is true~\cite{Bungart1992}*{Section~4} that the cup product~\eqref{eqn:bungart-cup} vanishes if $\lv J\rv \geq n+2$ and $j_k \leq n-1$ for all $k \in \{ 1, \dotsc, \ell \}$.

\begin{theorem}
 \label{cup-vanishing}
 Let $(M^{2n+1},T^{1,0})$ be a closed Sasakian manifold.
 If $J \in \bN^\ell$ is such that $\lv J\rv \geq n+1$ and $j_k \leq n$ for all $k \in \{ 1 , \dotsc, \ell \}$, then the cup product~\eqref{eqn:bungart-cup} vanishes.
\end{theorem}

\begin{proof}
 It suffices to show that if $[\omega] \in H_R^k(M;\bC)$ and $[\tau] \in H_R^\ell(M;\bC)$ are such that $k,\ell \leq n$ and $k+\ell \geq n+1$, then $\left[ \omega \right] \cup \left[ \tau \right] = 0$.
 Let $\theta$ be a torsion-free contact form on $(M^{2n+1},T^{1,0})$.
 \Cref{rumin-hodge-decomposition} implies that we may choose $\omega \in [\omega]$ and $\tau \in [\tau]$ to be $d$-harmonic.
 \Cref{kohn-laplacian-to-rumin-laplacian} implies that $\pi^{p,q}\omega$ and $\pi^{r,s}\tau$ are $d$-harmonic for all $p+q=k$ and $r+s=\ell$.
 By linearity, it suffices to consider the case $\omega \in \ker\Delta_b \cap \mR^{p,q}$ and $\tau \in \ker\Delta_b \cap \mR^{r,s}$.
 
 Since $k,\ell \leq n$, \cref{projection-to-E-expression,divergence-formula} imply that $\omega \in \Omega^{p,q}M$ and $\tau \in \Omega^{r,s}M$.
 Therefore $d(\omega \wedge \tau) = 0$ and $\omega \wedge \tau \in \Omega^{p+r,q+s}M$.
 The conclusion follows from \cref{rk:easy-recover-derham,basic-vanishing}.
\end{proof}

Recall that the \defn{cuplength} of a manifold $M$, denoted $\cuplength(M)$, is the largest integer $\ell$ for which there is a set $\{ [\omega_j] \}_{j=1}^\ell$ such that $[\omega_j] \in H^{k_j}(M;\bC)$, $k_j \geq 1$, and $\bigcup_j [\omega_j] \not= 0$.
\Cref{cup-vanishing} immediately gives a sharp upper bound on the cuplength of a closed Sasakian manifold.

\begin{corollary}
 \label{cup-length}
 Let $(M^{2n+1},T^{1,0})$ be a closed Sasakian manifold.
 Then
 \begin{equation}
  \label{eqn:cuplength}
  \cuplength(M) \leq n+1 .
 \end{equation}
 Moreover, the estimate is sharp.
\end{corollary}

\begin{remark}
 Rukimbira~\cite{Rukimbira1993} and Itoh~\cite{Itoh1997} previously proved the weaker estimate $\cuplength(M) \leq 2n$ under the assumptions of \cref{cup-length}.
 The upper bound~\eqref{eqn:cuplength} is attained by \cref{heisenberg-quotient}, which can be realized as a principal $S^1$-bundle over a complex torus~\cite{Folland2004}.
 This corrects an erroneous formula for $\cuplength(I^{2n+1})$ in Boyer and Galicki's book~\cite{BoyerGalicki2008}*{Example~8.1.13}.
\end{remark}

\begin{proof}
 Let $\{ [\omega_j] \}_{j=1}^\ell$ be such that $[\omega_j] \in H^{k_j}(M;\bC)$, $k_j \geq 1$, for all $j \in \{ 1, \dotsc, \ell \}$ and $\Omega := \bigcup_j [\omega_j] \not= 0$.
 Set $K := (k_1, \dotsc, k_j)$.
 Since $\Omega \not= 0$, it holds that $\lv K \rv \leq 2n+1$.
 By rearranging $K$ if necessary, we may assume that $k_j \leq n$ for $1 \leq j \leq \ell - 1$.
 Since $\bigcup_{j<\ell} [\omega_j] \not= 0$, we conclude from \cref{cup-vanishing} that
 \begin{equation*}
  \ell - 1 \leq \sum_{j=1}^{\ell-1} k_j \leq n .
 \end{equation*}
 Therefore $\cuplength(M) \leq n+1$.
 
 Now let $(I^{2n+1},T^{1,0})$ be as in \cref{heisenberg-quotient}.
 Then
 \begin{equation*}
  \bigl[ dz^1 \bigr] \cup \dotsm \cup \bigl[ dz^n \bigr] \cup \bigl[ \theta \wedge d\oz^{1} \wedge \dotsm \wedge d\oz^{n} \bigr] \not= 0 .
 \end{equation*}
 In particular, Inequality~\eqref{eqn:cuplength} is sharp.
\end{proof}

Our second application of \cref{dbc-vanishing,basic-vanishing} is a new vanishing result for the real Chern classes of a closed Sasakian manifold.
Given a multi-index $K \in \bN^\ell$, define the real Chern class $c_K(T^{1,0}) \in H^{2\lv K\rv}(M;\bR)$ by
\begin{equation*}
 c_K(T^{1,0}) := \prod_{j=1}^\ell \chern_{k_j}(T^{1,0}) ,
\end{equation*}
where $\chern_{k_j}(T^{1,0}) \in H^{2k_j}(M;\bR)$ is the real Chern character of degree $2k_j$.
It follows from \cref{basic-vanishing} that $c_K(T^{1,0})=0$ if $2\lv K\rv \geq n+1$:

\begin{theorem}
 \label{sasakian-chern-vanishing}
 Let $(M^{2n+1},T^{1,0})$ be a closed Sasakian manifold.
 If $K \in \bN^\ell$ is such that $2\lv K\rv \geq n+1$, then $c_K(T^{1,0})=0$ in $H^{2\lv K\rv}(M;\bR)$.
\end{theorem}

\begin{remark}
 While we cannot find the statement of \cref{sasakian-chern-vanishing} in the literature, it follows readily from results~\cite{BoyerGalicki2008}*{Theorem~7.2.9 and Lemma~7.5.22} in Boyer and Galicki's book.
\end{remark}

\begin{remark}
 The conclusion of \cref{sasakian-chern-vanishing} is false on closed, strictly pseudoconvex CR manifolds~\cite{Takeuchi2018}*{Proposition~4.2}.
 Instead, one only knows~\cite{Takeuchi2018}*{Theorem~1.1} that $c_K(T^{1,0})=0$ in $H^{2\lv K\rv}(M;\bR)$ if $2\lv K\rv \geq n+2$.
\end{remark}

\begin{proof}
 Let $\theta$ be a torsion-free contact form on $(M^{2n+1},T^{1,0})$.
 Given an admissible coframe $\{ \theta^\alpha \}_{\alpha=1}^n$, Equation~\eqref{eqn:curvature-form} gives
 \begin{equation}
  \label{eqn:sasakian-curvature-form}
  \Omega_\alpha{}^\gamma = R_\alpha{}^\gamma{}_{\mu\bar\nu} \, \theta^\mu \wedge \theta^{\bar\nu} .
 \end{equation}
 Equation~\eqref{eqn:bianchi} implies that
 \begin{equation*}
  \chi^{(K)} := \left(\frac{i}{2\pi}\right)^{\lv K\rv} \prod_{j=1}^\ell \Omega_{\mu_1}{}^{\mu_2} \Omega_{\mu_2}{}^{\mu_3} \dotsm \Omega_{\mu_{k_j}}{}^{\mu_1} .
 \end{equation*}
 is in $\ker d$.
 By definition, $c_K(T^{1,0}) = [\chi^{(K)}] \in H^{2\lv K\rv}(M;\bR)$.
 It follows from Equation~\eqref{eqn:sasakian-curvature-form} that $\chi^{(K)} \in \Omega^{\lv K\rv, \lv K \rv}$.
 Since $2\lv K\rv \geq n+1$, we conclude from \cref{basic-vanishing} that $\pi \chi^{(K)} \in \im d$.
 The conclusion follows from Equation~\eqref{eqn:projection-to-R}.
\end{proof}

Our third application of \cref{dbc-vanishing,basic-vanishing} is to a Sasakian analogue of the Hard Lefschetz Theorem.
Recall that in complex geometry, the Hard Lefschetz Theorem~\cite{Ballmann2006}*{Theorem~5.37} states that if $(X^{n},\omega)$ is a K\"ahler $n$-fold, then
\begin{equation}
 \label{eqn:kahler-lefschetz}
 H^{n-k}(X;\bR) \ni [\alpha] \mapsto [ \alpha \wedge \omega^k ] \in H^{n+k}(X;\bR)
\end{equation}
is an isomorphism.
It is clear that~\eqref{eqn:kahler-lefschetz} depends only on the symplectic structure $\omega$ on $X$, and hence the Hard Lefschetz Theorem gives an obstruction to a given symplectic manifold $(X,\omega)$ admitting an integrable complex structure with respect to which $\omega$ is K\"ahler.

On Sasakian $(2n+1)$-manifolds, it is not true that if $\alpha \in \ker d \cap \mR^{n-k}$, then $\theta \wedge \alpha \wedge d\theta^k \in \ker d$.
Instead, one defines an analogue of~\eqref{eqn:kahler-lefschetz} by choosing an appropriate representative of $[\alpha] \in H_R^{n-k}(M;\bR)$.
This is done using \cref{rumin-hodge-decomposition}, yielding a new proof of the Hard Lefschetz Theorem for Sasakian manifolds~\cite{CappellettiMontanoDeNicolaYudin2015}.
The key point is that on Sasakian manifolds, $\theta \wedge d\theta^k$ maps $d$-harmonic $(n-k)$-forms to $d$-harmonic $(n+1+k)$-forms and, moreover, \cref{dbc-vanishing} implies that the de Rham class of the image is independent of the choice of CR structure on $(M^{2n+1},\theta)$ with respect to which $\theta$ is torsion-free.

\begin{theorem}
 \label{hard-lefschetz}
 Let $(M^{2n+1},T^{1,0},\theta)$ be a Sasaki manifold.
 If $\omega \in \mH^k$, $k \leq n$, then $\theta \wedge \omega \wedge d\theta^{n-k} \in \mH^{2n+1-k}$.
 Moreover, the induced map $\Lef \colon H_R^k(M;\bR) \to H_R^{2n+1-k}(M;\bR)$,
 \begin{equation*}
  \Lef([\omega]) := [ \theta \wedge H\omega \wedge d\theta^{n-k} ] ,
 \end{equation*}
 is an isomorphism which depends only on $(M^{2n+1},\theta)$.
\end{theorem}

\begin{proof}
 Let $\omega \in \mH^k$, $k \leq n$, and let $p,q \in \bN_0$ be such that $p+q = k$.
 \Cref{kohn-laplacian-to-rumin-laplacian} implies that $\pi^{p,q}\omega \in \mH^k$.
 \Cref{hodge-star-sE} implies that $\hodge\pi^{p,q}\omega \in \mR^{n+1-q,n-p}$ is a nonzero multiple of $\theta \wedge \pi^{p,q}\omega \wedge d\theta^{n-k}$.
 We deduce from \cref{rumin-hodge-star} that $\theta \wedge \pi^{p,q}\omega \wedge d\theta^{n-k} \in \mH^{2n+1-k}$ is nonzero.
 Hence, by linearity, $\Lef$ is well-defined and injective.
 A similar argument shows that $\Lef$ is surjective.
 
 Suppose now that $(M^{2n+1},\hT^{1,0},\theta)$ is another Sasaki structure on $M$ with the same contact form $\theta$.
 Note that the spaces $\mR^j$, $j \in \{ 0 , \dotsc , 2n+1 \}$, are independent of the choice of CR structure.
 \Cref{rumin-hodge-decomposition} implies that there is a $\tau \in \mR^{k-1}$ such that $\omega + d\tau \in \hmH^k$, where $\hmH^k$ is the space of $d$-harmonic $k$-forms with respect to $(M^{2n+1},\hT^{1,0},\theta)$.
 It follows from the previous paragraph that $\theta \wedge \omega \wedge d\theta^{n-k}$ and $\theta \wedge (\omega + d\tau) \wedge d\theta^{n-k}$ are both $d$-closed.
 Therefore $\theta \wedge d\tau \wedge d\theta^{n-k}$ is $d$-closed.
 We conclude from \cref{dbc-vanishing} that $\theta \wedge d\tau \wedge d\theta^{n-k}$ is $d$-exact, and hence $\Lef = \widehat{\Lef}$.
\end{proof}

Recall that the $\partial\overline{\partial}$-lemma~\cite{Huybrechts2005}*{Corollary~3.2.10} implies that if $X$ is a closed K\"ahler manifold and if $\omega$ is a $d$-closed $(p,q)$-form, then $\omega \in \im d$ if and only if $\omega \in \im \partial\overline{\partial}$.
Sasakian analogues of the $\partial\overline{\partial}$-lemma for basic forms are known~\cite{Tievsky2008}*{Proposition~3.7}.
We conclude this section by discussing two Sasakian analogues --- chosen for their relevance to the Lee conjecture~\cite{Lee1988} as discussed in \cref{sec:lee-conjecture} --- of the $\partial\overline{\partial}$-lemma in terms of the bigraded Rumin complex.

To motivate our results, we give two reasons why the direct analogue of the $\partial\overline{\partial}$-lemma fails on Sasakian manifolds.
First, the existence of nonconstant CR functions on closed Sasakian manifolds implies that the direct analogue of the $\partial\overline{\partial}$-lemma cannot hold for $(1,0)$-forms.
More generally:

\begin{example}
 \label{ex:no-sasakian-dee-deebar}
 Let $\iota \colon S^{2n+1} \hookrightarrow \bC^{n+1}$ be the unit sphere with its standard CR structure $T^{1,0} := \bC TS^{2n+1} \cap T^{1,0}\bC^{n+1}$.
 On the one hand, $(S^{2n+1},T^{1,0})$ is Sasakian~\cite{BoyerGalicki2008}*{Example~7.1.5}.
 On the other hand, for each $p \in \{ 0, \dotsc, n \}$, it holds that
 \begin{equation*}
  \omega := \iota^\ast \left( dz^0 \wedge dz^1 \wedge \dotsm \wedge dz^p \right) = d \iota^\ast \left( z^0 \, dz^1 \wedge \dotsm \wedge dz^p \right)
 \end{equation*}
 is a nonzero element of $\im d \cap \mR^{p+1,0}$.
\end{example}

Second, it is not always the case that $\db\dbbar$-exact forms are $d$-closed.

\begin{example}
 \label{ex:paneitz}
 Let $(M^3,T^{1,0})$ be a CR manifold.
 On $\mR^{0,0}$, the operator $-d\db\dbbar$ coincides with~\cites{CaseYang2020,GrahamLee1988,Takeuchi2019} the CR Paneitz operator.
 In particular, $d\db\dbbar\not=0$.
\end{example}

\Cref{ex:no-sasakian-dee-deebar,ex:paneitz} imply that any Sasakian analogue of the $\partial\overline{\partial}$-lemma requires additional assumptions.
Here we impose additional assumptions involving the adjoints of the operators in the bigraded Rumin complex.
While the resulting statements are not CR invariant, they are sufficient for our purposes.

The key ingredient in our Sasakian analogues of the $\partial\overline{\partial}$-lemma is the following consequence of the Hodge Decomposition Theorem for the Popovici Laplacian.

\begin{lemma}
 \label{popovici-hodge-application}
 Let $(M^{2n+1},T^{1,0},\theta)$ be a closed Sasaki manifold.
 Let $\omega \in \mR^{p,q}$ be such that $\dbbar\omega = 0$.
 Then there are $\alpha \in \mR^{p,q-1}$ and $\beta \in \ker\dbbar\cap \mR^{p-1,q}$ and $\xi \in \ker\Box_b \cap \mR^{p+1,q}$ such that
 \begin{equation}
  \label{eqn:popovici-hodge-application}
  \omega = \cH\omega + \dbbar\alpha + d^\prime\beta + (d^\prime)^\ast\xi ,
 \end{equation}
 where $\cH$ is the $L^2$-orthogonal projection onto $\ker\cBoxb$ and $d^\prime$ is the operator~\eqref{eqn:defn-dprime}.
 In particular, if $p+q \leq n$, then
 \begin{equation}
  \label{eqn:popovici-hodge-application-d}
  d\omega = d\dbbar\alpha + d^\prime (d^\prime)^\ast \xi .
 \end{equation}
 Finally, if $d^\prime\omega \in \im \dbbar$, then we can take $\xi = 0$.
\end{lemma}

\begin{proof}
 Note that $(M^{2n+1},T^{1,0})$ is embeddable~\citelist{ \cite{Boutet1975}*{Th\'eor\`eme} \cite{MarinescuYeganefar2007}*{Theorem~1.4}}.

 \Cref{popovici-hodge-decomposition} implies that there are $\alpha \in \mR^{p,q-1}$ and $\beta \in \ker\dbbar\cap \mR^{p-1,q}$ and $\rho \in \mR^{p,q+1}$ and $\xi \in \ker\Box_b \cap \mR^{p+1,q}$ such that
 \begin{equation}
  \label{eqn:apply-popovici}
  \omega = \cH\omega + \dbbar\alpha + d^\prime\beta + \dbbar^\ast\rho + (d^\prime)^\ast\xi .
 \end{equation}
 Since $\dbbar\beta = 0$, \cref{justification-of-bigraded-complex} implies that $\dbbar d^\prime\beta = 0$.
 Since $\Box_b\xi = 0$, combining \cref{dbdbbarstar,dhordbbarstar} with Equations~\eqref{eqn:n-hodge-dbbar} and~\eqref{eqn:n-hodge-db} yields $\dbbar(d^\prime)^\ast\xi = 0$.
 We conclude from Equation~\eqref{eqn:sasakian-kernel-popovici} that
 \begin{equation*}
  0 = \dbbar\omega = \dbbar\dbbar^\ast\rho .
 \end{equation*}
 In particular, $\dbbar^\ast\rho=0$.
 Equation~\eqref{eqn:popovici-hodge-application} now follows from Equation~\eqref{eqn:apply-popovici}.
 
 Suppose now that $p+q \leq n$.
 Let $\alpha \in \mR^{p,q-1}$ and $\beta \in \ker\dbbar \cap \mR^{p-1,q}$ and $\xi \in \ker\Box_b \cap \mR^{p+1,q}$ be such that Equation~\eqref{eqn:popovici-hodge-application} holds.
 \Cref{justification-of-bigraded-complex} implies that $dd^\prime\beta = 0$.
 Equation~\eqref{eqn:sasakian-kernel-popovici} implies that $d\cH\omega = 0$.
 
 If $p+q \leq n-2$, then \cref{dbdbbarstar} implies that $\dbbar\db^\ast\xi = 0$.
 Hence $d(d^\prime)^\ast\xi = d^\prime(d^\prime)^\ast\xi$.
 
 If $p+q = n-1$, then Equation~\eqref{eqn:n-hodge-dbbar} implies that $\dbbar\db^\ast\xi=0$.
 Hence $d(d^\prime)^\ast\xi = d^\prime(d^\prime)^\ast\xi$.
 
 If $p+q = n$, then applying Equation~\eqref{eqn:n-hodge-dbbar} to $\hodge\xi$ yields $\db^\ast\xi=0$.
 We deduce from \cref{dual-justification} that $\db^\ast\dhor^\ast\xi = \dbbar^\ast\dhor^\ast\xi = 0$.
 Applying Equations~\eqref{eqn:n-hodge-dbbar} and~\eqref{eqn:n-hodge-db} to $\dhor^\ast\xi$ yield $\dbbar\dhor^\ast\xi = 0$ and $\db\dhor^\ast\xi=0$, respectively.
 Hence $d(d^\prime)^\ast\xi = d^\prime(d^\prime)^\ast\xi$.
 
 Since the previous three cases are exhaustive, we conclude that Equation~\eqref{eqn:popovici-hodge-application-d} holds.
 
 Suppose finally that $d^\prime\omega \in \im\dbbar$.
 \Cref{justification-of-bigraded-complex} implies that $d^\prime\dbbar\alpha \in \im \dbbar$ and, since $\dbbar\beta=0$, that $d^\prime d^\prime \beta \in \im \dbbar$.
 Equation~\eqref{eqn:sasakian-kernel-popovici} implies that $d^\prime\cH\omega = 0$.
 Equation~\eqref{eqn:popovici-hodge-application} then implies that $d^\prime (d^\prime)^\ast\xi \in \im \dbbar$.
 However, combining the assumption $\Box_b\xi = 0$ with \cref{justification-of-bigraded-complex,dhordbbarstar,critical-dbdbbarstar} yields $d^\prime(d^\prime)^\ast\xi \in \ker\Box_b$.
 Therefore $(d^\prime)^\ast\xi = 0$.
\end{proof}

We now present our Sasakian analogues of the $\partial\overline{\partial}$-lemma.
We first consider forms of degree at most $n$, where the conclusion is analogous to its K\"ahler analogue.

\begin{lemma}
 \label{dbdbbar-lemma-low}
 Let $(M^{2n+1},T^{1,0},\theta)$ be a closed Sasaki manifold.
 Let $\omega \in \mR^{p,q}$, $p + q \leq n$, be such that $\omega = d\tau$ for some $\tau \in \CmR^{p+q-1}$.
 If $\db^\ast\omega \in \im \dbbar$ and $\dbbar^\ast\omega \in \im \db$, then
 \begin{equation*}
  \omega \in \im \left( \db\dbbar \colon \mR^{p-1,q-1} \to \mR^{p,q} \right) .
 \end{equation*}
\end{lemma}

\begin{proof}
 Let $\omega \in \mR^{p,q}$, $p+q \leq n$, and let $\tau \in \CmR^{p+q-1}$ be such that $d\tau = \omega$.

 We claim that we may take $\tau \in \mR^{p,q-1} \oplus \mR^{p-1,q}$.
 Denote $\tau_j := \pi^{j-1,p+q-j}\tau$, $j \in \{ 0, \dotsc, p+q+1 \}$.
 Let $j$ be the largest element of $\{ 1, \dotsc , p-1 \}$ such that $\tau_\ell=0$ for all $\ell<j$;
 since $\tau_0=0$, $j$ exists.
 Since $j \leq p - 1$, it holds that
 \begin{align*}
  0 & = \pi^{j-1,p+q-j+1}\omega = \dbbar\tau_j , \\
  0 & = \pi^{j,p+q-j}\omega = \db\tau_j + \dbbar\tau_{j+1} .
 \end{align*}
 Applying \cref{popovici-hodge-application} to $\tau_j$ yields an $\alpha_j \in \mR^{j-1,p+q-j-1}$ such that
 \begin{equation*}
  d\tau_j = d\dbbar\alpha_j = -d\db\alpha_j .
 \end{equation*}
 In particular, $\ctau := \tau - \tau_j - \db\alpha_j$ is such that $d\ctau = \omega$ and $\pi^{j-1,p+q-j}\ctau=0$.
 It follows from induction that, by replacing $\tau$ if necessary, $\tau \in \bigoplus_{j\geq0} \mR^{p+j-1,q-j}$.
 Carrying out the analogous argument using larger values of $j$ and the conjugate of \cref{popovici-hodge-application} yields the claim.

 Now, applying \cref{popovici-hodge-application} to $\pi^{p-1,q}\tau$, the conjugate of \cref{popovici-hodge-application} to $\pi^{p,q-1}\tau$, and \cref{justification-of-bigraded-complex} yields $\alpha \in \mR^{p-1,q-1}$ and $\xi \in \ker\Box_b \cap \mR^{p,q}$ and $\zeta \in \ker\oBox_b \cap \mR^{p,q}$ such that
 \begin{equation}
  \label{eqn:dbdbbar-lemma-low-pre-equation}
  \omega = d\tau = \db\dbbar\alpha + \db\db^\ast\xi + \dbbar\dbbar^\ast\zeta .
 \end{equation}
 \Cref{dbdbbarstar} and the assumption $\dbbar^\ast\omega \in \im \db$ yield $\dbbar^\ast\dbbar\dbbar^\ast\zeta \in \im \db \cap \ker \oBox_b$.
 Therefore $\dbbar^\ast\zeta = 0$.
 Similarly, since $\db^\ast\omega \in \im \dbbar$, it holds that $\db^\ast\xi = 0$.
 The conclusion now follows from Equation~\eqref{eqn:dbdbbar-lemma-low-pre-equation}.
\end{proof}

For the remaining case $p+q = n+1$, set
\begin{equation*}
 \mB^{p,q} := \mR^{p+1,q-1} \oplus \mR^{p,q} .
\end{equation*}
Note that $\overline{\mB^{p,q}} = \mB^{q,p}$ and $d\dbbar(\mR^{p-1,q-1}) \subset \mB^{p,q}$.
Our Sasakian analogue of the $\partial\overline{\partial}$-lemma is analogous to \cref{dbdbbar-lemma-low}, but with $\db\dbbar$ and $\mR^{p,q}$ replaced by $d\dbbar$ and $\mB^{p,q}$, respectively, and with a different assumption on the divergences.
The latter change is motivated by the application to the Lee Conjecture in \cref{sec:lee-conjecture}.

\begin{lemma}
 \label{dbdbbar-lemma-middle}
 Let $(M^{2n+1},T^{1,0},\theta)$ be a closed Sasaki manifold.
 Let $\omega \in \mB^{p,q}$, $p + q = n + 1$, be such that $\omega = d\tau$ for some $\tau \in \CmR^{p+q-1}$.
 If
 \begin{equation*}
  \omega \in \im \bigl( \dbbar^\ast \colon \mR^{p+1,q} \to \mR^{p+1,q-1} \bigr) \oplus \im \bigl( \db^\ast \colon \mR^{p+1,q} \to \mR^{p,q} \bigr) ,
 \end{equation*}
 then
 \begin{equation*}
  \omega \in \im \left( d\dbbar \colon \mR^{p-1,q-1} \to \mB^{p,q} \right) .
 \end{equation*}
\end{lemma}

\begin{proof}
 Let $\omega \in \mB^{p,q}$, $p+q = n + 1$, and let $\tau \in \CmR^{p+q-1}$ be such that $d\tau = \omega$.

 We claim that we may take $\tau \in \mR^{p,q-1} \oplus \mR^{p-1,q}$.
 Denote $\tau_j := \pi^{j-1,p+q-j}\tau$, $j \in \{ 0, \dotsc, p+q+1 \}$.
 Let $j$ be the largest element of $\{ 1, \dotsc , p-1 \}$ such that $\tau_\ell=0$ for all $\ell<j$;
 since $\tau_0=0$, $j$ exists.
 Since $j \leq p - 1$, it holds that
 \begin{align*}
  0 & = \pi^{j-1,p+q-j+1}\omega = \dbbar\tau_j , \\
  0 & = \pi^{j,p+q-j}\omega = \dhor\tau_j + \dbbar\tau_{j+1} .
 \end{align*}
 Applying \cref{popovici-hodge-application} to $\tau_j$ yields an $\alpha_j \in \mR^{j-1,p+q-j-1}$ such that
 \begin{equation*}
  d\tau_j = d\dbbar\alpha_j = -d\db\alpha_j .
 \end{equation*}
 In particular, $\ctau := \tau - \tau_j - \db\alpha_j$ is such that $d\ctau = \omega$ and $\pi^{j-1,p+q-j}\ctau=0$.
 It follows from induction that, by replacing $\tau$ if necessary, $\tau \in \bigoplus_{j\geq0} \mR^{p+j-1,q-j}$.
 Carrying out the analogous argument using larger values of $j$ and the conjugate of \cref{popovici-hodge-application} yields the claim.

 Now, applying \cref{popovici-hodge-application} to $\pi^{p-1,q}\tau$ and the conjugate of \cref{popovici-hodge-application} to $\pi^{p,q-1}\tau$ yields $\alpha \in \mR^{p-1,q-1}$ and $\xi \in \ker\Box_b \cap \mR^{p,q}$ and $\zeta \in \ker\oBox_b \cap \mR^{p+1,q-1}$ such that
 \begin{equation}
  \label{eqn:dbdbbar-lemma-middle-pre-equation}
  \omega = d\tau = d\dbbar\alpha + \dhor\dhor^\ast\xi + \dhor\dhor^\ast\zeta .
 \end{equation}
 Combining the assumption $\pi^{p,q}\omega \in \im\db^\ast$ with \cref{justification-of-bigraded-complex} and Equation~\eqref{eqn:dbdbbar-lemma-middle-pre-equation} yields $\dhor\dhor^\ast\xi \in \im \dbbar + \im \db^\ast$.
 As observed in the proof of \cref{popovici-hodge-application}, it holds that $\db\dhor^\ast\xi = 0$.
 We conclude from \cref{justification-of-bigraded-complex,dual-justification,dhordbbarstar} that $\dhor^\ast\xi = 0$.
 Similarly, since $\pi^{p+1,q-1}\omega \in \im\dbbar^\ast$, it holds that $\dhor^\ast\zeta = 0$.
 The conclusion now follows from Equation~\eqref{eqn:dbdbbar-lemma-middle-pre-equation}.
\end{proof}

%% file: applications/lee-conjecture.tex
\section{Pseudo-Einstein contact forms and the Lee Conjecture}
\label{sec:lee-conjecture}

We conclude this article by discussing some cohomological properties of CR manifolds which admit a pseudo-Einstein contact form.
In particular, we extend to dimension three Lee's conjecture~\cite{Lee1988} of the equivalence of the vanishing of the real first Chern class $c_1(T^{1,0})$ and the existence of a pseudo-Einstein contact form.
We also present Lee's partial results~\cite{Lee1988} towards his conjecture and indicate which results generalize to dimension three.

Historically, pseudo-Einstein contact forms were first defined in dimension at least five~\cite{Lee1988}, and then extended to dimension three~\cites{FeffermanHirachi2003,CaseYang2012,Hirachi2013}.
The main reason for this difference is that the pseudo-Einstein condition is a second-order condition on the contact form in dimension at least five, but a third-order condition in dimension three.
These definitions are succinctly unified as follows:

\begin{definition}
 \label{defn:lee-form}
 Let $(M^{2n+1},T^{1,0},\theta)$ be a pseudohermitian manifold.
 The \defn{Lee form} is
 \[ \ell^\theta := -\frac{1}{n+2}\left( i\,\Omega_\mu{}^\mu - \frac{1}{n}d(R\theta) \right) . \]
 We say that $(M^{2n+1},T^{1,0},\theta)$ is \defn{pseudo-Einstein} if its Lee form is zero.
\end{definition}

Here $\Omega_\mu{}^\mu$ is the trace of the globally-defined curvature two-forms~\eqref{eqn:curvature-forms}, and hence is globally defined.

One main point of \cref{defn:lee-form} is that it is valid in all dimensions and signatures.
Another key point is that the Lee form is a closed element of the space $\mS^1$ of \cref{defn:pluriharmonic-cohomology}.
This can be seen by expressing the Lee form in terms of the pseudohermitian Ricci tensor;
this also recovers the usual definition~\cites{CaseYang2012,Lee1988,Hirachi2013} of a pseudo-Einstein contact form.

\begin{lemma}
 \label{lee-form}
 Let $(M^{2n+1},T^{1,0},\theta)$ be a pseudohermitian manifold.
 Then
 \begin{multline}
  \label{eqn:compute-lee-form}
  \ell^\theta = -\frac{1}{n+2} \biggl[ i\left( R_{\alpha\bar\beta} - \frac{R}{n}h_{\alpha\bar\beta}\right)\,\theta^\alpha \wedge \theta^{\bar\beta} \\
   + \left( \frac{1}{n}\nabla_\alpha R - i\nabla^\mu A_{\mu\alpha} \right)\,\theta\wedge\theta^\alpha + \left( \frac{1}{n}\nabla_{\bar\beta} R + i\nabla^{\bar\nu} A_{\bar\nu\bar\beta}\right)\,\theta\wedge\theta^{\bar\beta} \biggr] .
 \end{multline}
 In particular,
 \begin{enumerate}
  \item $\ell^\theta \in \ker \left( d \colon \mS^1 \to \mS^2 \right)$; and
  \item $(M^{2n+1},T^{1,0},\theta)$ is pseudo-Einstein if and only if
  \begin{equation}
   \label{eqn:pseudo-einstein}
   \begin{cases}
    R_{\alpha\bar\beta} = \frac{R}{n}h_{\alpha\bar\beta}, & \text{if $n\geq2$}, \\
    \nabla_\alpha R = i\nabla^\mu A_{\mu\alpha}, & \text{if $n=1$}.
   \end{cases}
  \end{equation}
 \end{enumerate}
\end{lemma}

\begin{proof}
 A straightforward computation using Equation~\eqref{eqn:curvature-form} yields Equation~\eqref{eqn:compute-lee-form}.
 
 That $d\ell^\theta=0$ follows immediately from Equation~\eqref{eqn:bianchi} and \cref{defn:lee-form}.
 Since $R_{\alpha\bar\beta}-\frac{R}{n}h_{\alpha\bar\beta}$ is trace-free, we deduce that $\theta\wedge\ell^\theta\wedge d\theta^{n-1}=0$.
 Therefore $\ell^\theta\in\mR^2$.
 Equation~\eqref{eqn:compute-lee-form} then implies that $\ell^\theta\in\mS^1$.
 
 It is clear that if $\ell^\theta=0$, then Equation~\eqref{eqn:pseudo-einstein} holds.
 Conversely, suppose that Equation~\eqref{eqn:pseudo-einstein} holds.
 If $n=1$, it trivially holds that $\ell^\theta=0$.
 If $n\geq2$, then $\ell^\theta\equiv0\mod\theta$.
 Since $\ell^\theta\in\mR^{1}$, we conclude from \cref{projection-to-R} that $\ell^\theta=0$.
\end{proof}

\Cref{lee-form} implies that $[\ell^\theta]$ determines an element of $H_R^1(M;\sP)$ for each choice of contact form $\theta$.
In fact, this class is independent of choice of contact form.
This follows by computing how the Lee form transforms under change of contact form.

\begin{lemma}
 \label{lee-form-transformation}
 Let $(M^{2n+1},T^{1,0},\theta)$ be a pseudohermitian manifold.
 Then
 \begin{equation}
  \label{eqn:lee-form-transformation}
  \ell^{e^u\theta} = \ell^\theta + id\dbbar u
 \end{equation}
 for all $u\in C^\infty(M)$.
\end{lemma}

\begin{proof}
 Set $\htheta:=e^u\theta$.
 Denote $E_{\alpha\bar\beta}:=R_{\alpha\bar\beta}-\frac{R}{n}h_{\alpha\bar\beta}$ and $W_\alpha:=\nabla_\alpha R - ni\nabla^\mu A_{\mu\alpha}$.
 
 Suppose first that $n\geq2$.
 Then $\ell^\theta \equiv -\frac{i}{n+2}E_{\alpha\bar\beta}\,\theta^\alpha\wedge\theta^{\bar\beta}\mod\theta$.
 \Cref{commutators,transformation} imply that
 \begin{equation*}
  \hE_{\alpha\bar\beta} = E_{\alpha\bar\beta} - (n+2)U_{\alpha\bar\beta} ,
 \end{equation*}
 where $U_{\alpha\bar\beta} := u_{\bar\beta\alpha} - \frac{1}{n}u_{\bar\nu}{}^{\bar\nu}h_{\alpha\bar\beta}$.
 \Cref{bigraded-operators} now yields Equation~\eqref{eqn:lee-form-transformation}.
 
 Suppose now that $n=1$.
 Then $\ell^\theta = -\frac{1}{3}(W_\alpha\,\theta\wedge\theta^\alpha + W_{\bar\beta}\,\theta\wedge\theta^{\bar\beta})$.
 \Cref{commutators,transformation} imply that
 \begin{equation*}
  e^u\hW_\alpha = W_\alpha - 3u^\gamma{}_{\gamma\alpha} - 3iu^\gamma A_{\gamma\alpha} = W_\alpha - 3u_\gamma{}^\gamma{}_\alpha + 3iu_{\alpha0} .
 \end{equation*}
 We again conclude from \cref{bigraded-operators} that Equation~\eqref{eqn:lee-form-transformation} holds.
\end{proof}

In particular, we recover the fact~\citelist{\cite{Lee1988}*{Proposition~5.1} \cite{CaseYang2012}*{Proposition~3.5}} that the space of pseudo-Einstein contact forms, when nonempty, is parameterized by the space of CR pluriharmonic functions.

\begin{lemma}
 Let $(M^{2n+1},T^{1,0},\theta)$ be a pseudo-Einstein manifold.
 Then $e^u\theta$ is pseudo-Einstein if and only if $u \in \mP$.
\end{lemma}

\begin{proof}
 This follows immediately from \cref{defn:lee-form,lee-form-transformation}.
\end{proof}

\Cref{lee-form-transformation} also implies that $[\ell^\theta]\in H_R^1(M;\sP)$ is independent of the choice of contact form.
This justifies the following definition:

\begin{definition}
 \label{defn:lee-class}
 Let $(M^{2n+1},T^{1,0})$ be an orientable CR manifold.
 The \defn{Lee class} is
 \begin{equation*}
  \mL := [\ell^\theta] \in H_R^1(M;\sP) .
 \end{equation*}
\end{definition}

Note that if $n\geq2$, then $H_R^1(M;\sP)$ can be interpreted as the CR analogue of the Bott--Chern class $H_{\mathrm{BC}}^{1,1}(M)$.
Hence the vanishing of $\mL$ is the CR analogue of the definition of a non-K\"ahler Calabi--Yau manifold~\cite{Tosatti2015}.

If $(M^{2n+1},T^{1,0})$ is a closed, strictly pseudoconvex manifold with $n\geq 2$ --- the setting considered by Lee~\cite{Lee1988} in his study of pseudo-Einstein manifolds --- then \cref{sP-resolution} implies that $\mL$ agrees, up to a dimensional constant, with the corresponding element of $H^1(M;\sP)$ introduced by Lee~\cite{Lee1988}*{Proposition~5.2}.
Importantly, the Lee class retains the property that it vanishes if and only if $(M^{2n+1},T^{1,0})$ admits a pseudo-Einstein contact form.

\begin{lemma}
 \label{characterize-pseudo-einstein}
 An orientable CR manifold $(M^{2n+1},T^{1,0})$ admits a pseudo-Einstein contact form if and only if its Lee class vanishes.
\end{lemma}

\begin{proof}
 If $\theta$ is a pseudo-Einstein contact form on $(M^{2n+1},T^{1,0})$, then $\mL=0$.
 
 Conversely, suppose that $\mL=0$.
 Let $\theta$ be a contact form on $(M^{2n+1},T^{1,0})$.
 Since $\mL=0$, there is a $u\in C^\infty(M)$ such that $\ell^\theta=id\dbbar u$.
 We conclude from \cref{lee-form-transformation} that $e^{-u}\theta$ is a pseudo-Einstein contact form on $(M^{2n+1},T^{1,0})$. 
\end{proof}

Recall from \cref{defn:cohomology_maps} that the identity map on $\mS^1$ induces a well-defined morphism $H_R^1(M;\sP)\to H^2(M;\bR)$.
The image of the Lee class under this morphism is a dimensional multiple of the real first Chern class of $T^{1,0}$.

\begin{lemma}
 \label{lee-class-to-H2}
 Let $(M^{2n+1},T^{1,0})$ be a CR manifold.
 Then $\mL\mapsto-\frac{2\pi}{n+2}c_1(T^{1,0})$ under the morphism $H_R^1(M;\sP)\to H^2(M;\bR)$.
\end{lemma}

\begin{proof}
 Recall that $c_1(T^{1,0}):=\frac{i}{2\pi}[\Omega_\mu{}^\mu]\in H^2(M;\bR)$.
 The conclusion follows immediately from \cref{defn:lee-form,defn:lee-class}.
\end{proof}

In particular, the vanishing of the first real Chern class is necessary for the existence of a pseudo-Einstein contact form.
Lee conjectured~\cite{Lee1988} that this condition is also sufficient for closed, strictly pseudoconvex CR manifolds of dimension at least five.
Based on our Hodge decomposition theorem, we propose the following strengthening of Lee's conjecture:

\begin{conjecture}[Lee Conjecture]
 \label{lee-conjecture}
 Let $(M^{2n+1},T^{1,0})$ be a closed, embeddable, strictly pseudoconvex CR manifold.
 Then $(M^{2n+1},T^{1,0})$ admits a pseudo-Einstein contact form if and only if $c_1(T^{1,0})=0$ in $H^2(M;\bR)$.
\end{conjecture}

By \cref{long-exact-sequence}, if $c_1(T^{1,0})=0$, then the Lee class is in the image of the morphism $H^{0,1}(M) \to H_R^1(M;\sP)$.
This leads to Lee's approach to the Lee Conjecture:
Prove that $0$ is in the preimage of $\mL$ under this morphism.
In particular, the Lee Conjecture is valid on CR manifolds for which $H^{0,1}(M)=0$.

\begin{proposition}
 \label{kr-to-lee}
 Let $(M^{2n+1},T^{1,0})$ be a CR manifold with $c_1(T^{1,0})=0$ in $H^2(M;\bR)$.
 If $b^{0,1}=0$, then $(M^{2n+1},T^{1,0})$ admits a pseudo-Einstein contact form.
\end{proposition}

\begin{proof}
 \Cref{long-exact-sequence,lee-class-to-H2} imply that $\mL\in\im\bigl( H^{0,1}(M) \to H_R^1(M;\sP) \bigr)$.
 Therefore $\mL=0$.
 The conclusion follows from \cref{characterize-pseudo-einstein}.
\end{proof}

However, the vanishing of $b^{0,1}$ is not necessary for the existence of a pseudo-Einstein contact form.
This is clear in the case $n=1$, when $H^{0,1}(M)$ is generally infinite-dimensional.
It is also true in higher dimensions:

\begin{example}
 \label{heisenberg-quotient-pseudo-einstein}
 Let $(I^{2n+1},T^{1,0},\theta)$ be as in \cref{heisenberg-quotient}.
 Clearly this is a closed, strictly pseudoconvex, pseudo-Einstein manifold with $b^{0,1}\not=0$.
\end{example}

Lee gave two sufficient conditions for the validity of the Lee Conjecture in terms of pseudohermitian invariants.
His first result~\cite{Lee1988}*{Theorem~E(i)} is that in dimension at least five, the existence of a contact form with nonnegative pseudohermitian Ricci curvature implies the validity of the Lee Conjecture.
The key point is that, under these assumptions, $\dbbar$-harmonic $(0,1)$-forms are parallel.

\begin{proposition}
 \label{lee-from-nonnegative-ricci}
 Let $(M^{2n+1},T^{1,0})$, $n \geq 2$, be a closed, strictly pseudoconvex manifold with $c_1(T^{1,0})=0$ in $H^2(M;\bR)$ which admits a contact form with $R_{\alpha\bar\beta}\geq0$.
 Then it admits a pseudo-Einstein contact form.
\end{proposition}

\begin{proof}
 \Cref{cohomology-map-trivial} implies that $H_R^1(M;\sP)\to H^2(M;\bR)$ is injective.
 The conclusion follows from \cref{lee-class-to-H2,characterize-pseudo-einstein}.
\end{proof}

Lee's second result~\cite{Lee1988}*{Theorem~E(ii)} is that the Lee Conjecture is true on Sasakian manifolds of dimension at least five.
We extend his result to all dimensions.

\begin{proposition}
 \label{lee-from-sasakian}
 Every closed Sasakian manifold with $c_1(T^{1,0})=0$ in $H^2(M;\bR)$ admits a pseudo-Einstein contact form.
\end{proposition}

\begin{proof}
 Let $\theta$ be a torsion-free contact form on $(M^{2n+1},T^{1,0})$.
 Since $c_1(T^{1,0})=0$, it holds that $\ell^\theta$ is $d$-exact.
 
 Suppose first that $n \geq 2$.
 Let $\omega \in \mS^1$.
 Equation~\eqref{eqn:projection-to-E-expression} and \cref{divergence-formula} imply that
 \begin{equation*}
  \omega - \frac{i}{n-1}\theta \wedge \db^\ast\omega + \frac{i}{n-1}\theta \wedge\dbbar^\ast\omega = \omega_{\alpha\bar\beta} \, \theta^\alpha \wedge \theta^{\bar\beta} .
 \end{equation*}
 Since $\theta$ is torsion-free, applying this to Equation~\eqref{eqn:compute-lee-form} yields
 \begin{align*}
  \db^\ast\ell^\theta & = \frac{n-1}{n(n+2)}i\,\db R , \\
  \dbbar^\ast\ell^\theta & = -\frac{n-1}{n(n+2)}i \, \dbbar R .
 \end{align*}
 We conclude from \cref{dbdbbar-lemma-low} that $\mL=0$.
 
 Suppose next that $n = 1$.
 Since $\theta$ is torsion-free, we deduce from \cref{divergence-formula} and Equation~\eqref{eqn:compute-lee-form} that $\ell^\theta \in \im \bigl(i\db^\ast - i\dbbar^\ast \colon \mR^{2,1} \to \CmR^2 \bigr)$.
 We conclude from \cref{dbdbbar-lemma-middle} that $\mL=0$.
 
 The conclusion now follows from \cref{characterize-pseudo-einstein}.
\end{proof}

\begin{remark}
 \label{sasakian-three-manifold}
 \Cref{sasakian-chern-vanishing,lee-from-sasakian} imply that every closed Sasakian three-manifold admits a pseudo-Einstein contact form.

 Indeed, there is a simple proof:
 Let $\theta$ be torsion-free.
 Let $u \in C^\infty(M)$ solve $\Delta_bu = \frac{2}{3}(R - r)$, where $r$ is the average of $R$ with respect to $\theta \wedge d\theta$.
 Since $\theta$ is torsion-free, we conclude from \cref{commutators} and Equation~\eqref{eqn:compute-lee-form} and the identity $d\ell^\theta=0$ that $[\nabla_0,\Delta_b]=0$ and $\nabla_0R = 0$.
 Therefore $\Delta_b\nabla_0u = 0$, and hence $\nabla_0u=0$.
 It follows that $\dhor\db u = \frac{i}{3}\nabla_\alpha R \, \theta \wedge d\theta$.
 We conclude from \cref{justification-of-bigraded-complex} and Equations~\eqref{eqn:compute-lee-form} and~\eqref{eqn:lee-form-transformation} that $e^u\theta$ is pseudo-Einstein.
\end{remark}